%% file: dHSIC_main.tex
\documentclass[10pt,a4paper]{article}

\usepackage[T1]{fontenc}
\usepackage[english]{babel}
\usepackage[utf8]{inputenc}
\usepackage{amsmath,amssymb,amsfonts,mathrsfs}
\usepackage[amsmath,thmmarks]{ntheorem}
\usepackage{mathtools}
\usepackage{dsfont}
\usepackage{algorithm}
\usepackage[noend]{algpseudocode}
\usepackage{a4wide}
\usepackage{authblk}
\usepackage{color}
\usepackage{paralist}
\usepackage{tikz}
\usepackage{tkz-graph}
\usetikzlibrary{shapes.geometric}
\usepackage[framemethod=TikZ]{mdframed}
\usepackage{subcaption}
\usepackage{placeins}
\usepackage{multirow}

\input{Content/macrosetup}
\input{Content/theoremsetup}

\newcommand\Peter[1]{{\color{red}{Peter: ``#1''}}}

\usepackage{natbib}
\begin{document}

\title{Kernel-based Tests for Joint Independence}
\author[1]{Niklas Pfister}
\affil[1]{Seminar für Statistik, ETH Zürich, Switzerland}
\author[1]{Peter B\"uhlmann}
\author[2]{Bernhard Sch\"olkopf}
\author[2]{Jonas Peters}
\affil[2]{MPI for Intelligent Systems, Tübingen, Germany \textit{and}
Department of Mathematical Sciences, University of Copenhagen,
Denmark.}
\date{\today}
\maketitle

\input{Content/abstract}

\input{Content/introduction}
\input{Content/dhsic}
\input{Content/estimation}
\input{Content/testing}
\input{Content/implementation}
\input{Content/experiments}
\input{Content/conclusions}

\section*{Acknowledgements}
The authors thank the Editor, an Associate Editor and some anonymous
reviewers for constructive comments, and Christoph Lampert, Arthur
Gretton, Gilles Blanchard and Aaditya Ramdas for helpful
discussion. JP received support from the ETH-MPI learning center.

\bibliography{Content/refs}

\appendix
\input{Content/appendix}


\end{document}

%% file: Content/macrosetup.tex
\newcommand{\N}{\mathbb{N}}

\newcommand{\R}{\mathbb{R}}

\DeclarePairedDelimiter\abs{\lvert}{\rvert}
\DeclarePairedDelimiter\norm{\lVert}{\rVert}
\DeclarePairedDelimiter\scalarprod{\langle}{\rangle}

\renewcommand{\epsilon}{\ensuremath\varepsilon}

\DeclareMathAlphabet\mathbfcal{OMS}{cmsy}{b}{n}

\renewcommand{\phi}{\ensuremath{\varphi}}

\renewcommand{\hat}[1]{\widehat{#1}}
\newcommand{\sm}{-}
\newcommand{\E}{\mathbb{E}}
\renewcommand{\P}{\mathbb{P}}
\newcommand{\Var}{\operatorname{Var}}
\newcommand{\iid}{\overset{\text{\tiny iid}}{\sim}}
\newcommand{\normaldist}[2]{\mathcal{N}\left(#1,#2\right)}
\newcommand{\mapset}[2]{\mathbf{M}_{#1}(#2)}
\newcommand{\permset}[2]{\mathbf{P}_{#1}(#2)}
\newcommand{\combset}[2]{\mathbf{C}_{#1}(#2)}
\newcommand{\permsize}[2]{(#1)_{#2}}

\newcommand{\Ustat}[2]{U_{#1}(#2)}
\newcommand{\Ustatb}[2]{U^*_{#1}(#2)}
\newcommand{\Ustatr}[2]{\tilde{U}_{#1}(#2)}
\newcommand{\Vstat}[2]{V_{#1}(#2)}
\newcommand{\Vstatr}[2]{\tilde{V}_{#1}(#2)}

\newcommand{\prodk}{\mathbf{k}}
\newcommand{\RKHS}{\mathcal{H}}
\newcommand{\prodRKHS}{\mathbfcal{H}}

\newcommand{\Bspace}{\mathcal{B}}
\newcommand{\kernelspace}{\mathcal{X}}

\newcommand{\prodkernelspace}{\mathbfcal{X}}
\newcommand{\uspace}{\mathcal{X}}

\newcommand{\lpspace}[3]{\mathcal{L}^{#1}(#2,#3)}
\newcommand{\Lpspace}[3]{L^{#1}(#2,#3)}
\newcommand{\bmspace}[1]{\mathcal{M}(#1)}
\newcommand{\pmspace}[1]{\mathcal{P}(#1)}

\newcommand{\linfF}{l^{\infty}(\mathcal{F})}
\newcommand{\linfG}{l^{\infty}(\mathcal{G})}

\newcommand{\mhsic}{\operatorname{dHSIC}}
\newcommand{\mhsicb}{\operatorname{\widehat{dHSIC}}}

\newcommand{\va}{\mathbf{a}}

\newcommand{\vN}{\mathbf{N}}
\newcommand{\vt}{\mathbf{t}}

\newcommand{\vx}{\mathbf{x}}
\newcommand{\vy}{\mathbf{y}}
\newcommand{\vz}{\mathbf{z}}
\newcommand{\vX}{\mathbf{X}}

\newcommand{\vR}{\mathbf{R}}

\newcommand{\Plaw}[1]{\mathbb{P}^{X^{#1}}}
\newcommand{\plaw}{\mathbb{P}^X}
\newcommand{\Plawpower}[1]{(\mathbb{P}^{X})^{\otimes #1}}
\newcommand{\prodlaw}{\mathbb{P}^{X^{1}}\otimes\dots\otimes\mathbb{P}^{X^{d}}}
\newcommand{\prodlawR}{\mathbb{P}^{R^{1}}\otimes\dots\otimes\mathbb{P}^{R^{d}}}
\newcommand{\prodlawbold}{\mathbb{P}^{\vX}}
\newcommand{\jointlaw}{\mathbb{P}^{(X^{1},\dots,X^{d})}}
\newcommand{\jointlawb}{\mathbb{P}^{\vX}}

\newcommand{\empprodlaw}{\hat{\mathbb{P}}_n^{X^1}\otimes\dots\otimes\hat{\mathbb{P}}_n^{X^d}}
\newcommand{\Pjoint}{\P}

\newcommand{\empjoint}{\hat{\Pjoint}_n}
\newcommand{\empprod}{\hat{\P}^{\otimes}_n}

\newcommand{\measure}[2]{\,#1(\text{d}#2)}

\newcommand{\HO}{H_0}
\newcommand{\HA}{H_A}

\newcommand{\landauO}[1]{\mathcal{O}\left(#1\right)}

\newcommand{\landauProbO}[1]{\mathcal{O}_{\P}\left(#1\right)}

\newcommand{\independent}{\perp\!\!\!\perp}

\newcommand{\parents}{\mathbf{PA}}


%% file: Content/theoremsetup.tex
\numberwithin{equation}{section}

\theoremstyle{break}
\newtheorem{theorem}{Theorem}[section]

\newtheorem{definition}[theorem]{Definition}
\newtheorem{lemma}[theorem]{Lemma}
\newtheorem{proposition}[theorem]{Proposition}

\newtheorem{setting}{Setting}
\theoremstyle{plain}
\newtheorem{remark}[theorem]{Remark}
\newtheorem*{bochner}{Bochner's Theorem}

\theoremstyle{nonumberplain}
\theorembodyfont{\normalfont}
\theoremsymbol{\ensuremath{\square}}
\newtheorem{proof}{Proof}

%% file: Content/abstract.tex
\begin{abstract}
  We investigate the problem of testing whether $d$ 
  possibly multivariate random variables, which may or may not be
  continuous, are jointly (or mutually) independent.  Our method
  builds on ideas of the two variable Hilbert-Schmidt independence
  criterion (HSIC) but allows for an arbitrary number of variables.
  We embed the joint distribution and the product of the marginals
  into a reproducing kernel Hilbert space and define the $d$-variable
  Hilbert-Schmidt independence criterion (dHSIC) as the squared
  distance between the embeddings.  In the population case, the value
  of dHSIC is zero if and only if the $d$ variables are jointly
  independent, as long as the kernel is characteristic. Based on an
  empirical estimate of dHSIC, we investigate three different
  non-parametric hypothesis tests: a permutation test, a bootstrap
  analogue and a procedure based on a Gamma approximation.
We apply nonparametric independence testing to a problem in causal
discovery and illustrate the new methods on simulated and a real dataset. 
\end{abstract}


%% file: Content/introduction.tex
\section{Introduction} \label{sec:intro}

We consider the problem of nonparametric testing for joint or mutual
independence of $d$ random variables. This is a very different and
more ambitious task than testing pairwise independence of a collection
of random variables.  Consistent pairwise nonparametric independence
tests date back to \citet{feuerverger1993} and \citet{Romano1986} and
have more recently received considerable attention using kernel-based
methods \citep{grettonHSIC,
  gretton08hsic}, 
and other related approaches for estimating or testing pairwise
(in-)dependence including distance correlations
\citep{szekely2009,szekely2014},
rank-based correlations \citep{bergsma2014, Leung2015, Nandy2016} or
also non- and semiparametric copula-based correlations
\citep{liu2012,xue2012,wegkamp2016,gaisser2010}.

One of our motivations to develop methods for nonparametric testing of
joint independence originates from the area of causal inference, and
we discuss this in Section~\ref{subsec.causal}: there, inferring
pairwise independence is not sufficient as those models assume the
existence of jointly independent noise variables. Our test can
therefore be used as a goodness-of-fit test and for model selection,
see Section~\ref{subsec.causal}. A further interesting application of
joint independence tests is independent component analysis. While many
algorithms use a step wise approach to construct the collection of
independent features, a more direct option is to explicitly minimize a
measure of mutual dependence (such as our dHSIC), for more details see
\citet{chen2006} or \citet{matteson2016}.\footnote{We thank an anonymous
  referee for pointing out this interesting application.}

For testing joint independence, consider the distribution $\jointlaw$
of the random vector $\vX=(X^1,\dots,X^d)$.\footnote{Throughout the
  paper, a superscript of $X$ always denotes an index rather than an
  exponent.}
By definition, $(X^1,\dots,X^d)$ are jointly or mutually independent
if and only if $\jointlaw = \prodlaw$.
For a given positive definite kernel, we map both distributions into
the reproducing kernel Hilbert space (see Section~\ref{sec:21} for
details) and consider their squared distance. Such a mapping can in
fact be seen as a generalization of the $L^2$ distance between
``traditional'' kernel density estimators, see the discussion on
page~732 in \citep{Gretton2012jmlr}. For characteristic kernels (e.g.,
the popular Gaussian kernel), the embedding of Borel probability
measures is injective and the squared distance is zero if and only if
the variables are jointly independent. For the finite sample case, we
compute a suitable estimator that can be used as a test statistic. We
then construct three statistical tests: two tests are based on
permutation and bootstrap procedures, respectively, and a third test
approximates the distribution of the test statistic under independence
with a Gamma distribution.  Our statistic extends the Hilbert-Schmidt
Independence Criterion (HSIC) \citep{grettonHSIC} and contains it as a
special case. We therefore call the corresponding test procedure
$d$-variable Hilbert-Schmidt Independence Criterion (dHSIC).  We prove
that the permutation based approach has correct level and that the
bootstrap approach has pointwise asymptotic level and is consistent in
the sense of \eqref{eq:consistencydef}.

In the literature, other mutual independence tests have been
proposed. One of them is a characteristic function based nonparametric
mutual independence test due to \citet{kankainen}, see
Section~\ref{sec:charfun} for details.  The dependence measure is a
weighted integral over the difference between the characteristic
functions of the joint and the product distribution.  Different weight
functions result in different quantities, all of them are special
cases of dHSIC for an appropriate choice of kernel.
We show
that our results carry over to the characteristic function
framework, whereas the opposite direction only works for a restrictive
class of kernels. Moreover, while \citet{kankainen} does prove similar
results about the asymptotic distribution of the test statistic as
given in Theorem~\ref{thm:asymptoticdist_HO}, her proof cannot be
directly extended to our more general framework. This is one of the
reasons, why we developed some of our general results about
V-statistics. The test in \citet{kankainen} is 
shown
to be
consistent, but the word consistency there refers to the property that the
asymptotic distribution of the test statistic under the alternative
hypothesis diverges; instead, we employ the commonly used 
definition 
that a test is consistent if the testing procedure itself (in our case
the bootstrap) is consistent in the sense of~\eqref{eq:consistencydef}. Our consistency results immediately carry over to the characteristic function framework, as it is contained as a special case of dHSIC.

\citet{bakirov2006} use an independence coefficient as measure of
dependence, 
which is strongly related to the characteristic function
approach of \citet{kankainen}. They approximate the asymptotic test
statistic, 
which is also a sum of chi-squared distributed random
variables, 
using tail bounds. This results in a test that has
(conservative) asymptotic level in the sense of
\eqref{eq:asympleveldef}. Due to the conservative bounds, however, the test is, in general, not consistent in the sense of~\eqref{eq:consistencydef}.

One test for which a consistency result~
\eqref{eq:consistencydef} has been shown, is an older method based on
work by \citet{beranmiller} and \citet[][p.27]{Romano1986}; it
does not seem to be used in practice very often. As a test statistic,
it takes the maximal difference between the empirical distribution and
the product of its marginals over a class of sets.  One then chooses a
distribution over sets and approximates this infinite class by
$C < \infty$ randomly chosen sets, see Section~\ref{sec:competing} for
details. This makes the construction impractical with a rather ad-hoc
computational implementation. In our experiments, we found that this
test has less power than dHSIC and is computationally more demanding,
even for moderate values of $C$. 

Both, this test and the
characteristic function based tests mentioned above are restricted
to the Euclidean space, dHSIC allows for more general kernels such as
kernels on graphs or strings \citep[see][]{gretton08hsic}.

Finally, it is possible to use the following alternative procedure
that constructs a joint independence test from a bivariate test: joint
independence holds if and only if for all $k \in \{2, \ldots, d\}$ we
have that $X^k$ is independent of $(X^1, \ldots, X^{k-1})$.  In order
to construct a mutual independence test, we can therefore perform
$d-1$ statistical tests and combine the results using a Bonferroni
correction.  However, such a procedure is asymmetric in the $d$ random
variables and depends on the order of the random variables.
Furthermore, it is known that the Bonferroni correction is often
conservative and due to performing $d-1$ tests, such a test is
of order $d$ times more computationally expensive than the direct
dHSIC approach, see Section~\ref{sec:computational}.


\subsection{Contribution}
This work extends the two variable HSIC
\citep{grettonHSIC,gretton08hsic,smola} to testing joint independence
for an arbitrary number of variables. The resulting test, moreover,
extends the work of \citet{kankainen} to the more flexible framework
of kernel methods (see Section~\ref{sec:charfun}) and establishes
consistency, as mentioned also in the previous section. While the
dHSIC test statistic has been briefly mentioned by \citet{lancester},
the derivation of the general results about asymptotic distributions
(Theorem~\ref{thm:asymptoticdist_HO} and
Theorem~\ref{thm:corollaryasymptoticdistmhsic}) as well as the
mathematical rigorous treatment of the permutation test and the
bootstrap test are novel: this concerns results for both types of tests about
their level (type I error) in
Proposition~\ref{thm:level_permutation_test} and
Theorem~\ref{thm:level_bootstrap_test}, and the consistency
(asymptotic power) of the bootstrap in
Theorem~\ref{thm:consistencybootstraptest}. In fact, the consistency
result is quite remarkable, establishing asymptotic consistency for any
fixed alternative. It is the first such result for kernel based
methods and maybe the first result for a practically feasible test for
joint independence having asymptotic error control and asymptotic
power equal to one.  We also prove that under the null hypothesis it
holds that $\xi_2(h)>0$, which has been implicitly assumed in for
example \citet[Theorem 2]{gretton08hsic}.
For the Gamma approximation based test, we compute general formulas
both for the mean and for the variance in
Proposition~\ref{thm:meanofmhsic} and
Proposition~\ref{thm:varianceofmhsic}, respectively. In order to make
our tests accessible we have created an R-package \citep{R} called
\texttt{dHSIC}, which is available on CRAN. Moreover, we have applied
our dHSIC to real data in causality, showing its usefulness also in
applied settings, both in terms of model selection and goodness-of-fit
test.

To establish these properties, we derive new results for V-statistics
which we collect in Appendix~\ref{appendix:vstat} and which are
of independent interest. In particular,
Lemma~\ref{thm:decompositionofmVstat} (asymptotic difference between
U- and V-statistics), Theorem~\ref{thm:var_vstat} (asymptotic variance
of a V-statistic), Theorem~\ref{thm:biasVstat} (asymptotic bias of a
V-statistic), Theorem~\ref{thm:asymptoticdist_vstat2} (asymptotic
distribution of a degenerate V-statistic) and
Theorem~\ref{thm:resamplingvstat} (asymptotic distribution of a
degenerate resampled V-statistic).

%


%% file: Content/dhsic.tex
\section{Hilbert-Schmidt independence criterion for $d$ variables} \label{sec:2}

\subsection{Reproducing kernel Hilbert spaces} \label{sec:21}
We present here a brief introduction to reproducing kernel Hilbert spaces
and the theory of mean embeddings. Given a set $\kernelspace$ we call
a function $k:\kernelspace\times\kernelspace\rightarrow\R$ a positive
semi-definite kernel if for any set of points
$(x_1,\dots,x_n)\in\kernelspace^n$ the corresponding Gram matrix
$(k(x_i,x_j))_{1\leq i,j\leq n}$ is symmetric and positive
semi-definite. Moreover, denote by $\mathcal{F}(\kernelspace)$ the
space of functions from $\kernelspace$ to $\R$. Reproducing kernel
Hilbert spaces on $\kernelspace$ are well-behaved sub-classes of
$\mathcal{F}(\kernelspace)$ defined as follows. 
\begin{definition}[Reproducing kernel Hilbert space]
  Let $\kernelspace$ be a set, let
  $\RKHS\subseteq\mathcal{F}(\kernelspace)$ be a Hilbert space. Then
  $\RKHS$ is called a \emph{reproducing kernel Hilbert space (RKHS)} if there
  exists a kernel $k$ on $\kernelspace$ satisfying
  \begin{compactitem}
    \item[(i)] $\forall x\in\kernelspace$: $k(x,\cdot)\in\RKHS$;
    \item[(ii)] $\forall f\in\RKHS,$ $\forall x\in\kernelspace$:
      $\scalarprod{f, k(x,\cdot)}_{\scriptscriptstyle \RKHS}=f(x)$.
  \end{compactitem}
  Moreover, we call $k$ a \emph{reproducing kernel} of $\RKHS$.
\end{definition}
It can be shown that for any positive semi-definite kernel $k$ there
exists an RKHS with reproducing kernel $k$. Given any positive
semi-definite kernel, we can therefore construct and use the
corresponding RKHS. A commonly used positive semi-definite kernel on
$\R^m$ is the Gaussian kernel, defined for all $x,y\in\R^m$ by
\begin{equation}
  \label{eq:GaussianKernel}
  k(x,y)=\operatorname{exp}\left(-\dfrac{\norm{x-y}^2_{\scriptscriptstyle \R^m}}{2\sigma^2}\right).
\end{equation}

One of the strengths of RKHS is that they can be used for embedding
complicated objects in order to use the Hilbert space structure to
analyze them. Being able to express inner products as function
evaluations via the reproducing property additionally simplifies
computation within an RKHS. In this paper, we use such an embedding
technique to analyze probability distributions. To this end,
we use the Bochner integral to define an embedding of
$\bmspace{\kernelspace}\coloneqq\{\mu\mid \mu \text{ is a finite Borel
  measure on }\kernelspace\}$ into an RKHS.
\begin{definition}[mean embedding function]
  Let $\kernelspace$ be a separable metric space, let $k$ be a
  continuous bounded positive semi-definite kernel and let $\RKHS$ be
  the RKHS with reproducing kernel $k$. Then, the function
  $\Pi:\bmspace{\kernelspace}\rightarrow\RKHS$ satisfying for all
  $\mu\in\bmspace{\kernelspace}$ that
  \begin{equation*}
    \Pi(\mu)=\int_{\kernelspace}k(x,\cdot)\measure{\mu}{x}.
  \end{equation*}
  is called the \emph{mean embedding} (associated to $k$).
\end{definition}
In order to infer that two distributions are equal given that their
embeddings coincide, it is necessary that the mean embedding is
injective. A kernel is called characteristic if the mean embedding
$\Pi$ is injective \citep[see][]{fukumizu07}. The Gaussian
kernel~\eqref{eq:GaussianKernel} on $\R^m$, for example, is
characteristic \citep[e.g.][Theorem 7]{gretton08}.

\subsection{Definition of dHSIC and independence property} \label{sec:22}
Our goal is to develop a non-parametric
hypothesis test to determine whether the components of a random vector
$\vX=(X^1,\dots,X^d)$ are mutually independent, based on $n$ iid
observations $\vX_1,\dots,\vX_n$ of the vector $\vX$. The variables
$X^1,\dots,X^d$ are mutually independent if and only if
\begin{equation*}
  \prodlaw=\jointlaw.
\end{equation*}
The central idea is to embed both $\prodlaw$ and $\jointlaw$ into an
appropriate RKHS and then check whether the embedded elements are
equal. To keep an overview of all our assumptions, we summarize the setting
used throughout the rest of this work.

\begin{setting}[dHSIC]
  \label{setting:mHSIC}
  For all $j\in\{1,\dots,d\}$, let $\kernelspace^j$ be a separable metric
  space and denote by
  $\prodkernelspace=\kernelspace^1\times\cdots\times\kernelspace^d$
  the product space. Let $(\Omega,\mathcal{F},\P)$ be a probability
  space and for every $j\in\{1,\dots,d\}$, let $X^j:\Omega \rightarrow
  \kernelspace^j$ be a random variable with law $\Plaw{j}$. Let
  $(\vX_i)_{i\in\N}$ be a sequence of iid copies of
  $\vX=(X^1,\dots,X^d)$. For $j\in\{1,\dots,d\}$, let
  $k^j:\kernelspace^j\times\kernelspace^j\rightarrow\R$ be a
  continuous, bounded, positive semi-definite
  kernel on $\kernelspace^j$ and denote by $\RKHS^j$ the corresponding
  RKHS. Moreover, assume that the tensor product of the kernels $k^j$
  denoted by $\prodk=k^1\otimes\cdots\otimes k^d$ is
  characteristic\footnote{\citet{gretton2015} claims that this can be
    weakened to simply assuming that the individual kernels are
    characteristic.} and let
  $\prodRKHS=\RKHS^1\otimes\cdots\otimes\RKHS^d$ be the (projective)
  tensor product of the RKHSs $\RKHS^j$. Let
  $\Pi:\bmspace{\prodkernelspace}\rightarrow\prodRKHS$ be the mean
  embedding function associated to $\prodk$.
\end{setting}
It is straightforward to show that this setting ensures that
$\prodRKHS$ is an RKHS with reproducing kernel~$\prodk$, that $\prodk$
is continuous and bounded, that $\prodRKHS$ is separable and only
contains continuous functions, and that $\Pi$ is injective. Using this
setting we can extend the Hilbert-Schmidt independence criterion
(HSIC) from two variables as described by \citet{gretton08hsic} to the
case of $d$ variables. The extension is based on the HSIC
characterization via the mean embedding described by \citet{smola}.

\begin{definition}[dHSIC]\label{def:dHSIC}
  Assume Setting \ref{setting:mHSIC}. Then, define
  the statistical functional
  \begin{equation*}
    \mhsic\left(\jointlaw\right)\coloneqq\norm[\Big]{\Pi\left(\prodlaw\right)-\Pi\left(\jointlaw\right)}_{\scriptscriptstyle \prodRKHS}^2
  \end{equation*}
  and call it the \emph{$d$-variable Hilbert-Schmidt independence criterion (dHSIC)}.
\end{definition}
Therefore, dHSIC is the distance between the
joint measure and the product measure after embedding them into an
RKHS. Since the mean embedding $\Pi$ is injective we get the following
relation between dHSIC and joint  independence.

\begin{proposition}[independence property of dHSIC]
  \label{thm:mhsic_independence}
  Assume Setting \ref{setting:mHSIC}. Then it holds
  that
  \begin{equation*}
    \mhsic(\jointlaw) =0 \quad\Longleftrightarrow\quad \prodlaw=\jointlaw.
  \end{equation*}
\end{proposition}

\begin{proof}
  This statement follows from the definiteness of the
  norm and the fact that $\Pi$ is injective.
\end{proof}
This proposition implies that we can use dHSIC as a measure of joint
dependence between variables. In order to make $\mhsic$
accessible for calculations, we express it in terms of the individual
kernels $k^1,\dots,k^d$. This expansion will be the basis of the estimator defined in Section~\ref{sec:23}.
\begin{proposition}[expansion of dHSIC]
  \label{thm:mhsickernel}
  Assume Setting \ref{setting:mHSIC}. Then it holds that
  \begin{equation*}
    \mhsic=\E\left(\prod_{j=1}^dk^j\left(X^j_1,X^j_2\right)\right)
    +\E\left(\prod_{j=1}^dk^j\left(X^j_{2j-1},X^j_{2j}\right)\right)
    -2\E\left(\prod_{j=1}^dk^j\left(X^j_1,X^j_{j+1}\right)\right)
  \end{equation*}
\end{proposition}
A proof is given in Appendix~\ref{proof:mhsickernel}.


%% file: Content/estimation.tex
\subsection{Estimating \lowercase{d}HSIC} \label{sec:23}

Our estimator will be constructed using several V-statistics. We
therefore start by summarizing a few well-known definitions and the
most important results from the theory of V-statistics. Readers familiar with these topics may skip directly to
Definition~\ref{def:mhsic_esimator}.

Let $n\in\N$, $q\in\{1,\dots,n\}$, let $\uspace$ be a metric space,
$(\Omega,\mathcal{F},\P)$ a probability space, $X:\Omega \rightarrow
\uspace$ a random variable with law $\Plaw{}$ and let $(X_i)_{i\in\N}$ be a
sequence of iid copies of $X$, i.e.,
$(X_i)_{i\in\N}\iid\P^{X}$. 
\if0 
Furthermore, define the set of all
combinations on $\{1,\dots,n\}$ by
\begin{equation*}
  \combset{q}{n}\coloneqq\left\{(i_1,\dots,i_q)\in\{1,\dots,n\}:i_1<\cdots<i_q\right\}
\end{equation*}
and the set of all mappings on $\{1,\dots,n\}$ by
\begin{equation*}
  \mapset{q}{n}\coloneqq\left\{1,\dots,n\right\}^q.
\end{equation*}
Observe that $\abs{\combset{q}{n}}=\binom{n}{q}$ and
$\abs{\mapset{q}{n}}=n^q$. Moreover, consider a measurable and symmetric
(i.e., invariant under any permutation of its input arguments) function
$g:\uspace^q\rightarrow\R$. Suppose we are interested in the
statistical functional
\begin{equation}
  \label{eq:statfun}
  \theta_g\coloneqq\theta_g\left(\Plaw{}\right)\coloneqq\E\left(g(X_1,\dots,X_q)\right).
\end{equation}
To this end we define two estimators: The U-statistic
\begin{equation}
  \label{eq:ustat}
  \Ustat{n}{g}\coloneqq\dbinom{n}{q}^{-1}\sum_{\combset{q}{n}}g(X_{i_1},\dots,X_{i_q}),
\end{equation}
and the V-statistic
\begin{equation}
  \label{eq:vstat}
  \Vstat{n}{g}\coloneqq\dfrac{1}{n^q}\sum_{\mapset{q}{n}}g(X_{i_1},\dots,X_{i_q}).
\end{equation}
It is straightforward to see that the U-statistic is unbiased. Both
U-and V-statistics are important. For practical applications, it is
generally easier to work with V-statistics because the sets
$\combset{q}{n}$ and $\permset{q}{n}$ become quite complicated for
$q>2$. Whenever deriving theoretical results, however, it turns out to
be easier to first consider U-statistics and then deduce the
corresponding result for V-statistics. The theory for U-statistics has
been considered in literature quite extensively and we refer to
\citet{serfling} for an exhaustive summary. V-statistics on the other
hand are more difficult to analyze and most results are only given for
$q=2$. For most applications this is sufficient, however, for dHSIC we
use V-statistics with $q=2d$ and therefore need more general
results. A comprehensive summary together with all proofs can be found
in Appendix~\ref{appendix:vstat}.
\else 
Furthermore, define $\mapset{q}{n}\coloneqq\left\{1,\dots,n\right\}^q$
as the $q$-fold Cartesian product of the set $\{1,\dots,n\}$.
Moreover, consider a measurable and symmetric
(i.e., invariant under any permutation of its input arguments) function
$g:\uspace^q\rightarrow\R$, which we denote as core function. 
The V-statistic
\begin{equation}
  \label{eq:vstat}
  \Vstat{n}{g}\coloneqq\dfrac{1}{n^q}\sum_{\mapset{q}{n}}g(X_{i_1},\dots,X_{i_q})
\end{equation}
estimates the 
statistical functional
\begin{equation*}
  \theta_g\coloneqq\theta_g\left(\Plaw{}\right)\coloneqq\E\left(g(X_1,\dots,X_q)\right).
\end{equation*}
As opposed to U-statistics, defined in~\eqref{eq:ustat} in
Appendix~\ref{appendix:vstat}, V-statistics are usually biased.  In
this work, we nevertheless consider a V-statistic because it can be
computed much faster than the corresponding U-statistic; this is in
particular the case if $q>2$. While U-statistics have been extensively studied
\citep[e.g.][]{serfling}, results for V-statistics are often
restricted to $q=2$. Since for dHSIC we use V-statistics with $q=2d$
(see Lemma~\ref{thm:properties_h}), we need more general
results that are derived in Appendix~\ref{appendix:vstat}.  \fi

The following notation appears throughout the paper in the context of
V-statistics and is also common in the theory of U-statistics,
see~\citet[Section 5.1.5]{serfling}. Given the core function 
$g:\uspace^q\rightarrow\R$ we define for every $c\in\{1,\dots,q-1\}$
the function $g_c:\uspace^c\rightarrow\R$ by
\begin{equation*}
  g_c(x_1,\dots,x_c)\coloneqq\E\left(g(x_1,\dots,x_c,X_{c+1},\dots,X_q)\right)
\end{equation*}
and $g_q\equiv g$. Then, $g_c$ is again a symmetric core function such that
for every $c\in\{1,\dots,q-1\}$, we have
\begin{equation*}
  \E\left(g_c(X_1,\dots,X_c)\right)=\E\left(g(X_1,\dots,X_q)\right)=\theta_g.
\end{equation*}
Further define $\tilde{g}\equiv g-\theta_g$ and
for all $c\in\{1,\dots,q\}$ define $\tilde{g}_c\equiv g_c-\theta_g$ to
be the centered versions of the core functions. Moreover, define for
every $c\in\{1,\dots,q\}$,
\begin{equation}
  \label{eq:xi_c}
  \xi_c\coloneqq\Var\left(g_c(X_1,\dots,X_c)\right)=\E\left(\tilde{g}_c(X_1,\dots,X_c)^2\right).
\end{equation}
We sometimes write $\xi_c(g)$ to make clear which core function
we are talking about.

We define an estimator for $\mhsic$ by estimating each of the
expectation terms in Proposition~\ref{thm:mhsickernel} by a V-statistic.
\begin{definition*}[$\mhsicb$]
  \label{def:mhsic_esimator}
  Assume Setting \ref{setting:mHSIC}.
  For all $(\vx_1,\dots,\vx_n)\in\prodkernelspace^n$ define the
  estimator $\mhsicb=(\mhsicb_n)_{n\in\N}$ as 
  \begin{equation*}
    \begin{split}
      \mhsicb_n(\vx_1,\dots,\vx_n)
      &\coloneqq\dfrac{1}{n^2}\sum_{\mapset{2}{n}}\prod_{j=1}^dk^j\left(x^j_{i_1},x^j_{i_2}\right)
      +\dfrac{1}{n^{2d}}\sum_{\mapset{2d}{n}}\prod_{j=1}^dk^j\left(x^j_{i_{2j\sm
            1}},x^j_{i_{2j}}\right)\\
      &\qquad-\dfrac{2}{n^{d+1}}\sum_{\mapset{d+1}{n}}\prod_{j=1}^dk^j\left(x^j_{i_1},x^j_{i_{j+1}}\right).
    \end{split}
  \end{equation*}
if  $n\in\{2d,2d+1,\dots\}$
  and
  as $\mhsicb_n(\vx_1,\dots,\vx_n)\coloneqq 0$ if
  $n\in\{1,\dots,2d-1\}$. 
\end{definition*}
Whenever it is clear from the context, we drop the
functional arguments and just write $\mhsicb_n$ instead of
$\mhsicb_n(\vX_1,\dots,\vX_n)$. In order to make this estimator more
accessible for analysis we can express it as a V-estimator with a
single core function. To this end, define
$h:\prodkernelspace^{2d}\rightarrow\R$ to be the function satisfying
for all $\vz_1,\dots,\vz_{2d}\in\prodkernelspace$ that
\begin{equation}
  \label{eq:hcore}
  \begin{split}
    h(\vz_1,\dots,\vz_{2d})=\dfrac{1}{(2d)!}\sum_{\pi\in
      S_{2d}}\Bigg[&\prod_{j=1}^dk^j\left(z^j_{\pi(1)},z^j_{\pi(2)}\right)+
    \prod_{j=1}^dk^j\left(z^j_{\pi(2j-1)},z^j_{\pi(2j)}\right)\\
    &\quad-2\prod_{j=1}^dk^j\left(z^j_{\pi(1)},z^j_{\pi(j+1)}\right)\Bigg],
  \end{split}
\end{equation}
where $S_{2d}$ is the set of permutations on $\{1,\dots,2d\}$. The
following proposition shows that $\mhsicb$ is a V-statistic with core
function $h$.

\begin{lemma}[properties of the core function $h$]
  \label{thm:properties_h}
  Assume Setting \ref{setting:mHSIC}. It holds that the function $h$
  defined in \eqref{eq:hcore} is symmetric, continuous, and there
  exists $C>0$ such that for all
  $\vz_1,\dots,\vz_{2d}\in\prodkernelspace$ we have
  $\abs{h(\vz_1,\dots,\vz_{2d})}<C$. Moreover,
  $\Vstat{n}{h}=\mhsicb_n$, see~\eqref{eq:vstat}, and
  $\theta_h=\E\left(h(\vX_1,\dots,\vX_{2d})\right)=\mhsic$.
\end{lemma}
A proof is given in Appendix~\ref{proof:properties_h}.

\subsection{Characteristic function framework}\label{sec:charfun}

\citet{kankainen} considers a characteristic function based mutual
independence test.
She considers a weighted integral over the difference between the characteristic
functions of the joint and the product distribution. For a weight
function $g$, the resulting
empirical test statistic \citep[p.25]{kankainen} is given by
\begin{equation*}
  \begin{split}
    T_n(\vx_1,\dots,\vx_n)
    &\coloneqq n\cdot\bigg[\dfrac{1}{n^2}\sum_{i_1,i_2}\prod_{j=1}^d\int_{\R}e^{it^j(x_{i_1}^j-x_{i_2}^j)}g_j(t^j)\text{d}t^j
    +\dfrac{1}{n^{2d}}\prod_{j=1}^d\sum_{i_1,i_2}\int_{\R}e^{it^j(x_{i_1}^j-x_{i_2}^j)}g_j(t^j)\text{d}t^j\\
    &\qquad-\dfrac{2}{n^{d+1}}\sum_{i_1}\prod_{j=1}^d\sum_{i_2}\int_{\R}e^{it^j(x_{i_1}^j-x_{i_2}^j)}g_j(t^j)\text{d}t^j\bigg].
  \end{split}
\end{equation*}
The characteristic function framework is a contained in the dHSIC framework as a special case.
We recover our dHSIC test statistic by choosing 
\begin{equation}
  \label{eq:connection_cf_k}
  \prodk(\vx,\vy)=\prod_{j=1}^d\int_{\R}e^{it^j(x^j-y^j)}g_j(t^j)\text{d}t^j.
\end{equation}
The choice of this kernel is justified by 
Bochner's theorem
\citep[e.g.][Theorem B.1]{unser}.
\begin{bochner}
  Let $f$ be a bounded continuous function on $\R^d$. Then, $f$ is
  positive semi-definite if and only if it is the (conjugate) Fourier
  transform of a nonnegative and finite Borel measure $\mu$, i.e.
  \begin{equation*}
    f(\vx)=\int_{\R^d}e^{i\scalarprod{\vx,\vt}}\measure{\mu}{\vt}.
  \end{equation*}
\end{bochner}
Given the characteristic function framework with a weight function $g$
satisfying the properties 1. to 5. in \citet[p. 25]{kankainen} it
holds that the measure
$\mu_g(\mathbf{B})\coloneqq
\int_{\mathbf{B}}\prod_{j=1}^dg_j(t^j)\text{d}t^j$ is a nonnegative
finite Borel measure on $\R^d$ and hence $\prodk$ defined as in
\eqref{eq:connection_cf_k} is a positive semi-definite kernel. The
setting given in \citet{kankainen} is thus entirely contained within
our dHSIC framework. 

Furthermore, the dHSIC framework is strictly more general.
To
see this, let $\prodk$ be a continuous bounded stationary positive
semi-definite kernel on $\R^d$. Then, by stationarity there exists a
continuous bounded function $f$ on $\R^d$ such that
$\prodk(\vx,\vy)=f(\vx-\vy)$ and hence by Bochner's theorem there
exists a measure $\mu\in\bmspace{\R^d}$ such that
\begin{equation*}
  \prodk(\vx,\vy)=\int_{\R^d}e^{i\scalarprod{\vx-\vy,\vt}}\measure{\mu}{\vt}.
\end{equation*}
This is, however, still more general than the setting in
\citet{kankainen} as there it is additionally assumed that the measure
$\mu$ is absolutely continuous with a density $g$ satisfying
properties 1.\ to 5.\ 
which in particular requires that $g$ is a simple
product, i.e.\ $g(t)=\prod_{j=1}^dg_j(t^j)$ and that the components
$g_j$ are even. Both of these conditions are essential to the
proofs given in \citet{kankainen}. Therefore the results from the
characteristic function framework in \citet{kankainen} cannot be
transferred to our more general dHSIC setting. Also note that the
characteristic function framework is restricted to real-valued
domains, while kernels are more flexible, e.g.\ kernels on graphs
or strings \citep[see][]{gretton08hsic}.


%% file: Content/testing.tex
\section{Statistical tests for joint independence}\label{sec:testing}

Assume Setting \ref{setting:mHSIC} and denote by
$\pmspace{\prodkernelspace}$ the space of Borel probability
measures. In this section we derive three statistical hypothesis tests for
the null hypothesis 
\begin{equation}
  \label{eq:HO}
  \HO\coloneqq\bigg\{\mu\in\pmspace{\prodkernelspace}\,\Big\rvert\, \vX
  \sim \mu = \jointlawb,\ \jointlawb=\prodlaw\bigg\}
\end{equation}
against the alternative
\begin{equation}
  \label{eq:HA}
  \HA\coloneqq\bigg\{\mu\in\pmspace{\prodkernelspace}\,\Big\rvert\, \vX
  \sim \mu = \jointlawb,\ \jointlawb \neq \prodlaw\bigg\}.
\end{equation}

Based on the asymptotic behavior given in Theorem
\ref{thm:asymptoticdist_HO}, we consider $n\cdot\mhsicb_n$ as test
statistic and define a 
decision rule $\phi=(\phi_n)_{n\in\N}$ encoding rejection of $H_0$ if
$\phi_n = 1$ and no rejection of $H_0$ if $\phi_n = 0$. For
all $n\in\{1,\dots,2d-1\}$ we define $\phi_n\coloneqq 0$ and for all
$n\in\{2d,2d+1,\dots\}$ and for all
$(\vx_1,\dots,\vx_n)\in\prodkernelspace^n$ we set
\begin{equation}
  \label{eq:hypothesis_test}
  \phi_n(\vx_1,\dots,\vx_n)\coloneqq\mathds{1}_{\{n\cdot\mhsicb_n(\vx_1,\dots,\vx_n)>c_n(\vx_1,\dots,\vx_n)\}}
\end{equation}
where the threshold $c=(c_n)_{n\in\N}$ remains to be chosen. Ideally,
for fixed $\alpha\in(0,1)$ the hypothesis test
should have (valid) level $\alpha$, i.e. for every $\mu = \prodlawbold \in \HO$ and all $n$,
$  \P\left(\phi_n(\vX_1,\dots,\vX_n)=1\right)\leq\alpha$, 
where $\vX_1,\vX_2,\dots\iid\prodlawbold \in \HO$. 
A weaker condition states that the test respects the level in the large
sample limit, i.e. for every $\mu = \prodlawbold \in \HO$,
\begin{equation}
  \label{eq:asympleveldef}
  \limsup_{n\rightarrow\infty}\P\left(\phi_n(\vX_1,\dots,\vX_n)=1\right)\leq\alpha,
\end{equation}
where $\vX_1,\vX_2,\dots\iid\prodlawbold \in \HO$. 
Such a test is
said to have pointwise asymptotic level. Additionally, 
the test is called pointwise consistent if
for all fixed $\prodlawbold\in\HA$ it holds that
\begin{equation}
  \label{eq:consistencydef}
  \lim_{n\rightarrow\infty}\P\left(\phi_n(\vX_1,\dots,\vX_n)=1\right)=1,
\end{equation}
where $\vX_1,\vX_2,\dots\iid\ \mbox{(fixed)}\ \prodlawbold \in \HA$.
The
following 
table summarizes the properties that our three tests satisfy.

\begin{center}
\renewcommand{\arraystretch}{1.2}
\begin{tabular}{|l||c|c|c|}
  \hline
  \textbf{Hypothesis test} & $\quad$ \textbf{consistency} $\quad$& \textbf{level} &
 $\qquad$ \textbf{speed} $\qquad$\\ \hline
  \multirow{2}{*}{Permutation\footnotemark[1]} & unknown&valid
&   \multirow{2}{*}{slow}\\
    &(Rem.~\ref{rmk:ptw_consistency_perm})&
   (Prop. \ref{thm:level_permutation_test}) & \\ \hline
  \multirow{2}{*}{Bootstrap\footnotemark[1]} & pointwise & pointwise asymptotic  & \multirow{2}{*}{slow} \\ 
  &(Thm. \ref{thm:consistencybootstraptest}) & (Thm. \ref{thm:level_bootstrap_test})&\\\hline
  Gamma approximation & no guarantee & no guarantee & fast \\ \hline
\end{tabular}
\end{center}
\renewcommand{\arraystretch}{1} \footnotetext[1]{For implementation
  purposes one can use the Monte-Carlo approximation. This leads to a
  reasonably fast implementation, while conserving the (asymptotic)
  level and consistency results. Further details are given in
  Section~\ref{sec:implementation_permboot}} In
Section~\ref{sec:asymptotics_teststat} we consider some of the
asymptotic properties of the test statistic $n\cdot\mhsicb_n$. In
particular, we show the existence of an asymptotic distribution under
$\HO$. We then construct three hypothesis tests of the form
\eqref{eq:hypothesis_test}. The first two are a permutation test and a
bootstrap test which are discussed in
Section~\ref{sec:boot_perm}. Both tests are based on resampling and
hence do not rely on an explicit knowledge of the asymptotic
distribution under $\HO$. In Section~\ref{sec:gamma} we consider a
third test which is based on an approximation of the asymptotic
distribution under $\HO$ using a Gamma distribution.

\subsection{Asymptotic behavior of the test statistic}\label{sec:asymptotics_teststat}

We first determine the asymptotic distribution of $n\cdot\mhsicb_n$ under
$\HO$, extending \citet[Theorem 2]{gretton08hsic} from HSIC to dHSIC.

\begin{theorem}[asymptotic distribution of $n\cdot\mhsicb_n$ under $\HO$]
  \label{thm:asymptoticdist_HO}
  Assume Setting \ref{setting:mHSIC}. Let $(Z_i)_{i\in\N}$ be a sequence of independent
  standard normal random variables on $\R$, let\footnote{Given a
    measure space $(\Omega,\mathcal{F},\mu)$ the space
    $\lpspace{r}{\mu}{\abs{\cdot}_{\R}}$ consists of all measurable
    functions $f:\Omega\rightarrow\R$ such that
    $\int_{\Omega}\abs{f(\omega)}^r\measure{\mu}{\omega}<\infty$. The
    corresponding space of equivalence classes of such functions is
    denoted by $\Lpspace{r}{\mu}{\abs{\cdot}_{\R}}$. Moreover, we
    denote the space of all linear bounded operators from a Banach
    space $\Bspace$ onto itself by $L(\Bspace)$.}
   $T_{h_2}\in
  L(\Lpspace{2}{\jointlaw}{\abs{\cdot}_{\R}})$ be such that
  for all $f\in\Lpspace{2}{\jointlaw}{\abs{\cdot}_{\R}}$ and for
  all $\vx\in\prodkernelspace$ it holds that
  \begin{equation*}
    \left(T_{h_2}(f)\right)(\vx)=\int_{\prodkernelspace}h_2(\vx,\vy)f(\vy)\measure{\jointlaw}{\vy}.
  \end{equation*}
  Denote by $(\lambda_i)_{i\in\N}$ the eigenvalues of
  $T_{h_2}$. Then under $\HO$ it holds that
  \begin{equation*}
    \xi_2(h)>0\qquad\text{and}\qquad
    n\cdot\mhsicb_n\overset{d}{\longrightarrow}\dbinom{2d}{2}\sum_{i=1}^{\infty}\lambda_iZ_i^2
    \quad \text{as } n\rightarrow\infty.
  \end{equation*}
\end{theorem}
The proof of this result relies on the
asymptotic distribution of degenerate V-statistics (see
Theorem~\ref{thm:asymptoticdist_vstat2}). In order to show that the
degenerate setting applies we need to prove that under $\HO$ it holds
that $\xi_1(h)=0$ and $\xi_2(h)>0$. The latter statement is of
interest in itself and has been for example implicitly assumed in
\citet[Theorem 2]{gretton08hsic}.
But while $\xi_1(h)=0$ follows more or less
directly from the independence assumption under $\HO$ (see
Lemma~\ref{thm:degeneracy}), the condition $\xi_2(h)>0$ is difficult
to verify directly due to the complicated form of the core function
$h$.  We therefore circumvent direct verification by using empirical
process theory to prove that the asymptotic distribution of
$n\cdot\mhsic_n$ has certain continuity properties (see
Theorem~\ref{thm:cont_strictly_increasing_cdf}) that are not satisfied
by the asymptotic distribution resulting from the theory of
V-statistics if both $\xi_1(h)$ and $\xi_2(h)$ were zero. A full proof
is given in Appendix~\ref{proof:asymptotic_dist}.

\begin{remark}[estimation of eigenvalues]
  It is possible to construct a test that estimates the eigenvalues of
  the integral operator in Theorem~\ref{thm:asymptoticdist_HO} by
  first estimating the eigenvalues of the Gram matrix corresponding to
  $h_2$ and then computing the asymptotic distribution using a
  bootstrap procedure, see \citet{eigenvalueMMD}. Given knowledge of
  the exact form of $h_2$ and under the assumption that $h_2$ is
  positive definite (can be shown for $d=2$, unknown for $d>2$) one can
  prove consistency, see \citet{masterthesis}. However, since $h_2$ is
  a complicated function (see Lemma~\ref{thm:expansionh2}) depending
  on the unknown distribution $\prodlawbold$ (as opposed to
  \citet{eigenvalueMMD}) one has to estimate $h_2$, which means one
  would have to additionally account for that approximation. In
  simulations, the eigenvalue estimation generally performed worse than
  the Gamma approximation in almost all our experiments. We have
  therefore decided not to include this approach in the paper. There
  is, however, an implementation in the \texttt{dHSIC} R-package.
\end{remark}

The following theorem is an important result required to establish
consistency (of the bootstrap test), stating that $n\cdot\mhsicb_n$
diverges under $\HA$.
\begin{theorem}[asymptotic distribution of $n\cdot\mhsicb_n$ under
  $\HA$]
  \label{thm:corollaryasymptoticdistmhsic}
  Assume Setting \ref{setting:mHSIC}. Then under $\HA$ it holds for all $t\in\R$ that
  \begin{equation*}
  \lim_{n\rightarrow\infty}\P\left(n\cdot\mhsicb_n\leq
    t\right)=0.
\end{equation*}
\end{theorem}
A proof is given in Appendix~\ref{proof:asymptotic_dist}.


\subsection{Resampling tests}\label{sec:boot_perm}

We first introduce the notation of a general resampling scheme which
encompasses a permutation and bootstrap method which we will use later. For
every 
function $\boldsymbol{\psi}=(\psi^1,\dots,\psi^d)$ such that for all
$i\in\{1,\dots,d\}$ it holds that
$\psi^i:\{1,\dots,n\}\rightarrow\{1,\dots,n\}$, define the function
$g_{n,\boldsymbol{\psi}}:\prodkernelspace^n\rightarrow\prodkernelspace^n$ 
\begin{equation}
  \label{eq:gmpsifun}
  g_{n,\boldsymbol{\psi}}(\vx_1,\dots,\vx_n)\coloneqq\left(\vx_{n,1}^{\boldsymbol{\psi}},\dots,\vx_{n,n}^{\boldsymbol{\psi}}\right),\
  (\vx_1,\dots,\vx_n)\in\prodkernelspace^n,
\end{equation}
where
$\vx_{n,i}^{\boldsymbol{\psi}}\coloneqq\left(x_{\psi^1(i)}^1,\dots,x_{\psi^d(i)}^d\right)$.
The diagram \eqref{eq:transformation_diagram} illustrates the mapping
$g_{n,\boldsymbol{\psi}}$. 
\begin{equation}
  \label{eq:transformation_diagram}
  \begin{array}{c|ccccc|ccc}
    \vx_1& x^1_1& \cdots& x^d_1& & \vx_{n,1}^{\boldsymbol{\psi}}&
    x^1_{\psi^1(1)}& \cdots& x^d_{\psi^d(1)}\\
    \vdots& \vdots& & \vdots&
  \quad   \overset{g_{n,\boldsymbol{\psi}}}{\longrightarrow} \quad &
    \vdots&  \vdots& & \vdots\\
    \vx_n& x^1_n& \cdots& x^d_n& & \vx_{n,n}^{\boldsymbol{\psi}}&
    x^1_{\psi^1(n)}& \cdots& x^d_{\psi^d(n)}
  \end{array}
\end{equation}
Define
\begin{equation}
  \label{eq:bootstrap_set}
  B_n\coloneqq\big\{\psi:\{1,\dots,n\}\rightarrow\{1,\dots,n\}\mid \psi\text{ is a
        function}\big\},
\end{equation}
then for a subset $A_n\subseteq B_n^d$ we call the family of functions
\begin{equation}
  \label{eq:resampling_method_ind}
  g\coloneqq(g_{n,\boldsymbol{\psi}})_{\boldsymbol{\psi}\in A_n}
\end{equation}
a resampling method.
In the following two sections we formulate the bootstrap and permutation tests in terms of this resampling method.

\subsubsection{Permutation test}\label{sec:permutationtest}

The permutation test is 
based on the resampling 
\eqref{eq:resampling_method_ind} with
$A_n=(S_n)^d$, where $S_n$ is the set of permutations on
$\{1,\dots,n\}$. More precisely, we have the following definition.
\begin{definition}[permutation test for dHSIC]
  \label{def:permutation_test}
  Assume Setting \ref{setting:mHSIC} and $\alpha\in(0,1)$. For all
  $\boldsymbol{\psi}\in(S_n)^d$, let 
  $g_{n,\boldsymbol{\psi}}$ be defined as in
  \eqref{eq:gmpsifun}. Moreover, for
$n\in\{2d,2d+1,\dots\}$, let
  $\hat{R}_n:\prodkernelspace^n\times\R\rightarrow [0,1]$ be the
  resampling distribution functions defined for all $t\in\R$ by
  \begin{equation} \label{eq:permtest}
    \hat{R}_n(\vx_1,\dots,\vx_n)(t)\coloneqq\dfrac{1}{(n!)^d}\sum_{\boldsymbol{\psi}\in
      (S_n)^d}\mathds{1}_{\{n\cdot\mhsicb_n(g_{n,\boldsymbol{\psi}}(\vx_1,\dots,\vx_n))\leq
      t\}}.
  \end{equation}
  Then the \emph{$\alpha$-permutation test for dHSIC} is defined by
  $\phi_n\coloneqq 0$ for $n\in\{1,\dots,2d-1\}$, and for
  $n\in\{2d,2d+1,\dots\}$ by 
    \begin{equation*}
    \phi_n(\vx_1,\dots,\vx_n)\coloneqq\mathds{1}_{\left\{n\cdot\mhsicb_n(\vx_1,\dots,\vx_n)>(\hat{R}_n(\vx_1,\dots,\vx_n))^{-1}(1-\alpha)\right\}}.
  \end{equation*}
\end{definition}
Given that the resampling method has a group structure
and additionally satisfies for all $\vX$ with $\prodlawbold\in\HO$
that
\begin{equation*}
  g_{n,\boldsymbol{\psi}}(\vX_1,\dots,\vX_n)\text{ is equal in distribution to }(\vX_1,\dots,\vX_n),
\end{equation*}
where $\vX_1,\vX_2,\dots\iid\prodlawbold$, it can be shown that tests
of this form have valid level. For the permutation test for dHSIC both
these properties are satisfied, hence it has valid level.
\begin{proposition}[permutation test for dHSIC has valid level]
  \label{thm:level_permutation_test}
  Assume Setting \ref{setting:mHSIC} and let $\HO$ and $\HA$ be
  defined as in \eqref{eq:HO} and \eqref{eq:HA}. Then for all
  $\alpha\in (0,1)$ the $\alpha$-permutation test for dHSIC has valid
  level $\alpha$ when testing $\HO$ against $\HA$.
\end{proposition}
A proof is given in Appendix~\ref{proof:permutation_test}. It is
important to note that the level property from Proposition
\ref{thm:level_permutation_test} is for the finite sample setting and
does not depend on the asymptotic behavior of the test statistics.

The size of the set $(S_n)^d$ is given by $(n!)^d$, therefore
computing~\eqref{eq:permtest} quickly becomes infeasible. For
implementation purposes we generally use a Monte-Carlo approximated
version, and the details are given in
Section~\ref{sec:implementation_permboot}. Surprisingly, it can be
shown that whenever the probability distribution $\prodlawbold$ is
continuous, the Monte-Carlo approximated permutation test also has
valid level, see Proposition~\ref{thm:exactlevelmontecarlo} and the
comments thereafter.
\if0
\begin{remark}[pointwise consistency of the permutation test]\label{rmk:ptw_consistency_perm}
  Based on our proof of the pointwise consistency of the bootstrap
  test (Theorem~\ref{thm:consistencybootstraptest}) and the
  similarities between bootstrap and permutation test it seems likely
  that a similar result should also hold for the permutation
  test. However, the proof of
  Theorem~\ref{thm:consistencybootstraptest} cannot be easily extended
  to cover the permutation case. A more promising approach would be to
  proceed similarly to \citet{romano1989permutation}, as the test
  statistics considered there are closely related to dHSIC, see
  \eqref{eq:Tn_rep_dhsic}. The essential idea there is to use the
  theory of empirical processes (see
  Appendix~\ref{proof:properties_of_dist_HO}) to prove the assumptions
  of \citet[Theorem 15.2.3]{lehmann}. Unfortunately, we were not able
  to extend the results in \citet{romano1989permutation} from
  V.C.-classes of sets to the required classes of functions. While
  most of the results extend more or less directly (see
  Appendix~\ref{proof:properties_of_dist_HO}) the difficulties lie in
  proving a similar representation for $S_n$ as the one given in the
  display of \citet[proof of Proposition 3.1]{romano1989permutation}
  as well as a result similar to \citet[Lemma
  5.1]{romano1989permutation}.

  It is also worth noting, that in order to extend the empirical process approach given in
  \citet{romano1988bootstrap} to give an alternative
  proof of Theorem~\ref{thm:consistencybootstraptest}, one would
  require a uniform Donsker-property for the unit ball of our RKHS.
\end{remark}
\else
\begin{remark}[pointwise consistency of the permutation test]\label{rmk:ptw_consistency_perm}
  Given the similarity between bootstrap and permutation tests, it
  seems likely that the permutation test for dHSIC is consistent, too.
  The proof of Theorem~\ref{thm:consistencybootstraptest}, however,
  cannot be easily extended. A more promising approach would be to
  proceed similarly to \citet{romano1989permutation}, as the test
  statistics considered there are closely related to dHSIC,
  see~\eqref{eq:BMR} and~\eqref{eq:Tn_rep_dhsic}. The essential idea
  there is to use the theory of empirical processes (see
  Appendix~\ref{proof:properties_of_dist_HO}) to prove the assumptions
  of \citet[Theorem 15.2.3]{lehmann}. Unfortunately, we were not able
  to extend the results in \citet{romano1989permutation} from
  VC-classes of sets to the required classes of functions. While many
  results extend more or less directly (see
  Appendix~\ref{proof:properties_of_dist_HO}), the difficulties lie in
  proving a similar representation for $S_n$ as the one given in the
  display of \citet[proof of Proposition 3.1]{romano1989permutation},
  as well as a result similar to \citet[Lemma
  5.1]{romano1989permutation}.  As a side remark, extending the
  empirical process approach given in \citet{romano1988bootstrap} to
  give an alternative proof of
  Theorem~\ref{thm:consistencybootstraptest} would require a uniform
  Donsker-property for the unit ball of the RKHS.
\end{remark}
\fi

\subsubsection{Bootstrap test}\label{sec:bootstraptest}

The bootstrap test is based on the resampling 
\eqref{eq:resampling_method_ind} with
$A_n=B_n^d$. 

\begin{definition}[bootstrap test for dHSIC]
  \label{def:bootstrap_test}
  Assume Setting \ref{setting:mHSIC} and $\alpha\in(0,1)$. For all
  $\boldsymbol{\psi}\in B_n^d$ let the 
  function $g_{n,\boldsymbol{\psi}}$ be defined as in
  \eqref{eq:gmpsifun}. Moreover, for $n\in\{2d,2d+1,\dots\}$, let
  $\hat{R}_n:\prodkernelspace^n\times\R\rightarrow [0,1]$ be the
  resampling distribution functions defined for all $t\in\R$ by
  \begin{equation*}
    \hat{R}_n(\vx_1,\dots,\vx_n)(t)\coloneqq\dfrac{1}{n^{nd}}\sum_{\boldsymbol{\psi}\in
      B_n^d}\mathds{1}_{\{n\cdot\mhsicb_n(g_{n,\boldsymbol{\psi}}(\vx_1,\dots,\vx_n))\leq
      t\}}.
  \end{equation*}
  Then the \emph{$\alpha$-bootstrap test for dHSIC} is defined by
  $\phi_n\coloneqq 0$ for all $n\in\{1,\dots,2d-1\}$, and for 
  $n\in\{2d,2d+1,\dots\}$ by 
\begin{equation*}
    \phi_n(\vx_1,\dots,\vx_n)\coloneqq\mathds{1}_{\left\{n\cdot\mhsicb_n(\vx_1,\dots,\vx_n)>(\hat{R}_n(\vx_1,\dots,\vx_n))^{-1}(1-\alpha)\right\}}. 
  \end{equation*}
\end{definition}

Unlike for the permutation test, the bootstrap resampling method no
longer exhibits a group. We can therefore not expect the bootstrap
test to have valid level. However, it is possible to show that it has
pointwise asymptotic level and even pointwise consistency. The reason
this can be done is that the resampling method in the bootstrap test
is connected to the empirical product distribution $\empprodlaw$. The
following theorem proves that the bootstrap test for dHSIC has
pointwise asymptotic level.

\begin{theorem}[bootstrap test for dHSIC has pointwise asymptotic level]
  \label{thm:level_bootstrap_test}
  Assume Setting \ref{setting:mHSIC} and let $\HO$ and $\HA$ be
  defined as in \eqref{eq:HO} and \eqref{eq:HA}. Then for all
  $\alpha\in(0,1)$ the $\alpha$-bootstrap test for dHSIC has pointwise
  asymptotic level $\alpha$ when testing $\HO$ against $\HA$.
\end{theorem}
A proof is given in Appendix~\ref{proof:bootstrap_test}. We now
establish that the bootstrap test for dHSIC is consistent. 

\begin{theorem}[consistency of the bootstrap test for dHSIC]
  \label{thm:consistencybootstraptest}
  Assume Setting \ref{setting:mHSIC} and let $\HO$ and $\HA$ be defined
  as in \eqref{eq:HO} and \eqref{eq:HA}. Then for all
  $\alpha\in(0,1)$ 
  the $\alpha$-bootstrap test is pointwise consistent when
  testing $\HO$ against $\HA$.
\end{theorem}
A proof is given in Appendix~\ref{proof:bootstrap_test}.
Similarly as for the permutation test, the size of the set
$(B_n)^d$ is $n^{nd}$ which grows quickly. Again, we may use a Monte-Carlo
approximated version, see 
Section~\ref{sec:implementation_permboot}.


In \citet{kacper} a similar consistency analysis has been performed
for a wild bootstrap approach on time series.

\subsection{Gamma approximation}\label{sec:gamma}
We showed in Theorem~\ref{thm:asymptoticdist_HO} that
the asymptotic distribution of $n\cdot\mhsicb_n$ equals
\begin{equation}
  \label{eq:chisquaredsum}
  \binom{2d}{2}\sum_{i=1}^{\infty}\lambda_iZ_i^2.
\end{equation}
The essential idea behind the Gamma approximation \citep[see also][]{kankainen,grettonHSIC} is that a
distribution of the form
$
  \sum_{i=1}^{\infty}\lambda_iZ_i^2
$
can be approximated fairly well by a Gamma distribution with matched
first and second moments \citep[see][for basic empirical
evidence]{gammaapproximation}. The intuition is that the Gamma
distribution would be correct if the sequence of eigenvalues
$\lambda_i$ from the integral operator contains only a finite number
of non-zero values, which implies that it is a good approximation as
long as the sequence of $\lambda_i$ decays fast enough. This has,
however, only been shown empirically and no guarantees in the large
sample limit are available.  In fact, it is rather unlikely such
guarantees even exist as it is not hard to find choices of $\lambda_i$
for which \eqref{eq:chisquaredsum} is not a Gamma distribution. It is
not as simple, however, to show that such values of $\lambda_i$ can
actually occur as solutions of the defining integral
equation. Nevertheless, the approximation seems to work well for
small~$d$, see Section~\ref{sec:experiments}, and the test can be
computed much faster than the other approaches.




The Gamma distribution with parameters $\alpha$ and $\beta$ is denoted
by $\operatorname{Gamma}(\alpha,\beta)$ and corresponds to the distribution with
density
$
  f(x)={x^{\alpha-1}e^{x/\beta}}/({\beta^{\alpha}\Gamma(\alpha)})
$,
where $\Gamma(t)=\int_0^{\infty}x^{t-1}e^{-x}dx$ is the Gamma
function. The first two moments of the
$\operatorname{Gamma}(\alpha,\beta)$-distributed random variable $Y$
are given by $\E(Y)=\alpha\beta$ and $\Var(Y)=\alpha\beta^2$. In order
to match the first two moments we define for
$\vX_1,\vX_2,\ldots\iid\prodlawbold\in\HO$ the two parameters
\begin{equation*}
  \alpha_n(\prodlawbold)\coloneqq\dfrac{\left(\E\left(\mhsicb_n\right)\right)^2}{
    \Var\left(\mhsicb_n\right)}\quad\text{and}\quad
  \beta_n(\prodlawbold)\coloneqq\dfrac{n\Var\left(\mhsicb_n\right)}{\E\left(\mhsicb_n\right)}.
\end{equation*}
Then we use the approximation
\begin{equation}
  \label{eq:gamma_approximation}
  n\cdot\mhsicb_n(\vX_1,\dots,\vX_n)\sim\operatorname{Gamma}\left(\alpha_n(\prodlawbold),\beta_n(\prodlawbold)\right).
\end{equation}
The following two propositions give expansions of the involved moments in
terms of the kernel. 
\begin{proposition}[mean of $\mhsicb$]
  \label{thm:meanofmhsic}
  Assume Setting \ref{setting:mHSIC}. Then under $\HO$ it holds that,   as $n\rightarrow\infty$,
  \begin{equation*}
    \E\left(\mhsicb_n\right)=\dfrac{1}{n}-\dfrac{1}{n}\sum_{r=1}^d
    \prod_{j\neq
        r}\E \left(k^j(X^j_1,X^j_2)\right)
        +\dfrac{d-1}{n}\prod_{j=1}^d\E\left(k^j(X^j_1,X^j_2)\right)+\landauO{n^{-2}}.
  \end{equation*}
\end{proposition}
A proof is given in Appendix~\ref{proof:momentsofdhsic}.

\begin{proposition}[variance of $\mhsicb$]
  \label{thm:varianceofmhsic}
  Assume Setting \ref{setting:mHSIC}. Then under $\HO$ it holds that,
  \begin{align*}
    \Var\left(\mhsicb_n\right)&=2\dfrac{(n-2d)!}{n!}\dfrac{(n-2d)!}{(n-4d+2)!}\Bigg[
    \prod_{j=1}^de_1(j)+(d-1)^2\prod_{j=1}^de_0(j)^2 
+2(d-1)\prod_{j=1}^de_2(j)\\
     &\qquad 
+\sum_{j=1}^de_1(j)\prod_{r\neq j}e_0(r)^2
    -2\sum_{j=1}^de_1(j)\prod_{r\neq j}e_2(r)-2(d-1)\sum_{j=1}^de_2(j)\prod_{r\neq j}e_0(r)^2\\
    &\qquad +\sum_{j\neq l}e_2(j)e_2(l)\prod_{r\neq
      j,l}e_0(r)^2\Bigg]+\landauO{n^{-\tfrac{5}{2}}}
  \end{align*}
  as $n\rightarrow\infty$ and where for all $j\in\{1,\dots,d\}$,
  \begin{equation*}
    e_0(j)=\E\left(k^j(X^j_1,X^j_2)\right),\quad
    e_1(j)=\E\left(k^j(X^j_1,X^j_2)^2\right),\quad
    e_2(j)=\E_{X^j_1}\left(\E_{X^j_2}\left(k^j(X^j_1,X^j_2)\right)^2\right).
  \end{equation*}
\end{proposition}
A proof is given in Appendix~\ref{proof:momentsofdhsic}.
Based on these two propositions we only need a method to estimate the terms
$e_0(j)$, $e_1(j)$ and $e_2(j)$ for all $j\in\{1,\dots,d\}$. One could
use a U-statistic~\eqref{eq:ustat} for each expectation term as this
would not add any bias. It turns out, however, that a computationally more
efficient V-statistic also does not add any asymptotic bias in this
particular case. This is due to 
Theorem~\ref{thm:biasVstat} describing that the bias of a
V-statistic is of order $\landauO{n^{-1}}$ and hence is consumed by
the error terms in Proposition~\ref{thm:meanofmhsic} and
Proposition~\ref{thm:varianceofmhsic}.  The V-statistics for these terms are
given for all $(\vx_1,\dots,\vx_n)\in\prodkernelspace^n$ by
\begin{enumerate}[(i)]
\item[] $\hat{e}_0(j)(\vx_1,\dots,\vx_n)\coloneqq\frac{1}{n^2}\sum_{i_1,i_2=1}^n k^j(x^j_{i_1},x^j_{i_2})$,
\item[] $\hat{e}_1(j)(\vx_1,\dots,\vx_n)\coloneqq\frac{1}{n^2}\sum_{i_1,i_2=1}^n k^j(x^j_{i_1},x^j_{i_2})^2$,
\item[] $\hat{e}_2(j)(\vx_1,\dots,\vx_n)\coloneqq\frac{1}{n^3}\sum_{i_2=1}^n\left(\sum_{i_1=1}^n k^j(x^j_{i_1},x^j_{i_2})\right)^2$.
\end{enumerate}
Based on these terms we define the estimators
$\hat{\E}_n$ and $\hat{\operatorname{Var}}_n$ for the mean and variance of
$\mhsicb_n$, respectively by
replacing all appearances of $e_0(j)$, $e_1(j)$ and $e_2(j)$ in
Propositions \ref{thm:meanofmhsic} and \ref{thm:varianceofmhsic} by
$\hat{e}_0(j)$, $\hat{e}_1(j)$ and $\hat{e}_2(j)$. We use the plug-in
estimators 
\begin{equation} \label{eq:alpha_estimator}
\hat{\alpha}_n(\vx_1,\dots,\vx_n)\coloneqq\dfrac{\hat{\E}_n(\vx_1,\dots,\vx_n)^2}{\hat{\operatorname{Var}}_n(\vx_1,\dots,\vx_n)},
\qquad \hat{\beta}_n(\vx_1,\dots,\vx_n) \coloneqq\dfrac{n\hat{\operatorname{Var}}_n(\vx_1,\dots,\vx_n)}{\hat{\E}_n(\vx_1,\dots,\vx_n)},
\end{equation}
and then define the following hypothesis test.

\begin{definition}[Gamma approximation based test for dHSIC]
  \label{def:gammaapproximation_test}
  Assume Setting \ref{setting:mHSIC} and $\alpha\in(0,1)$. Let
  $F_n(\vx_1,\dots,\vx_n)$ be the distribution function associated to
  the
  $\operatorname{Gamma}(\hat{\alpha}_n(\vx_1,\dots,\vx_n),\hat{\beta}_n(\vx_1,\dots,\vx_n))$-distribution,
  where $\hat{\alpha}_n$ and $\hat{\beta}_n$ are defined as in
  \eqref{eq:alpha_estimator}. 
  Then the \emph{$\alpha$-Gamma approximation based test for
  dHSIC} is defined by
  $\phi_n\coloneqq 0$ for all $n\in\{1,\dots,2d-1\}$, and for 
  $n\in\{2d,2d+1,\dots\}$ by 
  \begin{equation*}
    \phi_n(\vx_1,\dots,\vx_n)\coloneqq\mathds{1}_{\left\{n\cdot\mhsicb_n(\vx_1,\dots,\vx_n)>F_n(\vx_1,\dots,\vx_n)^{-1}(1-\alpha)\right\}}.
  \end{equation*}
\end{definition}


%% file: Content/implementation.tex
\section{Implementation} \label{sec:implementation}

We now discuss an efficient implementation of the
proposed tests and briefly comment on the choice of kernel. 
All
methods are available for the
\texttt{R}-language \citep{R} as package~\texttt{dHSIC}.

\subsection{dHSIC estimator}\label{sec:implementation_dhsic}

The dHSIC estimator ($\mhsicb$) can be computed in quadratic time. One such
efficient implementation is given in Algorithm \ref{alg:dHSIC}, where
$\mathds{1}_{k\times \ell}$ denotes a $k \times \ell$ matrix of ones,
the functions \texttt{Sum} and \texttt{ColumnSum} take the sums of all
elements in a matrix and its columns, respectively, and $*$ denotes
the element-wise multiplication operator. The variables term1, term2
and term3 are related to the three components of the sum in
Definition~\ref{def:mhsic_esimator}, after changing the order of
products and sums.
\begin{algorithm}
  \caption{computing the dHSIC V-estimator}\label{alg:dHSIC}
  {\fontsize{11}{15}\selectfont
  \begin{algorithmic}[1]
    \Procedure{$\texttt{dHSIC}$}{$\vx_1,\dots,\vx_n$}
    \For{$j=1:d$}
      \State $\mathbf{K}^j\gets\text{Gram matrix of kernel }k^j\text{
        given }\vx_1,\dots,\vx_n$
    \EndFor
    \State $\text{term1}\gets \mathds{1}_{n\times n}; \quad 
 \text{term2}\gets 1; \quad  
   \text{term3}\gets \frac{2}{n} \cdot \mathds{1}_{1\times n} $
    \For{$j=1:d$}
      \State $\text{term1}\gets\text{term1}*\mathbf{K}^j$
      \State $\text{term2}\gets\frac{1}{n^2}\cdot\text{term2}\cdot\texttt{Sum}(\mathbf{K}^j)$
      \State $\text{term3}\gets\frac{1}{n}\cdot\text{term3}*\texttt{ColumnSum}(\mathbf{K}^j)$
    \EndFor
    \State $\text{dHSIC}\gets\frac{1}{n^2}\cdot \texttt{Sum}(\text{term1}) +\text{term2}- \texttt{Sum}(\text{term3})$
    \State \Return $\text{dHSIC}$
    \EndProcedure
  \end{algorithmic}
  }
\end{algorithm}

\subsection{Resampling Tests}\label{sec:implementation_permboot}

From the definition of $\hat{R}_n$ we see that the permutation and bootstrap test involve $(n!)^d$ or $n^{nd}$ evaluations of $\mhsicb$, respectively. 
Instead of computing $\hat{R}_n$ explicitly one can use the Monte-Carlo approximation defined in Definition \ref{def:monte_carlo_resampling_dist}. 
This involves
calculating the $p$-value given by
\begin{equation*}
  \hat{p}_n(\vx_1,\dots,\vx_n)\coloneqq\frac{1+\abs[\big]{\{i\in\{1,\dots,B\}:\,
    \mhsicb(g_{n,\boldsymbol{\psi}_i}(\vx_1,\dots,\vx_n))\geq
    \mhsicb(\vx_1,\dots,\vx_n)\}}}{1+B},
\end{equation*}
where $(\boldsymbol{\psi}_i)_{i\in\N}$ is a sequence drawn from the
uniform distribution on $A_n$ (i.e. on $(S_n)^d$ for the permutation
test and on $B_n^d$ for the bootstrap test). The test then rejects the
null hypothesis whenever $\hat{p}_n(\vx_1,\dots,\vx_n)\leq\alpha$. 
Further details including critical values are provided in
Section~\ref{sec:mcresamplingtests}. 
\citet{davison} suggest to use $B$ between $99$ and $999$.

We now give some additional details specific
to the permutation and the bootstrap test.

\subsubsection{Permutation test}

As shown in the proof of Proposition~\ref{thm:level_permutation_test}
the resampling method $g$ for the permutation test is a resampling
group which satisfies the invariance conditions
\eqref{eq:invariance_1} and \eqref{eq:uniqueness_2}. This allows us to
apply Proposition~\ref{thm:exactlevelmontecarlo} to see that the
Monte-Carlo approximated permutation test has valid level for any
finite $B$, given that we have continuous random variables as input.
Algorithm~\ref{alg:pval_critval} in
Appendix~\ref{sec:mcresamplingtests} shows how to implement the
$p$-value and the critical value for the Monte-Carlo approximated
permutation test.

\subsubsection{Bootstrap test}

It can be shown using a standard concentration inequality argument
\citep[see e.g.][Theorem 11.2.18 and Example 11.2.13]{lehmann} that
the Monte-Carlo approximated version of the bootstrap test for dHSIC
still has pointwise asymptotic level and is pointwise consistent, as
both $n$ and $B$ go to infinity. Algorithm~\ref{alg:pval_critval} in
Appendix~\ref{sec:mcresamplingtests} shows how to implement the
$p$-value and the critical value for the Monte-Carlo approximated
bootstrap test.

\subsection{Gamma approximation test}
Implementing the $\alpha$-Gamma
approximation based test for dHSIC consists of four steps
(see Section~\ref{sec:gamma} for notation):
\begin{compactitem}
\item[1.] for all $j\in\{1,\dots,d\}$ implement the estimators
  $\hat{e}_0(j),\dots,\hat{e}_2(j)$,
\item[2.] compute the estimates
$\hat{\E}_n(\vx_1,\dots,\vx_n)$ and $\hat{\operatorname{Var}}_n(\vx_1,\dots,\vx_n)$,  
\item[3.] using \eqref{eq:alpha_estimator} 
  compute the estimates $\hat{\alpha}_n(\vx_1,\dots,\vx_n)$ and
  $\hat{\beta}_n(\vx_1,\dots,\vx_n)$ and
\item[4.] compute the $(1-\alpha)$-quantile of the $\operatorname{Gamma}(\hat{\alpha}_n(\vx_1,\dots,\vx_n),\hat{\beta}_n(\vx_1,\dots,\vx_n))$-distribution.
\end{compactitem}
The hypothesis test rejects $\HO$ if
$n\cdot\mhsicb_n(\vx_1,\dots,\vx_n)$ is larger than the $1-\alpha$ quantile of the
$\operatorname{Gamma}(\hat{\alpha}_n(\vx_1,\dots,\vx_n),\hat{\beta}_n(\vx_1,\dots,\vx_n))$-distribution
calculated in the last step.

\subsection{Choice of kernel}\label{sec:choice_kernel}
The choice of the kernel determines how well certain types of
dependence can be detected and therefore influences the practical
performance of dHSIC (see Simulation~6).  For continuous data a common
choice is a Gaussian kernel as defined in
\eqref{eq:GaussianKernel}. It is characteristic, which ensures that
all the above results hold. In particular, any type of dependence can
be detected in the large sample limit.
We use the median heuristic for choosing the bandwidth $\sigma$ by requiring that
$\operatorname{median}\left\{{\norm{\vx_i-\vx_j}^2_{\R^m}}\,:\,
  i<j\right\}={2\sigma^2}$.  This heuristic performs quite well in
many practical applications.  It may be possible, however, to extend
alternative approaches from two-sample testing to independence testing
\citep[e.g.][]{gretton2012optimal}.  For discrete data, we choose a
trivial kernel defined by $k(x,y)\coloneqq\mathds{1}_{\{x=y\}}$.

In practice, it is, moreover, possible and potentially beneficial to
also consider other (potentially non-characteristic) kernels that are
chosen in such a way that they are particularly powerful in detecting
certain types of dependencies.



%% file: Content/experiments.tex
\section{Experiments} \label{sec:experiments}

\subsection{Competing methods} \label{sec:competing}
For comparison purposes we consider an approach, which has been suggested by \citet{beranmiller} and \citet[][p.27]{Romano1986}.
For testing the joint independence of $d$ real valued random variables $X^1, \ldots, X^d$, they consider the test statistic
\begin{equation} \label{eq:BMR}
\widehat{\text{BMR}}_n := \sup_{\va \in \mathbb{R}^d}\abs[\big]{\empjoint(A_{\va}) - \empprod(A_{\va})},
\end{equation}
where
$A_{\va} := (-\infty, a^1] \times (-\infty, a^2] \times \ldots \times
(-\infty, a^d]$ is a subset of $\mathbb{R}^d$,
$\empjoint\coloneqq\frac{1}{n}\sum_{i=1}^n\delta_{\vX_i}$ is the
empirical joint measure and
$\empprod\coloneqq\prod_{j=1}^d(\frac{1}{n}\sum_{i=1}^n\delta_{X_i^j})$
is the empirical product measure. Usually,~\eqref{eq:BMR} cannot be
computed exactly and has to be approximated. One may choose a
distribution $\mu$ with full support on $\mathbb{R}^d$, for example,
and compute the supremum over $C < \infty$ randomly chosen
$\va_1, \ldots, \va_C \iid \mu$.  In our experiments, we mainly choose
$C=n$ since, for consistency, $C$ has to grow with $n$ and since then
the computational complexity is $\landauO{dn^2}$, which equals the
computational complexity of dHSIC, see
Section~\ref{sec:implementation_dhsic}.  As the authors do not provide
any other suggestion, we choose $\mu$ to be the $d$-dimensional
Gaussian distribution with parameters estimated by maximum
likelihood. The test itself is then based on a bootstrap procedure,
described in Section~\ref{sec:bootstraptest}.  In the remainder of
this section, we refer to this test as BMR-$C$.

Furthermore, we 
consider a multiple pairwise version of the two
variable HSIC test. In order to
test for joint 
independence we use the following testing sequence:
\vspace{-0.2cm}
$$
\begin{array}{cll}
1. & \text{use HSIC to test whether } X^d  \text{ is independent of }[X^1,\dots,X^{d-1}],&\qquad \qquad \qquad \text{ }\\
2. & \text{use HSIC to test whether }X^{d-1}\text{ is independent of }[X^1,\dots,X^{d-2}],&\\
\vdots &&\\
d-1.& \text{use HSIC to test whether }X^2 \text{ is independent of }X^1.&
\end{array} 
$$\vspace{-0.3cm}

Finally, we account for the increased family-wise error rate using the
Bonferroni correction, i.e. we perform all tests at level
${\alpha}/{(d-1)}$ and reject the null hypothesis if any of the
individual tests rejects the null hypothesis. In the following we simply
denote this method as HSIC. 
We have mentioned in Section~\ref{sec:intro} that the Bonferroni correction is often conservative: this becomes particularly
evident if this procedure is combined with a permutation test based
HSIC. In that case it can be shown that the smallest possible $p$-value
after the Bonferroni correction is given by $(d-1)/(B+1)$ and
hence for $B=100$ the test will not be able to reject the null
hypothesis at a level of $5\%$ if $d>6$. 

\subsection{Causal inference}\label{subsec.causal}

In causal discovery, one estimates the causal structure from an observed joint distribution.
Here,
we consider additive noise models \citep{peters2014} with additive nonlinear functions and Gaussian noise \citep{cam}; these are special cases of structural equation models~\citep{Pearl2009}.
Assume that the distribution $\prodlawbold = \mathbb{P}^{(X^1,\dots,X^d)}$ 
is induced by $d$ structural equations
\begin{equation} \label{eq:anms}
X^j := \sum_{k \in \parents^j} f^{j,k}(X^k) + N^j, \quad j \in \{1, \ldots, d\},
\end{equation}
with
$\parents^j$ being the parents of $j$ in the associated DAG $\mathcal{G}_0$.
The noise variables $N^1,\dots,N^d$ 
are normally distributed and are assumed to be jointly independent.
An
important question in causality is whether the causal structure, in
this case $\mathcal{G}_0$, can be inferred from the observational
distribution $\prodlawbold$. 
While this is impossible
for general structural equation models
\citep[e.g.][Proposition 9]{peters2014},
the additive noise structure renders the graph identifiable. That is, if $f^{j,k}$ are assumed to be nonlinear, any other 
additive noise model~\eqref{eq:anms} with a structure different from $\mathcal{G}_0$ cannot induce the distribution 
$\prodlawbold$
 \citep[see][Corollary
31, for the full result]{peters2014}. 
In other words, using conditional means as functions in the SEM, the corresponding residual variables will not be jointly independent.

We therefore propose the following method for structure learning using generalized additive model regression (GAM) \citep{gam}.
\begin{mdframed}[roundcorner=5pt, frametitle={DAG verification
    method}]
  Given: observations $\vX_1,\dots,\vX_n$ and a candidate DAG $\mathcal{G}$
  \begin{enumerate}[1)]
  \item Use generalized additive model regression (GAM) \citep{gam} to regress each
    node $X^j$ on all its parents $\parents^j$ and denote the resulting
    vector of residuals by $\mathbf{res}^j$.
  \item Perform a $d$-variable joint independence test (e.g. dHSIC) to
    test whether\\ $(\mathbf{res}^1,\dots,\mathbf{res}^d)$ is jointly
    independent.
  \item If $(\mathbf{res}^1,\dots,\mathbf{res}^d)$ is jointly
    independent, then the DAG $\mathcal{G}$ is not rejected.
  \end{enumerate}
\end{mdframed}
We can furthermore 
estimate the correct DAG by performing the
verification method for all possible DAGs with the correct number of nodes.  In
practice, we expect this method to accept also supergraphs of the
correct graph $\mathcal{G}_0$, which can be overcome by a variable
selection method. Since this work concentrates on the dependence
structure among the residuals, we instead consider only fully
connected DAGs in the experiments (Section~\ref{sec:exp:caus}).  In
practice, we do not want to iterate over all possible graphs. A more
efficient method, which is based on a similar idea, is the RESIT
(regression with subsequent independence test) algorithm described in
\citet[Section 4.1]{peters2014}. Also the computationally efficient
method CAM \citep{cam} could be equipped with a joint independence
test as a model check.

One issue deserves further attention.\footnote{We thank one of the anonymous referees for pointing this out.}
In the regression step 1), we only obtain an approximation of the correct function, which results in 
estimated and thus dependent residuals 
rather than the true noise values. 
We show that this does not affect the asymptotic ordering of $\mhsicb$, see Theorem~\ref{thm:noniid_resid} in Appendix~\ref{sec:approx_resid}.
If we are interested in asymptotically valid $p$-values, 
we can perform sample splitting, see Proposition~\ref{thm:asymptotic_dist_approx_resid} in Appendix~\ref{sec:approx_resid}.

The DAG verification method described above can also be used to
construct a statistical test for a more general causal hypothesis. For
example, the causal hypothesis ``X is a causal ancestor of Y'' can in
principle be tested by applying the DAG verification method to all
DAGs satisfying this ancestor relationship. One then reports the
largest of the $p$-values appearing in step 2) of the DAG verification
method. This test has, asymptotically, the correct size if there is
indeed an underlying additive noise model that generated the data
(again, using sample splitting, for example).  Under a (minor) model
misspecifcation, that is, if the additive noise assumption does not
hold, we might still find $p$-values that are much larger for the
correct causal statement than for the reversed statement, e.g. ``Y is
a causal ancestor of X'' \citep[see][]{Peters2011tpami}.

\subsection{Results}

We structure the experimental results into five parts: level analysis,
power analysis, runtime analysis and causal inference on simulated and a
real data set.

\subsubsection{Level analysis} \label{subsec:level}

We consider two examples of fixed elements $\prodlawbold\in\HO$
(Simulation~1 and Simulation~2).  In both examples we simulate
$m=1000$ independent realizations of
$\vX_1,\dots,\vX_n\iid\prodlawbold$ for different sample sizes $n$ and
check how often each of the three hypothesis tests reject the null
hypothesis.
\begin{mdframed}[roundcorner=5pt, frametitle={Simulation 1 (testing
    level - three continuous variables)}, nobreak=true]
  Consider $X^1,X^2,X^3\iid\mathcal{N}(0,1)$, then for $\vX=(X^1,X^2,X^3)$ it holds that
  \begin{equation*}
    \prodlawbold=\P^{X^1}\otimes\P^{X^2}\otimes\P^{X^3}\in\HO,
  \end{equation*}
  where $\HO$ is the null hypothesis defined in \eqref{eq:HO}. Set
  $\alpha=0.05$, $B=25$, 
  and \linebreak
  $n\in\{100,200,\dots,1000\}$. The 
  rejection rates for the corresponding three hypothesis tests
  (permutation, bootstrap and Gamma approximation) based
  on $m=1000$ repetitions are plotted in Figure~\ref{fig:levelcomparison1}.
\end{mdframed}

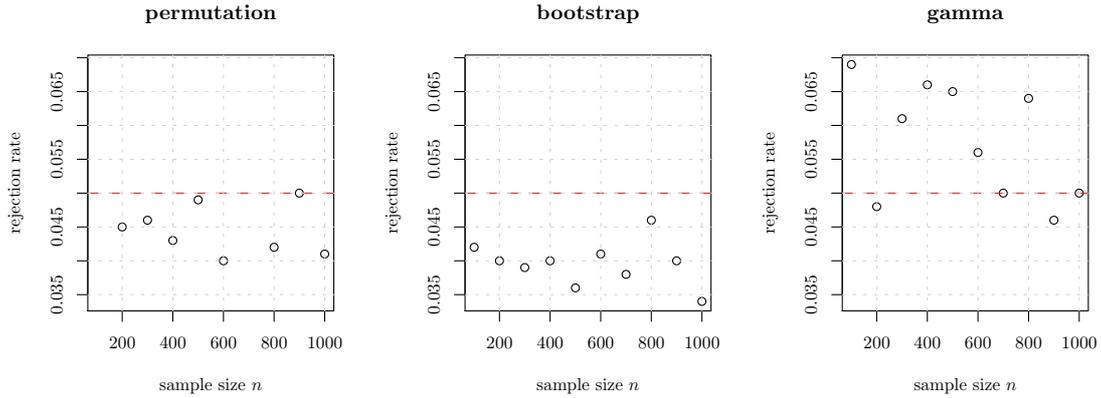
\begin{figure}[h!]
  \centering\vspace{-0.2cm}
  \input{Plots/levelcomparison1.tex}\vspace{-0.2cm}\\
  \caption{Simulation 1 (testing level - three continuous variables): Rejection rates, based on
    $m=1000$ repetitions, for each of the three hypothesis tests based
    on dHSIC. The test has valid level if the rejection rate does not lie
    far above the dotted red line at 0.05.}
  \label{fig:levelcomparison1}
\end{figure}

\begin{mdframed}[roundcorner=5pt, frametitle={Simulation 2 (testing
    level - continuous and discrete variables)}, nobreak=true]
  Consider $X^1\sim\mathcal{N}(0,1)$ and
  $X^2\sim\operatorname{Bin}(20,0.2)$ with $X^1$ and $ X^2$ independent. Then for $\vX=(X^1,X^2)$ it holds that
  \begin{equation*}
    \prodlawbold=\P^{X^1}\otimes\P^{X^2}\in\HO,
  \end{equation*}
  where $\HO$ is the null hypothesis defined in \eqref{eq:HO}. Set
  $\alpha=0.05$, $B=100$, 
    and \linebreak
    $n\in\{100,200,\dots,1000\}$. The
  rejection rates for the corresponding three hypothesis tests
  (permutation, bootstrap and Gamma approximation) based
  on $m=1000$ repetitions are plotted in Figure~\ref{fig:levelcomparison2}.
\end{mdframed}
\vspace{0.8cm}

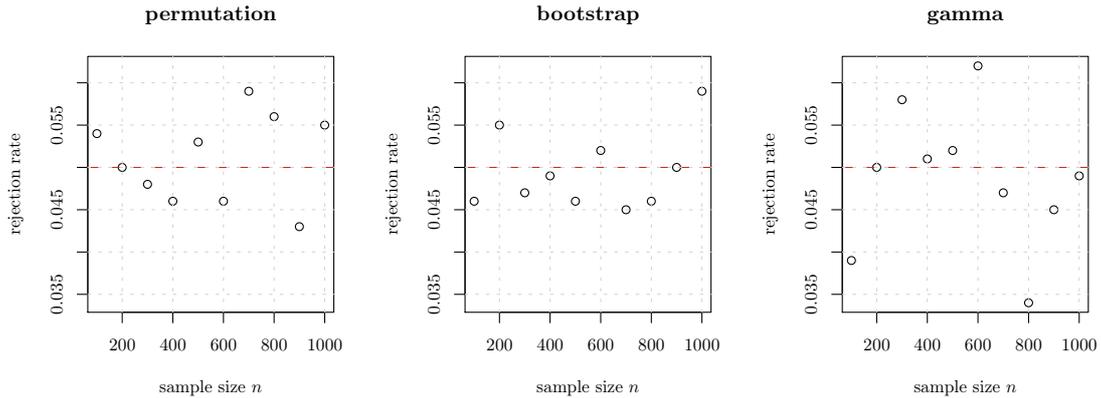
\begin{figure}[h!]
  \centering\vspace{-1cm}
  \input{Plots/levelcomparison2.tex}\vspace{-0.2cm}\\
  \caption{Simulation 2 (testing level - continuous and discrete variables): Rejection rates, based on
    $m=1000$ repetitions, for each of the three hypothesis tests based
    on dHSIC. The test has valid level if the rejection rate does not lie
    far above the dotted red line at 0.05.}
  \label{fig:levelcomparison2}
\end{figure}

In both simulations we get similar results. We collect the most
important observations.
\begin{compactitem}
\item[(i)] The permutation test achieves level $\alpha$. This corresponds
  to what has been proved in the previous section. As mentioned above,
  this result is rather surprising as it does
  not depend on the choice of $B$, which in Simulation~1 is very
  small ($B=25$).
\item[(ii)] The bootstrap test achieves level $\alpha$ in most
  cases, even though we only proved that it has pointwise asymptotic
  level. This is due to the conservative choice of the $p$-value in the
  Monte-Carlo approximation of the bootstrap test.
\item[(iii)] The Gamma approximation based test, at least in these two
  examples, has level close to $\alpha$ but often slightly exceeds the required level. 
For larger values of $d$ the Gamma approximation seems
  to break down. For instance, if we perform Simulation~1 with $10$
  variables instead of three the rejection rate for a sample size of
  $n=100$ is $0.40$ and even for $n=200$ it is still
  $0.21$. The bootstrap test on the other hand is not affected in this
  way (in the same setting we get $0.03$ for $n=100$ and $0.04$ for $n=200$).
\end{compactitem}

\subsubsection{Power analysis}
Assessing the power of a test requires to chose an alternative. Here,
we show two examples, one favoring dHSIC (Simulation~4), another one
favoring the multiple testing approach using HSIC $d-1$ times
(Simulation~3). In both simulations we use the BMR-$C$ test with $C=n$
as reference. Using a BMR-$C$ test with $C=1000$ (not shown here) only
brings marginal improvements which are not sufficient to beat HSIC in
either simulation.

\begin{mdframed}[roundcorner=5pt, frametitle={Simulation 3 (comparing
    power - single edge)}, nobreak=true]
For an additive noise model over random variables $X^1, \ldots, X^d$,
$$
X^j := \sum_{k \in \parents^j} f^{j,k}(X^k) + N^j, \quad j \in \{1, \ldots, d\},
$$
with corresponding DAG $\mathcal{G}$, we sample data in the following
way.  The noise variables are Gaussian with a standard deviation
sampled uniformly between $\sqrt{2}$ and $2$.  Nodes without parents
follow a Gaussian distribution with standard deviation sampled
uniformly between $5\sqrt{2}$ and $5\cdot 2$.  The functions $f^{j,k}$
are sampled from a Gaussian process with Gaussian kernel and bandwidth
one.  In this Simulation~3 we choose $d=4$, let $\mathcal{G}$ to be
the graph that contains $1 \rightarrow 2$ as a single edge and use
$m=1000$ repetitions to compute rejection rates, see
Figure~\ref{fig:power1}.  We expect this setting to favor the
multiple testing approach: due to the order of the
variables, it tests $X^1$ against $X^2$.
\end{mdframed}
\begin{figure}[h!]
  \centering\vspace{-0.6cm}
  \input{Plots/power_analysis1_romano.tex}\vspace{-0.2cm}\\
  \caption{Simulation 3 (comparing power - single edge): Rejection
    rates, based on $m=1000$ repetitions, for each of the three
    different hypothesis tests. The example (in particular the chosen
    order of variables) is constructed to favor the pairwise testing
    approach (HSIC). Nevertheless, it performs only slightly better
    than dHSIC.}
  \label{fig:power1}
\end{figure}
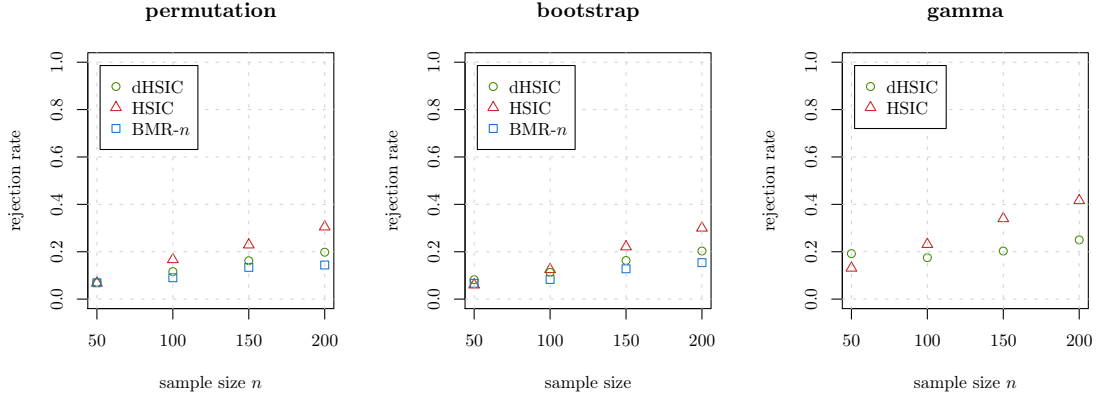


\begin{mdframed}[roundcorner=5pt, frametitle={Simulation 4 (comparing
    power - full DAG)}, nobreak=true]
  We simulate the data as described in Simulation~3 but this time
  using a (randomly chosen) full DAG $\mathcal{G}$ over $d=4$
  variables, i.e. every pair of two nodes is connected, see
  Figure~\ref{fig:power2}.  We expect that this setting favors
  dHSIC. Additionally, we fixed $n=100$, varied $d$ and used $m=1000$
  repetitions, see Figure~\ref{fig:curseofdim}.
\end{mdframed}

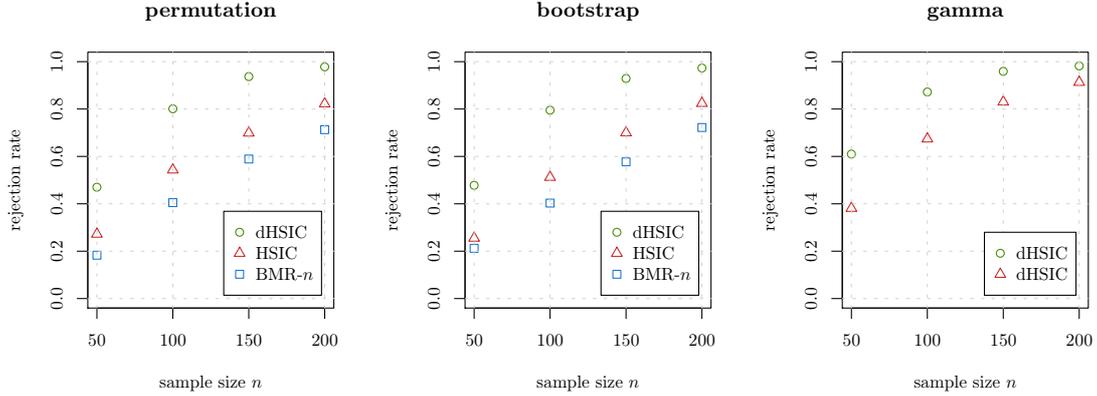
\begin{figure}[h!]
  \centering\vspace{-0.6cm}
  \input{Plots/power_analysis2_romano.tex}\vspace{-0.2cm}\\
  \caption{Simulation 4 (comparing power - full DAG): Rejection rates, based on
    $m=1000$ repetitions, for each of the three different hypothesis tests. 
    As expected, dHSIC outperforms the competing method HSIC that is based on pairwise independence tests.}
  \label{fig:power2}
\end{figure}





Next, we compare the empirical power properties of dHSIC for sparse
and dense alternatives (Simulation~5). In order to correct for the
dependence strength, we use the total variation
distance. Figure~\ref{fig:power_dense_sparse} shows that given a
sufficient amount of dependence the empirical dHSIC appears to have
more power against sparse alternatives.

\begin{mdframed}[roundcorner=5pt, frametitle={Simulation 5 (comparing
    power - dense and sparse alternatives)}, nobreak=true]
  We simulate data according to two SEMs given by
  \begin{equation*}
    \mathcal{S}_{\operatorname{dense},c}:
    \begin{cases}
      X^1&=H^2+c\cdot\epsilon_1\\
      X^2&=H^2+c\cdot\epsilon_2\\
      X^3&=H^2+c\cdot\epsilon_3\\
      &\vdots\\
      X^d&=H^2+c\cdot\epsilon_d
    \end{cases}
    \qquad\qquad
    \mathcal{S}_{\operatorname{sparse},c}:
    \begin{cases}
      X^1&=H^2+c\cdot\epsilon_1\\
      X^2&=H^2+c\cdot\epsilon_2\\
      X^3&=\epsilon_3\\
      &\vdots\\
      X^d&=\epsilon_d,
    \end{cases}
  \end{equation*}
  where $H,\epsilon_1,\dots,\epsilon_d\iid\mathcal{N}(0,1)$. Then, for
  different values of $c$ we compute both the total variation distance
  $\sup_{\vR}\abs{\prodlawbold(\vR)-\prodlaw(\vR)}$ and the rejection
  rate for $n=100$ of the dHSIC permutation test ($B=100$) based on
  $m=10000$ repetitions. The total variation is computed using
  numerical integration for the distribution functions and a
  Monte-Carlo approximation (with $C=10000$) to evaluate the
  supremum. The resulting plots for the cases $d=5$ and $d=10$ are
  given in Figure~\ref{fig:power_dense_sparse}.
\end{mdframed}

\begin{figure}[h!]
  \centering\vspace{-0.2cm}
  \resizebox{0.7\textwidth}{!}{\input{Plots/power_dense_sparse.tex}}\vspace{-0.2cm}
  \caption{Simulation 5 (comparing power - dense and sparse
    alternatives): Rejection rates, based on $m=10000$ repetitions, of
    the permutation test ($B=100$) for data from a dense SEM
    ($\mathcal{S}_{\operatorname{dense},c}$) and a sparse SEM
    ($\mathcal{S}_{\operatorname{sparse},c}$), where $c$ is tuned such
    that both settings have similar total variation distances.}
  \label{fig:power_dense_sparse}
\end{figure}
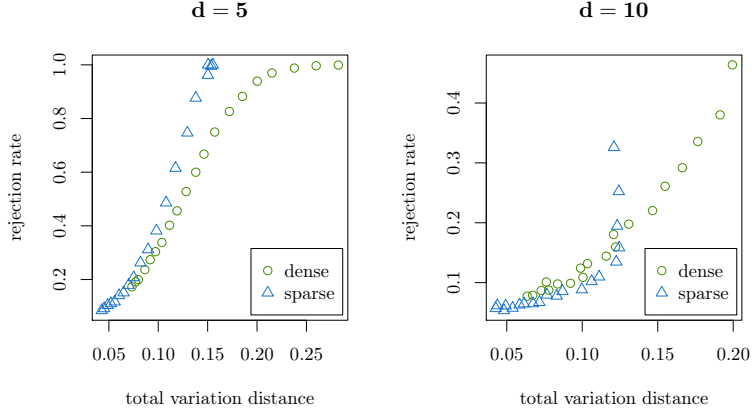

Finally, we analyze the influence of the choice of kernel on the
empirical power (Simulation~6). In this paper, we have mainly used the
Gaussian kernel with a bandwidth chosen by the median heuristic. As mentioned in
Section~\ref{sec:choice_kernel} this choice is not necessarily optimal. Using the Taylor expansion of the Gaussian
kernel we get for all $\vx,\vy\in\R^d$ that
\begin{equation*}
  \prodk(\vx,\vy)=1-\frac{1}{2\sigma^2}\sum_{j=1}^d\big(x^j-y^j\big)^2+\frac{1}{4\sigma^4}\sum_{j,k=1}^d\big(x^j-y^j\big)^2\big(x^k-y^k\big)^2+\landauO{\sigma^{-6}},
\end{equation*}
as $\sigma\rightarrow\infty$. Therefore, it can be shown using either
the representation in Definition~\ref{def:dHSIC} or the one in
Proposition~\ref{thm:mhsickernel} that for large $\sigma$ dHSIC using
the Gaussian kernel is approximately given by dHSIC using the
following kernel
$$
  \tilde{\prodk}(\vx,\vy)\coloneqq\frac{1}{4\sigma^4}\sum_{j,k=1}^d\big(x^j-y^j\big)^2\big(x^k-y^k\big)^2.
$$
Such a kernel can, however, only detect pairwise dependence
structures, and since the importance of this term becomes more
prominent as the size of the bandwidth increases we expect the power
of our dHSIC test to decrease when analyzing dependencies that have an
additional dependence structure beyond a pairwise dependence.
The following simulation illustrates this empirically based on three
dependencies: a pairwise dependence, a more complex dependence
due to a random nonlinear Gaussian SEM and a dependence on three
variables which is pairwise independent, see
Figure~\ref{fig:power_bandwidth}.

\begin{mdframed}[roundcorner=5pt, frametitle={Simulation 6 (comparing
    power - bandwidth)}, nobreak=true]
  We consider three dependencies and analyze the behavior of the
  empirical power of the dHSIC permutation test ($B=100$) based on
  different bandwidths for the Gaussian kernel. The first is
  generated by the linear Gaussian SEM
  \begin{equation*}
    \begin{cases}
      X^1&=H+\epsilon_1\\
      &\vdots\\
      X^4&=H+\epsilon_4
    \end{cases}
    \quad\text{where }H,\epsilon_1,\dots,\epsilon_4\iid\mathcal{N}(0,4),
  \end{equation*}
  and hence the only dependence is due to the confounder $H$. For the
  second dependence we use the same as in Simulation~4, which has a more
  evolved dependence structure due to potential chains of
  ancestors. The third dependence is given by the
  following probability density
  \begin{equation*}
    f(x^1,x^2,x^3)=
    \begin{cases}
      2\phi(x^1)\phi(x^2)\phi(x^3)\quad &\text{if } x^1,x^2,x^3\geq
      0,\text{ or }\exists !\, j\in\{1,2,3\}: x^j\geq 0,\\
      0\quad &\text{otherwise},
    \end{cases}
  \end{equation*}
  where $\phi$ is the standard normal density. The resulting
  distribution is, in particular, pairwise independent. For all
  examples we use a sample size of $n=100$ and $m=1000$ repetitions.
\end{mdframed}

\begin{figure}[h!]
  \centering\vspace{-0.2cm}
  \input{Plots/power_bandwidth.tex}\vspace{-0.2cm}
  \caption{Simulation 6 (comparing power - bandwidth): Rejection rates
    with $n=100$ for different bandwidths $\sigma$ in the Gaussian
    kernel, based on $m=1000$ repetitions, of the permutation test
    ($B=100$) for data containing only a pairwise dependence, for data
    from a random nonlinear Gaussian SEM and for data which is
    dependent but pairwise independent. The red lines are
    $95\%$-confidence intervals for the bandwidth selected using the
    median heuristic. The rejection rates resulting when using the
    median heuristic are $0.30$, $0.82$ and $1$ respectively.}
  \label{fig:power_bandwidth}
\end{figure}
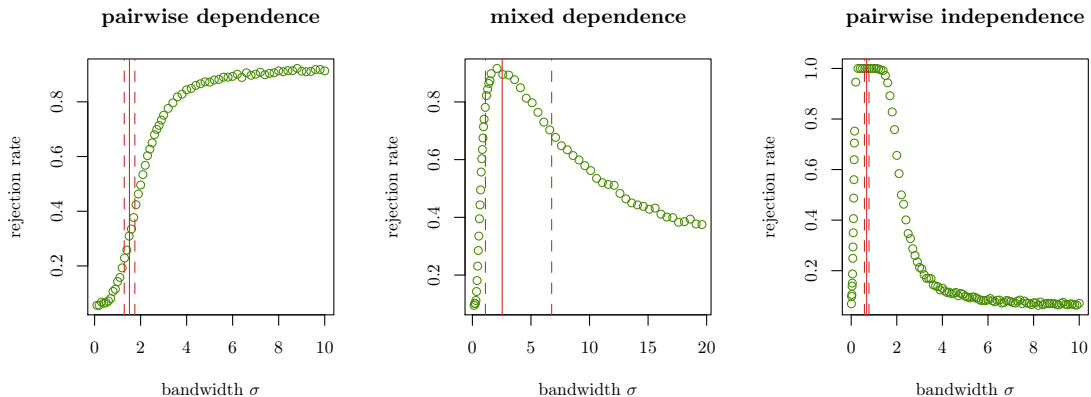

\subsubsection{Runtime analysis} \label{sec:computational}
The computational complexity for the dHSIC test statistic is
$ \landauO{dn^2} $ as can be seen from the considerations in
Section~\ref{sec:implementation_dhsic}. The multiple testing approach
for HSIC computes HSIC $d-1$ times, which appears to result in the
same computational complexity. But since the dimension of the input
variables for the HSIC tests generally depends on $d$, as well (at least in
common settings such as for the Gaussian kernel), the overall
complexity is $\landauO{d^2n^2}$.  We numerically test these
computational complexities by two simulations. In the first simulation
we fix $n$ and vary $d$; in the second simulation we fix $d$ and vary
$n$. The results are presented in Figure~\ref{fig:runtime}.


\begin{figure}[h!]
  \vspace{-0.2cm}
  \begin{minipage}[t]{.6\textwidth}
    \centering
    \resizebox{\textwidth}{!}{\input{Plots/runtime.tex}}
  \end{minipage}%
  \vspace{-0.2cm}
  \begin{minipage}[t]{.4\textwidth}
    \centering
    \resizebox{\textwidth}{!}{\input{Plots/power_analysis_extra.tex}}
  \end{minipage}
  \begin{minipage}[t]{.55\textwidth}
    \caption{runtime analysis: (left) varying number of variables and fixed
        sample size ($n=100$) and (right) varying sample size and fixed
        number of variables ($d=10$).}
    \label{fig:runtime}
  \end{minipage}\hspace{.05\textwidth}%
  \vspace{-0.2cm}
  \begin{minipage}[t]{.4\textwidth}
    \caption{Simulation 4 (comparing power - full DAG):
      Rejection rates, based on bootstrap ($B=100$, $n=100$). BMR-$C$
      suffers from the curse of dimensionality and the pairwise HSIC
      approach is not able to reject $\HO$ for $d>6$.}
    \label{fig:curseofdim}
  \end{minipage}
\end{figure}

\subsubsection{Causal inference (simulated data)} \label{sec:exp:caus}
We now apply both tests to the DAG verification method described in
Section~\ref{subsec.causal}.  As in Simulation~3, we simulate data
from an additive noise model. Here, we randomly choose a fully
connected DAG $\mathcal{G}$ over $d=4$ nodes and choose Gaussian
distributed noise variables with standard deviation sampled uniformly
between $1/5$ and $\sqrt{2}/5$ instead of $\sqrt{2}$ and $2$. We then
report how often (out of $m=1000$), the largest $p$-value leads to the
correct DAG. Because of its computational advantage, we use the tests
based on the Gamma approximation for dHSIC and the pairwise HSIC,
which work reasonably well for four nodes (strictly speaking, we only
use the relative size of the $p$-values). Most of the time was spent
computing the results for BMR-$n$ as we were forced to use a bootstrap
test ($B=100$) since no approximation is available for this test. The
proposed dHSIC recovers the correct DAG in more cases than the
pairwise approach and in even more cases than BMR-$n$, see
Figure~\ref{fig:causalex}.
\begin{figure}[h!]
  \centering\vspace{-1.6cm}
  \resizebox{0.7\textwidth}{!}{\input{Plots/causality_example_romano.tex}} \vspace{-0.2cm} 
  \caption{
   Causal inference example ($m=1000$ repetitions): (left)
    shows how often each method estimated the correct DAG, (right)
    shows the average structural intervention distance (SID, small is good) between
    the correct and estimated DAG \citep{petersSID15}.}
  \label{fig:causalex}
\end{figure}
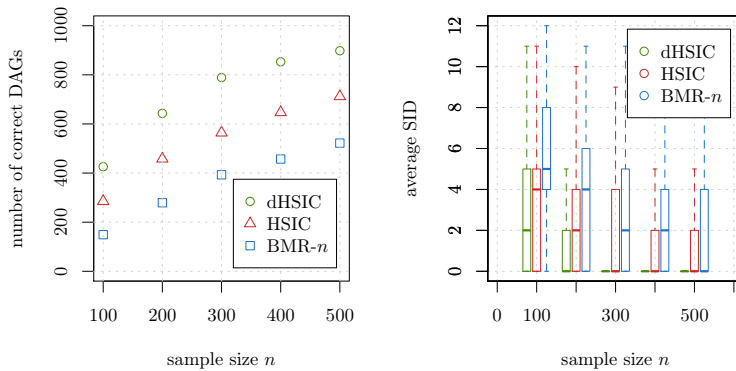

\subsubsection{Causal inference (real data example)}

We now apply the DAG verification method  
(see Section~\ref{subsec.causal})
to real world data. Given $349$ measurements of the variables Altitude,
Temperature and Sunshine\footnote{The dataset is taken from
  \citet[pair0001.txt and pair0004.txt]{dataset}.},  we try to determine the
correct causal structure out of 
25 possible DAGs. 
We use permutation based versions (with $B=1000$) of the dHSIC test,
the multiple testing approach for HSIC and the BMR-$1000$ test and
apply them to every possible DAG and compare the resulting
$p$-values. The result is shown in Figure~\ref{fig:realworld}~(left).
\tikzstyle{VertexStyle} = [shape = ellipse, minimum width = 11em,
draw] \tikzstyle{EdgeStyle} = [->,>=stealth',shorten > = 2pt]
\begin{figure}[h!]
  \centering\vspace{-0.6cm}
  \begin{subfigure}{0.4\textwidth}
    \hspace{-1cm}
    \resizebox{\textwidth}{!}{
    \input{Plots/realworld_romano.tex}}
  \end{subfigure}%
  \begin{subfigure}{0.4\textwidth}
    \raisebox{.14\textwidth}{\resizebox{\textwidth}{!}{
    \hspace{-1.4cm}\begin{tikzpicture}[scale=1.4] 
      \SetGraphUnit{2} 
      \Vertex{Altitude} 
      \SOWE(Altitude){Sunshine} \SOEA(Altitude){Temperature}
      \Edges(Altitude,Temperature,Sunshine) \Edges(Altitude,Sunshine)
    \end{tikzpicture}}}
  \end{subfigure}
  \vspace{-0.2cm}
  \caption{Real world data example: The left plot shows the $p$-values (on
    log-scale), for each
    DAG over three nodes, from the DAG verification
    method. Even for small $p$-value thresholds, dHISC is able to
    reject all 
    incorrect models, while the competing HSIC method cannot. The
    graphical representation of DAG 25 is shown on the right.}
  \label{fig:realworld}
\end{figure}
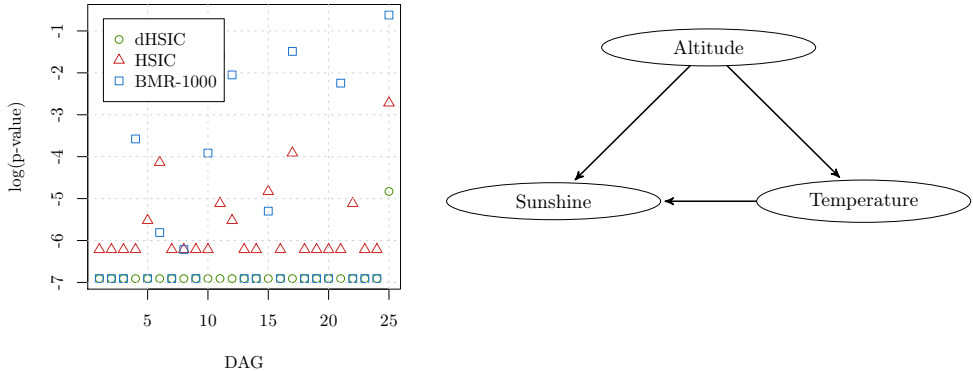

Figure~\ref{fig:realworld}~(right) shows DAG 25 that obtained the largest $p$-value. 
Based on physical background knowledge,
we expect altitude to effect both sunshine and temperature. The effect of
temperature on sunshine could be explained by intermediate latent
variables such as clouds or fog.

The plot illustrates that the dHSIC based test is able to reject
all incorrect models, even for very low $p$-value thresholds. The competing HSIC
and BMR-$1000$ methods on the other hand are not able to reject all incorrect DAGs, for
example DAG 12 has a $p$-value of about $0.01$ but contains an edge
from Sunshine to Altitude, which is clearly the wrong causal direction.





%% file: Plots/levelcomparison1.tex
\begin{tikzpicture}[x=1pt,y=1pt]
\definecolor{fillColor}{RGB}{255,255,255}
\path[use as bounding box,fill=fillColor,fill opacity=0.00] (0,0) rectangle (423.50,169.40);
\begin{scope}
\path[clip] ( 32.47, 40.39) rectangle (124.54,136.93);
\definecolor{drawColor}{RGB}{0,0,0}

\path[draw=drawColor,line width= 0.4pt,line join=round,line cap=round] ( 35.88, 18.43) circle (  1.49);

\path[draw=drawColor,line width= 0.4pt,line join=round,line cap=round] ( 45.35, 72.06) circle (  1.49);

\path[draw=drawColor,line width= 0.4pt,line join=round,line cap=round] ( 54.82, 74.61) circle (  1.49);

\path[draw=drawColor,line width= 0.4pt,line join=round,line cap=round] ( 64.30, 66.95) circle (  1.49);

\path[draw=drawColor,line width= 0.4pt,line join=round,line cap=round] ( 73.77, 82.28) circle (  1.49);

\path[draw=drawColor,line width= 0.4pt,line join=round,line cap=round] ( 83.24, 59.29) circle (  1.49);

\path[draw=drawColor,line width= 0.4pt,line join=round,line cap=round] ( 92.71, 36.31) circle (  1.49);

\path[draw=drawColor,line width= 0.4pt,line join=round,line cap=round] (102.18, 64.40) circle (  1.49);

\path[draw=drawColor,line width= 0.4pt,line join=round,line cap=round] (111.65, 84.83) circle (  1.49);

\path[draw=drawColor,line width= 0.4pt,line join=round,line cap=round] (121.13, 61.84) circle (  1.49);
\end{scope}
\begin{scope}
\path[clip] (  0.00,  0.00) rectangle (423.50,169.40);
\definecolor{drawColor}{RGB}{0,0,0}

\path[draw=drawColor,line width= 0.4pt,line join=round,line cap=round] ( 45.35, 40.39) -- (121.13, 40.39);

\path[draw=drawColor,line width= 0.4pt,line join=round,line cap=round] ( 45.35, 40.39) -- ( 45.35, 36.43);

\path[draw=drawColor,line width= 0.4pt,line join=round,line cap=round] ( 64.30, 40.39) -- ( 64.30, 36.43);

\path[draw=drawColor,line width= 0.4pt,line join=round,line cap=round] ( 83.24, 40.39) -- ( 83.24, 36.43);

\path[draw=drawColor,line width= 0.4pt,line join=round,line cap=round] (102.18, 40.39) -- (102.18, 36.43);

\path[draw=drawColor,line width= 0.4pt,line join=round,line cap=round] (121.13, 40.39) -- (121.13, 36.43);

\node[text=drawColor,anchor=base,inner sep=0pt, outer sep=0pt, scale=  0.66] at ( 45.35, 26.14) {200};

\node[text=drawColor,anchor=base,inner sep=0pt, outer sep=0pt, scale=  0.66] at ( 64.30, 26.14) {400};

\node[text=drawColor,anchor=base,inner sep=0pt, outer sep=0pt, scale=  0.66] at ( 83.24, 26.14) {600};

\node[text=drawColor,anchor=base,inner sep=0pt, outer sep=0pt, scale=  0.66] at (102.18, 26.14) {800};

\node[text=drawColor,anchor=base,inner sep=0pt, outer sep=0pt, scale=  0.66] at (121.13, 26.14) {1000};

\path[draw=drawColor,line width= 0.4pt,line join=round,line cap=round] ( 32.47, 46.52) -- ( 32.47,135.91);

\path[draw=drawColor,line width= 0.4pt,line join=round,line cap=round] ( 32.47, 46.52) -- ( 28.51, 46.52);

\path[draw=drawColor,line width= 0.4pt,line join=round,line cap=round] ( 32.47, 59.29) -- ( 28.51, 59.29);

\path[draw=drawColor,line width= 0.4pt,line join=round,line cap=round] ( 32.47, 72.06) -- ( 28.51, 72.06);

\path[draw=drawColor,line width= 0.4pt,line join=round,line cap=round] ( 32.47, 84.83) -- ( 28.51, 84.83);

\path[draw=drawColor,line width= 0.4pt,line join=round,line cap=round] ( 32.47, 97.60) -- ( 28.51, 97.60);

\path[draw=drawColor,line width= 0.4pt,line join=round,line cap=round] ( 32.47,110.37) -- ( 28.51,110.37);

\path[draw=drawColor,line width= 0.4pt,line join=round,line cap=round] ( 32.47,123.14) -- ( 28.51,123.14);

\path[draw=drawColor,line width= 0.4pt,line join=round,line cap=round] ( 32.47,135.91) -- ( 28.51,135.91);

\node[text=drawColor,rotate= 90.00,anchor=base,inner sep=0pt, outer sep=0pt, scale=  0.66] at ( 22.97, 46.52) {0.035};

\node[text=drawColor,rotate= 90.00,anchor=base,inner sep=0pt, outer sep=0pt, scale=  0.66] at ( 22.97, 72.06) {0.045};

\node[text=drawColor,rotate= 90.00,anchor=base,inner sep=0pt, outer sep=0pt, scale=  0.66] at ( 22.97, 97.60) {0.055};

\node[text=drawColor,rotate= 90.00,anchor=base,inner sep=0pt, outer sep=0pt, scale=  0.66] at ( 22.97,123.14) {0.065};

\path[draw=drawColor,line width= 0.4pt,line join=round,line cap=round] ( 32.47, 40.39) --
	(124.54, 40.39) --
	(124.54,136.93) --
	( 32.47,136.93) --
	( 32.47, 40.39);
\end{scope}
\begin{scope}
\path[clip] (  0.00,  0.00) rectangle (141.17,169.40);
\definecolor{drawColor}{RGB}{0,0,0}

\node[text=drawColor,anchor=base,inner sep=0pt, outer sep=0pt, scale=  0.79] at ( 78.50,150.43) {\bfseries permutation};

\node[text=drawColor,anchor=base,inner sep=0pt, outer sep=0pt, scale=  0.66] at ( 78.50, 10.30) {sample size $n$};

\node[text=drawColor,rotate= 90.00,anchor=base,inner sep=0pt, outer sep=0pt, scale=  0.66] at (  7.13, 88.66) {rejection rate};
\end{scope}
\begin{scope}
\path[clip] ( 32.47, 40.39) rectangle (124.54,136.93);
\definecolor{drawColor}{RGB}{205,38,38}

\path[draw=drawColor,line width= 0.4pt,dash pattern=on 4pt off 4pt ,line join=round,line cap=round] ( 32.47, 84.83) -- (124.54, 84.83);
\definecolor{drawColor}{RGB}{211,211,211}

\path[draw=drawColor,line width= 0.4pt,dash pattern=on 1pt off 3pt ,line join=round,line cap=round] ( 45.35, 40.39) -- ( 45.35,136.93);

\path[draw=drawColor,line width= 0.4pt,dash pattern=on 1pt off 3pt ,line join=round,line cap=round] ( 64.30, 40.39) -- ( 64.30,136.93);

\path[draw=drawColor,line width= 0.4pt,dash pattern=on 1pt off 3pt ,line join=round,line cap=round] ( 83.24, 40.39) -- ( 83.24,136.93);

\path[draw=drawColor,line width= 0.4pt,dash pattern=on 1pt off 3pt ,line join=round,line cap=round] (102.18, 40.39) -- (102.18,136.93);

\path[draw=drawColor,line width= 0.4pt,dash pattern=on 1pt off 3pt ,line join=round,line cap=round] (121.13, 40.39) -- (121.13,136.93);

\path[draw=drawColor,line width= 0.4pt,dash pattern=on 1pt off 3pt ,line join=round,line cap=round] ( 32.47, 46.52) -- (124.54, 46.52);

\path[draw=drawColor,line width= 0.4pt,dash pattern=on 1pt off 3pt ,line join=round,line cap=round] ( 32.47, 59.29) -- (124.54, 59.29);

\path[draw=drawColor,line width= 0.4pt,dash pattern=on 1pt off 3pt ,line join=round,line cap=round] ( 32.47, 72.06) -- (124.54, 72.06);

\path[draw=drawColor,line width= 0.4pt,dash pattern=on 1pt off 3pt ,line join=round,line cap=round] ( 32.47, 84.83) -- (124.54, 84.83);

\path[draw=drawColor,line width= 0.4pt,dash pattern=on 1pt off 3pt ,line join=round,line cap=round] ( 32.47, 97.60) -- (124.54, 97.60);

\path[draw=drawColor,line width= 0.4pt,dash pattern=on 1pt off 3pt ,line join=round,line cap=round] ( 32.47,110.37) -- (124.54,110.37);

\path[draw=drawColor,line width= 0.4pt,dash pattern=on 1pt off 3pt ,line join=round,line cap=round] ( 32.47,123.14) -- (124.54,123.14);

\path[draw=drawColor,line width= 0.4pt,dash pattern=on 1pt off 3pt ,line join=round,line cap=round] ( 32.47,135.91) -- (124.54,135.91);
\end{scope}
\begin{scope}
\path[clip] (173.64, 40.39) rectangle (265.70,136.93);
\definecolor{drawColor}{RGB}{0,0,0}

\path[draw=drawColor,line width= 0.4pt,line join=round,line cap=round] (177.05, 64.40) circle (  1.49);

\path[draw=drawColor,line width= 0.4pt,line join=round,line cap=round] (186.52, 59.29) circle (  1.49);

\path[draw=drawColor,line width= 0.4pt,line join=round,line cap=round] (195.99, 56.74) circle (  1.49);

\path[draw=drawColor,line width= 0.4pt,line join=round,line cap=round] (205.46, 59.29) circle (  1.49);

\path[draw=drawColor,line width= 0.4pt,line join=round,line cap=round] (214.94, 49.08) circle (  1.49);

\path[draw=drawColor,line width= 0.4pt,line join=round,line cap=round] (224.41, 61.84) circle (  1.49);

\path[draw=drawColor,line width= 0.4pt,line join=round,line cap=round] (233.88, 54.18) circle (  1.49);

\path[draw=drawColor,line width= 0.4pt,line join=round,line cap=round] (243.35, 74.61) circle (  1.49);

\path[draw=drawColor,line width= 0.4pt,line join=round,line cap=round] (252.82, 59.29) circle (  1.49);

\path[draw=drawColor,line width= 0.4pt,line join=round,line cap=round] (262.29, 43.97) circle (  1.49);
\end{scope}
\begin{scope}
\path[clip] (  0.00,  0.00) rectangle (423.50,169.40);
\definecolor{drawColor}{RGB}{0,0,0}

\path[draw=drawColor,line width= 0.4pt,line join=round,line cap=round] (186.52, 40.39) -- (262.29, 40.39);

\path[draw=drawColor,line width= 0.4pt,line join=round,line cap=round] (186.52, 40.39) -- (186.52, 36.43);

\path[draw=drawColor,line width= 0.4pt,line join=round,line cap=round] (205.46, 40.39) -- (205.46, 36.43);

\path[draw=drawColor,line width= 0.4pt,line join=round,line cap=round] (224.41, 40.39) -- (224.41, 36.43);

\path[draw=drawColor,line width= 0.4pt,line join=round,line cap=round] (243.35, 40.39) -- (243.35, 36.43);

\path[draw=drawColor,line width= 0.4pt,line join=round,line cap=round] (262.29, 40.39) -- (262.29, 36.43);

\node[text=drawColor,anchor=base,inner sep=0pt, outer sep=0pt, scale=  0.66] at (186.52, 26.14) {200};

\node[text=drawColor,anchor=base,inner sep=0pt, outer sep=0pt, scale=  0.66] at (205.46, 26.14) {400};

\node[text=drawColor,anchor=base,inner sep=0pt, outer sep=0pt, scale=  0.66] at (224.41, 26.14) {600};

\node[text=drawColor,anchor=base,inner sep=0pt, outer sep=0pt, scale=  0.66] at (243.35, 26.14) {800};

\node[text=drawColor,anchor=base,inner sep=0pt, outer sep=0pt, scale=  0.66] at (262.29, 26.14) {1000};

\path[draw=drawColor,line width= 0.4pt,line join=round,line cap=round] (173.64, 46.52) -- (173.64,135.91);

\path[draw=drawColor,line width= 0.4pt,line join=round,line cap=round] (173.64, 46.52) -- (169.68, 46.52);

\path[draw=drawColor,line width= 0.4pt,line join=round,line cap=round] (173.64, 59.29) -- (169.68, 59.29);

\path[draw=drawColor,line width= 0.4pt,line join=round,line cap=round] (173.64, 72.06) -- (169.68, 72.06);

\path[draw=drawColor,line width= 0.4pt,line join=round,line cap=round] (173.64, 84.83) -- (169.68, 84.83);

\path[draw=drawColor,line width= 0.4pt,line join=round,line cap=round] (173.64, 97.60) -- (169.68, 97.60);

\path[draw=drawColor,line width= 0.4pt,line join=round,line cap=round] (173.64,110.37) -- (169.68,110.37);

\path[draw=drawColor,line width= 0.4pt,line join=round,line cap=round] (173.64,123.14) -- (169.68,123.14);

\path[draw=drawColor,line width= 0.4pt,line join=round,line cap=round] (173.64,135.91) -- (169.68,135.91);

\node[text=drawColor,rotate= 90.00,anchor=base,inner sep=0pt, outer sep=0pt, scale=  0.66] at (164.14, 46.52) {0.035};

\node[text=drawColor,rotate= 90.00,anchor=base,inner sep=0pt, outer sep=0pt, scale=  0.66] at (164.14, 72.06) {0.045};

\node[text=drawColor,rotate= 90.00,anchor=base,inner sep=0pt, outer sep=0pt, scale=  0.66] at (164.14, 97.60) {0.055};

\node[text=drawColor,rotate= 90.00,anchor=base,inner sep=0pt, outer sep=0pt, scale=  0.66] at (164.14,123.14) {0.065};

\path[draw=drawColor,line width= 0.4pt,line join=round,line cap=round] (173.64, 40.39) --
	(265.70, 40.39) --
	(265.70,136.93) --
	(173.64,136.93) --
	(173.64, 40.39);
\end{scope}
\begin{scope}
\path[clip] (141.17,  0.00) rectangle (282.33,169.40);
\definecolor{drawColor}{RGB}{0,0,0}

\node[text=drawColor,anchor=base,inner sep=0pt, outer sep=0pt, scale=  0.79] at (219.67,150.43) {\bfseries bootstrap};

\node[text=drawColor,anchor=base,inner sep=0pt, outer sep=0pt, scale=  0.66] at (219.67, 10.30) {sample size $n$};

\node[text=drawColor,rotate= 90.00,anchor=base,inner sep=0pt, outer sep=0pt, scale=  0.66] at (148.30, 88.66) {rejection rate};
\end{scope}
\begin{scope}
\path[clip] (173.64, 40.39) rectangle (265.70,136.93);
\definecolor{drawColor}{RGB}{205,38,38}

\path[draw=drawColor,line width= 0.4pt,dash pattern=on 4pt off 4pt ,line join=round,line cap=round] (173.64, 84.83) -- (265.70, 84.83);
\definecolor{drawColor}{RGB}{211,211,211}

\path[draw=drawColor,line width= 0.4pt,dash pattern=on 1pt off 3pt ,line join=round,line cap=round] (186.52, 40.39) -- (186.52,136.93);

\path[draw=drawColor,line width= 0.4pt,dash pattern=on 1pt off 3pt ,line join=round,line cap=round] (205.46, 40.39) -- (205.46,136.93);

\path[draw=drawColor,line width= 0.4pt,dash pattern=on 1pt off 3pt ,line join=round,line cap=round] (224.41, 40.39) -- (224.41,136.93);

\path[draw=drawColor,line width= 0.4pt,dash pattern=on 1pt off 3pt ,line join=round,line cap=round] (243.35, 40.39) -- (243.35,136.93);

\path[draw=drawColor,line width= 0.4pt,dash pattern=on 1pt off 3pt ,line join=round,line cap=round] (262.29, 40.39) -- (262.29,136.93);

\path[draw=drawColor,line width= 0.4pt,dash pattern=on 1pt off 3pt ,line join=round,line cap=round] (173.64, 46.52) -- (265.70, 46.52);

\path[draw=drawColor,line width= 0.4pt,dash pattern=on 1pt off 3pt ,line join=round,line cap=round] (173.64, 59.29) -- (265.70, 59.29);

\path[draw=drawColor,line width= 0.4pt,dash pattern=on 1pt off 3pt ,line join=round,line cap=round] (173.64, 72.06) -- (265.70, 72.06);

\path[draw=drawColor,line width= 0.4pt,dash pattern=on 1pt off 3pt ,line join=round,line cap=round] (173.64, 84.83) -- (265.70, 84.83);

\path[draw=drawColor,line width= 0.4pt,dash pattern=on 1pt off 3pt ,line join=round,line cap=round] (173.64, 97.60) -- (265.70, 97.60);

\path[draw=drawColor,line width= 0.4pt,dash pattern=on 1pt off 3pt ,line join=round,line cap=round] (173.64,110.37) -- (265.70,110.37);

\path[draw=drawColor,line width= 0.4pt,dash pattern=on 1pt off 3pt ,line join=round,line cap=round] (173.64,123.14) -- (265.70,123.14);

\path[draw=drawColor,line width= 0.4pt,dash pattern=on 1pt off 3pt ,line join=round,line cap=round] (173.64,135.91) -- (265.70,135.91);
\end{scope}
\begin{scope}
\path[clip] (314.81, 40.39) rectangle (406.87,136.93);
\definecolor{drawColor}{RGB}{0,0,0}

\path[draw=drawColor,line width= 0.4pt,line join=round,line cap=round] (318.22,133.35) circle (  1.49);

\path[draw=drawColor,line width= 0.4pt,line join=round,line cap=round] (327.69, 79.72) circle (  1.49);

\path[draw=drawColor,line width= 0.4pt,line join=round,line cap=round] (337.16,112.92) circle (  1.49);

\path[draw=drawColor,line width= 0.4pt,line join=round,line cap=round] (346.63,125.69) circle (  1.49);

\path[draw=drawColor,line width= 0.4pt,line join=round,line cap=round] (356.10,123.14) circle (  1.49);

\path[draw=drawColor,line width= 0.4pt,line join=round,line cap=round] (365.57,100.15) circle (  1.49);

\path[draw=drawColor,line width= 0.4pt,line join=round,line cap=round] (375.05, 84.83) circle (  1.49);

\path[draw=drawColor,line width= 0.4pt,line join=round,line cap=round] (384.52,120.58) circle (  1.49);

\path[draw=drawColor,line width= 0.4pt,line join=round,line cap=round] (393.99, 74.61) circle (  1.49);

\path[draw=drawColor,line width= 0.4pt,line join=round,line cap=round] (403.46, 84.83) circle (  1.49);
\end{scope}
\begin{scope}
\path[clip] (  0.00,  0.00) rectangle (423.50,169.40);
\definecolor{drawColor}{RGB}{0,0,0}

\path[draw=drawColor,line width= 0.4pt,line join=round,line cap=round] (327.69, 40.39) -- (403.46, 40.39);

\path[draw=drawColor,line width= 0.4pt,line join=round,line cap=round] (327.69, 40.39) -- (327.69, 36.43);

\path[draw=drawColor,line width= 0.4pt,line join=round,line cap=round] (346.63, 40.39) -- (346.63, 36.43);

\path[draw=drawColor,line width= 0.4pt,line join=round,line cap=round] (365.57, 40.39) -- (365.57, 36.43);

\path[draw=drawColor,line width= 0.4pt,line join=round,line cap=round] (384.52, 40.39) -- (384.52, 36.43);

\path[draw=drawColor,line width= 0.4pt,line join=round,line cap=round] (403.46, 40.39) -- (403.46, 36.43);

\node[text=drawColor,anchor=base,inner sep=0pt, outer sep=0pt, scale=  0.66] at (327.69, 26.14) {200};

\node[text=drawColor,anchor=base,inner sep=0pt, outer sep=0pt, scale=  0.66] at (346.63, 26.14) {400};

\node[text=drawColor,anchor=base,inner sep=0pt, outer sep=0pt, scale=  0.66] at (365.57, 26.14) {600};

\node[text=drawColor,anchor=base,inner sep=0pt, outer sep=0pt, scale=  0.66] at (384.52, 26.14) {800};

\node[text=drawColor,anchor=base,inner sep=0pt, outer sep=0pt, scale=  0.66] at (403.46, 26.14) {1000};

\path[draw=drawColor,line width= 0.4pt,line join=round,line cap=round] (314.81, 46.52) -- (314.81,135.91);

\path[draw=drawColor,line width= 0.4pt,line join=round,line cap=round] (314.81, 46.52) -- (310.85, 46.52);

\path[draw=drawColor,line width= 0.4pt,line join=round,line cap=round] (314.81, 59.29) -- (310.85, 59.29);

\path[draw=drawColor,line width= 0.4pt,line join=round,line cap=round] (314.81, 72.06) -- (310.85, 72.06);

\path[draw=drawColor,line width= 0.4pt,line join=round,line cap=round] (314.81, 84.83) -- (310.85, 84.83);

\path[draw=drawColor,line width= 0.4pt,line join=round,line cap=round] (314.81, 97.60) -- (310.85, 97.60);

\path[draw=drawColor,line width= 0.4pt,line join=round,line cap=round] (314.81,110.37) -- (310.85,110.37);

\path[draw=drawColor,line width= 0.4pt,line join=round,line cap=round] (314.81,123.14) -- (310.85,123.14);

\path[draw=drawColor,line width= 0.4pt,line join=round,line cap=round] (314.81,135.91) -- (310.85,135.91);

\node[text=drawColor,rotate= 90.00,anchor=base,inner sep=0pt, outer sep=0pt, scale=  0.66] at (305.30, 46.52) {0.035};

\node[text=drawColor,rotate= 90.00,anchor=base,inner sep=0pt, outer sep=0pt, scale=  0.66] at (305.30, 72.06) {0.045};

\node[text=drawColor,rotate= 90.00,anchor=base,inner sep=0pt, outer sep=0pt, scale=  0.66] at (305.30, 97.60) {0.055};

\node[text=drawColor,rotate= 90.00,anchor=base,inner sep=0pt, outer sep=0pt, scale=  0.66] at (305.30,123.14) {0.065};

\path[draw=drawColor,line width= 0.4pt,line join=round,line cap=round] (314.81, 40.39) --
	(406.87, 40.39) --
	(406.87,136.93) --
	(314.81,136.93) --
	(314.81, 40.39);
\end{scope}
\begin{scope}
\path[clip] (282.33,  0.00) rectangle (423.50,169.40);
\definecolor{drawColor}{RGB}{0,0,0}

\node[text=drawColor,anchor=base,inner sep=0pt, outer sep=0pt, scale=  0.79] at (360.84,150.43) {\bfseries gamma};

\node[text=drawColor,anchor=base,inner sep=0pt, outer sep=0pt, scale=  0.66] at (360.84, 10.30) {sample size $n$};

\node[text=drawColor,rotate= 90.00,anchor=base,inner sep=0pt, outer sep=0pt, scale=  0.66] at (289.46, 88.66) {rejection rate};
\end{scope}
\begin{scope}
\path[clip] (314.81, 40.39) rectangle (406.87,136.93);
\definecolor{drawColor}{RGB}{205,38,38}

\path[draw=drawColor,line width= 0.4pt,dash pattern=on 4pt off 4pt ,line join=round,line cap=round] (314.81, 84.83) -- (406.87, 84.83);
\definecolor{drawColor}{RGB}{211,211,211}

\path[draw=drawColor,line width= 0.4pt,dash pattern=on 1pt off 3pt ,line join=round,line cap=round] (327.69, 40.39) -- (327.69,136.93);

\path[draw=drawColor,line width= 0.4pt,dash pattern=on 1pt off 3pt ,line join=round,line cap=round] (346.63, 40.39) -- (346.63,136.93);

\path[draw=drawColor,line width= 0.4pt,dash pattern=on 1pt off 3pt ,line join=round,line cap=round] (365.57, 40.39) -- (365.57,136.93);

\path[draw=drawColor,line width= 0.4pt,dash pattern=on 1pt off 3pt ,line join=round,line cap=round] (384.52, 40.39) -- (384.52,136.93);

\path[draw=drawColor,line width= 0.4pt,dash pattern=on 1pt off 3pt ,line join=round,line cap=round] (403.46, 40.39) -- (403.46,136.93);

\path[draw=drawColor,line width= 0.4pt,dash pattern=on 1pt off 3pt ,line join=round,line cap=round] (314.81, 46.52) -- (406.87, 46.52);

\path[draw=drawColor,line width= 0.4pt,dash pattern=on 1pt off 3pt ,line join=round,line cap=round] (314.81, 59.29) -- (406.87, 59.29);

\path[draw=drawColor,line width= 0.4pt,dash pattern=on 1pt off 3pt ,line join=round,line cap=round] (314.81, 72.06) -- (406.87, 72.06);

\path[draw=drawColor,line width= 0.4pt,dash pattern=on 1pt off 3pt ,line join=round,line cap=round] (314.81, 84.83) -- (406.87, 84.83);

\path[draw=drawColor,line width= 0.4pt,dash pattern=on 1pt off 3pt ,line join=round,line cap=round] (314.81, 97.60) -- (406.87, 97.60);

\path[draw=drawColor,line width= 0.4pt,dash pattern=on 1pt off 3pt ,line join=round,line cap=round] (314.81,110.37) -- (406.87,110.37);

\path[draw=drawColor,line width= 0.4pt,dash pattern=on 1pt off 3pt ,line join=round,line cap=round] (314.81,123.14) -- (406.87,123.14);

\path[draw=drawColor,line width= 0.4pt,dash pattern=on 1pt off 3pt ,line join=round,line cap=round] (314.81,135.91) -- (406.87,135.91);
\end{scope}
\end{tikzpicture}

%% file: Plots/levelcomparison2.tex
\begin{tikzpicture}[x=1pt,y=1pt]
\definecolor{fillColor}{RGB}{255,255,255}
\path[use as bounding box,fill=fillColor,fill opacity=0.00] (0,0) rectangle (423.50,169.40);
\begin{scope}
\path[clip] ( 32.47, 40.39) rectangle (124.54,136.93);
\definecolor{drawColor}{RGB}{0,0,0}

\path[draw=drawColor,line width= 0.4pt,line join=round,line cap=round] ( 35.88,107.81) circle (  1.49);

\path[draw=drawColor,line width= 0.4pt,line join=round,line cap=round] ( 45.35, 95.05) circle (  1.49);

\path[draw=drawColor,line width= 0.4pt,line join=round,line cap=round] ( 54.82, 88.66) circle (  1.49);

\path[draw=drawColor,line width= 0.4pt,line join=round,line cap=round] ( 64.30, 82.28) circle (  1.49);

\path[draw=drawColor,line width= 0.4pt,line join=round,line cap=round] ( 73.77,104.62) circle (  1.49);

\path[draw=drawColor,line width= 0.4pt,line join=round,line cap=round] ( 83.24, 82.28) circle (  1.49);

\path[draw=drawColor,line width= 0.4pt,line join=round,line cap=round] ( 92.71,123.78) circle (  1.49);

\path[draw=drawColor,line width= 0.4pt,line join=round,line cap=round] (102.18,114.20) circle (  1.49);

\path[draw=drawColor,line width= 0.4pt,line join=round,line cap=round] (111.65, 72.70) circle (  1.49);

\path[draw=drawColor,line width= 0.4pt,line join=round,line cap=round] (121.13,111.01) circle (  1.49);
\end{scope}
\begin{scope}
\path[clip] (  0.00,  0.00) rectangle (423.50,169.40);
\definecolor{drawColor}{RGB}{0,0,0}

\path[draw=drawColor,line width= 0.4pt,line join=round,line cap=round] ( 45.35, 40.39) -- (121.13, 40.39);

\path[draw=drawColor,line width= 0.4pt,line join=round,line cap=round] ( 45.35, 40.39) -- ( 45.35, 36.43);

\path[draw=drawColor,line width= 0.4pt,line join=round,line cap=round] ( 64.30, 40.39) -- ( 64.30, 36.43);

\path[draw=drawColor,line width= 0.4pt,line join=round,line cap=round] ( 83.24, 40.39) -- ( 83.24, 36.43);

\path[draw=drawColor,line width= 0.4pt,line join=round,line cap=round] (102.18, 40.39) -- (102.18, 36.43);

\path[draw=drawColor,line width= 0.4pt,line join=round,line cap=round] (121.13, 40.39) -- (121.13, 36.43);

\node[text=drawColor,anchor=base,inner sep=0pt, outer sep=0pt, scale=  0.66] at ( 45.35, 26.14) {200};

\node[text=drawColor,anchor=base,inner sep=0pt, outer sep=0pt, scale=  0.66] at ( 64.30, 26.14) {400};

\node[text=drawColor,anchor=base,inner sep=0pt, outer sep=0pt, scale=  0.66] at ( 83.24, 26.14) {600};

\node[text=drawColor,anchor=base,inner sep=0pt, outer sep=0pt, scale=  0.66] at (102.18, 26.14) {800};

\node[text=drawColor,anchor=base,inner sep=0pt, outer sep=0pt, scale=  0.66] at (121.13, 26.14) {1000};

\path[draw=drawColor,line width= 0.4pt,line join=round,line cap=round] ( 32.47, 47.16) -- ( 32.47,126.97);

\path[draw=drawColor,line width= 0.4pt,line join=round,line cap=round] ( 32.47, 47.16) -- ( 28.51, 47.16);

\path[draw=drawColor,line width= 0.4pt,line join=round,line cap=round] ( 32.47, 63.12) -- ( 28.51, 63.12);

\path[draw=drawColor,line width= 0.4pt,line join=round,line cap=round] ( 32.47, 79.08) -- ( 28.51, 79.08);

\path[draw=drawColor,line width= 0.4pt,line join=round,line cap=round] ( 32.47, 95.05) -- ( 28.51, 95.05);

\path[draw=drawColor,line width= 0.4pt,line join=round,line cap=round] ( 32.47,111.01) -- ( 28.51,111.01);

\path[draw=drawColor,line width= 0.4pt,line join=round,line cap=round] ( 32.47,126.97) -- ( 28.51,126.97);

\node[text=drawColor,rotate= 90.00,anchor=base,inner sep=0pt, outer sep=0pt, scale=  0.66] at ( 22.97, 47.16) {0.035};

\node[text=drawColor,rotate= 90.00,anchor=base,inner sep=0pt, outer sep=0pt, scale=  0.66] at ( 22.97, 79.08) {0.045};

\node[text=drawColor,rotate= 90.00,anchor=base,inner sep=0pt, outer sep=0pt, scale=  0.66] at ( 22.97,111.01) {0.055};

\path[draw=drawColor,line width= 0.4pt,line join=round,line cap=round] ( 32.47, 40.39) --
	(124.54, 40.39) --
	(124.54,136.93) --
	( 32.47,136.93) --
	( 32.47, 40.39);
\end{scope}
\begin{scope}
\path[clip] (  0.00,  0.00) rectangle (141.17,169.40);
\definecolor{drawColor}{RGB}{0,0,0}

\node[text=drawColor,anchor=base,inner sep=0pt, outer sep=0pt, scale=  0.79] at ( 78.50,150.43) {\bfseries permutation};

\node[text=drawColor,anchor=base,inner sep=0pt, outer sep=0pt, scale=  0.66] at ( 78.50, 10.30) {sample size $n$};

\node[text=drawColor,rotate= 90.00,anchor=base,inner sep=0pt, outer sep=0pt, scale=  0.66] at (  7.13, 88.66) {rejection rate};
\end{scope}
\begin{scope}
\path[clip] ( 32.47, 40.39) rectangle (124.54,136.93);
\definecolor{drawColor}{RGB}{205,38,38}

\path[draw=drawColor,line width= 0.4pt,dash pattern=on 4pt off 4pt ,line join=round,line cap=round] ( 32.47, 95.05) -- (124.54, 95.05);
\definecolor{drawColor}{RGB}{211,211,211}

\path[draw=drawColor,line width= 0.4pt,dash pattern=on 1pt off 3pt ,line join=round,line cap=round] ( 45.35, 40.39) -- ( 45.35,136.93);

\path[draw=drawColor,line width= 0.4pt,dash pattern=on 1pt off 3pt ,line join=round,line cap=round] ( 64.30, 40.39) -- ( 64.30,136.93);

\path[draw=drawColor,line width= 0.4pt,dash pattern=on 1pt off 3pt ,line join=round,line cap=round] ( 83.24, 40.39) -- ( 83.24,136.93);

\path[draw=drawColor,line width= 0.4pt,dash pattern=on 1pt off 3pt ,line join=round,line cap=round] (102.18, 40.39) -- (102.18,136.93);

\path[draw=drawColor,line width= 0.4pt,dash pattern=on 1pt off 3pt ,line join=round,line cap=round] (121.13, 40.39) -- (121.13,136.93);

\path[draw=drawColor,line width= 0.4pt,dash pattern=on 1pt off 3pt ,line join=round,line cap=round] ( 32.47, 47.16) -- (124.54, 47.16);

\path[draw=drawColor,line width= 0.4pt,dash pattern=on 1pt off 3pt ,line join=round,line cap=round] ( 32.47, 63.12) -- (124.54, 63.12);

\path[draw=drawColor,line width= 0.4pt,dash pattern=on 1pt off 3pt ,line join=round,line cap=round] ( 32.47, 79.08) -- (124.54, 79.08);

\path[draw=drawColor,line width= 0.4pt,dash pattern=on 1pt off 3pt ,line join=round,line cap=round] ( 32.47, 95.05) -- (124.54, 95.05);

\path[draw=drawColor,line width= 0.4pt,dash pattern=on 1pt off 3pt ,line join=round,line cap=round] ( 32.47,111.01) -- (124.54,111.01);

\path[draw=drawColor,line width= 0.4pt,dash pattern=on 1pt off 3pt ,line join=round,line cap=round] ( 32.47,126.97) -- (124.54,126.97);
\end{scope}
\begin{scope}
\path[clip] (173.64, 40.39) rectangle (265.70,136.93);
\definecolor{drawColor}{RGB}{0,0,0}

\path[draw=drawColor,line width= 0.4pt,line join=round,line cap=round] (177.05, 82.28) circle (  1.49);

\path[draw=drawColor,line width= 0.4pt,line join=round,line cap=round] (186.52,111.01) circle (  1.49);

\path[draw=drawColor,line width= 0.4pt,line join=round,line cap=round] (195.99, 85.47) circle (  1.49);

\path[draw=drawColor,line width= 0.4pt,line join=round,line cap=round] (205.46, 91.85) circle (  1.49);

\path[draw=drawColor,line width= 0.4pt,line join=round,line cap=round] (214.94, 82.28) circle (  1.49);

\path[draw=drawColor,line width= 0.4pt,line join=round,line cap=round] (224.41,101.43) circle (  1.49);

\path[draw=drawColor,line width= 0.4pt,line join=round,line cap=round] (233.88, 79.08) circle (  1.49);

\path[draw=drawColor,line width= 0.4pt,line join=round,line cap=round] (243.35, 82.28) circle (  1.49);

\path[draw=drawColor,line width= 0.4pt,line join=round,line cap=round] (252.82, 95.05) circle (  1.49);

\path[draw=drawColor,line width= 0.4pt,line join=round,line cap=round] (262.29,123.78) circle (  1.49);
\end{scope}
\begin{scope}
\path[clip] (  0.00,  0.00) rectangle (423.50,169.40);
\definecolor{drawColor}{RGB}{0,0,0}

\path[draw=drawColor,line width= 0.4pt,line join=round,line cap=round] (186.52, 40.39) -- (262.29, 40.39);

\path[draw=drawColor,line width= 0.4pt,line join=round,line cap=round] (186.52, 40.39) -- (186.52, 36.43);

\path[draw=drawColor,line width= 0.4pt,line join=round,line cap=round] (205.46, 40.39) -- (205.46, 36.43);

\path[draw=drawColor,line width= 0.4pt,line join=round,line cap=round] (224.41, 40.39) -- (224.41, 36.43);

\path[draw=drawColor,line width= 0.4pt,line join=round,line cap=round] (243.35, 40.39) -- (243.35, 36.43);

\path[draw=drawColor,line width= 0.4pt,line join=round,line cap=round] (262.29, 40.39) -- (262.29, 36.43);

\node[text=drawColor,anchor=base,inner sep=0pt, outer sep=0pt, scale=  0.66] at (186.52, 26.14) {200};

\node[text=drawColor,anchor=base,inner sep=0pt, outer sep=0pt, scale=  0.66] at (205.46, 26.14) {400};

\node[text=drawColor,anchor=base,inner sep=0pt, outer sep=0pt, scale=  0.66] at (224.41, 26.14) {600};

\node[text=drawColor,anchor=base,inner sep=0pt, outer sep=0pt, scale=  0.66] at (243.35, 26.14) {800};

\node[text=drawColor,anchor=base,inner sep=0pt, outer sep=0pt, scale=  0.66] at (262.29, 26.14) {1000};

\path[draw=drawColor,line width= 0.4pt,line join=round,line cap=round] (173.64, 47.16) -- (173.64,126.97);

\path[draw=drawColor,line width= 0.4pt,line join=round,line cap=round] (173.64, 47.16) -- (169.68, 47.16);

\path[draw=drawColor,line width= 0.4pt,line join=round,line cap=round] (173.64, 63.12) -- (169.68, 63.12);

\path[draw=drawColor,line width= 0.4pt,line join=round,line cap=round] (173.64, 79.08) -- (169.68, 79.08);

\path[draw=drawColor,line width= 0.4pt,line join=round,line cap=round] (173.64, 95.05) -- (169.68, 95.05);

\path[draw=drawColor,line width= 0.4pt,line join=round,line cap=round] (173.64,111.01) -- (169.68,111.01);

\path[draw=drawColor,line width= 0.4pt,line join=round,line cap=round] (173.64,126.97) -- (169.68,126.97);

\node[text=drawColor,rotate= 90.00,anchor=base,inner sep=0pt, outer sep=0pt, scale=  0.66] at (164.14, 47.16) {0.035};

\node[text=drawColor,rotate= 90.00,anchor=base,inner sep=0pt, outer sep=0pt, scale=  0.66] at (164.14, 79.08) {0.045};

\node[text=drawColor,rotate= 90.00,anchor=base,inner sep=0pt, outer sep=0pt, scale=  0.66] at (164.14,111.01) {0.055};

\path[draw=drawColor,line width= 0.4pt,line join=round,line cap=round] (173.64, 40.39) --
	(265.70, 40.39) --
	(265.70,136.93) --
	(173.64,136.93) --
	(173.64, 40.39);
\end{scope}
\begin{scope}
\path[clip] (141.17,  0.00) rectangle (282.33,169.40);
\definecolor{drawColor}{RGB}{0,0,0}

\node[text=drawColor,anchor=base,inner sep=0pt, outer sep=0pt, scale=  0.79] at (219.67,150.43) {\bfseries bootstrap};

\node[text=drawColor,anchor=base,inner sep=0pt, outer sep=0pt, scale=  0.66] at (219.67, 10.30) {sample size $n$};

\node[text=drawColor,rotate= 90.00,anchor=base,inner sep=0pt, outer sep=0pt, scale=  0.66] at (148.30, 88.66) {rejection rate};
\end{scope}
\begin{scope}
\path[clip] (173.64, 40.39) rectangle (265.70,136.93);
\definecolor{drawColor}{RGB}{205,38,38}

\path[draw=drawColor,line width= 0.4pt,dash pattern=on 4pt off 4pt ,line join=round,line cap=round] (173.64, 95.05) -- (265.70, 95.05);
\definecolor{drawColor}{RGB}{211,211,211}

\path[draw=drawColor,line width= 0.4pt,dash pattern=on 1pt off 3pt ,line join=round,line cap=round] (186.52, 40.39) -- (186.52,136.93);

\path[draw=drawColor,line width= 0.4pt,dash pattern=on 1pt off 3pt ,line join=round,line cap=round] (205.46, 40.39) -- (205.46,136.93);

\path[draw=drawColor,line width= 0.4pt,dash pattern=on 1pt off 3pt ,line join=round,line cap=round] (224.41, 40.39) -- (224.41,136.93);

\path[draw=drawColor,line width= 0.4pt,dash pattern=on 1pt off 3pt ,line join=round,line cap=round] (243.35, 40.39) -- (243.35,136.93);

\path[draw=drawColor,line width= 0.4pt,dash pattern=on 1pt off 3pt ,line join=round,line cap=round] (262.29, 40.39) -- (262.29,136.93);

\path[draw=drawColor,line width= 0.4pt,dash pattern=on 1pt off 3pt ,line join=round,line cap=round] (173.64, 47.16) -- (265.70, 47.16);

\path[draw=drawColor,line width= 0.4pt,dash pattern=on 1pt off 3pt ,line join=round,line cap=round] (173.64, 63.12) -- (265.70, 63.12);

\path[draw=drawColor,line width= 0.4pt,dash pattern=on 1pt off 3pt ,line join=round,line cap=round] (173.64, 79.08) -- (265.70, 79.08);

\path[draw=drawColor,line width= 0.4pt,dash pattern=on 1pt off 3pt ,line join=round,line cap=round] (173.64, 95.05) -- (265.70, 95.05);

\path[draw=drawColor,line width= 0.4pt,dash pattern=on 1pt off 3pt ,line join=round,line cap=round] (173.64,111.01) -- (265.70,111.01);

\path[draw=drawColor,line width= 0.4pt,dash pattern=on 1pt off 3pt ,line join=round,line cap=round] (173.64,126.97) -- (265.70,126.97);
\end{scope}
\begin{scope}
\path[clip] (314.81, 40.39) rectangle (406.87,136.93);
\definecolor{drawColor}{RGB}{0,0,0}

\path[draw=drawColor,line width= 0.4pt,line join=round,line cap=round] (318.22, 59.93) circle (  1.49);

\path[draw=drawColor,line width= 0.4pt,line join=round,line cap=round] (327.69, 95.05) circle (  1.49);

\path[draw=drawColor,line width= 0.4pt,line join=round,line cap=round] (337.16,120.58) circle (  1.49);

\path[draw=drawColor,line width= 0.4pt,line join=round,line cap=round] (346.63, 98.24) circle (  1.49);

\path[draw=drawColor,line width= 0.4pt,line join=round,line cap=round] (356.10,101.43) circle (  1.49);

\path[draw=drawColor,line width= 0.4pt,line join=round,line cap=round] (365.57,133.35) circle (  1.49);

\path[draw=drawColor,line width= 0.4pt,line join=round,line cap=round] (375.05, 85.47) circle (  1.49);

\path[draw=drawColor,line width= 0.4pt,line join=round,line cap=round] (384.52, 43.97) circle (  1.49);

\path[draw=drawColor,line width= 0.4pt,line join=round,line cap=round] (393.99, 79.08) circle (  1.49);

\path[draw=drawColor,line width= 0.4pt,line join=round,line cap=round] (403.46, 91.85) circle (  1.49);
\end{scope}
\begin{scope}
\path[clip] (  0.00,  0.00) rectangle (423.50,169.40);
\definecolor{drawColor}{RGB}{0,0,0}

\path[draw=drawColor,line width= 0.4pt,line join=round,line cap=round] (327.69, 40.39) -- (403.46, 40.39);

\path[draw=drawColor,line width= 0.4pt,line join=round,line cap=round] (327.69, 40.39) -- (327.69, 36.43);

\path[draw=drawColor,line width= 0.4pt,line join=round,line cap=round] (346.63, 40.39) -- (346.63, 36.43);

\path[draw=drawColor,line width= 0.4pt,line join=round,line cap=round] (365.57, 40.39) -- (365.57, 36.43);

\path[draw=drawColor,line width= 0.4pt,line join=round,line cap=round] (384.52, 40.39) -- (384.52, 36.43);

\path[draw=drawColor,line width= 0.4pt,line join=round,line cap=round] (403.46, 40.39) -- (403.46, 36.43);

\node[text=drawColor,anchor=base,inner sep=0pt, outer sep=0pt, scale=  0.66] at (327.69, 26.14) {200};

\node[text=drawColor,anchor=base,inner sep=0pt, outer sep=0pt, scale=  0.66] at (346.63, 26.14) {400};

\node[text=drawColor,anchor=base,inner sep=0pt, outer sep=0pt, scale=  0.66] at (365.57, 26.14) {600};

\node[text=drawColor,anchor=base,inner sep=0pt, outer sep=0pt, scale=  0.66] at (384.52, 26.14) {800};

\node[text=drawColor,anchor=base,inner sep=0pt, outer sep=0pt, scale=  0.66] at (403.46, 26.14) {1000};

\path[draw=drawColor,line width= 0.4pt,line join=round,line cap=round] (314.81, 47.16) -- (314.81,126.97);

\path[draw=drawColor,line width= 0.4pt,line join=round,line cap=round] (314.81, 47.16) -- (310.85, 47.16);

\path[draw=drawColor,line width= 0.4pt,line join=round,line cap=round] (314.81, 63.12) -- (310.85, 63.12);

\path[draw=drawColor,line width= 0.4pt,line join=round,line cap=round] (314.81, 79.08) -- (310.85, 79.08);

\path[draw=drawColor,line width= 0.4pt,line join=round,line cap=round] (314.81, 95.05) -- (310.85, 95.05);

\path[draw=drawColor,line width= 0.4pt,line join=round,line cap=round] (314.81,111.01) -- (310.85,111.01);

\path[draw=drawColor,line width= 0.4pt,line join=round,line cap=round] (314.81,126.97) -- (310.85,126.97);

\node[text=drawColor,rotate= 90.00,anchor=base,inner sep=0pt, outer sep=0pt, scale=  0.66] at (305.30, 47.16) {0.035};

\node[text=drawColor,rotate= 90.00,anchor=base,inner sep=0pt, outer sep=0pt, scale=  0.66] at (305.30, 79.08) {0.045};

\node[text=drawColor,rotate= 90.00,anchor=base,inner sep=0pt, outer sep=0pt, scale=  0.66] at (305.30,111.01) {0.055};

\path[draw=drawColor,line width= 0.4pt,line join=round,line cap=round] (314.81, 40.39) --
	(406.87, 40.39) --
	(406.87,136.93) --
	(314.81,136.93) --
	(314.81, 40.39);
\end{scope}
\begin{scope}
\path[clip] (282.33,  0.00) rectangle (423.50,169.40);
\definecolor{drawColor}{RGB}{0,0,0}

\node[text=drawColor,anchor=base,inner sep=0pt, outer sep=0pt, scale=  0.79] at (360.84,150.43) {\bfseries gamma};

\node[text=drawColor,anchor=base,inner sep=0pt, outer sep=0pt, scale=  0.66] at (360.84, 10.30) {sample size $n$};

\node[text=drawColor,rotate= 90.00,anchor=base,inner sep=0pt, outer sep=0pt, scale=  0.66] at (289.46, 88.66) {rejection rate};
\end{scope}
\begin{scope}
\path[clip] (314.81, 40.39) rectangle (406.87,136.93);
\definecolor{drawColor}{RGB}{205,38,38}

\path[draw=drawColor,line width= 0.4pt,dash pattern=on 4pt off 4pt ,line join=round,line cap=round] (314.81, 95.05) -- (406.87, 95.05);
\definecolor{drawColor}{RGB}{211,211,211}

\path[draw=drawColor,line width= 0.4pt,dash pattern=on 1pt off 3pt ,line join=round,line cap=round] (327.69, 40.39) -- (327.69,136.93);

\path[draw=drawColor,line width= 0.4pt,dash pattern=on 1pt off 3pt ,line join=round,line cap=round] (346.63, 40.39) -- (346.63,136.93);

\path[draw=drawColor,line width= 0.4pt,dash pattern=on 1pt off 3pt ,line join=round,line cap=round] (365.57, 40.39) -- (365.57,136.93);

\path[draw=drawColor,line width= 0.4pt,dash pattern=on 1pt off 3pt ,line join=round,line cap=round] (384.52, 40.39) -- (384.52,136.93);

\path[draw=drawColor,line width= 0.4pt,dash pattern=on 1pt off 3pt ,line join=round,line cap=round] (403.46, 40.39) -- (403.46,136.93);

\path[draw=drawColor,line width= 0.4pt,dash pattern=on 1pt off 3pt ,line join=round,line cap=round] (314.81, 47.16) -- (406.87, 47.16);

\path[draw=drawColor,line width= 0.4pt,dash pattern=on 1pt off 3pt ,line join=round,line cap=round] (314.81, 63.12) -- (406.87, 63.12);

\path[draw=drawColor,line width= 0.4pt,dash pattern=on 1pt off 3pt ,line join=round,line cap=round] (314.81, 79.08) -- (406.87, 79.08);

\path[draw=drawColor,line width= 0.4pt,dash pattern=on 1pt off 3pt ,line join=round,line cap=round] (314.81, 95.05) -- (406.87, 95.05);

\path[draw=drawColor,line width= 0.4pt,dash pattern=on 1pt off 3pt ,line join=round,line cap=round] (314.81,111.01) -- (406.87,111.01);

\path[draw=drawColor,line width= 0.4pt,dash pattern=on 1pt off 3pt ,line join=round,line cap=round] (314.81,126.97) -- (406.87,126.97);
\end{scope}
\end{tikzpicture}

%% file: Plots/power_analysis1_romano.tex
\begin{tikzpicture}[x=1pt,y=1pt]
\definecolor{fillColor}{RGB}{255,255,255}
\path[use as bounding box,fill=fillColor,fill opacity=0.00] (0,0) rectangle (423.50,169.40);
\begin{scope}
\path[clip] ( 32.47, 40.39) rectangle (124.54,136.93);
\definecolor{drawColor}{RGB}{205,38,38}

\path[draw=drawColor,line width= 0.4pt,line join=round,line cap=round] ( 35.88, 52.36) --
	( 37.88, 48.89) --
	( 33.88, 48.89) --
	( 35.88, 52.36);

\path[draw=drawColor,line width= 0.4pt,line join=round,line cap=round] ( 64.30, 61.20) --
	( 66.30, 57.74) --
	( 62.30, 57.74) --
	( 64.30, 61.20);

\path[draw=drawColor,line width= 0.4pt,line join=round,line cap=round] ( 92.71, 66.75) --
	( 94.71, 63.28) --
	( 90.71, 63.28) --
	( 92.71, 66.75);

\path[draw=drawColor,line width= 0.4pt,line join=round,line cap=round] (121.13, 73.54) --
	(123.13, 70.08) --
	(119.13, 70.08) --
	(121.13, 73.54);
\end{scope}
\begin{scope}
\path[clip] (  0.00,  0.00) rectangle (423.50,169.40);
\definecolor{drawColor}{RGB}{0,0,0}

\path[draw=drawColor,line width= 0.4pt,line join=round,line cap=round] ( 35.88, 40.39) -- (121.13, 40.39);

\path[draw=drawColor,line width= 0.4pt,line join=round,line cap=round] ( 35.88, 40.39) -- ( 35.88, 36.43);

\path[draw=drawColor,line width= 0.4pt,line join=round,line cap=round] ( 64.30, 40.39) -- ( 64.30, 36.43);

\path[draw=drawColor,line width= 0.4pt,line join=round,line cap=round] ( 92.71, 40.39) -- ( 92.71, 36.43);

\path[draw=drawColor,line width= 0.4pt,line join=round,line cap=round] (121.13, 40.39) -- (121.13, 36.43);

\node[text=drawColor,anchor=base,inner sep=0pt, outer sep=0pt, scale=  0.66] at ( 35.88, 26.14) {50};

\node[text=drawColor,anchor=base,inner sep=0pt, outer sep=0pt, scale=  0.66] at ( 64.30, 26.14) {100};

\node[text=drawColor,anchor=base,inner sep=0pt, outer sep=0pt, scale=  0.66] at ( 92.71, 26.14) {150};

\node[text=drawColor,anchor=base,inner sep=0pt, outer sep=0pt, scale=  0.66] at (121.13, 26.14) {200};

\path[draw=drawColor,line width= 0.4pt,line join=round,line cap=round] ( 32.47, 43.97) -- ( 32.47,133.35);

\path[draw=drawColor,line width= 0.4pt,line join=round,line cap=round] ( 32.47, 43.97) -- ( 28.51, 43.97);

\path[draw=drawColor,line width= 0.4pt,line join=round,line cap=round] ( 32.47, 61.84) -- ( 28.51, 61.84);

\path[draw=drawColor,line width= 0.4pt,line join=round,line cap=round] ( 32.47, 79.72) -- ( 28.51, 79.72);

\path[draw=drawColor,line width= 0.4pt,line join=round,line cap=round] ( 32.47, 97.60) -- ( 28.51, 97.60);

\path[draw=drawColor,line width= 0.4pt,line join=round,line cap=round] ( 32.47,115.48) -- ( 28.51,115.48);

\path[draw=drawColor,line width= 0.4pt,line join=round,line cap=round] ( 32.47,133.35) -- ( 28.51,133.35);

\node[text=drawColor,rotate= 90.00,anchor=base,inner sep=0pt, outer sep=0pt, scale=  0.66] at ( 22.97, 43.97) {0.0};

\node[text=drawColor,rotate= 90.00,anchor=base,inner sep=0pt, outer sep=0pt, scale=  0.66] at ( 22.97, 61.84) {0.2};

\node[text=drawColor,rotate= 90.00,anchor=base,inner sep=0pt, outer sep=0pt, scale=  0.66] at ( 22.97, 79.72) {0.4};

\node[text=drawColor,rotate= 90.00,anchor=base,inner sep=0pt, outer sep=0pt, scale=  0.66] at ( 22.97, 97.60) {0.6};

\node[text=drawColor,rotate= 90.00,anchor=base,inner sep=0pt, outer sep=0pt, scale=  0.66] at ( 22.97,115.48) {0.8};

\node[text=drawColor,rotate= 90.00,anchor=base,inner sep=0pt, outer sep=0pt, scale=  0.66] at ( 22.97,133.35) {1.0};

\path[draw=drawColor,line width= 0.4pt,line join=round,line cap=round] ( 32.47, 40.39) --
	(124.54, 40.39) --
	(124.54,136.93) --
	( 32.47,136.93) --
	( 32.47, 40.39);
\end{scope}
\begin{scope}
\path[clip] (  0.00,  0.00) rectangle (141.17,169.40);
\definecolor{drawColor}{RGB}{0,0,0}

\node[text=drawColor,anchor=base,inner sep=0pt, outer sep=0pt, scale=  0.79] at ( 78.50,150.43) {\bfseries permutation};

\node[text=drawColor,anchor=base,inner sep=0pt, outer sep=0pt, scale=  0.66] at ( 78.50, 10.30) {sample size $n$};

\node[text=drawColor,rotate= 90.00,anchor=base,inner sep=0pt, outer sep=0pt, scale=  0.66] at (  7.13, 88.66) {rejection rate};
\end{scope}
\begin{scope}
\path[clip] ( 32.47, 40.39) rectangle (124.54,136.93);
\definecolor{drawColor}{RGB}{69,139,0}

\path[draw=drawColor,line width= 0.4pt,line join=round,line cap=round] ( 35.88, 50.22) circle (  1.49);

\path[draw=drawColor,line width= 0.4pt,line join=round,line cap=round] ( 64.30, 54.34) circle (  1.49);

\path[draw=drawColor,line width= 0.4pt,line join=round,line cap=round] ( 92.71, 58.45) circle (  1.49);

\path[draw=drawColor,line width= 0.4pt,line join=round,line cap=round] (121.13, 61.67) circle (  1.49);
\definecolor{drawColor}{RGB}{24,116,205}

\path[draw=drawColor,line width= 0.4pt,line join=round,line cap=round] ( 34.40, 48.65) rectangle ( 37.37, 51.62);

\path[draw=drawColor,line width= 0.4pt,line join=round,line cap=round] ( 62.81, 50.53) rectangle ( 65.78, 53.50);

\path[draw=drawColor,line width= 0.4pt,line join=round,line cap=round] ( 91.23, 54.46) rectangle ( 94.20, 57.43);

\path[draw=drawColor,line width= 0.4pt,line join=round,line cap=round] (119.64, 55.35) rectangle (122.61, 58.32);
\definecolor{drawColor}{RGB}{211,211,211}

\path[draw=drawColor,line width= 0.4pt,dash pattern=on 1pt off 3pt ,line join=round,line cap=round] ( 35.88, 40.39) -- ( 35.88,136.93);

\path[draw=drawColor,line width= 0.4pt,dash pattern=on 1pt off 3pt ,line join=round,line cap=round] ( 64.30, 40.39) -- ( 64.30,136.93);

\path[draw=drawColor,line width= 0.4pt,dash pattern=on 1pt off 3pt ,line join=round,line cap=round] ( 92.71, 40.39) -- ( 92.71,136.93);

\path[draw=drawColor,line width= 0.4pt,dash pattern=on 1pt off 3pt ,line join=round,line cap=round] (121.13, 40.39) -- (121.13,136.93);

\path[draw=drawColor,line width= 0.4pt,dash pattern=on 1pt off 3pt ,line join=round,line cap=round] ( 32.47, 43.97) -- (124.54, 43.97);

\path[draw=drawColor,line width= 0.4pt,dash pattern=on 1pt off 3pt ,line join=round,line cap=round] ( 32.47, 61.84) -- (124.54, 61.84);

\path[draw=drawColor,line width= 0.4pt,dash pattern=on 1pt off 3pt ,line join=round,line cap=round] ( 32.47, 79.72) -- (124.54, 79.72);

\path[draw=drawColor,line width= 0.4pt,dash pattern=on 1pt off 3pt ,line join=round,line cap=round] ( 32.47, 97.60) -- (124.54, 97.60);

\path[draw=drawColor,line width= 0.4pt,dash pattern=on 1pt off 3pt ,line join=round,line cap=round] ( 32.47,115.48) -- (124.54,115.48);

\path[draw=drawColor,line width= 0.4pt,dash pattern=on 1pt off 3pt ,line join=round,line cap=round] ( 32.47,133.35) -- (124.54,133.35);
\definecolor{drawColor}{RGB}{0,0,0}
\definecolor{fillColor}{RGB}{255,255,255}

\path[draw=drawColor,line width= 0.4pt,line join=round,line cap=round,fill=fillColor] ( 37.08,132.10) rectangle ( 73.67,100.42);
\definecolor{drawColor}{RGB}{69,139,0}

\path[draw=drawColor,line width= 0.4pt,line join=round,line cap=round] ( 43.02,124.18) circle (  1.49);
\definecolor{drawColor}{RGB}{205,38,38}

\path[draw=drawColor,line width= 0.4pt,line join=round,line cap=round] ( 43.02,118.57) --
	( 45.02,115.11) --
	( 41.02,115.11) --
	( 43.02,118.57);
\definecolor{drawColor}{RGB}{24,116,205}

\path[draw=drawColor,line width= 0.4pt,line join=round,line cap=round] ( 41.53,106.86) rectangle ( 44.50,109.83);
\definecolor{drawColor}{RGB}{0,0,0}

\node[text=drawColor,anchor=base west,inner sep=0pt, outer sep=0pt, scale=  0.66] at ( 48.96,121.91) {dHSIC};

\node[text=drawColor,anchor=base west,inner sep=0pt, outer sep=0pt, scale=  0.66] at ( 48.96,113.99) {HSIC};

\node[text=drawColor,anchor=base west,inner sep=0pt, outer sep=0pt, scale=  0.66] at ( 48.96,106.07) {BMR-$n$};
\end{scope}
\begin{scope}
\path[clip] (173.64, 40.39) rectangle (265.70,136.93);
\definecolor{drawColor}{RGB}{205,38,38}

\path[draw=drawColor,line width= 0.4pt,line join=round,line cap=round] (177.05, 51.64) --
	(179.05, 48.18) --
	(175.05, 48.18) --
	(177.05, 51.64);

\path[draw=drawColor,line width= 0.4pt,line join=round,line cap=round] (205.46, 57.45) --
	(207.46, 53.99) --
	(203.46, 53.99) --
	(205.46, 57.45);

\path[draw=drawColor,line width= 0.4pt,line join=round,line cap=round] (233.88, 66.03) --
	(235.88, 62.57) --
	(231.88, 62.57) --
	(233.88, 66.03);

\path[draw=drawColor,line width= 0.4pt,line join=round,line cap=round] (262.29, 73.09) --
	(264.29, 69.63) --
	(260.29, 69.63) --
	(262.29, 73.09);
\end{scope}
\begin{scope}
\path[clip] (  0.00,  0.00) rectangle (423.50,169.40);
\definecolor{drawColor}{RGB}{0,0,0}

\path[draw=drawColor,line width= 0.4pt,line join=round,line cap=round] (177.05, 40.39) -- (262.29, 40.39);

\path[draw=drawColor,line width= 0.4pt,line join=round,line cap=round] (177.05, 40.39) -- (177.05, 36.43);

\path[draw=drawColor,line width= 0.4pt,line join=round,line cap=round] (205.46, 40.39) -- (205.46, 36.43);

\path[draw=drawColor,line width= 0.4pt,line join=round,line cap=round] (233.88, 40.39) -- (233.88, 36.43);

\path[draw=drawColor,line width= 0.4pt,line join=round,line cap=round] (262.29, 40.39) -- (262.29, 36.43);

\node[text=drawColor,anchor=base,inner sep=0pt, outer sep=0pt, scale=  0.66] at (177.05, 26.14) {50};

\node[text=drawColor,anchor=base,inner sep=0pt, outer sep=0pt, scale=  0.66] at (205.46, 26.14) {100};

\node[text=drawColor,anchor=base,inner sep=0pt, outer sep=0pt, scale=  0.66] at (233.88, 26.14) {150};

\node[text=drawColor,anchor=base,inner sep=0pt, outer sep=0pt, scale=  0.66] at (262.29, 26.14) {200};

\path[draw=drawColor,line width= 0.4pt,line join=round,line cap=round] (173.64, 43.97) -- (173.64,133.35);

\path[draw=drawColor,line width= 0.4pt,line join=round,line cap=round] (173.64, 43.97) -- (169.68, 43.97);

\path[draw=drawColor,line width= 0.4pt,line join=round,line cap=round] (173.64, 61.84) -- (169.68, 61.84);

\path[draw=drawColor,line width= 0.4pt,line join=round,line cap=round] (173.64, 79.72) -- (169.68, 79.72);

\path[draw=drawColor,line width= 0.4pt,line join=round,line cap=round] (173.64, 97.60) -- (169.68, 97.60);

\path[draw=drawColor,line width= 0.4pt,line join=round,line cap=round] (173.64,115.48) -- (169.68,115.48);

\path[draw=drawColor,line width= 0.4pt,line join=round,line cap=round] (173.64,133.35) -- (169.68,133.35);

\node[text=drawColor,rotate= 90.00,anchor=base,inner sep=0pt, outer sep=0pt, scale=  0.66] at (164.14, 43.97) {0.0};

\node[text=drawColor,rotate= 90.00,anchor=base,inner sep=0pt, outer sep=0pt, scale=  0.66] at (164.14, 61.84) {0.2};

\node[text=drawColor,rotate= 90.00,anchor=base,inner sep=0pt, outer sep=0pt, scale=  0.66] at (164.14, 79.72) {0.4};

\node[text=drawColor,rotate= 90.00,anchor=base,inner sep=0pt, outer sep=0pt, scale=  0.66] at (164.14, 97.60) {0.6};

\node[text=drawColor,rotate= 90.00,anchor=base,inner sep=0pt, outer sep=0pt, scale=  0.66] at (164.14,115.48) {0.8};

\node[text=drawColor,rotate= 90.00,anchor=base,inner sep=0pt, outer sep=0pt, scale=  0.66] at (164.14,133.35) {1.0};

\path[draw=drawColor,line width= 0.4pt,line join=round,line cap=round] (173.64, 40.39) --
	(265.70, 40.39) --
	(265.70,136.93) --
	(173.64,136.93) --
	(173.64, 40.39);
\end{scope}
\begin{scope}
\path[clip] (141.17,  0.00) rectangle (282.33,169.40);
\definecolor{drawColor}{RGB}{0,0,0}

\node[text=drawColor,anchor=base,inner sep=0pt, outer sep=0pt, scale=  0.79] at (219.67,150.43) {\bfseries bootstrap};

\node[text=drawColor,anchor=base,inner sep=0pt, outer sep=0pt, scale=  0.66] at (219.67, 10.30) {sample size};

\node[text=drawColor,rotate= 90.00,anchor=base,inner sep=0pt, outer sep=0pt, scale=  0.66] at (148.30, 88.66) {rejection rate};
\end{scope}
\begin{scope}
\path[clip] (173.64, 40.39) rectangle (265.70,136.93);
\definecolor{drawColor}{RGB}{69,139,0}

\path[draw=drawColor,line width= 0.4pt,line join=round,line cap=round] (177.05, 51.30) circle (  1.49);

\path[draw=drawColor,line width= 0.4pt,line join=round,line cap=round] (205.46, 54.07) circle (  1.49);

\path[draw=drawColor,line width= 0.4pt,line join=round,line cap=round] (233.88, 58.54) circle (  1.49);

\path[draw=drawColor,line width= 0.4pt,line join=round,line cap=round] (262.29, 62.11) circle (  1.49);
\definecolor{drawColor}{RGB}{24,116,205}

\path[draw=drawColor,line width= 0.4pt,line join=round,line cap=round] (175.56, 48.47) rectangle (178.53, 51.44);

\path[draw=drawColor,line width= 0.4pt,line join=round,line cap=round] (203.98, 49.90) rectangle (206.95, 52.87);

\path[draw=drawColor,line width= 0.4pt,line join=round,line cap=round] (232.39, 53.92) rectangle (235.36, 56.89);

\path[draw=drawColor,line width= 0.4pt,line join=round,line cap=round] (260.81, 56.25) rectangle (263.78, 59.22);
\definecolor{drawColor}{RGB}{211,211,211}

\path[draw=drawColor,line width= 0.4pt,dash pattern=on 1pt off 3pt ,line join=round,line cap=round] (177.05, 40.39) -- (177.05,136.93);

\path[draw=drawColor,line width= 0.4pt,dash pattern=on 1pt off 3pt ,line join=round,line cap=round] (205.46, 40.39) -- (205.46,136.93);

\path[draw=drawColor,line width= 0.4pt,dash pattern=on 1pt off 3pt ,line join=round,line cap=round] (233.88, 40.39) -- (233.88,136.93);

\path[draw=drawColor,line width= 0.4pt,dash pattern=on 1pt off 3pt ,line join=round,line cap=round] (262.29, 40.39) -- (262.29,136.93);

\path[draw=drawColor,line width= 0.4pt,dash pattern=on 1pt off 3pt ,line join=round,line cap=round] (173.64, 43.97) -- (265.70, 43.97);

\path[draw=drawColor,line width= 0.4pt,dash pattern=on 1pt off 3pt ,line join=round,line cap=round] (173.64, 61.84) -- (265.70, 61.84);

\path[draw=drawColor,line width= 0.4pt,dash pattern=on 1pt off 3pt ,line join=round,line cap=round] (173.64, 79.72) -- (265.70, 79.72);

\path[draw=drawColor,line width= 0.4pt,dash pattern=on 1pt off 3pt ,line join=round,line cap=round] (173.64, 97.60) -- (265.70, 97.60);

\path[draw=drawColor,line width= 0.4pt,dash pattern=on 1pt off 3pt ,line join=round,line cap=round] (173.64,115.48) -- (265.70,115.48);

\path[draw=drawColor,line width= 0.4pt,dash pattern=on 1pt off 3pt ,line join=round,line cap=round] (173.64,133.35) -- (265.70,133.35);
\definecolor{drawColor}{RGB}{0,0,0}
\definecolor{fillColor}{RGB}{255,255,255}

\path[draw=drawColor,line width= 0.4pt,line join=round,line cap=round,fill=fillColor] (178.24,132.10) rectangle (214.83,100.42);
\definecolor{drawColor}{RGB}{69,139,0}

\path[draw=drawColor,line width= 0.4pt,line join=round,line cap=round] (184.18,124.18) circle (  1.49);
\definecolor{drawColor}{RGB}{205,38,38}

\path[draw=drawColor,line width= 0.4pt,line join=round,line cap=round] (184.18,118.57) --
	(186.18,115.11) --
	(182.18,115.11) --
	(184.18,118.57);
\definecolor{drawColor}{RGB}{24,116,205}

\path[draw=drawColor,line width= 0.4pt,line join=round,line cap=round] (182.70,106.86) rectangle (185.67,109.83);
\definecolor{drawColor}{RGB}{0,0,0}

\node[text=drawColor,anchor=base west,inner sep=0pt, outer sep=0pt, scale=  0.66] at (190.12,121.91) {dHSIC};

\node[text=drawColor,anchor=base west,inner sep=0pt, outer sep=0pt, scale=  0.66] at (190.12,113.99) {HSIC};

\node[text=drawColor,anchor=base west,inner sep=0pt, outer sep=0pt, scale=  0.66] at (190.12,106.07) {BMR-$n$};
\end{scope}
\begin{scope}
\path[clip] (314.81, 40.39) rectangle (406.87,136.93);
\definecolor{drawColor}{RGB}{205,38,38}

\path[draw=drawColor,line width= 0.4pt,line join=round,line cap=round] (318.22, 57.99) --
	(320.22, 54.52) --
	(316.22, 54.52) --
	(318.22, 57.99);

\path[draw=drawColor,line width= 0.4pt,line join=round,line cap=round] (346.63, 66.92) --
	(348.63, 63.46) --
	(344.63, 63.46) --
	(346.63, 66.92);

\path[draw=drawColor,line width= 0.4pt,line join=round,line cap=round] (375.05, 76.67) --
	(377.05, 73.20) --
	(373.05, 73.20) --
	(375.05, 76.67);

\path[draw=drawColor,line width= 0.4pt,line join=round,line cap=round] (403.46, 83.46) --
	(405.46, 80.00) --
	(401.46, 80.00) --
	(403.46, 83.46);
\end{scope}
\begin{scope}
\path[clip] (  0.00,  0.00) rectangle (423.50,169.40);
\definecolor{drawColor}{RGB}{0,0,0}

\path[draw=drawColor,line width= 0.4pt,line join=round,line cap=round] (318.22, 40.39) -- (403.46, 40.39);

\path[draw=drawColor,line width= 0.4pt,line join=round,line cap=round] (318.22, 40.39) -- (318.22, 36.43);

\path[draw=drawColor,line width= 0.4pt,line join=round,line cap=round] (346.63, 40.39) -- (346.63, 36.43);

\path[draw=drawColor,line width= 0.4pt,line join=round,line cap=round] (375.05, 40.39) -- (375.05, 36.43);

\path[draw=drawColor,line width= 0.4pt,line join=round,line cap=round] (403.46, 40.39) -- (403.46, 36.43);

\node[text=drawColor,anchor=base,inner sep=0pt, outer sep=0pt, scale=  0.66] at (318.22, 26.14) {50};

\node[text=drawColor,anchor=base,inner sep=0pt, outer sep=0pt, scale=  0.66] at (346.63, 26.14) {100};

\node[text=drawColor,anchor=base,inner sep=0pt, outer sep=0pt, scale=  0.66] at (375.05, 26.14) {150};

\node[text=drawColor,anchor=base,inner sep=0pt, outer sep=0pt, scale=  0.66] at (403.46, 26.14) {200};

\path[draw=drawColor,line width= 0.4pt,line join=round,line cap=round] (314.81, 43.97) -- (314.81,133.35);

\path[draw=drawColor,line width= 0.4pt,line join=round,line cap=round] (314.81, 43.97) -- (310.85, 43.97);

\path[draw=drawColor,line width= 0.4pt,line join=round,line cap=round] (314.81, 61.84) -- (310.85, 61.84);

\path[draw=drawColor,line width= 0.4pt,line join=round,line cap=round] (314.81, 79.72) -- (310.85, 79.72);

\path[draw=drawColor,line width= 0.4pt,line join=round,line cap=round] (314.81, 97.60) -- (310.85, 97.60);

\path[draw=drawColor,line width= 0.4pt,line join=round,line cap=round] (314.81,115.48) -- (310.85,115.48);

\path[draw=drawColor,line width= 0.4pt,line join=round,line cap=round] (314.81,133.35) -- (310.85,133.35);

\node[text=drawColor,rotate= 90.00,anchor=base,inner sep=0pt, outer sep=0pt, scale=  0.66] at (305.30, 43.97) {0.0};

\node[text=drawColor,rotate= 90.00,anchor=base,inner sep=0pt, outer sep=0pt, scale=  0.66] at (305.30, 61.84) {0.2};

\node[text=drawColor,rotate= 90.00,anchor=base,inner sep=0pt, outer sep=0pt, scale=  0.66] at (305.30, 79.72) {0.4};

\node[text=drawColor,rotate= 90.00,anchor=base,inner sep=0pt, outer sep=0pt, scale=  0.66] at (305.30, 97.60) {0.6};

\node[text=drawColor,rotate= 90.00,anchor=base,inner sep=0pt, outer sep=0pt, scale=  0.66] at (305.30,115.48) {0.8};

\node[text=drawColor,rotate= 90.00,anchor=base,inner sep=0pt, outer sep=0pt, scale=  0.66] at (305.30,133.35) {1.0};

\path[draw=drawColor,line width= 0.4pt,line join=round,line cap=round] (314.81, 40.39) --
	(406.87, 40.39) --
	(406.87,136.93) --
	(314.81,136.93) --
	(314.81, 40.39);
\end{scope}
\begin{scope}
\path[clip] (282.33,  0.00) rectangle (423.50,169.40);
\definecolor{drawColor}{RGB}{0,0,0}

\node[text=drawColor,anchor=base,inner sep=0pt, outer sep=0pt, scale=  0.79] at (360.84,150.43) {\bfseries gamma};

\node[text=drawColor,anchor=base,inner sep=0pt, outer sep=0pt, scale=  0.66] at (360.84, 10.30) {sample size $n$};

\node[text=drawColor,rotate= 90.00,anchor=base,inner sep=0pt, outer sep=0pt, scale=  0.66] at (289.46, 88.66) {rejection rate};
\end{scope}
\begin{scope}
\path[clip] (314.81, 40.39) rectangle (406.87,136.93);
\definecolor{drawColor}{RGB}{69,139,0}

\path[draw=drawColor,line width= 0.4pt,line join=round,line cap=round] (318.22, 61.13) circle (  1.49);

\path[draw=drawColor,line width= 0.4pt,line join=round,line cap=round] (346.63, 59.61) circle (  1.49);

\path[draw=drawColor,line width= 0.4pt,line join=round,line cap=round] (375.05, 62.11) circle (  1.49);

\path[draw=drawColor,line width= 0.4pt,line join=round,line cap=round] (403.46, 66.31) circle (  1.49);
\definecolor{drawColor}{RGB}{211,211,211}

\path[draw=drawColor,line width= 0.4pt,dash pattern=on 1pt off 3pt ,line join=round,line cap=round] (318.22, 40.39) -- (318.22,136.93);

\path[draw=drawColor,line width= 0.4pt,dash pattern=on 1pt off 3pt ,line join=round,line cap=round] (346.63, 40.39) -- (346.63,136.93);

\path[draw=drawColor,line width= 0.4pt,dash pattern=on 1pt off 3pt ,line join=round,line cap=round] (375.05, 40.39) -- (375.05,136.93);

\path[draw=drawColor,line width= 0.4pt,dash pattern=on 1pt off 3pt ,line join=round,line cap=round] (403.46, 40.39) -- (403.46,136.93);

\path[draw=drawColor,line width= 0.4pt,dash pattern=on 1pt off 3pt ,line join=round,line cap=round] (314.81, 43.97) -- (406.87, 43.97);

\path[draw=drawColor,line width= 0.4pt,dash pattern=on 1pt off 3pt ,line join=round,line cap=round] (314.81, 61.84) -- (406.87, 61.84);

\path[draw=drawColor,line width= 0.4pt,dash pattern=on 1pt off 3pt ,line join=round,line cap=round] (314.81, 79.72) -- (406.87, 79.72);

\path[draw=drawColor,line width= 0.4pt,dash pattern=on 1pt off 3pt ,line join=round,line cap=round] (314.81, 97.60) -- (406.87, 97.60);

\path[draw=drawColor,line width= 0.4pt,dash pattern=on 1pt off 3pt ,line join=round,line cap=round] (314.81,115.48) -- (406.87,115.48);

\path[draw=drawColor,line width= 0.4pt,dash pattern=on 1pt off 3pt ,line join=round,line cap=round] (314.81,133.35) -- (406.87,133.35);
\definecolor{drawColor}{RGB}{0,0,0}
\definecolor{fillColor}{RGB}{255,255,255}

\path[draw=drawColor,line width= 0.4pt,line join=round,line cap=round,fill=fillColor] (319.41,132.10) rectangle (353.69,108.34);
\definecolor{drawColor}{RGB}{69,139,0}

\path[draw=drawColor,line width= 0.4pt,line join=round,line cap=round] (325.35,124.18) circle (  1.49);
\definecolor{drawColor}{RGB}{205,38,38}

\path[draw=drawColor,line width= 0.4pt,line join=round,line cap=round] (325.35,118.57) --
	(327.35,115.11) --
	(323.35,115.11) --
	(325.35,118.57);
\definecolor{drawColor}{RGB}{0,0,0}

\node[text=drawColor,anchor=base west,inner sep=0pt, outer sep=0pt, scale=  0.66] at (331.29,121.91) {dHSIC};

\node[text=drawColor,anchor=base west,inner sep=0pt, outer sep=0pt, scale=  0.66] at (331.29,113.99) {HSIC};
\end{scope}
\end{tikzpicture}

%% file: Plots/power_analysis2_romano.tex
\begin{tikzpicture}[x=1pt,y=1pt]
\definecolor{fillColor}{RGB}{255,255,255}
\path[use as bounding box,fill=fillColor,fill opacity=0.00] (0,0) rectangle (423.50,169.40);
\begin{scope}
\path[clip] ( 32.47, 40.39) rectangle (124.54,136.93);
\definecolor{drawColor}{RGB}{205,38,38}

\path[draw=drawColor,line width= 0.4pt,line join=round,line cap=round] ( 35.88, 70.59) --
	( 37.88, 67.13) --
	( 33.88, 67.13) --
	( 35.88, 70.59);

\path[draw=drawColor,line width= 0.4pt,line join=round,line cap=round] ( 64.30, 94.81) --
	( 66.30, 91.35) --
	( 62.30, 91.35) --
	( 64.30, 94.81);

\path[draw=drawColor,line width= 0.4pt,line join=round,line cap=round] ( 92.71,108.76) --
	( 94.71,105.29) --
	( 90.71,105.29) --
	( 92.71,108.76);

\path[draw=drawColor,line width= 0.4pt,line join=round,line cap=round] (121.13,119.75) --
	(123.13,116.29) --
	(119.13,116.29) --
	(121.13,119.75);
\end{scope}
\begin{scope}
\path[clip] (  0.00,  0.00) rectangle (423.50,169.40);
\definecolor{drawColor}{RGB}{0,0,0}

\path[draw=drawColor,line width= 0.4pt,line join=round,line cap=round] ( 35.88, 40.39) -- (121.13, 40.39);

\path[draw=drawColor,line width= 0.4pt,line join=round,line cap=round] ( 35.88, 40.39) -- ( 35.88, 36.43);

\path[draw=drawColor,line width= 0.4pt,line join=round,line cap=round] ( 64.30, 40.39) -- ( 64.30, 36.43);

\path[draw=drawColor,line width= 0.4pt,line join=round,line cap=round] ( 92.71, 40.39) -- ( 92.71, 36.43);

\path[draw=drawColor,line width= 0.4pt,line join=round,line cap=round] (121.13, 40.39) -- (121.13, 36.43);

\node[text=drawColor,anchor=base,inner sep=0pt, outer sep=0pt, scale=  0.66] at ( 35.88, 26.14) {50};

\node[text=drawColor,anchor=base,inner sep=0pt, outer sep=0pt, scale=  0.66] at ( 64.30, 26.14) {100};

\node[text=drawColor,anchor=base,inner sep=0pt, outer sep=0pt, scale=  0.66] at ( 92.71, 26.14) {150};

\node[text=drawColor,anchor=base,inner sep=0pt, outer sep=0pt, scale=  0.66] at (121.13, 26.14) {200};

\path[draw=drawColor,line width= 0.4pt,line join=round,line cap=round] ( 32.47, 43.97) -- ( 32.47,133.35);

\path[draw=drawColor,line width= 0.4pt,line join=round,line cap=round] ( 32.47, 43.97) -- ( 28.51, 43.97);

\path[draw=drawColor,line width= 0.4pt,line join=round,line cap=round] ( 32.47, 61.84) -- ( 28.51, 61.84);

\path[draw=drawColor,line width= 0.4pt,line join=round,line cap=round] ( 32.47, 79.72) -- ( 28.51, 79.72);

\path[draw=drawColor,line width= 0.4pt,line join=round,line cap=round] ( 32.47, 97.60) -- ( 28.51, 97.60);

\path[draw=drawColor,line width= 0.4pt,line join=round,line cap=round] ( 32.47,115.48) -- ( 28.51,115.48);

\path[draw=drawColor,line width= 0.4pt,line join=round,line cap=round] ( 32.47,133.35) -- ( 28.51,133.35);

\node[text=drawColor,rotate= 90.00,anchor=base,inner sep=0pt, outer sep=0pt, scale=  0.66] at ( 22.97, 43.97) {0.0};

\node[text=drawColor,rotate= 90.00,anchor=base,inner sep=0pt, outer sep=0pt, scale=  0.66] at ( 22.97, 61.84) {0.2};

\node[text=drawColor,rotate= 90.00,anchor=base,inner sep=0pt, outer sep=0pt, scale=  0.66] at ( 22.97, 79.72) {0.4};

\node[text=drawColor,rotate= 90.00,anchor=base,inner sep=0pt, outer sep=0pt, scale=  0.66] at ( 22.97, 97.60) {0.6};

\node[text=drawColor,rotate= 90.00,anchor=base,inner sep=0pt, outer sep=0pt, scale=  0.66] at ( 22.97,115.48) {0.8};

\node[text=drawColor,rotate= 90.00,anchor=base,inner sep=0pt, outer sep=0pt, scale=  0.66] at ( 22.97,133.35) {1.0};

\path[draw=drawColor,line width= 0.4pt,line join=round,line cap=round] ( 32.47, 40.39) --
	(124.54, 40.39) --
	(124.54,136.93) --
	( 32.47,136.93) --
	( 32.47, 40.39);
\end{scope}
\begin{scope}
\path[clip] (  0.00,  0.00) rectangle (141.17,169.40);
\definecolor{drawColor}{RGB}{0,0,0}

\node[text=drawColor,anchor=base,inner sep=0pt, outer sep=0pt, scale=  0.79] at ( 78.50,150.43) {\bfseries permutation};

\node[text=drawColor,anchor=base,inner sep=0pt, outer sep=0pt, scale=  0.66] at ( 78.50, 10.30) {sample size $n$};

\node[text=drawColor,rotate= 90.00,anchor=base,inner sep=0pt, outer sep=0pt, scale=  0.66] at (  7.13, 88.66) {rejection rate};
\end{scope}
\begin{scope}
\path[clip] ( 32.47, 40.39) rectangle (124.54,136.93);
\definecolor{drawColor}{RGB}{69,139,0}

\path[draw=drawColor,line width= 0.4pt,line join=round,line cap=round] ( 35.88, 85.98) circle (  1.49);

\path[draw=drawColor,line width= 0.4pt,line join=round,line cap=round] ( 64.30,115.57) circle (  1.49);

\path[draw=drawColor,line width= 0.4pt,line join=round,line cap=round] ( 92.71,127.72) circle (  1.49);

\path[draw=drawColor,line width= 0.4pt,line join=round,line cap=round] (121.13,131.39) circle (  1.49);
\definecolor{drawColor}{RGB}{24,116,205}

\path[draw=drawColor,line width= 0.4pt,line join=round,line cap=round] ( 34.40, 58.84) rectangle ( 37.37, 61.81);

\path[draw=drawColor,line width= 0.4pt,line join=round,line cap=round] ( 62.81, 78.68) rectangle ( 65.78, 81.65);

\path[draw=drawColor,line width= 0.4pt,line join=round,line cap=round] ( 91.23, 95.13) rectangle ( 94.20, 98.10);

\path[draw=drawColor,line width= 0.4pt,line join=round,line cap=round] (119.64,106.21) rectangle (122.61,109.18);
\definecolor{drawColor}{RGB}{211,211,211}

\path[draw=drawColor,line width= 0.4pt,dash pattern=on 1pt off 3pt ,line join=round,line cap=round] ( 35.88, 40.39) -- ( 35.88,136.93);

\path[draw=drawColor,line width= 0.4pt,dash pattern=on 1pt off 3pt ,line join=round,line cap=round] ( 64.30, 40.39) -- ( 64.30,136.93);

\path[draw=drawColor,line width= 0.4pt,dash pattern=on 1pt off 3pt ,line join=round,line cap=round] ( 92.71, 40.39) -- ( 92.71,136.93);

\path[draw=drawColor,line width= 0.4pt,dash pattern=on 1pt off 3pt ,line join=round,line cap=round] (121.13, 40.39) -- (121.13,136.93);

\path[draw=drawColor,line width= 0.4pt,dash pattern=on 1pt off 3pt ,line join=round,line cap=round] ( 32.47, 43.97) -- (124.54, 43.97);

\path[draw=drawColor,line width= 0.4pt,dash pattern=on 1pt off 3pt ,line join=round,line cap=round] ( 32.47, 61.84) -- (124.54, 61.84);

\path[draw=drawColor,line width= 0.4pt,dash pattern=on 1pt off 3pt ,line join=round,line cap=round] ( 32.47, 79.72) -- (124.54, 79.72);

\path[draw=drawColor,line width= 0.4pt,dash pattern=on 1pt off 3pt ,line join=round,line cap=round] ( 32.47, 97.60) -- (124.54, 97.60);

\path[draw=drawColor,line width= 0.4pt,dash pattern=on 1pt off 3pt ,line join=round,line cap=round] ( 32.47,115.48) -- (124.54,115.48);

\path[draw=drawColor,line width= 0.4pt,dash pattern=on 1pt off 3pt ,line join=round,line cap=round] ( 32.47,133.35) -- (124.54,133.35);
\definecolor{drawColor}{RGB}{0,0,0}
\definecolor{fillColor}{RGB}{255,255,255}

\path[draw=drawColor,line width= 0.4pt,line join=round,line cap=round,fill=fillColor] ( 83.34, 76.90) rectangle (119.93, 45.22);
\definecolor{drawColor}{RGB}{69,139,0}

\path[draw=drawColor,line width= 0.4pt,line join=round,line cap=round] ( 89.28, 68.98) circle (  1.49);
\definecolor{drawColor}{RGB}{205,38,38}

\path[draw=drawColor,line width= 0.4pt,line join=round,line cap=round] ( 89.28, 63.37) --
	( 91.28, 59.90) --
	( 87.28, 59.90) --
	( 89.28, 63.37);
\definecolor{drawColor}{RGB}{24,116,205}

\path[draw=drawColor,line width= 0.4pt,line join=round,line cap=round] ( 87.80, 51.65) rectangle ( 90.77, 54.62);
\definecolor{drawColor}{RGB}{0,0,0}

\node[text=drawColor,anchor=base west,inner sep=0pt, outer sep=0pt, scale=  0.66] at ( 95.22, 66.71) {dHSIC};

\node[text=drawColor,anchor=base west,inner sep=0pt, outer sep=0pt, scale=  0.66] at ( 95.22, 58.79) {HSIC};

\node[text=drawColor,anchor=base west,inner sep=0pt, outer sep=0pt, scale=  0.66] at ( 95.22, 50.87) {BMR-$n$};
\end{scope}
\begin{scope}
\path[clip] (173.64, 40.39) rectangle (265.70,136.93);
\definecolor{drawColor}{RGB}{205,38,38}

\path[draw=drawColor,line width= 0.4pt,line join=round,line cap=round] (177.05, 69.07) --
	(179.05, 65.61) --
	(175.05, 65.61) --
	(177.05, 69.07);

\path[draw=drawColor,line width= 0.4pt,line join=round,line cap=round] (205.46, 92.04) --
	(207.46, 88.58) --
	(203.46, 88.58) --
	(205.46, 92.04);

\path[draw=drawColor,line width= 0.4pt,line join=round,line cap=round] (233.88,108.85) --
	(235.88,105.38) --
	(231.88,105.38) --
	(233.88,108.85);

\path[draw=drawColor,line width= 0.4pt,line join=round,line cap=round] (262.29,119.93) --
	(264.29,116.47) --
	(260.29,116.47) --
	(262.29,119.93);
\end{scope}
\begin{scope}
\path[clip] (  0.00,  0.00) rectangle (423.50,169.40);
\definecolor{drawColor}{RGB}{0,0,0}

\path[draw=drawColor,line width= 0.4pt,line join=round,line cap=round] (177.05, 40.39) -- (262.29, 40.39);

\path[draw=drawColor,line width= 0.4pt,line join=round,line cap=round] (177.05, 40.39) -- (177.05, 36.43);

\path[draw=drawColor,line width= 0.4pt,line join=round,line cap=round] (205.46, 40.39) -- (205.46, 36.43);

\path[draw=drawColor,line width= 0.4pt,line join=round,line cap=round] (233.88, 40.39) -- (233.88, 36.43);

\path[draw=drawColor,line width= 0.4pt,line join=round,line cap=round] (262.29, 40.39) -- (262.29, 36.43);

\node[text=drawColor,anchor=base,inner sep=0pt, outer sep=0pt, scale=  0.66] at (177.05, 26.14) {50};

\node[text=drawColor,anchor=base,inner sep=0pt, outer sep=0pt, scale=  0.66] at (205.46, 26.14) {100};

\node[text=drawColor,anchor=base,inner sep=0pt, outer sep=0pt, scale=  0.66] at (233.88, 26.14) {150};

\node[text=drawColor,anchor=base,inner sep=0pt, outer sep=0pt, scale=  0.66] at (262.29, 26.14) {200};

\path[draw=drawColor,line width= 0.4pt,line join=round,line cap=round] (173.64, 43.97) -- (173.64,133.35);

\path[draw=drawColor,line width= 0.4pt,line join=round,line cap=round] (173.64, 43.97) -- (169.68, 43.97);

\path[draw=drawColor,line width= 0.4pt,line join=round,line cap=round] (173.64, 61.84) -- (169.68, 61.84);

\path[draw=drawColor,line width= 0.4pt,line join=round,line cap=round] (173.64, 79.72) -- (169.68, 79.72);

\path[draw=drawColor,line width= 0.4pt,line join=round,line cap=round] (173.64, 97.60) -- (169.68, 97.60);

\path[draw=drawColor,line width= 0.4pt,line join=round,line cap=round] (173.64,115.48) -- (169.68,115.48);

\path[draw=drawColor,line width= 0.4pt,line join=round,line cap=round] (173.64,133.35) -- (169.68,133.35);

\node[text=drawColor,rotate= 90.00,anchor=base,inner sep=0pt, outer sep=0pt, scale=  0.66] at (164.14, 43.97) {0.0};

\node[text=drawColor,rotate= 90.00,anchor=base,inner sep=0pt, outer sep=0pt, scale=  0.66] at (164.14, 61.84) {0.2};

\node[text=drawColor,rotate= 90.00,anchor=base,inner sep=0pt, outer sep=0pt, scale=  0.66] at (164.14, 79.72) {0.4};

\node[text=drawColor,rotate= 90.00,anchor=base,inner sep=0pt, outer sep=0pt, scale=  0.66] at (164.14, 97.60) {0.6};

\node[text=drawColor,rotate= 90.00,anchor=base,inner sep=0pt, outer sep=0pt, scale=  0.66] at (164.14,115.48) {0.8};

\node[text=drawColor,rotate= 90.00,anchor=base,inner sep=0pt, outer sep=0pt, scale=  0.66] at (164.14,133.35) {1.0};

\path[draw=drawColor,line width= 0.4pt,line join=round,line cap=round] (173.64, 40.39) --
	(265.70, 40.39) --
	(265.70,136.93) --
	(173.64,136.93) --
	(173.64, 40.39);
\end{scope}
\begin{scope}
\path[clip] (141.17,  0.00) rectangle (282.33,169.40);
\definecolor{drawColor}{RGB}{0,0,0}

\node[text=drawColor,anchor=base,inner sep=0pt, outer sep=0pt, scale=  0.79] at (219.67,150.43) {\bfseries bootstrap};

\node[text=drawColor,anchor=base,inner sep=0pt, outer sep=0pt, scale=  0.66] at (219.67, 10.30) {sample size $n$};

\node[text=drawColor,rotate= 90.00,anchor=base,inner sep=0pt, outer sep=0pt, scale=  0.66] at (148.30, 88.66) {rejection rate};
\end{scope}
\begin{scope}
\path[clip] (173.64, 40.39) rectangle (265.70,136.93);
\definecolor{drawColor}{RGB}{69,139,0}

\path[draw=drawColor,line width= 0.4pt,line join=round,line cap=round] (177.05, 86.69) circle (  1.49);

\path[draw=drawColor,line width= 0.4pt,line join=round,line cap=round] (205.46,115.03) circle (  1.49);

\path[draw=drawColor,line width= 0.4pt,line join=round,line cap=round] (233.88,127.01) circle (  1.49);

\path[draw=drawColor,line width= 0.4pt,line join=round,line cap=round] (262.29,130.94) circle (  1.49);
\definecolor{drawColor}{RGB}{24,116,205}

\path[draw=drawColor,line width= 0.4pt,line join=round,line cap=round] (175.56, 61.43) rectangle (178.53, 64.40);

\path[draw=drawColor,line width= 0.4pt,line join=round,line cap=round] (203.98, 78.50) rectangle (206.95, 81.47);

\path[draw=drawColor,line width= 0.4pt,line join=round,line cap=round] (232.39, 94.06) rectangle (235.36, 97.03);

\path[draw=drawColor,line width= 0.4pt,line join=round,line cap=round] (260.81,107.02) rectangle (263.78,109.99);
\definecolor{drawColor}{RGB}{211,211,211}

\path[draw=drawColor,line width= 0.4pt,dash pattern=on 1pt off 3pt ,line join=round,line cap=round] (177.05, 40.39) -- (177.05,136.93);

\path[draw=drawColor,line width= 0.4pt,dash pattern=on 1pt off 3pt ,line join=round,line cap=round] (205.46, 40.39) -- (205.46,136.93);

\path[draw=drawColor,line width= 0.4pt,dash pattern=on 1pt off 3pt ,line join=round,line cap=round] (233.88, 40.39) -- (233.88,136.93);

\path[draw=drawColor,line width= 0.4pt,dash pattern=on 1pt off 3pt ,line join=round,line cap=round] (262.29, 40.39) -- (262.29,136.93);

\path[draw=drawColor,line width= 0.4pt,dash pattern=on 1pt off 3pt ,line join=round,line cap=round] (173.64, 43.97) -- (265.70, 43.97);

\path[draw=drawColor,line width= 0.4pt,dash pattern=on 1pt off 3pt ,line join=round,line cap=round] (173.64, 61.84) -- (265.70, 61.84);

\path[draw=drawColor,line width= 0.4pt,dash pattern=on 1pt off 3pt ,line join=round,line cap=round] (173.64, 79.72) -- (265.70, 79.72);

\path[draw=drawColor,line width= 0.4pt,dash pattern=on 1pt off 3pt ,line join=round,line cap=round] (173.64, 97.60) -- (265.70, 97.60);

\path[draw=drawColor,line width= 0.4pt,dash pattern=on 1pt off 3pt ,line join=round,line cap=round] (173.64,115.48) -- (265.70,115.48);

\path[draw=drawColor,line width= 0.4pt,dash pattern=on 1pt off 3pt ,line join=round,line cap=round] (173.64,133.35) -- (265.70,133.35);
\definecolor{drawColor}{RGB}{0,0,0}
\definecolor{fillColor}{RGB}{255,255,255}

\path[draw=drawColor,line width= 0.4pt,line join=round,line cap=round,fill=fillColor] (224.51, 76.90) rectangle (261.10, 45.22);
\definecolor{drawColor}{RGB}{69,139,0}

\path[draw=drawColor,line width= 0.4pt,line join=round,line cap=round] (230.45, 68.98) circle (  1.49);
\definecolor{drawColor}{RGB}{205,38,38}

\path[draw=drawColor,line width= 0.4pt,line join=round,line cap=round] (230.45, 63.37) --
	(232.45, 59.90) --
	(228.45, 59.90) --
	(230.45, 63.37);
\definecolor{drawColor}{RGB}{24,116,205}

\path[draw=drawColor,line width= 0.4pt,line join=round,line cap=round] (228.96, 51.65) rectangle (231.93, 54.62);
\definecolor{drawColor}{RGB}{0,0,0}

\node[text=drawColor,anchor=base west,inner sep=0pt, outer sep=0pt, scale=  0.66] at (236.39, 66.71) {dHSIC};

\node[text=drawColor,anchor=base west,inner sep=0pt, outer sep=0pt, scale=  0.66] at (236.39, 58.79) {HSIC};

\node[text=drawColor,anchor=base west,inner sep=0pt, outer sep=0pt, scale=  0.66] at (236.39, 50.87) {BMR-$n$};
\end{scope}
\begin{scope}
\path[clip] (314.81, 40.39) rectangle (406.87,136.93);
\definecolor{drawColor}{RGB}{205,38,38}

\path[draw=drawColor,line width= 0.4pt,line join=round,line cap=round] (318.22, 80.33) --
	(320.22, 76.87) --
	(316.22, 76.87) --
	(318.22, 80.33);

\path[draw=drawColor,line width= 0.4pt,line join=round,line cap=round] (346.63,106.52) --
	(348.63,103.06) --
	(344.63,103.06) --
	(346.63,106.52);

\path[draw=drawColor,line width= 0.4pt,line join=round,line cap=round] (375.05,120.47) --
	(377.05,117.00) --
	(373.05,117.00) --
	(375.05,120.47);

\path[draw=drawColor,line width= 0.4pt,line join=round,line cap=round] (403.46,127.89) --
	(405.46,124.42) --
	(401.46,124.42) --
	(403.46,127.89);
\end{scope}
\begin{scope}
\path[clip] (  0.00,  0.00) rectangle (423.50,169.40);
\definecolor{drawColor}{RGB}{0,0,0}

\path[draw=drawColor,line width= 0.4pt,line join=round,line cap=round] (318.22, 40.39) -- (403.46, 40.39);

\path[draw=drawColor,line width= 0.4pt,line join=round,line cap=round] (318.22, 40.39) -- (318.22, 36.43);

\path[draw=drawColor,line width= 0.4pt,line join=round,line cap=round] (346.63, 40.39) -- (346.63, 36.43);

\path[draw=drawColor,line width= 0.4pt,line join=round,line cap=round] (375.05, 40.39) -- (375.05, 36.43);

\path[draw=drawColor,line width= 0.4pt,line join=round,line cap=round] (403.46, 40.39) -- (403.46, 36.43);

\node[text=drawColor,anchor=base,inner sep=0pt, outer sep=0pt, scale=  0.66] at (318.22, 26.14) {50};

\node[text=drawColor,anchor=base,inner sep=0pt, outer sep=0pt, scale=  0.66] at (346.63, 26.14) {100};

\node[text=drawColor,anchor=base,inner sep=0pt, outer sep=0pt, scale=  0.66] at (375.05, 26.14) {150};

\node[text=drawColor,anchor=base,inner sep=0pt, outer sep=0pt, scale=  0.66] at (403.46, 26.14) {200};

\path[draw=drawColor,line width= 0.4pt,line join=round,line cap=round] (314.81, 43.97) -- (314.81,133.35);

\path[draw=drawColor,line width= 0.4pt,line join=round,line cap=round] (314.81, 43.97) -- (310.85, 43.97);

\path[draw=drawColor,line width= 0.4pt,line join=round,line cap=round] (314.81, 61.84) -- (310.85, 61.84);

\path[draw=drawColor,line width= 0.4pt,line join=round,line cap=round] (314.81, 79.72) -- (310.85, 79.72);

\path[draw=drawColor,line width= 0.4pt,line join=round,line cap=round] (314.81, 97.60) -- (310.85, 97.60);

\path[draw=drawColor,line width= 0.4pt,line join=round,line cap=round] (314.81,115.48) -- (310.85,115.48);

\path[draw=drawColor,line width= 0.4pt,line join=round,line cap=round] (314.81,133.35) -- (310.85,133.35);

\node[text=drawColor,rotate= 90.00,anchor=base,inner sep=0pt, outer sep=0pt, scale=  0.66] at (305.30, 43.97) {0.0};

\node[text=drawColor,rotate= 90.00,anchor=base,inner sep=0pt, outer sep=0pt, scale=  0.66] at (305.30, 61.84) {0.2};

\node[text=drawColor,rotate= 90.00,anchor=base,inner sep=0pt, outer sep=0pt, scale=  0.66] at (305.30, 79.72) {0.4};

\node[text=drawColor,rotate= 90.00,anchor=base,inner sep=0pt, outer sep=0pt, scale=  0.66] at (305.30, 97.60) {0.6};

\node[text=drawColor,rotate= 90.00,anchor=base,inner sep=0pt, outer sep=0pt, scale=  0.66] at (305.30,115.48) {0.8};

\node[text=drawColor,rotate= 90.00,anchor=base,inner sep=0pt, outer sep=0pt, scale=  0.66] at (305.30,133.35) {1.0};

\path[draw=drawColor,line width= 0.4pt,line join=round,line cap=round] (314.81, 40.39) --
	(406.87, 40.39) --
	(406.87,136.93) --
	(314.81,136.93) --
	(314.81, 40.39);
\end{scope}
\begin{scope}
\path[clip] (282.33,  0.00) rectangle (423.50,169.40);
\definecolor{drawColor}{RGB}{0,0,0}

\node[text=drawColor,anchor=base,inner sep=0pt, outer sep=0pt, scale=  0.79] at (360.84,150.43) {\bfseries gamma};

\node[text=drawColor,anchor=base,inner sep=0pt, outer sep=0pt, scale=  0.66] at (360.84, 10.30) {sample size $n$};

\node[text=drawColor,rotate= 90.00,anchor=base,inner sep=0pt, outer sep=0pt, scale=  0.66] at (289.46, 88.66) {rejection rate};
\end{scope}
\begin{scope}
\path[clip] (314.81, 40.39) rectangle (406.87,136.93);
\definecolor{drawColor}{RGB}{69,139,0}

\path[draw=drawColor,line width= 0.4pt,line join=round,line cap=round] (318.22, 98.49) circle (  1.49);

\path[draw=drawColor,line width= 0.4pt,line join=round,line cap=round] (346.63,121.91) circle (  1.49);

\path[draw=drawColor,line width= 0.4pt,line join=round,line cap=round] (375.05,129.69) circle (  1.49);

\path[draw=drawColor,line width= 0.4pt,line join=round,line cap=round] (403.46,131.74) circle (  1.49);
\definecolor{drawColor}{RGB}{211,211,211}

\path[draw=drawColor,line width= 0.4pt,dash pattern=on 1pt off 3pt ,line join=round,line cap=round] (318.22, 40.39) -- (318.22,136.93);

\path[draw=drawColor,line width= 0.4pt,dash pattern=on 1pt off 3pt ,line join=round,line cap=round] (346.63, 40.39) -- (346.63,136.93);

\path[draw=drawColor,line width= 0.4pt,dash pattern=on 1pt off 3pt ,line join=round,line cap=round] (375.05, 40.39) -- (375.05,136.93);

\path[draw=drawColor,line width= 0.4pt,dash pattern=on 1pt off 3pt ,line join=round,line cap=round] (403.46, 40.39) -- (403.46,136.93);

\path[draw=drawColor,line width= 0.4pt,dash pattern=on 1pt off 3pt ,line join=round,line cap=round] (314.81, 43.97) -- (406.87, 43.97);

\path[draw=drawColor,line width= 0.4pt,dash pattern=on 1pt off 3pt ,line join=round,line cap=round] (314.81, 61.84) -- (406.87, 61.84);

\path[draw=drawColor,line width= 0.4pt,dash pattern=on 1pt off 3pt ,line join=round,line cap=round] (314.81, 79.72) -- (406.87, 79.72);

\path[draw=drawColor,line width= 0.4pt,dash pattern=on 1pt off 3pt ,line join=round,line cap=round] (314.81, 97.60) -- (406.87, 97.60);

\path[draw=drawColor,line width= 0.4pt,dash pattern=on 1pt off 3pt ,line join=round,line cap=round] (314.81,115.48) -- (406.87,115.48);

\path[draw=drawColor,line width= 0.4pt,dash pattern=on 1pt off 3pt ,line join=round,line cap=round] (314.81,133.35) -- (406.87,133.35);
\definecolor{drawColor}{RGB}{0,0,0}
\definecolor{fillColor}{RGB}{255,255,255}

\path[draw=drawColor,line width= 0.4pt,line join=round,line cap=round,fill=fillColor] (367.99, 68.98) rectangle (402.27, 45.22);
\definecolor{drawColor}{RGB}{69,139,0}

\path[draw=drawColor,line width= 0.4pt,line join=round,line cap=round] (373.93, 61.06) circle (  1.49);
\definecolor{drawColor}{RGB}{205,38,38}

\path[draw=drawColor,line width= 0.4pt,line join=round,line cap=round] (373.93, 55.45) --
	(375.93, 51.98) --
	(371.93, 51.98) --
	(373.93, 55.45);
\definecolor{drawColor}{RGB}{0,0,0}

\node[text=drawColor,anchor=base west,inner sep=0pt, outer sep=0pt, scale=  0.66] at (379.87, 58.79) {dHSIC};

\node[text=drawColor,anchor=base west,inner sep=0pt, outer sep=0pt, scale=  0.66] at (379.87, 50.87) {dHSIC};
\end{scope}
\end{tikzpicture}

%% file: Plots/power_dense_sparse.tex
\begin{tikzpicture}[x=1pt,y=1pt]
\definecolor{fillColor}{RGB}{255,255,255}
\path[use as bounding box,fill=fillColor,fill opacity=0.00] (0,0) rectangle (423.50,254.10);
\begin{scope}
\path[clip] ( 49.20, 61.20) rectangle (186.55,204.90);
\definecolor{drawColor}{RGB}{69,139,0}

\path[draw=drawColor,line width= 0.4pt,line join=round,line cap=round] (181.46,199.49) circle (  2.25);

\path[draw=drawColor,line width= 0.4pt,line join=round,line cap=round] (169.56,199.08) circle (  2.25);

\path[draw=drawColor,line width= 0.4pt,line join=round,line cap=round] (157.81,197.83) circle (  2.25);

\path[draw=drawColor,line width= 0.4pt,line join=round,line cap=round] (145.79,195.17) circle (  2.25);

\path[draw=drawColor,line width= 0.4pt,line join=round,line cap=round] (137.87,190.75) circle (  2.25);

\path[draw=drawColor,line width= 0.4pt,line join=round,line cap=round] (129.92,182.47) circle (  2.25);

\path[draw=drawColor,line width= 0.4pt,line join=round,line cap=round] (122.91,174.31) circle (  2.25);

\path[draw=drawColor,line width= 0.4pt,line join=round,line cap=round] (115.04,163.17) circle (  2.25);

\path[draw=drawColor,line width= 0.4pt,line join=round,line cap=round] (109.24,151.20) circle (  2.25);

\path[draw=drawColor,line width= 0.4pt,line join=round,line cap=round] (104.82,141.40) circle (  2.25);

\path[draw=drawColor,line width= 0.4pt,line join=round,line cap=round] ( 99.67,130.92) circle (  2.25);

\path[draw=drawColor,line width= 0.4pt,line join=round,line cap=round] ( 94.80,120.48) circle (  2.25);

\path[draw=drawColor,line width= 0.4pt,line join=round,line cap=round] ( 90.73,112.66) circle (  2.25);

\path[draw=drawColor,line width= 0.4pt,line join=round,line cap=round] ( 86.64,103.34) circle (  2.25);

\path[draw=drawColor,line width= 0.4pt,line join=round,line cap=round] ( 83.18, 98.35) circle (  2.25);

\path[draw=drawColor,line width= 0.4pt,line join=round,line cap=round] ( 80.35, 94.02) circle (  2.25);

\path[draw=drawColor,line width= 0.4pt,line join=round,line cap=round] ( 77.46, 88.64) circle (  2.25);

\path[draw=drawColor,line width= 0.4pt,line join=round,line cap=round] ( 74.01, 83.24) circle (  2.25);

\path[draw=drawColor,line width= 0.4pt,line join=round,line cap=round] ( 72.48, 81.98) circle (  2.25);

\path[draw=drawColor,line width= 0.4pt,line join=round,line cap=round] ( 70.52, 79.33) circle (  2.25);
\end{scope}
\begin{scope}
\path[clip] (  0.00,  0.00) rectangle (423.50,254.10);
\definecolor{drawColor}{RGB}{0,0,0}

\path[draw=drawColor,line width= 0.4pt,line join=round,line cap=round] ( 58.14, 61.20) -- (164.30, 61.20);

\path[draw=drawColor,line width= 0.4pt,line join=round,line cap=round] ( 58.14, 61.20) -- ( 58.14, 55.20);

\path[draw=drawColor,line width= 0.4pt,line join=round,line cap=round] ( 84.68, 61.20) -- ( 84.68, 55.20);

\path[draw=drawColor,line width= 0.4pt,line join=round,line cap=round] (111.22, 61.20) -- (111.22, 55.20);

\path[draw=drawColor,line width= 0.4pt,line join=round,line cap=round] (137.76, 61.20) -- (137.76, 55.20);

\path[draw=drawColor,line width= 0.4pt,line join=round,line cap=round] (164.30, 61.20) -- (164.30, 55.20);

\node[text=drawColor,anchor=base,inner sep=0pt, outer sep=0pt, scale=  1.00] at ( 58.14, 39.60) {0.05};

\node[text=drawColor,anchor=base,inner sep=0pt, outer sep=0pt, scale=  1.00] at ( 84.68, 39.60) {0.10};

\node[text=drawColor,anchor=base,inner sep=0pt, outer sep=0pt, scale=  1.00] at (111.22, 39.60) {0.15};

\node[text=drawColor,anchor=base,inner sep=0pt, outer sep=0pt, scale=  1.00] at (137.76, 39.60) {0.20};

\node[text=drawColor,anchor=base,inner sep=0pt, outer sep=0pt, scale=  1.00] at (164.30, 39.60) {0.25};

\path[draw=drawColor,line width= 0.4pt,line join=round,line cap=round] ( 49.20, 83.26) -- ( 49.20,199.58);

\path[draw=drawColor,line width= 0.4pt,line join=round,line cap=round] ( 49.20, 83.26) -- ( 43.20, 83.26);

\path[draw=drawColor,line width= 0.4pt,line join=round,line cap=round] ( 49.20,112.34) -- ( 43.20,112.34);

\path[draw=drawColor,line width= 0.4pt,line join=round,line cap=round] ( 49.20,141.42) -- ( 43.20,141.42);

\path[draw=drawColor,line width= 0.4pt,line join=round,line cap=round] ( 49.20,170.50) -- ( 43.20,170.50);

\path[draw=drawColor,line width= 0.4pt,line join=round,line cap=round] ( 49.20,199.58) -- ( 43.20,199.58);

\node[text=drawColor,rotate= 90.00,anchor=base,inner sep=0pt, outer sep=0pt, scale=  1.00] at ( 34.80, 83.26) {0.2};

\node[text=drawColor,rotate= 90.00,anchor=base,inner sep=0pt, outer sep=0pt, scale=  1.00] at ( 34.80,112.34) {0.4};

\node[text=drawColor,rotate= 90.00,anchor=base,inner sep=0pt, outer sep=0pt, scale=  1.00] at ( 34.80,141.42) {0.6};

\node[text=drawColor,rotate= 90.00,anchor=base,inner sep=0pt, outer sep=0pt, scale=  1.00] at ( 34.80,170.50) {0.8};

\node[text=drawColor,rotate= 90.00,anchor=base,inner sep=0pt, outer sep=0pt, scale=  1.00] at ( 34.80,199.58) {1.0};

\path[draw=drawColor,line width= 0.4pt,line join=round,line cap=round] ( 49.20, 61.20) --
	(186.55, 61.20) --
	(186.55,204.90) --
	( 49.20,204.90) --
	( 49.20, 61.20);
\end{scope}
\begin{scope}
\path[clip] (  0.00,  0.00) rectangle (211.75,254.10);
\definecolor{drawColor}{RGB}{0,0,0}

\node[text=drawColor,anchor=base,inner sep=0pt, outer sep=0pt, scale=  1.20] at (117.88,225.36) {\bfseries $\mathbf{d=5}$};

\node[text=drawColor,anchor=base,inner sep=0pt, outer sep=0pt, scale=  1.00] at (117.88, 15.60) {total variation distance};

\node[text=drawColor,rotate= 90.00,anchor=base,inner sep=0pt, outer sep=0pt, scale=  1.00] at ( 10.80,133.05) {rejection rate};
\end{scope}
\begin{scope}
\path[clip] ( 49.20, 61.20) rectangle (186.55,204.90);
\definecolor{drawColor}{RGB}{24,116,205}

\path[draw=drawColor,line width= 0.4pt,line join=round,line cap=round] (111.48,203.08) --
	(114.51,197.83) --
	(108.45,197.83) --
	(111.48,203.08);

\path[draw=drawColor,line width= 0.4pt,line join=round,line cap=round] (114.00,203.05) --
	(117.03,197.80) --
	(110.97,197.80) --
	(114.00,203.05);

\path[draw=drawColor,line width= 0.4pt,line join=round,line cap=round] (113.11,202.41) --
	(116.14,197.16) --
	(110.08,197.16) --
	(113.11,202.41);

\path[draw=drawColor,line width= 0.4pt,line join=round,line cap=round] (111.36,197.55) --
	(114.39,192.30) --
	(108.32,192.30) --
	(111.36,197.55);

\path[draw=drawColor,line width= 0.4pt,line join=round,line cap=round] (104.81,185.14) --
	(107.85,179.89) --
	(101.78,179.89) --
	(104.81,185.14);

\path[draw=drawColor,line width= 0.4pt,line join=round,line cap=round] (100.32,166.26) --
	(103.35,161.01) --
	( 97.29,161.01) --
	(100.32,166.26);

\path[draw=drawColor,line width= 0.4pt,line join=round,line cap=round] ( 93.96,147.13) --
	( 96.99,141.88) --
	( 90.93,141.88) --
	( 93.96,147.13);

\path[draw=drawColor,line width= 0.4pt,line join=round,line cap=round] ( 88.90,128.37) --
	( 91.93,123.12) --
	( 85.87,123.12) --
	( 88.90,128.37);

\path[draw=drawColor,line width= 0.4pt,line join=round,line cap=round] ( 83.70,113.19) --
	( 86.73,107.94) --
	( 80.67,107.94) --
	( 83.70,113.19);

\path[draw=drawColor,line width= 0.4pt,line join=round,line cap=round] ( 79.25,103.03) --
	( 82.29, 97.78) --
	( 76.22, 97.78) --
	( 79.25,103.03);

\path[draw=drawColor,line width= 0.4pt,line join=round,line cap=round] ( 75.13, 95.86) --
	( 78.16, 90.61) --
	( 72.10, 90.61) --
	( 75.13, 95.86);

\path[draw=drawColor,line width= 0.4pt,line join=round,line cap=round] ( 71.49, 87.86) --
	( 74.52, 82.61) --
	( 68.46, 82.61) --
	( 71.49, 87.86);

\path[draw=drawColor,line width= 0.4pt,line join=round,line cap=round] ( 68.72, 83.73) --
	( 71.75, 78.48) --
	( 65.69, 78.48) --
	( 68.72, 83.73);

\path[draw=drawColor,line width= 0.4pt,line join=round,line cap=round] ( 66.04, 79.65) --
	( 69.07, 74.40) --
	( 63.01, 74.40) --
	( 66.04, 79.65);

\path[draw=drawColor,line width= 0.4pt,line join=round,line cap=round] ( 63.57, 78.13) --
	( 66.60, 72.89) --
	( 60.54, 72.89) --
	( 63.57, 78.13);

\path[draw=drawColor,line width= 0.4pt,line join=round,line cap=round] ( 61.27, 74.73) --
	( 64.30, 69.48) --
	( 58.24, 69.48) --
	( 61.27, 74.73);

\path[draw=drawColor,line width= 0.4pt,line join=round,line cap=round] ( 59.20, 73.77) --
	( 62.23, 68.52) --
	( 56.17, 68.52) --
	( 59.20, 73.77);

\path[draw=drawColor,line width= 0.4pt,line join=round,line cap=round] ( 57.48, 72.84) --
	( 60.51, 67.59) --
	( 54.45, 67.59) --
	( 57.48, 72.84);

\path[draw=drawColor,line width= 0.4pt,line join=round,line cap=round] ( 55.63, 71.05) --
	( 58.66, 65.81) --
	( 52.59, 65.81) --
	( 55.63, 71.05);

\path[draw=drawColor,line width= 0.4pt,line join=round,line cap=round] ( 54.29, 70.02) --
	( 57.32, 64.77) --
	( 51.26, 64.77) --
	( 54.29, 70.02);
\definecolor{drawColor}{RGB}{0,0,0}

\path[draw=drawColor,line width= 0.4pt,line join=round,line cap=round] (134.51,100.07) rectangle (183.80, 64.07);
\definecolor{drawColor}{RGB}{69,139,0}

\path[draw=drawColor,line width= 0.4pt,line join=round,line cap=round] (143.51, 88.07) circle (  2.25);
\definecolor{drawColor}{RGB}{24,116,205}

\path[draw=drawColor,line width= 0.4pt,line join=round,line cap=round] (143.51, 79.57) --
	(146.54, 74.32) --
	(140.47, 74.32) --
	(143.51, 79.57);
\definecolor{drawColor}{RGB}{0,0,0}

\node[text=drawColor,anchor=base west,inner sep=0pt, outer sep=0pt, scale=  1.00] at (152.51, 84.63) {dense};

\node[text=drawColor,anchor=base west,inner sep=0pt, outer sep=0pt, scale=  1.00] at (152.51, 72.63) {sparse};
\end{scope}
\begin{scope}
\path[clip] (260.95, 61.20) rectangle (398.30,204.90);
\definecolor{drawColor}{RGB}{69,139,0}

\path[draw=drawColor,line width= 0.4pt,line join=round,line cap=round] (393.22,199.58) circle (  2.25);

\path[draw=drawColor,line width= 0.4pt,line join=round,line cap=round] (386.60,172.49) circle (  2.25);

\path[draw=drawColor,line width= 0.4pt,line join=round,line cap=round] (374.62,158.08) circle (  2.25);

\path[draw=drawColor,line width= 0.4pt,line join=round,line cap=round] (366.33,143.84) circle (  2.25);

\path[draw=drawColor,line width= 0.4pt,line join=round,line cap=round] (357.04,133.81) circle (  2.25);

\path[draw=drawColor,line width= 0.4pt,line join=round,line cap=round] (350.30,120.61) circle (  2.25);

\path[draw=drawColor,line width= 0.4pt,line join=round,line cap=round] (337.50,113.28) circle (  2.25);

\path[draw=drawColor,line width= 0.4pt,line join=round,line cap=round] (329.37,107.69) circle (  2.25);

\path[draw=drawColor,line width= 0.4pt,line join=round,line cap=round] (330.45,100.95) circle (  2.25);

\path[draw=drawColor,line width= 0.4pt,line join=round,line cap=round] (325.41, 95.92) circle (  2.25);

\path[draw=drawColor,line width= 0.4pt,line join=round,line cap=round] (315.34, 91.86) circle (  2.25);

\path[draw=drawColor,line width= 0.4pt,line join=round,line cap=round] (311.62, 89.43) circle (  2.25);

\path[draw=drawColor,line width= 0.4pt,line join=round,line cap=round] (312.90, 84.43) circle (  2.25);

\path[draw=drawColor,line width= 0.4pt,line join=round,line cap=round] (306.22, 81.22) circle (  2.25);

\path[draw=drawColor,line width= 0.4pt,line join=round,line cap=round] (299.27, 80.86) circle (  2.25);

\path[draw=drawColor,line width= 0.4pt,line join=round,line cap=round] (293.18, 81.93) circle (  2.25);

\path[draw=drawColor,line width= 0.4pt,line join=round,line cap=round] (294.53, 77.55) circle (  2.25);

\path[draw=drawColor,line width= 0.4pt,line join=round,line cap=round] (290.48, 77.33) circle (  2.25);

\path[draw=drawColor,line width= 0.4pt,line join=round,line cap=round] (282.94, 74.24) circle (  2.25);

\path[draw=drawColor,line width= 0.4pt,line join=round,line cap=round] (285.91, 74.76) circle (  2.25);
\end{scope}
\begin{scope}
\path[clip] (  0.00,  0.00) rectangle (423.50,254.10);
\definecolor{drawColor}{RGB}{0,0,0}

\path[draw=drawColor,line width= 0.4pt,line join=round,line cap=round] (271.96, 61.20) -- (393.65, 61.20);

\path[draw=drawColor,line width= 0.4pt,line join=round,line cap=round] (271.96, 61.20) -- (271.96, 55.20);

\path[draw=drawColor,line width= 0.4pt,line join=round,line cap=round] (312.52, 61.20) -- (312.52, 55.20);

\path[draw=drawColor,line width= 0.4pt,line join=round,line cap=round] (353.09, 61.20) -- (353.09, 55.20);

\path[draw=drawColor,line width= 0.4pt,line join=round,line cap=round] (393.65, 61.20) -- (393.65, 55.20);

\node[text=drawColor,anchor=base,inner sep=0pt, outer sep=0pt, scale=  1.00] at (271.96, 39.60) {0.05};

\node[text=drawColor,anchor=base,inner sep=0pt, outer sep=0pt, scale=  1.00] at (312.52, 39.60) {0.10};

\node[text=drawColor,anchor=base,inner sep=0pt, outer sep=0pt, scale=  1.00] at (353.09, 39.60) {0.15};

\node[text=drawColor,anchor=base,inner sep=0pt, outer sep=0pt, scale=  1.00] at (393.65, 39.60) {0.20};

\path[draw=drawColor,line width= 0.4pt,line join=round,line cap=round] (260.95, 81.58) -- (260.95,178.91);

\path[draw=drawColor,line width= 0.4pt,line join=round,line cap=round] (260.95, 81.58) -- (254.95, 81.58);

\path[draw=drawColor,line width= 0.4pt,line join=round,line cap=round] (260.95,114.02) -- (254.95,114.02);

\path[draw=drawColor,line width= 0.4pt,line join=round,line cap=round] (260.95,146.47) -- (254.95,146.47);

\path[draw=drawColor,line width= 0.4pt,line join=round,line cap=round] (260.95,178.91) -- (254.95,178.91);

\node[text=drawColor,rotate= 90.00,anchor=base,inner sep=0pt, outer sep=0pt, scale=  1.00] at (246.55, 81.58) {0.1};

\node[text=drawColor,rotate= 90.00,anchor=base,inner sep=0pt, outer sep=0pt, scale=  1.00] at (246.55,114.02) {0.2};

\node[text=drawColor,rotate= 90.00,anchor=base,inner sep=0pt, outer sep=0pt, scale=  1.00] at (246.55,146.47) {0.3};

\node[text=drawColor,rotate= 90.00,anchor=base,inner sep=0pt, outer sep=0pt, scale=  1.00] at (246.55,178.91) {0.4};

\path[draw=drawColor,line width= 0.4pt,line join=round,line cap=round] (260.95, 61.20) --
	(398.30, 61.20) --
	(398.30,204.90) --
	(260.95,204.90) --
	(260.95, 61.20);
\end{scope}
\begin{scope}
\path[clip] (211.75,  0.00) rectangle (423.50,254.10);
\definecolor{drawColor}{RGB}{0,0,0}

\node[text=drawColor,anchor=base,inner sep=0pt, outer sep=0pt, scale=  1.20] at (329.63,225.36) {\bfseries $\mathbf{d=10}$};

\node[text=drawColor,anchor=base,inner sep=0pt, outer sep=0pt, scale=  1.00] at (329.63, 15.60) {total variation distance};

\node[text=drawColor,rotate= 90.00,anchor=base,inner sep=0pt, outer sep=0pt, scale=  1.00] at (222.55,133.05) {rejection rate};
\end{scope}
\begin{scope}
\path[clip] (260.95, 61.20) rectangle (398.30,204.90);
\definecolor{drawColor}{RGB}{24,116,205}

\path[draw=drawColor,line width= 0.4pt,line join=round,line cap=round] (329.56,158.40) --
	(332.59,153.15) --
	(326.53,153.15) --
	(329.56,158.40);

\path[draw=drawColor,line width= 0.4pt,line join=round,line cap=round] (332.17,134.49) --
	(335.20,129.24) --
	(329.14,129.24) --
	(332.17,134.49);

\path[draw=drawColor,line width= 0.4pt,line join=round,line cap=round] (331.31,115.77) --
	(334.34,110.52) --
	(328.28,110.52) --
	(331.31,115.77);

\path[draw=drawColor,line width= 0.4pt,line join=round,line cap=round] (332.37,104.02) --
	(335.40, 98.78) --
	(329.34, 98.78) --
	(332.37,104.02);

\path[draw=drawColor,line width= 0.4pt,line join=round,line cap=round] (330.84, 96.33) --
	(333.87, 91.09) --
	(327.81, 91.09) --
	(330.84, 96.33);

\path[draw=drawColor,line width= 0.4pt,line join=round,line cap=round] (321.59, 88.26) --
	(324.62, 83.01) --
	(318.56, 83.01) --
	(321.59, 88.26);

\path[draw=drawColor,line width= 0.4pt,line join=round,line cap=round] (317.57, 85.72) --
	(320.60, 80.48) --
	(314.54, 80.48) --
	(317.57, 85.72);

\path[draw=drawColor,line width= 0.4pt,line join=round,line cap=round] (312.30, 81.31) --
	(315.33, 76.06) --
	(309.27, 76.06) --
	(312.30, 81.31);

\path[draw=drawColor,line width= 0.4pt,line join=round,line cap=round] (301.93, 80.37) --
	(304.97, 75.12) --
	(298.90, 75.12) --
	(301.93, 80.37);

\path[draw=drawColor,line width= 0.4pt,line join=round,line cap=round] (298.79, 77.87) --
	(301.82, 72.62) --
	(295.76, 72.62) --
	(298.79, 77.87);

\path[draw=drawColor,line width= 0.4pt,line join=round,line cap=round] (293.51, 78.42) --
	(296.54, 73.18) --
	(290.48, 73.18) --
	(293.51, 78.42);

\path[draw=drawColor,line width= 0.4pt,line join=round,line cap=round] (289.51, 74.43) --
	(292.54, 69.19) --
	(286.48, 69.19) --
	(289.51, 74.43);

\path[draw=drawColor,line width= 0.4pt,line join=round,line cap=round] (286.10, 73.75) --
	(289.13, 68.50) --
	(283.07, 68.50) --
	(286.10, 73.75);

\path[draw=drawColor,line width= 0.4pt,line join=round,line cap=round] (281.25, 73.66) --
	(284.28, 68.41) --
	(278.22, 68.41) --
	(281.25, 73.66);

\path[draw=drawColor,line width= 0.4pt,line join=round,line cap=round] (278.82, 73.07) --
	(281.85, 67.82) --
	(275.79, 67.82) --
	(278.82, 73.07);

\path[draw=drawColor,line width= 0.4pt,line join=round,line cap=round] (275.12, 71.32) --
	(278.15, 66.07) --
	(272.09, 66.07) --
	(275.12, 71.32);

\path[draw=drawColor,line width= 0.4pt,line join=round,line cap=round] (271.36, 72.52) --
	(274.39, 67.27) --
	(268.33, 67.27) --
	(271.36, 72.52);

\path[draw=drawColor,line width= 0.4pt,line join=round,line cap=round] (270.45, 70.02) --
	(273.48, 64.77) --
	(267.42, 64.77) --
	(270.45, 70.02);

\path[draw=drawColor,line width= 0.4pt,line join=round,line cap=round] (266.87, 72.68) --
	(269.90, 67.43) --
	(263.84, 67.43) --
	(266.87, 72.68);

\path[draw=drawColor,line width= 0.4pt,line join=round,line cap=round] (266.04, 71.12) --
	(269.07, 65.88) --
	(263.01, 65.88) --
	(266.04, 71.12);
\definecolor{drawColor}{RGB}{0,0,0}

\path[draw=drawColor,line width= 0.4pt,line join=round,line cap=round] (346.26,100.07) rectangle (395.56, 64.07);
\definecolor{drawColor}{RGB}{69,139,0}

\path[draw=drawColor,line width= 0.4pt,line join=round,line cap=round] (355.26, 88.07) circle (  2.25);
\definecolor{drawColor}{RGB}{24,116,205}

\path[draw=drawColor,line width= 0.4pt,line join=round,line cap=round] (355.26, 79.57) --
	(358.29, 74.32) --
	(352.23, 74.32) --
	(355.26, 79.57);
\definecolor{drawColor}{RGB}{0,0,0}

\node[text=drawColor,anchor=base west,inner sep=0pt, outer sep=0pt, scale=  1.00] at (364.26, 84.63) {dense};

\node[text=drawColor,anchor=base west,inner sep=0pt, outer sep=0pt, scale=  1.00] at (364.26, 72.63) {sparse};
\end{scope}
\end{tikzpicture}

%% file: Plots/power_bandwidth.tex
\begin{tikzpicture}[x=1pt,y=1pt]
\definecolor{fillColor}{RGB}{255,255,255}
\path[use as bounding box,fill=fillColor,fill opacity=0.00] (0,0) rectangle (423.50,169.40);
\begin{scope}
\path[clip] ( 32.47, 40.39) rectangle (124.54,136.93);
\definecolor{drawColor}{RGB}{69,139,0}

\path[draw=drawColor,line width= 0.4pt,line join=round,line cap=round] ( 35.88, 43.97) circle (  1.49);

\path[draw=drawColor,line width= 0.4pt,line join=round,line cap=round] ( 36.74, 43.97) circle (  1.49);

\path[draw=drawColor,line width= 0.4pt,line join=round,line cap=round] ( 37.60, 45.21) circle (  1.49);

\path[draw=drawColor,line width= 0.4pt,line join=round,line cap=round] ( 38.46, 44.79) circle (  1.49);

\path[draw=drawColor,line width= 0.4pt,line join=round,line cap=round] ( 39.33, 45.00) circle (  1.49);

\path[draw=drawColor,line width= 0.4pt,line join=round,line cap=round] ( 40.19, 45.62) circle (  1.49);

\path[draw=drawColor,line width= 0.4pt,line join=round,line cap=round] ( 41.05, 46.55) circle (  1.49);

\path[draw=drawColor,line width= 0.4pt,line join=round,line cap=round] ( 41.91, 49.23) circle (  1.49);

\path[draw=drawColor,line width= 0.4pt,line join=round,line cap=round] ( 42.77, 50.06) circle (  1.49);

\path[draw=drawColor,line width= 0.4pt,line join=round,line cap=round] ( 43.63, 52.95) circle (  1.49);

\path[draw=drawColor,line width= 0.4pt,line join=round,line cap=round] ( 44.49, 54.50) circle (  1.49);

\path[draw=drawColor,line width= 0.4pt,line join=round,line cap=round] ( 45.35, 58.11) circle (  1.49);

\path[draw=drawColor,line width= 0.4pt,line join=round,line cap=round] ( 46.21, 61.82) circle (  1.49);

\path[draw=drawColor,line width= 0.4pt,line join=round,line cap=round] ( 47.08, 64.82) circle (  1.49);

\path[draw=drawColor,line width= 0.4pt,line join=round,line cap=round] ( 47.94, 70.08) circle (  1.49);

\path[draw=drawColor,line width= 0.4pt,line join=round,line cap=round] ( 48.80, 72.77) circle (  1.49);

\path[draw=drawColor,line width= 0.4pt,line join=round,line cap=round] ( 49.66, 77.10) circle (  1.49);

\path[draw=drawColor,line width= 0.4pt,line join=round,line cap=round] ( 50.52, 81.95) circle (  1.49);

\path[draw=drawColor,line width= 0.4pt,line join=round,line cap=round] ( 51.38, 85.98) circle (  1.49);

\path[draw=drawColor,line width= 0.4pt,line join=round,line cap=round] ( 52.24, 89.38) circle (  1.49);

\path[draw=drawColor,line width= 0.4pt,line join=round,line cap=round] ( 53.10, 93.31) circle (  1.49);

\path[draw=drawColor,line width= 0.4pt,line join=round,line cap=round] ( 53.96, 96.81) circle (  1.49);

\path[draw=drawColor,line width= 0.4pt,line join=round,line cap=round] ( 54.82,100.43) circle (  1.49);

\path[draw=drawColor,line width= 0.4pt,line join=round,line cap=round] ( 55.69,102.90) circle (  1.49);

\path[draw=drawColor,line width= 0.4pt,line join=round,line cap=round] ( 56.55,105.28) circle (  1.49);

\path[draw=drawColor,line width= 0.4pt,line join=round,line cap=round] ( 57.41,108.37) circle (  1.49);

\path[draw=drawColor,line width= 0.4pt,line join=round,line cap=round] ( 58.27,110.34) circle (  1.49);

\path[draw=drawColor,line width= 0.4pt,line join=round,line cap=round] ( 59.13,111.88) circle (  1.49);

\path[draw=drawColor,line width= 0.4pt,line join=round,line cap=round] ( 59.99,113.85) circle (  1.49);

\path[draw=drawColor,line width= 0.4pt,line join=round,line cap=round] ( 60.85,115.70) circle (  1.49);

\path[draw=drawColor,line width= 0.4pt,line join=round,line cap=round] ( 62.57,118.28) circle (  1.49);

\path[draw=drawColor,line width= 0.4pt,line join=round,line cap=round] ( 64.30,120.35) circle (  1.49);

\path[draw=drawColor,line width= 0.4pt,line join=round,line cap=round] ( 66.02,122.62) circle (  1.49);

\path[draw=drawColor,line width= 0.4pt,line join=round,line cap=round] ( 67.74,123.65) circle (  1.49);

\path[draw=drawColor,line width= 0.4pt,line join=round,line cap=round] ( 69.46,125.30) circle (  1.49);

\path[draw=drawColor,line width= 0.4pt,line join=round,line cap=round] ( 71.18,125.82) circle (  1.49);

\path[draw=drawColor,line width= 0.4pt,line join=round,line cap=round] ( 72.91,126.95) circle (  1.49);

\path[draw=drawColor,line width= 0.4pt,line join=round,line cap=round] ( 74.63,127.57) circle (  1.49);

\path[draw=drawColor,line width= 0.4pt,line join=round,line cap=round] ( 76.35,128.30) circle (  1.49);

\path[draw=drawColor,line width= 0.4pt,line join=round,line cap=round] ( 78.07,128.19) circle (  1.49);

\path[draw=drawColor,line width= 0.4pt,line join=round,line cap=round] ( 79.80,129.02) circle (  1.49);

\path[draw=drawColor,line width= 0.4pt,line join=round,line cap=round] ( 81.52,129.22) circle (  1.49);

\path[draw=drawColor,line width= 0.4pt,line join=round,line cap=round] ( 83.24,130.05) circle (  1.49);

\path[draw=drawColor,line width= 0.4pt,line join=round,line cap=round] ( 84.96,130.05) circle (  1.49);

\path[draw=drawColor,line width= 0.4pt,line join=round,line cap=round] ( 86.68,130.36) circle (  1.49);

\path[draw=drawColor,line width= 0.4pt,line join=round,line cap=round] ( 88.41,131.19) circle (  1.49);

\path[draw=drawColor,line width= 0.4pt,line join=round,line cap=round] ( 90.13,129.95) circle (  1.49);

\path[draw=drawColor,line width= 0.4pt,line join=round,line cap=round] ( 91.85,131.70) circle (  1.49);

\path[draw=drawColor,line width= 0.4pt,line join=round,line cap=round] ( 93.57,130.77) circle (  1.49);

\path[draw=drawColor,line width= 0.4pt,line join=round,line cap=round] ( 95.29,131.19) circle (  1.49);

\path[draw=drawColor,line width= 0.4pt,line join=round,line cap=round] ( 97.02,131.91) circle (  1.49);

\path[draw=drawColor,line width= 0.4pt,line join=round,line cap=round] ( 98.74,130.98) circle (  1.49);

\path[draw=drawColor,line width= 0.4pt,line join=round,line cap=round] (100.46,131.50) circle (  1.49);

\path[draw=drawColor,line width= 0.4pt,line join=round,line cap=round] (102.18,131.70) circle (  1.49);

\path[draw=drawColor,line width= 0.4pt,line join=round,line cap=round] (103.90,132.42) circle (  1.49);

\path[draw=drawColor,line width= 0.4pt,line join=round,line cap=round] (105.63,132.11) circle (  1.49);

\path[draw=drawColor,line width= 0.4pt,line join=round,line cap=round] (107.35,132.53) circle (  1.49);

\path[draw=drawColor,line width= 0.4pt,line join=round,line cap=round] (109.07,132.53) circle (  1.49);

\path[draw=drawColor,line width= 0.4pt,line join=round,line cap=round] (110.79,133.35) circle (  1.49);

\path[draw=drawColor,line width= 0.4pt,line join=round,line cap=round] (112.52,132.32) circle (  1.49);

\path[draw=drawColor,line width= 0.4pt,line join=round,line cap=round] (114.24,132.11) circle (  1.49);

\path[draw=drawColor,line width= 0.4pt,line join=round,line cap=round] (115.96,132.22) circle (  1.49);

\path[draw=drawColor,line width= 0.4pt,line join=round,line cap=round] (117.68,132.94) circle (  1.49);

\path[draw=drawColor,line width= 0.4pt,line join=round,line cap=round] (119.40,132.94) circle (  1.49);

\path[draw=drawColor,line width= 0.4pt,line join=round,line cap=round] (121.13,132.42) circle (  1.49);
\end{scope}
\begin{scope}
\path[clip] (  0.00,  0.00) rectangle (423.50,169.40);
\definecolor{drawColor}{RGB}{0,0,0}

\path[draw=drawColor,line width= 0.4pt,line join=round,line cap=round] ( 35.02, 40.39) -- (121.13, 40.39);

\path[draw=drawColor,line width= 0.4pt,line join=round,line cap=round] ( 35.02, 40.39) -- ( 35.02, 36.43);

\path[draw=drawColor,line width= 0.4pt,line join=round,line cap=round] ( 52.24, 40.39) -- ( 52.24, 36.43);

\path[draw=drawColor,line width= 0.4pt,line join=round,line cap=round] ( 69.46, 40.39) -- ( 69.46, 36.43);

\path[draw=drawColor,line width= 0.4pt,line join=round,line cap=round] ( 86.68, 40.39) -- ( 86.68, 36.43);

\path[draw=drawColor,line width= 0.4pt,line join=round,line cap=round] (103.90, 40.39) -- (103.90, 36.43);

\path[draw=drawColor,line width= 0.4pt,line join=round,line cap=round] (121.13, 40.39) -- (121.13, 36.43);

\node[text=drawColor,anchor=base,inner sep=0pt, outer sep=0pt, scale=  0.66] at ( 35.02, 26.14) {0};

\node[text=drawColor,anchor=base,inner sep=0pt, outer sep=0pt, scale=  0.66] at ( 52.24, 26.14) {2};

\node[text=drawColor,anchor=base,inner sep=0pt, outer sep=0pt, scale=  0.66] at ( 69.46, 26.14) {4};

\node[text=drawColor,anchor=base,inner sep=0pt, outer sep=0pt, scale=  0.66] at ( 86.68, 26.14) {6};

\node[text=drawColor,anchor=base,inner sep=0pt, outer sep=0pt, scale=  0.66] at (103.90, 26.14) {8};

\node[text=drawColor,anchor=base,inner sep=0pt, outer sep=0pt, scale=  0.66] at (121.13, 26.14) {10};

\path[draw=drawColor,line width= 0.4pt,line join=round,line cap=round] ( 32.47, 58.83) -- ( 32.47,120.76);

\path[draw=drawColor,line width= 0.4pt,line join=round,line cap=round] ( 32.47, 58.83) -- ( 28.51, 58.83);

\path[draw=drawColor,line width= 0.4pt,line join=round,line cap=round] ( 32.47, 79.47) -- ( 28.51, 79.47);

\path[draw=drawColor,line width= 0.4pt,line join=round,line cap=round] ( 32.47,100.12) -- ( 28.51,100.12);

\path[draw=drawColor,line width= 0.4pt,line join=round,line cap=round] ( 32.47,120.76) -- ( 28.51,120.76);

\node[text=drawColor,rotate= 90.00,anchor=base,inner sep=0pt, outer sep=0pt, scale=  0.66] at ( 22.97, 58.83) {0.2};

\node[text=drawColor,rotate= 90.00,anchor=base,inner sep=0pt, outer sep=0pt, scale=  0.66] at ( 22.97, 79.47) {0.4};

\node[text=drawColor,rotate= 90.00,anchor=base,inner sep=0pt, outer sep=0pt, scale=  0.66] at ( 22.97,100.12) {0.6};

\node[text=drawColor,rotate= 90.00,anchor=base,inner sep=0pt, outer sep=0pt, scale=  0.66] at ( 22.97,120.76) {0.8};

\path[draw=drawColor,line width= 0.4pt,line join=round,line cap=round] ( 32.47, 40.39) --
	(124.54, 40.39) --
	(124.54,136.93) --
	( 32.47,136.93) --
	( 32.47, 40.39);
\end{scope}
\begin{scope}
\path[clip] (  0.00,  0.00) rectangle (141.17,169.40);
\definecolor{drawColor}{RGB}{0,0,0}

\node[text=drawColor,anchor=base,inner sep=0pt, outer sep=0pt, scale=  0.79] at ( 78.50,150.43) {\bfseries pairwise dependence};

\node[text=drawColor,anchor=base,inner sep=0pt, outer sep=0pt, scale=  0.66] at ( 78.50, 10.30) {bandwidth $\sigma$};

\node[text=drawColor,rotate= 90.00,anchor=base,inner sep=0pt, outer sep=0pt, scale=  0.66] at (  7.13, 88.66) {rejection rate};
\end{scope}
\begin{scope}
\path[clip] ( 32.47, 40.39) rectangle (124.54,136.93);
\definecolor{drawColor}{RGB}{205,38,38}

\path[draw=drawColor,line width= 0.4pt,dash pattern=on 4pt off 4pt ,line join=round,line cap=round] ( 46.11, 40.39) -- ( 46.11,136.93);

\path[draw=drawColor,line width= 0.4pt,dash pattern=on 4pt off 4pt ,line join=round,line cap=round] ( 50.08, 40.39) -- ( 50.08,136.93);

\path[draw=drawColor,line width= 0.4pt,line join=round,line cap=round] ( 48.06, 40.39) -- ( 48.06,136.93);
\end{scope}
\begin{scope}
\path[clip] (173.64, 40.39) rectangle (265.70,136.93);
\definecolor{drawColor}{RGB}{69,139,0}

\path[draw=drawColor,line width= 0.4pt,line join=round,line cap=round] (177.05, 43.97) circle (  1.49);

\path[draw=drawColor,line width= 0.4pt,line join=round,line cap=round] (177.27, 44.73) circle (  1.49);

\path[draw=drawColor,line width= 0.4pt,line join=round,line cap=round] (177.49, 44.62) circle (  1.49);

\path[draw=drawColor,line width= 0.4pt,line join=round,line cap=round] (177.71, 45.82) circle (  1.49);

\path[draw=drawColor,line width= 0.4pt,line join=round,line cap=round] (177.93, 49.19) circle (  1.49);

\path[draw=drawColor,line width= 0.4pt,line join=round,line cap=round] (178.14, 53.33) circle (  1.49);

\path[draw=drawColor,line width= 0.4pt,line join=round,line cap=round] (178.36, 58.77) circle (  1.49);

\path[draw=drawColor,line width= 0.4pt,line join=round,line cap=round] (178.58, 64.65) circle (  1.49);

\path[draw=drawColor,line width= 0.4pt,line join=round,line cap=round] (178.80, 70.10) circle (  1.49);

\path[draw=drawColor,line width= 0.4pt,line join=round,line cap=round] (179.02, 76.63) circle (  1.49);

\path[draw=drawColor,line width= 0.4pt,line join=round,line cap=round] (179.24, 81.86) circle (  1.49);

\path[draw=drawColor,line width= 0.4pt,line join=round,line cap=round] (179.46, 87.52) circle (  1.49);

\path[draw=drawColor,line width= 0.4pt,line join=round,line cap=round] (179.68, 94.27) circle (  1.49);

\path[draw=drawColor,line width= 0.4pt,line join=round,line cap=round] (179.90, 99.28) circle (  1.49);

\path[draw=drawColor,line width= 0.4pt,line join=round,line cap=round] (180.12,102.76) circle (  1.49);

\path[draw=drawColor,line width= 0.4pt,line join=round,line cap=round] (180.34,107.11) circle (  1.49);

\path[draw=drawColor,line width= 0.4pt,line join=round,line cap=round] (180.56,111.36) circle (  1.49);

\path[draw=drawColor,line width= 0.4pt,line join=round,line cap=round] (180.77,114.19) circle (  1.49);

\path[draw=drawColor,line width= 0.4pt,line join=round,line cap=round] (181.21,118.66) circle (  1.49);

\path[draw=drawColor,line width= 0.4pt,line join=round,line cap=round] (181.65,123.12) circle (  1.49);

\path[draw=drawColor,line width= 0.4pt,line join=round,line cap=round] (182.09,125.51) circle (  1.49);

\path[draw=drawColor,line width= 0.4pt,line join=round,line cap=round] (182.53,127.69) circle (  1.49);

\path[draw=drawColor,line width= 0.4pt,line join=round,line cap=round] (182.97,128.67) circle (  1.49);

\path[draw=drawColor,line width= 0.4pt,line join=round,line cap=round] (183.40,131.39) circle (  1.49);

\path[draw=drawColor,line width= 0.4pt,line join=round,line cap=round] (185.60,133.35) circle (  1.49);

\path[draw=drawColor,line width= 0.4pt,line join=round,line cap=round] (187.79,131.07) circle (  1.49);

\path[draw=drawColor,line width= 0.4pt,line join=round,line cap=round] (189.98,130.85) circle (  1.49);

\path[draw=drawColor,line width= 0.4pt,line join=round,line cap=round] (192.17,129.11) circle (  1.49);

\path[draw=drawColor,line width= 0.4pt,line join=round,line cap=round] (194.36,126.06) circle (  1.49);

\path[draw=drawColor,line width= 0.4pt,line join=round,line cap=round] (196.55,122.14) circle (  1.49);

\path[draw=drawColor,line width= 0.4pt,line join=round,line cap=round] (198.74,120.40) circle (  1.49);

\path[draw=drawColor,line width= 0.4pt,line join=round,line cap=round] (200.93,116.80) circle (  1.49);

\path[draw=drawColor,line width= 0.4pt,line join=round,line cap=round] (203.13,113.10) circle (  1.49);

\path[draw=drawColor,line width= 0.4pt,line join=round,line cap=round] (205.32,110.16) circle (  1.49);

\path[draw=drawColor,line width= 0.4pt,line join=round,line cap=round] (207.51,107.33) circle (  1.49);

\path[draw=drawColor,line width= 0.4pt,line join=round,line cap=round] (209.70,104.18) circle (  1.49);

\path[draw=drawColor,line width= 0.4pt,line join=round,line cap=round] (211.89,102.65) circle (  1.49);

\path[draw=drawColor,line width= 0.4pt,line join=round,line cap=round] (214.08,100.47) circle (  1.49);

\path[draw=drawColor,line width= 0.4pt,line join=round,line cap=round] (216.27, 98.84) circle (  1.49);

\path[draw=drawColor,line width= 0.4pt,line join=round,line cap=round] (218.47, 96.66) circle (  1.49);

\path[draw=drawColor,line width= 0.4pt,line join=round,line cap=round] (220.66, 94.81) circle (  1.49);

\path[draw=drawColor,line width= 0.4pt,line join=round,line cap=round] (222.85, 91.87) circle (  1.49);

\path[draw=drawColor,line width= 0.4pt,line join=round,line cap=round] (225.04, 90.24) circle (  1.49);

\path[draw=drawColor,line width= 0.4pt,line join=round,line cap=round] (227.23, 89.59) circle (  1.49);

\path[draw=drawColor,line width= 0.4pt,line join=round,line cap=round] (229.42, 89.26) circle (  1.49);

\path[draw=drawColor,line width= 0.4pt,line join=round,line cap=round] (231.61, 86.21) circle (  1.49);

\path[draw=drawColor,line width= 0.4pt,line join=round,line cap=round] (233.81, 84.14) circle (  1.49);

\path[draw=drawColor,line width= 0.4pt,line join=round,line cap=round] (236.00, 82.62) circle (  1.49);

\path[draw=drawColor,line width= 0.4pt,line join=round,line cap=round] (238.19, 81.86) circle (  1.49);

\path[draw=drawColor,line width= 0.4pt,line join=round,line cap=round] (240.38, 81.31) circle (  1.49);

\path[draw=drawColor,line width= 0.4pt,line join=round,line cap=round] (242.57, 80.22) circle (  1.49);

\path[draw=drawColor,line width= 0.4pt,line join=round,line cap=round] (244.76, 80.66) circle (  1.49);

\path[draw=drawColor,line width= 0.4pt,line join=round,line cap=round] (246.95, 78.37) circle (  1.49);

\path[draw=drawColor,line width= 0.4pt,line join=round,line cap=round] (249.14, 77.28) circle (  1.49);

\path[draw=drawColor,line width= 0.4pt,line join=round,line cap=round] (251.34, 77.07) circle (  1.49);

\path[draw=drawColor,line width= 0.4pt,line join=round,line cap=round] (253.53, 75.32) circle (  1.49);

\path[draw=drawColor,line width= 0.4pt,line join=round,line cap=round] (255.72, 75.54) circle (  1.49);

\path[draw=drawColor,line width= 0.4pt,line join=round,line cap=round] (257.91, 76.52) circle (  1.49);

\path[draw=drawColor,line width= 0.4pt,line join=round,line cap=round] (260.10, 74.67) circle (  1.49);

\path[draw=drawColor,line width= 0.4pt,line join=round,line cap=round] (262.29, 74.45) circle (  1.49);
\end{scope}
\begin{scope}
\path[clip] (  0.00,  0.00) rectangle (423.50,169.40);
\definecolor{drawColor}{RGB}{0,0,0}

\path[draw=drawColor,line width= 0.4pt,line join=round,line cap=round] (176.39, 40.39) -- (264.05, 40.39);

\path[draw=drawColor,line width= 0.4pt,line join=round,line cap=round] (176.39, 40.39) -- (176.39, 36.43);

\path[draw=drawColor,line width= 0.4pt,line join=round,line cap=round] (198.31, 40.39) -- (198.31, 36.43);

\path[draw=drawColor,line width= 0.4pt,line join=round,line cap=round] (220.22, 40.39) -- (220.22, 36.43);

\path[draw=drawColor,line width= 0.4pt,line join=round,line cap=round] (242.13, 40.39) -- (242.13, 36.43);

\path[draw=drawColor,line width= 0.4pt,line join=round,line cap=round] (264.05, 40.39) -- (264.05, 36.43);

\node[text=drawColor,anchor=base,inner sep=0pt, outer sep=0pt, scale=  0.66] at (176.39, 26.14) {0};

\node[text=drawColor,anchor=base,inner sep=0pt, outer sep=0pt, scale=  0.66] at (198.31, 26.14) {5};

\node[text=drawColor,anchor=base,inner sep=0pt, outer sep=0pt, scale=  0.66] at (220.22, 26.14) {10};

\node[text=drawColor,anchor=base,inner sep=0pt, outer sep=0pt, scale=  0.66] at (242.13, 26.14) {15};

\node[text=drawColor,anchor=base,inner sep=0pt, outer sep=0pt, scale=  0.66] at (264.05, 26.14) {20};

\path[draw=drawColor,line width= 0.4pt,line join=round,line cap=round] (173.64, 55.40) -- (173.64,120.72);

\path[draw=drawColor,line width= 0.4pt,line join=round,line cap=round] (173.64, 55.40) -- (169.68, 55.40);

\path[draw=drawColor,line width= 0.4pt,line join=round,line cap=round] (173.64, 77.17) -- (169.68, 77.17);

\path[draw=drawColor,line width= 0.4pt,line join=round,line cap=round] (173.64, 98.95) -- (169.68, 98.95);

\path[draw=drawColor,line width= 0.4pt,line join=round,line cap=round] (173.64,120.72) -- (169.68,120.72);

\node[text=drawColor,rotate= 90.00,anchor=base,inner sep=0pt, outer sep=0pt, scale=  0.66] at (164.14, 55.40) {0.2};

\node[text=drawColor,rotate= 90.00,anchor=base,inner sep=0pt, outer sep=0pt, scale=  0.66] at (164.14, 77.17) {0.4};

\node[text=drawColor,rotate= 90.00,anchor=base,inner sep=0pt, outer sep=0pt, scale=  0.66] at (164.14, 98.95) {0.6};

\node[text=drawColor,rotate= 90.00,anchor=base,inner sep=0pt, outer sep=0pt, scale=  0.66] at (164.14,120.72) {0.8};

\path[draw=drawColor,line width= 0.4pt,line join=round,line cap=round] (173.64, 40.39) --
	(265.70, 40.39) --
	(265.70,136.93) --
	(173.64,136.93) --
	(173.64, 40.39);
\end{scope}
\begin{scope}
\path[clip] (141.17,  0.00) rectangle (282.33,169.40);
\definecolor{drawColor}{RGB}{0,0,0}

\node[text=drawColor,anchor=base,inner sep=0pt, outer sep=0pt, scale=  0.79] at (219.67,150.43) {\bfseries mixed dependence};

\node[text=drawColor,anchor=base,inner sep=0pt, outer sep=0pt, scale=  0.66] at (219.67, 10.30) {bandwidth $\sigma$};

\node[text=drawColor,rotate= 90.00,anchor=base,inner sep=0pt, outer sep=0pt, scale=  0.66] at (148.30, 88.66) {rejection rate};
\end{scope}
\begin{scope}
\path[clip] (173.64, 40.39) rectangle (265.70,136.93);
\definecolor{drawColor}{RGB}{205,38,38}

\path[draw=drawColor,line width= 0.4pt,dash pattern=on 4pt off 4pt ,line join=round,line cap=round] (181.26, 40.39) -- (181.26,136.93);

\path[draw=drawColor,line width= 0.4pt,dash pattern=on 4pt off 4pt ,line join=round,line cap=round] (206.05, 40.39) -- (206.05,136.93);

\path[draw=drawColor,line width= 0.4pt,line join=round,line cap=round] (187.54, 40.39) -- (187.54,136.93);
\end{scope}
\begin{scope}
\path[clip] (314.81, 40.39) rectangle (406.87,136.93);
\definecolor{drawColor}{RGB}{69,139,0}

\path[draw=drawColor,line width= 0.4pt,line join=round,line cap=round] (318.22, 44.54) circle (  1.49);

\path[draw=drawColor,line width= 0.4pt,line join=round,line cap=round] (318.30, 47.31) circle (  1.49);

\path[draw=drawColor,line width= 0.4pt,line join=round,line cap=round] (318.39, 48.07) circle (  1.49);

\path[draw=drawColor,line width= 0.4pt,line join=round,line cap=round] (318.47, 51.22) circle (  1.49);

\path[draw=drawColor,line width= 0.4pt,line join=round,line cap=round] (318.56, 52.36) circle (  1.49);

\path[draw=drawColor,line width= 0.4pt,line join=round,line cap=round] (318.64, 55.80) circle (  1.49);

\path[draw=drawColor,line width= 0.4pt,line join=round,line cap=round] (318.73, 61.71) circle (  1.49);

\path[draw=drawColor,line width= 0.4pt,line join=round,line cap=round] (318.81, 66.00) circle (  1.49);

\path[draw=drawColor,line width= 0.4pt,line join=round,line cap=round] (318.90, 71.35) circle (  1.49);

\path[draw=drawColor,line width= 0.4pt,line join=round,line cap=round] (318.98, 76.69) circle (  1.49);

\path[draw=drawColor,line width= 0.4pt,line join=round,line cap=round] (319.07, 84.42) circle (  1.49);

\path[draw=drawColor,line width= 0.4pt,line join=round,line cap=round] (319.16, 91.38) circle (  1.49);

\path[draw=drawColor,line width= 0.4pt,line join=round,line cap=round] (319.24, 98.92) circle (  1.49);

\path[draw=drawColor,line width= 0.4pt,line join=round,line cap=round] (319.33,105.21) circle (  1.49);

\path[draw=drawColor,line width= 0.4pt,line join=round,line cap=round] (319.41,109.70) circle (  1.49);

\path[draw=drawColor,line width= 0.4pt,line join=round,line cap=round] (319.84,128.20) circle (  1.49);

\path[draw=drawColor,line width= 0.4pt,line join=round,line cap=round] (320.69,133.35) circle (  1.49);

\path[draw=drawColor,line width= 0.4pt,line join=round,line cap=round] (321.54,133.35) circle (  1.49);

\path[draw=drawColor,line width= 0.4pt,line join=round,line cap=round] (322.40,133.35) circle (  1.49);

\path[draw=drawColor,line width= 0.4pt,line join=round,line cap=round] (323.25,133.35) circle (  1.49);

\path[draw=drawColor,line width= 0.4pt,line join=round,line cap=round] (324.10,133.35) circle (  1.49);

\path[draw=drawColor,line width= 0.4pt,line join=round,line cap=round] (324.96,133.35) circle (  1.49);

\path[draw=drawColor,line width= 0.4pt,line join=round,line cap=round] (325.81,133.35) circle (  1.49);

\path[draw=drawColor,line width= 0.4pt,line join=round,line cap=round] (326.66,133.35) circle (  1.49);

\path[draw=drawColor,line width= 0.4pt,line join=round,line cap=round] (327.52,133.35) circle (  1.49);

\path[draw=drawColor,line width= 0.4pt,line join=round,line cap=round] (328.37,133.16) circle (  1.49);

\path[draw=drawColor,line width= 0.4pt,line join=round,line cap=round] (329.22,132.78) circle (  1.49);

\path[draw=drawColor,line width= 0.4pt,line join=round,line cap=round] (330.08,132.59) circle (  1.49);

\path[draw=drawColor,line width= 0.4pt,line join=round,line cap=round] (330.93,130.68) circle (  1.49);

\path[draw=drawColor,line width= 0.4pt,line join=round,line cap=round] (331.78,127.92) circle (  1.49);

\path[draw=drawColor,line width= 0.4pt,line join=round,line cap=round] (332.64,123.05) circle (  1.49);

\path[draw=drawColor,line width= 0.4pt,line join=round,line cap=round] (333.49,116.95) circle (  1.49);

\path[draw=drawColor,line width= 0.4pt,line join=round,line cap=round] (334.34,110.27) circle (  1.49);

\path[draw=drawColor,line width= 0.4pt,line join=round,line cap=round] (335.20,100.54) circle (  1.49);

\path[draw=drawColor,line width= 0.4pt,line join=round,line cap=round] (336.05, 93.67) circle (  1.49);

\path[draw=drawColor,line width= 0.4pt,line join=round,line cap=round] (336.90, 85.66) circle (  1.49);

\path[draw=drawColor,line width= 0.4pt,line join=round,line cap=round] (337.76, 82.13) circle (  1.49);

\path[draw=drawColor,line width= 0.4pt,line join=round,line cap=round] (338.61, 76.21) circle (  1.49);

\path[draw=drawColor,line width= 0.4pt,line join=round,line cap=round] (339.46, 71.06) circle (  1.49);

\path[draw=drawColor,line width= 0.4pt,line join=round,line cap=round] (340.32, 69.15) circle (  1.49);

\path[draw=drawColor,line width= 0.4pt,line join=round,line cap=round] (341.17, 65.34) circle (  1.49);

\path[draw=drawColor,line width= 0.4pt,line join=round,line cap=round] (342.02, 62.76) circle (  1.49);

\path[draw=drawColor,line width= 0.4pt,line join=round,line cap=round] (342.88, 60.38) circle (  1.49);

\path[draw=drawColor,line width= 0.4pt,line join=round,line cap=round] (343.73, 58.28) circle (  1.49);

\path[draw=drawColor,line width= 0.4pt,line join=round,line cap=round] (344.58, 57.70) circle (  1.49);

\path[draw=drawColor,line width= 0.4pt,line join=round,line cap=round] (345.44, 55.51) circle (  1.49);

\path[draw=drawColor,line width= 0.4pt,line join=round,line cap=round] (346.29, 54.17) circle (  1.49);

\path[draw=drawColor,line width= 0.4pt,line join=round,line cap=round] (347.14, 54.17) circle (  1.49);

\path[draw=drawColor,line width= 0.4pt,line join=round,line cap=round] (348.00, 54.08) circle (  1.49);

\path[draw=drawColor,line width= 0.4pt,line join=round,line cap=round] (348.85, 51.69) circle (  1.49);

\path[draw=drawColor,line width= 0.4pt,line join=round,line cap=round] (349.70, 51.31) circle (  1.49);

\path[draw=drawColor,line width= 0.4pt,line join=round,line cap=round] (350.56, 51.03) circle (  1.49);

\path[draw=drawColor,line width= 0.4pt,line join=round,line cap=round] (351.41, 49.79) circle (  1.49);

\path[draw=drawColor,line width= 0.4pt,line join=round,line cap=round] (352.26, 50.45) circle (  1.49);

\path[draw=drawColor,line width= 0.4pt,line join=round,line cap=round] (353.12, 49.41) circle (  1.49);

\path[draw=drawColor,line width= 0.4pt,line join=round,line cap=round] (353.97, 48.83) circle (  1.49);

\path[draw=drawColor,line width= 0.4pt,line join=round,line cap=round] (354.82, 48.93) circle (  1.49);

\path[draw=drawColor,line width= 0.4pt,line join=round,line cap=round] (355.68, 48.36) circle (  1.49);

\path[draw=drawColor,line width= 0.4pt,line join=round,line cap=round] (356.53, 48.55) circle (  1.49);

\path[draw=drawColor,line width= 0.4pt,line join=round,line cap=round] (357.38, 48.83) circle (  1.49);

\path[draw=drawColor,line width= 0.4pt,line join=round,line cap=round] (358.24, 47.59) circle (  1.49);

\path[draw=drawColor,line width= 0.4pt,line join=round,line cap=round] (359.09, 48.45) circle (  1.49);

\path[draw=drawColor,line width= 0.4pt,line join=round,line cap=round] (359.94, 48.16) circle (  1.49);

\path[draw=drawColor,line width= 0.4pt,line join=round,line cap=round] (360.80, 47.40) circle (  1.49);

\path[draw=drawColor,line width= 0.4pt,line join=round,line cap=round] (361.65, 47.12) circle (  1.49);

\path[draw=drawColor,line width= 0.4pt,line join=round,line cap=round] (362.50, 46.73) circle (  1.49);

\path[draw=drawColor,line width= 0.4pt,line join=round,line cap=round] (363.36, 47.12) circle (  1.49);

\path[draw=drawColor,line width= 0.4pt,line join=round,line cap=round] (364.21, 47.12) circle (  1.49);

\path[draw=drawColor,line width= 0.4pt,line join=round,line cap=round] (365.06, 46.64) circle (  1.49);

\path[draw=drawColor,line width= 0.4pt,line join=round,line cap=round] (365.92, 46.26) circle (  1.49);

\path[draw=drawColor,line width= 0.4pt,line join=round,line cap=round] (366.77, 45.78) circle (  1.49);

\path[draw=drawColor,line width= 0.4pt,line join=round,line cap=round] (367.62, 45.88) circle (  1.49);

\path[draw=drawColor,line width= 0.4pt,line join=round,line cap=round] (368.48, 45.78) circle (  1.49);

\path[draw=drawColor,line width= 0.4pt,line join=round,line cap=round] (369.33, 46.07) circle (  1.49);

\path[draw=drawColor,line width= 0.4pt,line join=round,line cap=round] (370.18, 46.54) circle (  1.49);

\path[draw=drawColor,line width= 0.4pt,line join=round,line cap=round] (371.04, 45.68) circle (  1.49);

\path[draw=drawColor,line width= 0.4pt,line join=round,line cap=round] (371.89, 45.30) circle (  1.49);

\path[draw=drawColor,line width= 0.4pt,line join=round,line cap=round] (372.74, 45.88) circle (  1.49);

\path[draw=drawColor,line width= 0.4pt,line join=round,line cap=round] (373.60, 45.97) circle (  1.49);

\path[draw=drawColor,line width= 0.4pt,line join=round,line cap=round] (374.45, 44.83) circle (  1.49);

\path[draw=drawColor,line width= 0.4pt,line join=round,line cap=round] (375.30, 45.02) circle (  1.49);

\path[draw=drawColor,line width= 0.4pt,line join=round,line cap=round] (376.16, 45.78) circle (  1.49);

\path[draw=drawColor,line width= 0.4pt,line join=round,line cap=round] (377.01, 45.68) circle (  1.49);

\path[draw=drawColor,line width= 0.4pt,line join=round,line cap=round] (377.86, 45.40) circle (  1.49);

\path[draw=drawColor,line width= 0.4pt,line join=round,line cap=round] (378.71, 45.02) circle (  1.49);

\path[draw=drawColor,line width= 0.4pt,line join=round,line cap=round] (379.57, 44.83) circle (  1.49);

\path[draw=drawColor,line width= 0.4pt,line join=round,line cap=round] (380.42, 45.59) circle (  1.49);

\path[draw=drawColor,line width= 0.4pt,line join=round,line cap=round] (381.27, 45.97) circle (  1.49);

\path[draw=drawColor,line width= 0.4pt,line join=round,line cap=round] (382.13, 44.64) circle (  1.49);

\path[draw=drawColor,line width= 0.4pt,line join=round,line cap=round] (382.98, 45.59) circle (  1.49);

\path[draw=drawColor,line width= 0.4pt,line join=round,line cap=round] (383.83, 44.83) circle (  1.49);

\path[draw=drawColor,line width= 0.4pt,line join=round,line cap=round] (384.69, 44.73) circle (  1.49);

\path[draw=drawColor,line width= 0.4pt,line join=round,line cap=round] (385.54, 44.25) circle (  1.49);

\path[draw=drawColor,line width= 0.4pt,line join=round,line cap=round] (386.39, 44.92) circle (  1.49);

\path[draw=drawColor,line width= 0.4pt,line join=round,line cap=round] (387.25, 44.92) circle (  1.49);

\path[draw=drawColor,line width= 0.4pt,line join=round,line cap=round] (388.10, 43.97) circle (  1.49);

\path[draw=drawColor,line width= 0.4pt,line join=round,line cap=round] (388.95, 45.21) circle (  1.49);

\path[draw=drawColor,line width= 0.4pt,line join=round,line cap=round] (389.81, 44.16) circle (  1.49);

\path[draw=drawColor,line width= 0.4pt,line join=round,line cap=round] (390.66, 44.54) circle (  1.49);

\path[draw=drawColor,line width= 0.4pt,line join=round,line cap=round] (391.51, 44.73) circle (  1.49);

\path[draw=drawColor,line width= 0.4pt,line join=round,line cap=round] (392.37, 44.54) circle (  1.49);

\path[draw=drawColor,line width= 0.4pt,line join=round,line cap=round] (393.22, 44.73) circle (  1.49);

\path[draw=drawColor,line width= 0.4pt,line join=round,line cap=round] (394.07, 44.64) circle (  1.49);

\path[draw=drawColor,line width= 0.4pt,line join=round,line cap=round] (394.93, 44.35) circle (  1.49);

\path[draw=drawColor,line width= 0.4pt,line join=round,line cap=round] (395.78, 44.06) circle (  1.49);

\path[draw=drawColor,line width= 0.4pt,line join=round,line cap=round] (396.63, 44.73) circle (  1.49);

\path[draw=drawColor,line width= 0.4pt,line join=round,line cap=round] (397.49, 45.11) circle (  1.49);

\path[draw=drawColor,line width= 0.4pt,line join=round,line cap=round] (398.34, 44.35) circle (  1.49);

\path[draw=drawColor,line width= 0.4pt,line join=round,line cap=round] (399.19, 44.16) circle (  1.49);

\path[draw=drawColor,line width= 0.4pt,line join=round,line cap=round] (400.05, 44.54) circle (  1.49);

\path[draw=drawColor,line width= 0.4pt,line join=round,line cap=round] (400.90, 44.44) circle (  1.49);

\path[draw=drawColor,line width= 0.4pt,line join=round,line cap=round] (401.75, 44.16) circle (  1.49);

\path[draw=drawColor,line width= 0.4pt,line join=round,line cap=round] (402.61, 44.06) circle (  1.49);

\path[draw=drawColor,line width= 0.4pt,line join=round,line cap=round] (403.46, 44.73) circle (  1.49);
\end{scope}
\begin{scope}
\path[clip] (  0.00,  0.00) rectangle (423.50,169.40);
\definecolor{drawColor}{RGB}{0,0,0}

\path[draw=drawColor,line width= 0.4pt,line join=round,line cap=round] (318.13, 40.39) -- (403.46, 40.39);

\path[draw=drawColor,line width= 0.4pt,line join=round,line cap=round] (318.13, 40.39) -- (318.13, 36.43);

\path[draw=drawColor,line width= 0.4pt,line join=round,line cap=round] (335.20, 40.39) -- (335.20, 36.43);

\path[draw=drawColor,line width= 0.4pt,line join=round,line cap=round] (352.26, 40.39) -- (352.26, 36.43);

\path[draw=drawColor,line width= 0.4pt,line join=round,line cap=round] (369.33, 40.39) -- (369.33, 36.43);

\path[draw=drawColor,line width= 0.4pt,line join=round,line cap=round] (386.39, 40.39) -- (386.39, 36.43);

\path[draw=drawColor,line width= 0.4pt,line join=round,line cap=round] (403.46, 40.39) -- (403.46, 36.43);

\node[text=drawColor,anchor=base,inner sep=0pt, outer sep=0pt, scale=  0.66] at (318.13, 26.14) {0};

\node[text=drawColor,anchor=base,inner sep=0pt, outer sep=0pt, scale=  0.66] at (335.20, 26.14) {2};

\node[text=drawColor,anchor=base,inner sep=0pt, outer sep=0pt, scale=  0.66] at (352.26, 26.14) {4};

\node[text=drawColor,anchor=base,inner sep=0pt, outer sep=0pt, scale=  0.66] at (369.33, 26.14) {6};

\node[text=drawColor,anchor=base,inner sep=0pt, outer sep=0pt, scale=  0.66] at (386.39, 26.14) {8};

\node[text=drawColor,anchor=base,inner sep=0pt, outer sep=0pt, scale=  0.66] at (403.46, 26.14) {10};

\path[draw=drawColor,line width= 0.4pt,line join=round,line cap=round] (314.81, 57.04) -- (314.81,133.35);

\path[draw=drawColor,line width= 0.4pt,line join=round,line cap=round] (314.81, 57.04) -- (310.85, 57.04);

\path[draw=drawColor,line width= 0.4pt,line join=round,line cap=round] (314.81, 76.12) -- (310.85, 76.12);

\path[draw=drawColor,line width= 0.4pt,line join=round,line cap=round] (314.81, 95.20) -- (310.85, 95.20);

\path[draw=drawColor,line width= 0.4pt,line join=round,line cap=round] (314.81,114.27) -- (310.85,114.27);

\path[draw=drawColor,line width= 0.4pt,line join=round,line cap=round] (314.81,133.35) -- (310.85,133.35);

\node[text=drawColor,rotate= 90.00,anchor=base,inner sep=0pt, outer sep=0pt, scale=  0.66] at (305.30, 57.04) {0.2};

\node[text=drawColor,rotate= 90.00,anchor=base,inner sep=0pt, outer sep=0pt, scale=  0.66] at (305.30, 76.12) {0.4};

\node[text=drawColor,rotate= 90.00,anchor=base,inner sep=0pt, outer sep=0pt, scale=  0.66] at (305.30, 95.20) {0.6};

\node[text=drawColor,rotate= 90.00,anchor=base,inner sep=0pt, outer sep=0pt, scale=  0.66] at (305.30,114.27) {0.8};

\node[text=drawColor,rotate= 90.00,anchor=base,inner sep=0pt, outer sep=0pt, scale=  0.66] at (305.30,133.35) {1.0};

\path[draw=drawColor,line width= 0.4pt,line join=round,line cap=round] (314.81, 40.39) --
	(406.87, 40.39) --
	(406.87,136.93) --
	(314.81,136.93) --
	(314.81, 40.39);
\end{scope}
\begin{scope}
\path[clip] (282.33,  0.00) rectangle (423.50,169.40);
\definecolor{drawColor}{RGB}{0,0,0}

\node[text=drawColor,anchor=base,inner sep=0pt, outer sep=0pt, scale=  0.79] at (360.84,150.43) {\bfseries pairwise independence};

\node[text=drawColor,anchor=base,inner sep=0pt, outer sep=0pt, scale=  0.66] at (360.84, 10.30) {bandwidth $\sigma$};

\node[text=drawColor,rotate= 90.00,anchor=base,inner sep=0pt, outer sep=0pt, scale=  0.66] at (289.46, 88.66) {rejection rate};
\end{scope}
\begin{scope}
\path[clip] (314.81, 40.39) rectangle (406.87,136.93);
\definecolor{drawColor}{RGB}{205,38,38}

\path[draw=drawColor,line width= 0.4pt,dash pattern=on 4pt off 4pt ,line join=round,line cap=round] (323.09, 40.39) -- (323.09,136.93);

\path[draw=drawColor,line width= 0.4pt,dash pattern=on 4pt off 4pt ,line join=round,line cap=round] (324.79, 40.39) -- (324.79,136.93);

\path[draw=drawColor,line width= 0.4pt,line join=round,line cap=round] (323.93, 40.39) -- (323.93,136.93);
\end{scope}
\end{tikzpicture}

%% file: Plots/power_analysis_extra.tex
\begin{tikzpicture}[x=1pt,y=1pt]
\definecolor{fillColor}{RGB}{255,255,255}
\path[use as bounding box,fill=fillColor,fill opacity=0.00] (0,0) rectangle (275.28,275.28);
\begin{scope}
\path[clip] ( 49.20, 61.20) rectangle (250.08,226.08);
\definecolor{drawColor}{RGB}{69,139,0}

\path[draw=drawColor,line width= 0.4pt,line join=round,line cap=round] ( 56.64,189.90) circle (  2.25);

\path[draw=drawColor,line width= 0.4pt,line join=round,line cap=round] (118.64,208.67) circle (  2.25);

\path[draw=drawColor,line width= 0.4pt,line join=round,line cap=round] (180.64,211.73) circle (  2.25);

\path[draw=drawColor,line width= 0.4pt,line join=round,line cap=round] (242.64,207.30) circle (  2.25);
\end{scope}
\begin{scope}
\path[clip] (  0.00,  0.00) rectangle (275.28,275.28);
\definecolor{drawColor}{RGB}{0,0,0}

\path[draw=drawColor,line width= 0.4pt,line join=round,line cap=round] ( 56.64, 61.20) -- (242.64, 61.20);

\path[draw=drawColor,line width= 0.4pt,line join=round,line cap=round] ( 56.64, 61.20) -- ( 56.64, 55.20);

\path[draw=drawColor,line width= 0.4pt,line join=round,line cap=round] ( 87.64, 61.20) -- ( 87.64, 55.20);

\path[draw=drawColor,line width= 0.4pt,line join=round,line cap=round] (118.64, 61.20) -- (118.64, 55.20);

\path[draw=drawColor,line width= 0.4pt,line join=round,line cap=round] (149.64, 61.20) -- (149.64, 55.20);

\path[draw=drawColor,line width= 0.4pt,line join=round,line cap=round] (180.64, 61.20) -- (180.64, 55.20);

\path[draw=drawColor,line width= 0.4pt,line join=round,line cap=round] (211.64, 61.20) -- (211.64, 55.20);

\path[draw=drawColor,line width= 0.4pt,line join=round,line cap=round] (242.64, 61.20) -- (242.64, 55.20);

\node[text=drawColor,anchor=base,inner sep=0pt, outer sep=0pt, scale=  1.00] at ( 56.64, 39.60) {4};

\node[text=drawColor,anchor=base,inner sep=0pt, outer sep=0pt, scale=  1.00] at ( 87.64, 39.60) {5};

\node[text=drawColor,anchor=base,inner sep=0pt, outer sep=0pt, scale=  1.00] at (118.64, 39.60) {6};

\node[text=drawColor,anchor=base,inner sep=0pt, outer sep=0pt, scale=  1.00] at (149.64, 39.60) {7};

\node[text=drawColor,anchor=base,inner sep=0pt, outer sep=0pt, scale=  1.00] at (180.64, 39.60) {8};

\node[text=drawColor,anchor=base,inner sep=0pt, outer sep=0pt, scale=  1.00] at (211.64, 39.60) {9};

\node[text=drawColor,anchor=base,inner sep=0pt, outer sep=0pt, scale=  1.00] at (242.64, 39.60) {10};

\path[draw=drawColor,line width= 0.4pt,line join=round,line cap=round] ( 49.20, 67.31) -- ( 49.20,219.97);

\path[draw=drawColor,line width= 0.4pt,line join=round,line cap=round] ( 49.20, 67.31) -- ( 43.20, 67.31);

\path[draw=drawColor,line width= 0.4pt,line join=round,line cap=round] ( 49.20, 97.84) -- ( 43.20, 97.84);

\path[draw=drawColor,line width= 0.4pt,line join=round,line cap=round] ( 49.20,128.37) -- ( 43.20,128.37);

\path[draw=drawColor,line width= 0.4pt,line join=round,line cap=round] ( 49.20,158.90) -- ( 43.20,158.90);

\path[draw=drawColor,line width= 0.4pt,line join=round,line cap=round] ( 49.20,189.44) -- ( 43.20,189.44);

\path[draw=drawColor,line width= 0.4pt,line join=round,line cap=round] ( 49.20,219.97) -- ( 43.20,219.97);

\node[text=drawColor,rotate= 90.00,anchor=base,inner sep=0pt, outer sep=0pt, scale=  1.00] at ( 34.80, 67.31) {0.0};

\node[text=drawColor,rotate= 90.00,anchor=base,inner sep=0pt, outer sep=0pt, scale=  1.00] at ( 34.80, 97.84) {0.2};

\node[text=drawColor,rotate= 90.00,anchor=base,inner sep=0pt, outer sep=0pt, scale=  1.00] at ( 34.80,128.37) {0.4};

\node[text=drawColor,rotate= 90.00,anchor=base,inner sep=0pt, outer sep=0pt, scale=  1.00] at ( 34.80,158.90) {0.6};

\node[text=drawColor,rotate= 90.00,anchor=base,inner sep=0pt, outer sep=0pt, scale=  1.00] at ( 34.80,189.44) {0.8};

\node[text=drawColor,rotate= 90.00,anchor=base,inner sep=0pt, outer sep=0pt, scale=  1.00] at ( 34.80,219.97) {1.0};

\path[draw=drawColor,line width= 0.4pt,line join=round,line cap=round] ( 49.20, 61.20) --
	(250.08, 61.20) --
	(250.08,226.08) --
	( 49.20,226.08) --
	( 49.20, 61.20);
\end{scope}
\begin{scope}
\path[clip] (  0.00,  0.00) rectangle (275.28,275.28);
\definecolor{drawColor}{RGB}{0,0,0}

\node[text=drawColor,anchor=base,inner sep=0pt, outer sep=0pt, scale=  1.00] at (149.64, 15.60) {number of variables $d$};

\node[text=drawColor,rotate= 90.00,anchor=base,inner sep=0pt, outer sep=0pt, scale=  1.00] at ( 10.80,143.64) {rejection rate};
\end{scope}
\begin{scope}
\path[clip] ( 49.20, 61.20) rectangle (250.08,226.08);
\definecolor{drawColor}{RGB}{205,38,38}

\path[draw=drawColor,line width= 0.4pt,line join=round,line cap=round] ( 56.64,149.27) --
	( 59.67,144.03) --
	( 53.61,144.03) --
	( 56.64,149.27);

\path[draw=drawColor,line width= 0.4pt,line join=round,line cap=round] (118.64,188.05) --
	(121.67,182.80) --
	(115.61,182.80) --
	(118.64,188.05);

\path[draw=drawColor,line width= 0.4pt,line join=round,line cap=round] (180.64, 70.81) --
	(183.67, 65.56) --
	(177.61, 65.56) --
	(180.64, 70.81);

\path[draw=drawColor,line width= 0.4pt,line join=round,line cap=round] (242.64, 70.81) --
	(245.67, 65.56) --
	(239.61, 65.56) --
	(242.64, 70.81);
\definecolor{drawColor}{RGB}{24,116,205}

\path[draw=drawColor,line width= 0.4pt,line join=round,line cap=round] ( 54.39,126.73) rectangle ( 58.89,131.23);

\path[draw=drawColor,line width= 0.4pt,line join=round,line cap=round] (116.39,103.53) rectangle (120.89,108.03);

\path[draw=drawColor,line width= 0.4pt,line join=round,line cap=round] (178.39, 83.99) rectangle (182.89, 88.49);

\path[draw=drawColor,line width= 0.4pt,line join=round,line cap=round] (240.39, 77.27) rectangle (244.89, 81.77);

\path[draw=drawColor,line width= 0.4pt,line join=round,line cap=round] ( 53.46,145.01) --
	( 56.64,148.19) --
	( 59.82,145.01) --
	( 56.64,141.83) --
	( 53.46,145.01);

\path[draw=drawColor,line width= 0.4pt,line join=round,line cap=round] (115.46,124.40) --
	(118.64,127.58) --
	(121.82,124.40) --
	(118.64,121.22) --
	(115.46,124.40);

\path[draw=drawColor,line width= 0.4pt,line join=round,line cap=round] (177.46, 99.67) --
	(180.64,102.85) --
	(183.82, 99.67) --
	(180.64, 96.49) --
	(177.46, 99.67);

\path[draw=drawColor,line width= 0.4pt,line join=round,line cap=round] (239.45, 84.10) --
	(242.64, 87.28) --
	(245.82, 84.10) --
	(242.64, 80.92) --
	(239.45, 84.10);
\definecolor{drawColor}{RGB}{211,211,211}

\path[draw=drawColor,line width= 0.4pt,dash pattern=on 1pt off 3pt ,line join=round,line cap=round] ( 56.64, 61.20) -- ( 56.64,226.08);

\path[draw=drawColor,line width= 0.4pt,dash pattern=on 1pt off 3pt ,line join=round,line cap=round] ( 87.64, 61.20) -- ( 87.64,226.08);

\path[draw=drawColor,line width= 0.4pt,dash pattern=on 1pt off 3pt ,line join=round,line cap=round] (118.64, 61.20) -- (118.64,226.08);

\path[draw=drawColor,line width= 0.4pt,dash pattern=on 1pt off 3pt ,line join=round,line cap=round] (149.64, 61.20) -- (149.64,226.08);

\path[draw=drawColor,line width= 0.4pt,dash pattern=on 1pt off 3pt ,line join=round,line cap=round] (180.64, 61.20) -- (180.64,226.08);

\path[draw=drawColor,line width= 0.4pt,dash pattern=on 1pt off 3pt ,line join=round,line cap=round] (211.64, 61.20) -- (211.64,226.08);

\path[draw=drawColor,line width= 0.4pt,dash pattern=on 1pt off 3pt ,line join=round,line cap=round] (242.64, 61.20) -- (242.64,226.08);

\path[draw=drawColor,line width= 0.4pt,dash pattern=on 1pt off 3pt ,line join=round,line cap=round] ( 49.20, 67.31) -- (250.08, 67.31);

\path[draw=drawColor,line width= 0.4pt,dash pattern=on 1pt off 3pt ,line join=round,line cap=round] ( 49.20, 97.84) -- (250.08, 97.84);

\path[draw=drawColor,line width= 0.4pt,dash pattern=on 1pt off 3pt ,line join=round,line cap=round] ( 49.20,128.37) -- (250.08,128.37);

\path[draw=drawColor,line width= 0.4pt,dash pattern=on 1pt off 3pt ,line join=round,line cap=round] ( 49.20,158.90) -- (250.08,158.90);

\path[draw=drawColor,line width= 0.4pt,dash pattern=on 1pt off 3pt ,line join=round,line cap=round] ( 49.20,189.44) -- (250.08,189.44);

\path[draw=drawColor,line width= 0.4pt,dash pattern=on 1pt off 3pt ,line join=round,line cap=round] ( 49.20,219.97) -- (250.08,219.97);
\definecolor{drawColor}{RGB}{0,0,0}
\definecolor{fillColor}{RGB}{255,255,255}

\path[draw=drawColor,line width= 0.4pt,line join=round,line cap=round,fill=fillColor] (170.59,173.64) rectangle (240.03,113.64);
\definecolor{drawColor}{RGB}{69,139,0}

\path[draw=drawColor,line width= 0.4pt,line join=round,line cap=round] (179.59,161.64) circle (  2.25);
\definecolor{drawColor}{RGB}{205,38,38}

\path[draw=drawColor,line width= 0.4pt,line join=round,line cap=round] (179.59,153.14) --
	(182.62,147.89) --
	(176.56,147.89) --
	(179.59,153.14);
\definecolor{drawColor}{RGB}{24,116,205}

\path[draw=drawColor,line width= 0.4pt,line join=round,line cap=round] (177.34,135.39) rectangle (181.84,139.89);

\path[draw=drawColor,line width= 0.4pt,line join=round,line cap=round] (176.41,125.64) --
	(179.59,128.82) --
	(182.78,125.64) --
	(179.59,122.46) --
	(176.41,125.64);
\definecolor{drawColor}{RGB}{0,0,0}

\node[text=drawColor,anchor=base west,inner sep=0pt, outer sep=0pt, scale=  1.00] at (188.59,158.19) {dHSIC};

\node[text=drawColor,anchor=base west,inner sep=0pt, outer sep=0pt, scale=  1.00] at (188.59,146.19) {HSIC};

\node[text=drawColor,anchor=base west,inner sep=0pt, outer sep=0pt, scale=  1.00] at (188.59,134.19) {BMR-$n$};

\node[text=drawColor,anchor=base west,inner sep=0pt, outer sep=0pt, scale=  1.00] at (188.59,122.19) {BMR-$1000$};
\end{scope}
\end{tikzpicture}

%% file: Plots/causality_example_romano.tex
\begin{tikzpicture}[x=1pt,y=1pt]
\definecolor{fillColor}{RGB}{255,255,255}
\path[use as bounding box,fill=fillColor,fill opacity=0.00] (0,0) rectangle (423.50,254.10);
\begin{scope}
\path[clip] ( 49.20, 61.20) rectangle (186.55,204.90);
\definecolor{drawColor}{RGB}{69,139,0}

\path[draw=drawColor,line width= 0.4pt,line join=round,line cap=round] ( 54.29,123.20) circle (  2.25);

\path[draw=drawColor,line width= 0.4pt,line join=round,line cap=round] ( 86.08,152.08) circle (  2.25);

\path[draw=drawColor,line width= 0.4pt,line join=round,line cap=round] (117.88,171.50) circle (  2.25);

\path[draw=drawColor,line width= 0.4pt,line join=round,line cap=round] (149.67,180.02) circle (  2.25);

\path[draw=drawColor,line width= 0.4pt,line join=round,line cap=round] (181.46,186.01) circle (  2.25);
\end{scope}
\begin{scope}
\path[clip] (  0.00,  0.00) rectangle (423.50,254.10);
\definecolor{drawColor}{RGB}{0,0,0}

\path[draw=drawColor,line width= 0.4pt,line join=round,line cap=round] ( 54.29, 61.20) -- (181.46, 61.20);

\path[draw=drawColor,line width= 0.4pt,line join=round,line cap=round] ( 54.29, 61.20) -- ( 54.29, 55.20);

\path[draw=drawColor,line width= 0.4pt,line join=round,line cap=round] ( 86.08, 61.20) -- ( 86.08, 55.20);

\path[draw=drawColor,line width= 0.4pt,line join=round,line cap=round] (117.88, 61.20) -- (117.88, 55.20);

\path[draw=drawColor,line width= 0.4pt,line join=round,line cap=round] (149.67, 61.20) -- (149.67, 55.20);

\path[draw=drawColor,line width= 0.4pt,line join=round,line cap=round] (181.46, 61.20) -- (181.46, 55.20);

\node[text=drawColor,anchor=base,inner sep=0pt, outer sep=0pt, scale=  1.00] at ( 54.29, 39.60) {100};

\node[text=drawColor,anchor=base,inner sep=0pt, outer sep=0pt, scale=  1.00] at ( 86.08, 39.60) {200};

\node[text=drawColor,anchor=base,inner sep=0pt, outer sep=0pt, scale=  1.00] at (117.88, 39.60) {300};

\node[text=drawColor,anchor=base,inner sep=0pt, outer sep=0pt, scale=  1.00] at (149.67, 39.60) {400};

\node[text=drawColor,anchor=base,inner sep=0pt, outer sep=0pt, scale=  1.00] at (181.46, 39.60) {500};

\path[draw=drawColor,line width= 0.4pt,line join=round,line cap=round] ( 49.20, 66.52) -- ( 49.20,199.58);

\path[draw=drawColor,line width= 0.4pt,line join=round,line cap=round] ( 49.20, 66.52) -- ( 43.20, 66.52);

\path[draw=drawColor,line width= 0.4pt,line join=round,line cap=round] ( 49.20, 93.13) -- ( 43.20, 93.13);

\path[draw=drawColor,line width= 0.4pt,line join=round,line cap=round] ( 49.20,119.74) -- ( 43.20,119.74);

\path[draw=drawColor,line width= 0.4pt,line join=round,line cap=round] ( 49.20,146.36) -- ( 43.20,146.36);

\path[draw=drawColor,line width= 0.4pt,line join=round,line cap=round] ( 49.20,172.97) -- ( 43.20,172.97);

\path[draw=drawColor,line width= 0.4pt,line join=round,line cap=round] ( 49.20,199.58) -- ( 43.20,199.58);

\node[text=drawColor,rotate= 90.00,anchor=base,inner sep=0pt, outer sep=0pt, scale=  1.00] at ( 34.80, 66.52) {0};

\node[text=drawColor,rotate= 90.00,anchor=base,inner sep=0pt, outer sep=0pt, scale=  1.00] at ( 34.80, 93.13) {200};

\node[text=drawColor,rotate= 90.00,anchor=base,inner sep=0pt, outer sep=0pt, scale=  1.00] at ( 34.80,119.74) {400};

\node[text=drawColor,rotate= 90.00,anchor=base,inner sep=0pt, outer sep=0pt, scale=  1.00] at ( 34.80,146.36) {600};

\node[text=drawColor,rotate= 90.00,anchor=base,inner sep=0pt, outer sep=0pt, scale=  1.00] at ( 34.80,172.97) {800};

\node[text=drawColor,rotate= 90.00,anchor=base,inner sep=0pt, outer sep=0pt, scale=  1.00] at ( 34.80,199.58) {1000};

\path[draw=drawColor,line width= 0.4pt,line join=round,line cap=round] ( 49.20, 61.20) --
	(186.55, 61.20) --
	(186.55,204.90) --
	( 49.20,204.90) --
	( 49.20, 61.20);
\end{scope}
\begin{scope}
\path[clip] (  0.00,  0.00) rectangle (211.75,254.10);
\definecolor{drawColor}{RGB}{0,0,0}

\node[text=drawColor,anchor=base,inner sep=0pt, outer sep=0pt, scale=  1.00] at (117.88, 15.60) {sample size $n$};

\node[text=drawColor,rotate= 90.00,anchor=base,inner sep=0pt, outer sep=0pt, scale=  1.00] at ( 10.80,133.05) {number of correct DAGs};
\end{scope}
\begin{scope}
\path[clip] ( 49.20, 61.20) rectangle (186.55,204.90);
\definecolor{drawColor}{RGB}{205,38,38}

\path[draw=drawColor,line width= 0.4pt,line join=round,line cap=round] ( 54.29,107.94) --
	( 57.32,102.69) --
	( 51.26,102.69) --
	( 54.29,107.94);

\path[draw=drawColor,line width= 0.4pt,line join=round,line cap=round] ( 86.08,130.96) --
	( 89.11,125.71) --
	( 83.05,125.71) --
	( 86.08,130.96);

\path[draw=drawColor,line width= 0.4pt,line join=round,line cap=round] (117.88,145.07) --
	(120.91,139.82) --
	(114.85,139.82) --
	(117.88,145.07);

\path[draw=drawColor,line width= 0.4pt,line join=round,line cap=round] (149.67,156.11) --
	(152.70,150.86) --
	(146.64,150.86) --
	(149.67,156.11);

\path[draw=drawColor,line width= 0.4pt,line join=round,line cap=round] (181.46,164.76) --
	(184.49,159.51) --
	(178.43,159.51) --
	(181.46,164.76);
\definecolor{drawColor}{RGB}{24,116,205}

\path[draw=drawColor,line width= 0.4pt,line join=round,line cap=round] ( 52.04, 84.10) rectangle ( 56.54, 88.60);

\path[draw=drawColor,line width= 0.4pt,line join=round,line cap=round] ( 83.83,101.40) rectangle ( 88.33,105.90);

\path[draw=drawColor,line width= 0.4pt,line join=round,line cap=round] (115.63,116.56) rectangle (120.13,121.06);

\path[draw=drawColor,line width= 0.4pt,line join=round,line cap=round] (147.42,125.08) rectangle (151.92,129.58);

\path[draw=drawColor,line width= 0.4pt,line join=round,line cap=round] (179.21,133.73) rectangle (183.71,138.23);
\definecolor{drawColor}{RGB}{211,211,211}

\path[draw=drawColor,line width= 0.4pt,dash pattern=on 1pt off 3pt ,line join=round,line cap=round] ( 54.29, 61.20) -- ( 54.29,204.90);

\path[draw=drawColor,line width= 0.4pt,dash pattern=on 1pt off 3pt ,line join=round,line cap=round] ( 86.08, 61.20) -- ( 86.08,204.90);

\path[draw=drawColor,line width= 0.4pt,dash pattern=on 1pt off 3pt ,line join=round,line cap=round] (117.88, 61.20) -- (117.88,204.90);

\path[draw=drawColor,line width= 0.4pt,dash pattern=on 1pt off 3pt ,line join=round,line cap=round] (149.67, 61.20) -- (149.67,204.90);

\path[draw=drawColor,line width= 0.4pt,dash pattern=on 1pt off 3pt ,line join=round,line cap=round] (181.46, 61.20) -- (181.46,204.90);

\path[draw=drawColor,line width= 0.4pt,dash pattern=on 1pt off 3pt ,line join=round,line cap=round] ( 49.20, 66.52) -- (186.55, 66.52);

\path[draw=drawColor,line width= 0.4pt,dash pattern=on 1pt off 3pt ,line join=round,line cap=round] ( 49.20, 93.13) -- (186.55, 93.13);

\path[draw=drawColor,line width= 0.4pt,dash pattern=on 1pt off 3pt ,line join=round,line cap=round] ( 49.20,119.74) -- (186.55,119.74);

\path[draw=drawColor,line width= 0.4pt,dash pattern=on 1pt off 3pt ,line join=round,line cap=round] ( 49.20,146.36) -- (186.55,146.36);

\path[draw=drawColor,line width= 0.4pt,dash pattern=on 1pt off 3pt ,line join=round,line cap=round] ( 49.20,172.97) -- (186.55,172.97);

\path[draw=drawColor,line width= 0.4pt,dash pattern=on 1pt off 3pt ,line join=round,line cap=round] ( 49.20,199.58) -- (186.55,199.58);
\definecolor{drawColor}{RGB}{0,0,0}
\definecolor{fillColor}{RGB}{255,255,255}

\path[draw=drawColor,line width= 0.4pt,line join=round,line cap=round,fill=fillColor] (124.24,116.39) rectangle (179.68, 68.39);
\definecolor{drawColor}{RGB}{69,139,0}

\path[draw=drawColor,line width= 0.4pt,line join=round,line cap=round] (133.24,104.39) circle (  2.25);
\definecolor{drawColor}{RGB}{205,38,38}

\path[draw=drawColor,line width= 0.4pt,line join=round,line cap=round] (133.24, 95.88) --
	(136.27, 90.64) --
	(130.21, 90.64) --
	(133.24, 95.88);
\definecolor{drawColor}{RGB}{24,116,205}

\path[draw=drawColor,line width= 0.4pt,line join=round,line cap=round] (130.99, 78.14) rectangle (135.49, 82.64);
\definecolor{drawColor}{RGB}{0,0,0}

\node[text=drawColor,anchor=base west,inner sep=0pt, outer sep=0pt, scale=  1.00] at (142.24,100.94) {dHSIC};

\node[text=drawColor,anchor=base west,inner sep=0pt, outer sep=0pt, scale=  1.00] at (142.24, 88.94) {HSIC};

\node[text=drawColor,anchor=base west,inner sep=0pt, outer sep=0pt, scale=  1.00] at (142.24, 76.94) {BMR-$n$};
\end{scope}
\begin{scope}
\path[clip] (  0.00,  0.00) rectangle (423.50,254.10);
\definecolor{drawColor}{RGB}{0,0,0}

\path[draw=drawColor,line width= 0.4pt,line join=round,line cap=round] (266.04, 61.20) -- (393.22, 61.20);

\path[draw=drawColor,line width= 0.4pt,line join=round,line cap=round] (266.04, 61.20) -- (266.04, 55.20);

\path[draw=drawColor,line width= 0.4pt,line join=round,line cap=round] (287.23, 61.20) -- (287.23, 55.20);

\path[draw=drawColor,line width= 0.4pt,line join=round,line cap=round] (308.43, 61.20) -- (308.43, 55.20);

\path[draw=drawColor,line width= 0.4pt,line join=round,line cap=round] (329.63, 61.20) -- (329.63, 55.20);

\path[draw=drawColor,line width= 0.4pt,line join=round,line cap=round] (350.82, 61.20) -- (350.82, 55.20);

\path[draw=drawColor,line width= 0.4pt,line join=round,line cap=round] (372.02, 61.20) -- (372.02, 55.20);

\path[draw=drawColor,line width= 0.4pt,line join=round,line cap=round] (393.22, 61.20) -- (393.22, 55.20);

\node[text=drawColor,anchor=base,inner sep=0pt, outer sep=0pt, scale=  1.00] at (266.04, 39.60) {0};

\node[text=drawColor,anchor=base,inner sep=0pt, outer sep=0pt, scale=  1.00] at (287.23, 39.60) {100};

\node[text=drawColor,anchor=base,inner sep=0pt, outer sep=0pt, scale=  1.00] at (329.63, 39.60) {300};

\node[text=drawColor,anchor=base,inner sep=0pt, outer sep=0pt, scale=  1.00] at (372.02, 39.60) {500};

\path[draw=drawColor,line width= 0.4pt,line join=round,line cap=round] (260.95, 66.52) -- (260.95,199.58);

\path[draw=drawColor,line width= 0.4pt,line join=round,line cap=round] (260.95, 66.52) -- (254.95, 66.52);

\path[draw=drawColor,line width= 0.4pt,line join=round,line cap=round] (260.95, 88.70) -- (254.95, 88.70);

\path[draw=drawColor,line width= 0.4pt,line join=round,line cap=round] (260.95,110.87) -- (254.95,110.87);

\path[draw=drawColor,line width= 0.4pt,line join=round,line cap=round] (260.95,133.05) -- (254.95,133.05);

\path[draw=drawColor,line width= 0.4pt,line join=round,line cap=round] (260.95,155.23) -- (254.95,155.23);

\path[draw=drawColor,line width= 0.4pt,line join=round,line cap=round] (260.95,177.40) -- (254.95,177.40);

\path[draw=drawColor,line width= 0.4pt,line join=round,line cap=round] (260.95,199.58) -- (254.95,199.58);

\node[text=drawColor,rotate= 90.00,anchor=base,inner sep=0pt, outer sep=0pt, scale=  1.00] at (246.55, 66.52) {0};

\node[text=drawColor,rotate= 90.00,anchor=base,inner sep=0pt, outer sep=0pt, scale=  1.00] at (246.55, 88.70) {2};

\node[text=drawColor,rotate= 90.00,anchor=base,inner sep=0pt, outer sep=0pt, scale=  1.00] at (246.55,110.87) {4};

\node[text=drawColor,rotate= 90.00,anchor=base,inner sep=0pt, outer sep=0pt, scale=  1.00] at (246.55,133.05) {6};

\node[text=drawColor,rotate= 90.00,anchor=base,inner sep=0pt, outer sep=0pt, scale=  1.00] at (246.55,155.23) {8};

\node[text=drawColor,rotate= 90.00,anchor=base,inner sep=0pt, outer sep=0pt, scale=  1.00] at (246.55,177.40) {10};

\node[text=drawColor,rotate= 90.00,anchor=base,inner sep=0pt, outer sep=0pt, scale=  1.00] at (246.55,199.58) {12};

\path[draw=drawColor,line width= 0.4pt,line join=round,line cap=round] (260.95, 61.20) --
	(398.30, 61.20) --
	(398.30,204.90) --
	(260.95,204.90) --
	(260.95, 61.20);
\end{scope}
\begin{scope}
\path[clip] (211.75,  0.00) rectangle (423.50,254.10);
\definecolor{drawColor}{RGB}{0,0,0}

\node[text=drawColor,anchor=base,inner sep=0pt, outer sep=0pt, scale=  1.00] at (329.63, 15.60) {sample size $n$};

\node[text=drawColor,rotate= 90.00,anchor=base,inner sep=0pt, outer sep=0pt, scale=  1.00] at (222.55,133.05) {average SID};
\end{scope}
\begin{scope}
\path[clip] (260.95, 61.20) rectangle (398.30,204.90);
\definecolor{drawColor}{RGB}{211,211,211}

\path[draw=drawColor,line width= 0.4pt,dash pattern=on 1pt off 3pt ,line join=round,line cap=round] (266.04, 61.20) -- (266.04,204.90);

\path[draw=drawColor,line width= 0.4pt,dash pattern=on 1pt off 3pt ,line join=round,line cap=round] (287.23, 61.20) -- (287.23,204.90);

\path[draw=drawColor,line width= 0.4pt,dash pattern=on 1pt off 3pt ,line join=round,line cap=round] (308.43, 61.20) -- (308.43,204.90);

\path[draw=drawColor,line width= 0.4pt,dash pattern=on 1pt off 3pt ,line join=round,line cap=round] (329.63, 61.20) -- (329.63,204.90);

\path[draw=drawColor,line width= 0.4pt,dash pattern=on 1pt off 3pt ,line join=round,line cap=round] (350.82, 61.20) -- (350.82,204.90);

\path[draw=drawColor,line width= 0.4pt,dash pattern=on 1pt off 3pt ,line join=round,line cap=round] (372.02, 61.20) -- (372.02,204.90);

\path[draw=drawColor,line width= 0.4pt,dash pattern=on 1pt off 3pt ,line join=round,line cap=round] (393.22, 61.20) -- (393.22,204.90);

\path[draw=drawColor,line width= 0.4pt,dash pattern=on 1pt off 3pt ,line join=round,line cap=round] (260.95, 66.52) -- (398.30, 66.52);

\path[draw=drawColor,line width= 0.4pt,dash pattern=on 1pt off 3pt ,line join=round,line cap=round] (260.95, 88.70) -- (398.30, 88.70);

\path[draw=drawColor,line width= 0.4pt,dash pattern=on 1pt off 3pt ,line join=round,line cap=round] (260.95,110.87) -- (398.30,110.87);

\path[draw=drawColor,line width= 0.4pt,dash pattern=on 1pt off 3pt ,line join=round,line cap=round] (260.95,133.05) -- (398.30,133.05);

\path[draw=drawColor,line width= 0.4pt,dash pattern=on 1pt off 3pt ,line join=round,line cap=round] (260.95,155.23) -- (398.30,155.23);

\path[draw=drawColor,line width= 0.4pt,dash pattern=on 1pt off 3pt ,line join=round,line cap=round] (260.95,177.40) -- (398.30,177.40);

\path[draw=drawColor,line width= 0.4pt,dash pattern=on 1pt off 3pt ,line join=round,line cap=round] (260.95,199.58) -- (398.30,199.58);
\definecolor{drawColor}{RGB}{69,139,0}

\path[draw=drawColor,line width= 1.2pt,line join=round] (279.82, 88.70) -- (284.05, 88.70);

\path[draw=drawColor,line width= 0.4pt,dash pattern=on 4pt off 4pt ,line join=round,line cap=round] (281.94, 66.52) -- (281.94, 66.52);

\path[draw=drawColor,line width= 0.4pt,dash pattern=on 4pt off 4pt ,line join=round,line cap=round] (281.94,188.49) -- (281.94,121.96);

\path[draw=drawColor,line width= 0.4pt,line join=round,line cap=round] (280.88, 66.52) -- (283.00, 66.52);

\path[draw=drawColor,line width= 0.4pt,line join=round,line cap=round] (280.88,188.49) -- (283.00,188.49);

\path[draw=drawColor,line width= 0.4pt,line join=round,line cap=round] (279.82, 66.52) --
	(284.05, 66.52) --
	(284.05,121.96) --
	(279.82,121.96) --
	(279.82, 66.52);

\path[draw=drawColor,line width= 1.2pt,line join=round] (301.01, 66.52) -- (305.25, 66.52);

\path[draw=drawColor,line width= 0.4pt,dash pattern=on 4pt off 4pt ,line join=round,line cap=round] (303.13, 66.52) -- (303.13, 66.52);

\path[draw=drawColor,line width= 0.4pt,dash pattern=on 4pt off 4pt ,line join=round,line cap=round] (303.13,121.96) -- (303.13, 88.70);

\path[draw=drawColor,line width= 0.4pt,line join=round,line cap=round] (302.07, 66.52) -- (304.19, 66.52);

\path[draw=drawColor,line width= 0.4pt,line join=round,line cap=round] (302.07,121.96) -- (304.19,121.96);

\path[draw=drawColor,line width= 0.4pt,line join=round,line cap=round] (301.01, 66.52) --
	(305.25, 66.52) --
	(305.25, 88.70) --
	(301.01, 88.70) --
	(301.01, 66.52);

\path[draw=drawColor,line width= 1.2pt,line join=round] (322.21, 66.52) -- (326.45, 66.52);

\path[draw=drawColor,line width= 0.4pt,dash pattern=on 4pt off 4pt ,line join=round,line cap=round] (324.33, 66.52) -- (324.33, 66.52);

\path[draw=drawColor,line width= 0.4pt,dash pattern=on 4pt off 4pt ,line join=round,line cap=round] (324.33, 66.52) -- (324.33, 66.52);

\path[draw=drawColor,line width= 0.4pt,line join=round,line cap=round] (323.27, 66.52) -- (325.39, 66.52);

\path[draw=drawColor,line width= 0.4pt,line join=round,line cap=round] (323.27, 66.52) -- (325.39, 66.52);

\path[draw=drawColor,line width= 0.4pt,line join=round,line cap=round] (322.21, 66.52) --
	(326.45, 66.52) --
	(326.45, 66.52) --
	(322.21, 66.52) --
	(322.21, 66.52);

\path[draw=drawColor,line width= 1.2pt,line join=round] (343.40, 66.52) -- (347.64, 66.52);

\path[draw=drawColor,line width= 0.4pt,dash pattern=on 4pt off 4pt ,line join=round,line cap=round] (345.52, 66.52) -- (345.52, 66.52);

\path[draw=drawColor,line width= 0.4pt,dash pattern=on 4pt off 4pt ,line join=round,line cap=round] (345.52, 66.52) -- (345.52, 66.52);

\path[draw=drawColor,line width= 0.4pt,line join=round,line cap=round] (344.46, 66.52) -- (346.58, 66.52);

\path[draw=drawColor,line width= 0.4pt,line join=round,line cap=round] (344.46, 66.52) -- (346.58, 66.52);

\path[draw=drawColor,line width= 0.4pt,line join=round,line cap=round] (343.40, 66.52) --
	(347.64, 66.52) --
	(347.64, 66.52) --
	(343.40, 66.52) --
	(343.40, 66.52);

\path[draw=drawColor,line width= 1.2pt,line join=round] (364.60, 66.52) -- (368.84, 66.52);

\path[draw=drawColor,line width= 0.4pt,dash pattern=on 4pt off 4pt ,line join=round,line cap=round] (366.72, 66.52) -- (366.72, 66.52);

\path[draw=drawColor,line width= 0.4pt,dash pattern=on 4pt off 4pt ,line join=round,line cap=round] (366.72, 66.52) -- (366.72, 66.52);

\path[draw=drawColor,line width= 0.4pt,line join=round,line cap=round] (365.66, 66.52) -- (367.78, 66.52);

\path[draw=drawColor,line width= 0.4pt,line join=round,line cap=round] (365.66, 66.52) -- (367.78, 66.52);

\path[draw=drawColor,line width= 0.4pt,line join=round,line cap=round] (364.60, 66.52) --
	(368.84, 66.52) --
	(368.84, 66.52) --
	(364.60, 66.52) --
	(364.60, 66.52);
\end{scope}
\begin{scope}
\path[clip] (  0.00,  0.00) rectangle (423.50,254.10);
\definecolor{drawColor}{RGB}{0,0,0}

\path[draw=drawColor,line width= 0.4pt,line join=round,line cap=round] (260.95, 66.52) -- (260.95,199.58);

\path[draw=drawColor,line width= 0.4pt,line join=round,line cap=round] (260.95, 66.52) -- (254.95, 66.52);

\path[draw=drawColor,line width= 0.4pt,line join=round,line cap=round] (260.95, 88.70) -- (254.95, 88.70);

\path[draw=drawColor,line width= 0.4pt,line join=round,line cap=round] (260.95,110.87) -- (254.95,110.87);

\path[draw=drawColor,line width= 0.4pt,line join=round,line cap=round] (260.95,133.05) -- (254.95,133.05);

\path[draw=drawColor,line width= 0.4pt,line join=round,line cap=round] (260.95,155.23) -- (254.95,155.23);

\path[draw=drawColor,line width= 0.4pt,line join=round,line cap=round] (260.95,177.40) -- (254.95,177.40);

\path[draw=drawColor,line width= 0.4pt,line join=round,line cap=round] (260.95,199.58) -- (254.95,199.58);

\node[text=drawColor,rotate= 90.00,anchor=base,inner sep=0pt, outer sep=0pt, scale=  1.00] at (246.55, 66.52) {0};

\node[text=drawColor,rotate= 90.00,anchor=base,inner sep=0pt, outer sep=0pt, scale=  1.00] at (246.55, 88.70) {2};

\node[text=drawColor,rotate= 90.00,anchor=base,inner sep=0pt, outer sep=0pt, scale=  1.00] at (246.55,110.87) {4};

\node[text=drawColor,rotate= 90.00,anchor=base,inner sep=0pt, outer sep=0pt, scale=  1.00] at (246.55,133.05) {6};

\node[text=drawColor,rotate= 90.00,anchor=base,inner sep=0pt, outer sep=0pt, scale=  1.00] at (246.55,155.23) {8};

\node[text=drawColor,rotate= 90.00,anchor=base,inner sep=0pt, outer sep=0pt, scale=  1.00] at (246.55,177.40) {10};

\node[text=drawColor,rotate= 90.00,anchor=base,inner sep=0pt, outer sep=0pt, scale=  1.00] at (246.55,199.58) {12};

\path[draw=drawColor,line width= 0.4pt,line join=round,line cap=round] (260.95, 61.20) --
	(398.30, 61.20) --
	(398.30,204.90) --
	(260.95,204.90) --
	(260.95, 61.20);
\end{scope}
\begin{scope}
\path[clip] (260.95, 61.20) rectangle (398.30,204.90);
\definecolor{drawColor}{RGB}{205,38,38}

\path[draw=drawColor,line width= 1.2pt,line join=round] (285.11,110.87) -- (289.35,110.87);

\path[draw=drawColor,line width= 0.4pt,dash pattern=on 4pt off 4pt ,line join=round,line cap=round] (287.23, 66.52) -- (287.23, 66.52);

\path[draw=drawColor,line width= 0.4pt,dash pattern=on 4pt off 4pt ,line join=round,line cap=round] (287.23,188.49) -- (287.23,121.96);

\path[draw=drawColor,line width= 0.4pt,line join=round,line cap=round] (286.17, 66.52) -- (288.29, 66.52);

\path[draw=drawColor,line width= 0.4pt,line join=round,line cap=round] (286.17,188.49) -- (288.29,188.49);

\path[draw=drawColor,line width= 0.4pt,line join=round,line cap=round] (285.11, 66.52) --
	(289.35, 66.52) --
	(289.35,121.96) --
	(285.11,121.96) --
	(285.11, 66.52);

\path[draw=drawColor,line width= 1.2pt,line join=round] (306.31, 88.70) -- (310.55, 88.70);

\path[draw=drawColor,line width= 0.4pt,dash pattern=on 4pt off 4pt ,line join=round,line cap=round] (308.43, 66.52) -- (308.43, 66.52);

\path[draw=drawColor,line width= 0.4pt,dash pattern=on 4pt off 4pt ,line join=round,line cap=round] (308.43,177.40) -- (308.43,110.87);

\path[draw=drawColor,line width= 0.4pt,line join=round,line cap=round] (307.37, 66.52) -- (309.49, 66.52);

\path[draw=drawColor,line width= 0.4pt,line join=round,line cap=round] (307.37,177.40) -- (309.49,177.40);

\path[draw=drawColor,line width= 0.4pt,line join=round,line cap=round] (306.31, 66.52) --
	(310.55, 66.52) --
	(310.55,110.87) --
	(306.31,110.87) --
	(306.31, 66.52);

\path[draw=drawColor,line width= 1.2pt,line join=round] (327.51, 66.52) -- (331.75, 66.52);

\path[draw=drawColor,line width= 0.4pt,dash pattern=on 4pt off 4pt ,line join=round,line cap=round] (329.63, 66.52) -- (329.63, 66.52);

\path[draw=drawColor,line width= 0.4pt,dash pattern=on 4pt off 4pt ,line join=round,line cap=round] (329.63,166.31) -- (329.63,110.87);

\path[draw=drawColor,line width= 0.4pt,line join=round,line cap=round] (328.57, 66.52) -- (330.69, 66.52);

\path[draw=drawColor,line width= 0.4pt,line join=round,line cap=round] (328.57,166.31) -- (330.69,166.31);

\path[draw=drawColor,line width= 0.4pt,line join=round,line cap=round] (327.51, 66.52) --
	(331.75, 66.52) --
	(331.75,110.87) --
	(327.51,110.87) --
	(327.51, 66.52);

\path[draw=drawColor,line width= 1.2pt,line join=round] (348.70, 66.52) -- (352.94, 66.52);

\path[draw=drawColor,line width= 0.4pt,dash pattern=on 4pt off 4pt ,line join=round,line cap=round] (350.82, 66.52) -- (350.82, 66.52);

\path[draw=drawColor,line width= 0.4pt,dash pattern=on 4pt off 4pt ,line join=round,line cap=round] (350.82,121.96) -- (350.82, 88.70);

\path[draw=drawColor,line width= 0.4pt,line join=round,line cap=round] (349.76, 66.52) -- (351.88, 66.52);

\path[draw=drawColor,line width= 0.4pt,line join=round,line cap=round] (349.76,121.96) -- (351.88,121.96);

\path[draw=drawColor,line width= 0.4pt,line join=round,line cap=round] (348.70, 66.52) --
	(352.94, 66.52) --
	(352.94, 88.70) --
	(348.70, 88.70) --
	(348.70, 66.52);

\path[draw=drawColor,line width= 1.2pt,line join=round] (369.90, 66.52) -- (374.14, 66.52);

\path[draw=drawColor,line width= 0.4pt,dash pattern=on 4pt off 4pt ,line join=round,line cap=round] (372.02, 66.52) -- (372.02, 66.52);

\path[draw=drawColor,line width= 0.4pt,dash pattern=on 4pt off 4pt ,line join=round,line cap=round] (372.02,121.96) -- (372.02, 88.70);

\path[draw=drawColor,line width= 0.4pt,line join=round,line cap=round] (370.96, 66.52) -- (373.08, 66.52);

\path[draw=drawColor,line width= 0.4pt,line join=round,line cap=round] (370.96,121.96) -- (373.08,121.96);

\path[draw=drawColor,line width= 0.4pt,line join=round,line cap=round] (369.90, 66.52) --
	(374.14, 66.52) --
	(374.14, 88.70) --
	(369.90, 88.70) --
	(369.90, 66.52);
\end{scope}
\begin{scope}
\path[clip] (  0.00,  0.00) rectangle (423.50,254.10);
\definecolor{drawColor}{RGB}{0,0,0}

\path[draw=drawColor,line width= 0.4pt,line join=round,line cap=round] (260.95, 61.20) --
	(398.30, 61.20) --
	(398.30,204.90) --
	(260.95,204.90) --
	(260.95, 61.20);
\end{scope}
\begin{scope}
\path[clip] (260.95, 61.20) rectangle (398.30,204.90);
\definecolor{drawColor}{RGB}{24,116,205}

\path[draw=drawColor,line width= 1.2pt,line join=round] (290.41,121.96) -- (294.65,121.96);

\path[draw=drawColor,line width= 0.4pt,dash pattern=on 4pt off 4pt ,line join=round,line cap=round] (292.53, 66.52) -- (292.53,110.87);

\path[draw=drawColor,line width= 0.4pt,dash pattern=on 4pt off 4pt ,line join=round,line cap=round] (292.53,199.58) -- (292.53,155.23);

\path[draw=drawColor,line width= 0.4pt,line join=round,line cap=round] (291.47, 66.52) -- (293.59, 66.52);

\path[draw=drawColor,line width= 0.4pt,line join=round,line cap=round] (291.47,199.58) -- (293.59,199.58);

\path[draw=drawColor,line width= 0.4pt,line join=round,line cap=round] (290.41,110.87) --
	(294.65,110.87) --
	(294.65,155.23) --
	(290.41,155.23) --
	(290.41,110.87);

\path[draw=drawColor,line width= 1.2pt,line join=round] (311.61,110.87) -- (315.85,110.87);

\path[draw=drawColor,line width= 0.4pt,dash pattern=on 4pt off 4pt ,line join=round,line cap=round] (313.73, 66.52) -- (313.73, 66.52);

\path[draw=drawColor,line width= 0.4pt,dash pattern=on 4pt off 4pt ,line join=round,line cap=round] (313.73,188.49) -- (313.73,133.05);

\path[draw=drawColor,line width= 0.4pt,line join=round,line cap=round] (312.67, 66.52) -- (314.79, 66.52);

\path[draw=drawColor,line width= 0.4pt,line join=round,line cap=round] (312.67,188.49) -- (314.79,188.49);

\path[draw=drawColor,line width= 0.4pt,line join=round,line cap=round] (311.61, 66.52) --
	(315.85, 66.52) --
	(315.85,133.05) --
	(311.61,133.05) --
	(311.61, 66.52);

\path[draw=drawColor,line width= 1.2pt,line join=round] (332.81, 88.70) -- (337.05, 88.70);

\path[draw=drawColor,line width= 0.4pt,dash pattern=on 4pt off 4pt ,line join=round,line cap=round] (334.93, 66.52) -- (334.93, 66.52);

\path[draw=drawColor,line width= 0.4pt,dash pattern=on 4pt off 4pt ,line join=round,line cap=round] (334.93,188.49) -- (334.93,121.96);

\path[draw=drawColor,line width= 0.4pt,line join=round,line cap=round] (333.87, 66.52) -- (335.99, 66.52);

\path[draw=drawColor,line width= 0.4pt,line join=round,line cap=round] (333.87,188.49) -- (335.99,188.49);

\path[draw=drawColor,line width= 0.4pt,line join=round,line cap=round] (332.81, 66.52) --
	(337.05, 66.52) --
	(337.05,121.96) --
	(332.81,121.96) --
	(332.81, 66.52);

\path[draw=drawColor,line width= 1.2pt,line join=round] (354.00, 88.70) -- (358.24, 88.70);

\path[draw=drawColor,line width= 0.4pt,dash pattern=on 4pt off 4pt ,line join=round,line cap=round] (356.12, 66.52) -- (356.12, 66.52);

\path[draw=drawColor,line width= 0.4pt,dash pattern=on 4pt off 4pt ,line join=round,line cap=round] (356.12,177.40) -- (356.12,110.87);

\path[draw=drawColor,line width= 0.4pt,line join=round,line cap=round] (355.06, 66.52) -- (357.18, 66.52);

\path[draw=drawColor,line width= 0.4pt,line join=round,line cap=round] (355.06,177.40) -- (357.18,177.40);

\path[draw=drawColor,line width= 0.4pt,line join=round,line cap=round] (354.00, 66.52) --
	(358.24, 66.52) --
	(358.24,110.87) --
	(354.00,110.87) --
	(354.00, 66.52);

\path[draw=drawColor,line width= 1.2pt,line join=round] (375.20, 66.52) -- (379.44, 66.52);

\path[draw=drawColor,line width= 0.4pt,dash pattern=on 4pt off 4pt ,line join=round,line cap=round] (377.32, 66.52) -- (377.32, 66.52);

\path[draw=drawColor,line width= 0.4pt,dash pattern=on 4pt off 4pt ,line join=round,line cap=round] (377.32,177.40) -- (377.32,110.87);

\path[draw=drawColor,line width= 0.4pt,line join=round,line cap=round] (376.26, 66.52) -- (378.38, 66.52);

\path[draw=drawColor,line width= 0.4pt,line join=round,line cap=round] (376.26,177.40) -- (378.38,177.40);

\path[draw=drawColor,line width= 0.4pt,line join=round,line cap=round] (375.20, 66.52) --
	(379.44, 66.52) --
	(379.44,110.87) --
	(375.20,110.87) --
	(375.20, 66.52);
\end{scope}
\begin{scope}
\path[clip] (  0.00,  0.00) rectangle (423.50,254.10);
\definecolor{drawColor}{RGB}{0,0,0}

\path[draw=drawColor,line width= 0.4pt,line join=round,line cap=round] (260.95, 61.20) --
	(398.30, 61.20) --
	(398.30,204.90) --
	(260.95,204.90) --
	(260.95, 61.20);
\end{scope}
\begin{scope}
\path[clip] (260.95, 61.20) rectangle (398.30,204.90);
\definecolor{drawColor}{RGB}{0,0,0}
\definecolor{fillColor}{RGB}{255,255,255}

\path[draw=drawColor,line width= 0.4pt,line join=round,line cap=round,fill=fillColor] (335.99,197.72) rectangle (391.43,149.72);
\definecolor{drawColor}{RGB}{69,139,0}

\path[draw=drawColor,line width= 0.4pt,line join=round,line cap=round] (344.99,185.72) circle (  2.25);
\definecolor{drawColor}{RGB}{205,38,38}

\path[draw=drawColor,line width= 0.4pt,line join=round,line cap=round] (344.99,173.72) circle (  2.25);
\definecolor{drawColor}{RGB}{24,116,205}

\path[draw=drawColor,line width= 0.4pt,line join=round,line cap=round] (344.99,161.72) circle (  2.25);
\definecolor{drawColor}{RGB}{0,0,0}

\node[text=drawColor,anchor=base west,inner sep=0pt, outer sep=0pt, scale=  1.00] at (353.99,182.27) {dHSIC};

\node[text=drawColor,anchor=base west,inner sep=0pt, outer sep=0pt, scale=  1.00] at (353.99,170.27) {HSIC};

\node[text=drawColor,anchor=base west,inner sep=0pt, outer sep=0pt, scale=  1.00] at (353.99,158.27) {BMR-$n$};
\end{scope}
\end{tikzpicture}

%% file: Plots/realworld_romano.tex
\begin{tikzpicture}[x=1pt,y=1pt]
\definecolor{fillColor}{RGB}{255,255,255}
\path[use as bounding box,fill=fillColor,fill opacity=0.00] (0,0) rectangle (254.10,275.28);
\begin{scope}
\path[clip] ( 49.20, 61.20) rectangle (228.90,226.08);
\definecolor{drawColor}{RGB}{69,139,0}

\path[draw=drawColor,line width= 0.4pt,line join=round,line cap=round] ( 55.86, 67.31) circle (  2.25);

\path[draw=drawColor,line width= 0.4pt,line join=round,line cap=round] ( 62.79, 67.31) circle (  2.25);

\path[draw=drawColor,line width= 0.4pt,line join=round,line cap=round] ( 69.72, 67.31) circle (  2.25);

\path[draw=drawColor,line width= 0.4pt,line join=round,line cap=round] ( 76.65, 67.31) circle (  2.25);

\path[draw=drawColor,line width= 0.4pt,line join=round,line cap=round] ( 83.59, 67.31) circle (  2.25);

\path[draw=drawColor,line width= 0.4pt,line join=round,line cap=round] ( 90.52, 67.31) circle (  2.25);

\path[draw=drawColor,line width= 0.4pt,line join=round,line cap=round] ( 97.45, 67.31) circle (  2.25);

\path[draw=drawColor,line width= 0.4pt,line join=round,line cap=round] (104.39, 67.31) circle (  2.25);

\path[draw=drawColor,line width= 0.4pt,line join=round,line cap=round] (111.32, 67.31) circle (  2.25);

\path[draw=drawColor,line width= 0.4pt,line join=round,line cap=round] (118.25, 67.31) circle (  2.25);

\path[draw=drawColor,line width= 0.4pt,line join=round,line cap=round] (125.18, 67.31) circle (  2.25);

\path[draw=drawColor,line width= 0.4pt,line join=round,line cap=round] (132.12, 67.31) circle (  2.25);

\path[draw=drawColor,line width= 0.4pt,line join=round,line cap=round] (139.05, 67.31) circle (  2.25);

\path[draw=drawColor,line width= 0.4pt,line join=round,line cap=round] (145.98, 67.31) circle (  2.25);

\path[draw=drawColor,line width= 0.4pt,line join=round,line cap=round] (152.92, 67.31) circle (  2.25);

\path[draw=drawColor,line width= 0.4pt,line join=round,line cap=round] (159.85, 67.31) circle (  2.25);

\path[draw=drawColor,line width= 0.4pt,line join=round,line cap=round] (166.78, 67.31) circle (  2.25);

\path[draw=drawColor,line width= 0.4pt,line join=round,line cap=round] (173.72, 67.31) circle (  2.25);

\path[draw=drawColor,line width= 0.4pt,line join=round,line cap=round] (180.65, 67.31) circle (  2.25);

\path[draw=drawColor,line width= 0.4pt,line join=round,line cap=round] (187.58, 67.31) circle (  2.25);

\path[draw=drawColor,line width= 0.4pt,line join=round,line cap=round] (194.51, 67.31) circle (  2.25);

\path[draw=drawColor,line width= 0.4pt,line join=round,line cap=round] (201.45, 67.31) circle (  2.25);

\path[draw=drawColor,line width= 0.4pt,line join=round,line cap=round] (208.38, 67.31) circle (  2.25);

\path[draw=drawColor,line width= 0.4pt,line join=round,line cap=round] (215.31, 67.31) circle (  2.25);

\path[draw=drawColor,line width= 0.4pt,line join=round,line cap=round] (222.25,117.79) circle (  2.25);
\end{scope}
\begin{scope}
\path[clip] (  0.00,  0.00) rectangle (254.10,275.28);
\definecolor{drawColor}{RGB}{0,0,0}

\path[draw=drawColor,line width= 0.4pt,line join=round,line cap=round] ( 83.59, 61.20) -- (222.25, 61.20);

\path[draw=drawColor,line width= 0.4pt,line join=round,line cap=round] ( 83.59, 61.20) -- ( 83.59, 55.20);

\path[draw=drawColor,line width= 0.4pt,line join=round,line cap=round] (118.25, 61.20) -- (118.25, 55.20);

\path[draw=drawColor,line width= 0.4pt,line join=round,line cap=round] (152.92, 61.20) -- (152.92, 55.20);

\path[draw=drawColor,line width= 0.4pt,line join=round,line cap=round] (187.58, 61.20) -- (187.58, 55.20);

\path[draw=drawColor,line width= 0.4pt,line join=round,line cap=round] (222.25, 61.20) -- (222.25, 55.20);

\node[text=drawColor,anchor=base,inner sep=0pt, outer sep=0pt, scale=  1.00] at ( 83.59, 39.60) {5};

\node[text=drawColor,anchor=base,inner sep=0pt, outer sep=0pt, scale=  1.00] at (118.25, 39.60) {10};

\node[text=drawColor,anchor=base,inner sep=0pt, outer sep=0pt, scale=  1.00] at (152.92, 39.60) {15};

\node[text=drawColor,anchor=base,inner sep=0pt, outer sep=0pt, scale=  1.00] at (187.58, 39.60) {20};

\node[text=drawColor,anchor=base,inner sep=0pt, outer sep=0pt, scale=  1.00] at (222.25, 39.60) {25};

\path[draw=drawColor,line width= 0.4pt,line join=round,line cap=round] ( 49.20, 65.09) -- ( 49.20,210.77);

\path[draw=drawColor,line width= 0.4pt,line join=round,line cap=round] ( 49.20, 65.09) -- ( 43.20, 65.09);

\path[draw=drawColor,line width= 0.4pt,line join=round,line cap=round] ( 49.20, 89.37) -- ( 43.20, 89.37);

\path[draw=drawColor,line width= 0.4pt,line join=round,line cap=round] ( 49.20,113.65) -- ( 43.20,113.65);

\path[draw=drawColor,line width= 0.4pt,line join=round,line cap=round] ( 49.20,137.93) -- ( 43.20,137.93);

\path[draw=drawColor,line width= 0.4pt,line join=round,line cap=round] ( 49.20,162.21) -- ( 43.20,162.21);

\path[draw=drawColor,line width= 0.4pt,line join=round,line cap=round] ( 49.20,186.49) -- ( 43.20,186.49);

\path[draw=drawColor,line width= 0.4pt,line join=round,line cap=round] ( 49.20,210.77) -- ( 43.20,210.77);

\node[text=drawColor,rotate= 90.00,anchor=base,inner sep=0pt, outer sep=0pt, scale=  1.00] at ( 34.80, 65.09) {-7};

\node[text=drawColor,rotate= 90.00,anchor=base,inner sep=0pt, outer sep=0pt, scale=  1.00] at ( 34.80, 89.37) {-6};

\node[text=drawColor,rotate= 90.00,anchor=base,inner sep=0pt, outer sep=0pt, scale=  1.00] at ( 34.80,113.65) {-5};

\node[text=drawColor,rotate= 90.00,anchor=base,inner sep=0pt, outer sep=0pt, scale=  1.00] at ( 34.80,137.93) {-4};

\node[text=drawColor,rotate= 90.00,anchor=base,inner sep=0pt, outer sep=0pt, scale=  1.00] at ( 34.80,162.21) {-3};

\node[text=drawColor,rotate= 90.00,anchor=base,inner sep=0pt, outer sep=0pt, scale=  1.00] at ( 34.80,186.49) {-2};

\node[text=drawColor,rotate= 90.00,anchor=base,inner sep=0pt, outer sep=0pt, scale=  1.00] at ( 34.80,210.77) {-1};

\path[draw=drawColor,line width= 0.4pt,line join=round,line cap=round] ( 49.20, 61.20) --
	(228.90, 61.20) --
	(228.90,226.08) --
	( 49.20,226.08) --
	( 49.20, 61.20);
\end{scope}
\begin{scope}
\path[clip] (  0.00,  0.00) rectangle (254.10,275.28);
\definecolor{drawColor}{RGB}{0,0,0}

\node[text=drawColor,anchor=base,inner sep=0pt, outer sep=0pt, scale=  1.00] at (139.05, 15.60) {DAG};

\node[text=drawColor,rotate= 90.00,anchor=base,inner sep=0pt, outer sep=0pt, scale=  1.00] at ( 10.80,143.64) {log(p-value)};
\end{scope}
\begin{scope}
\path[clip] ( 49.20, 61.20) rectangle (228.90,226.08);
\definecolor{drawColor}{RGB}{205,38,38}

\path[draw=drawColor,line width= 0.4pt,line join=round,line cap=round] ( 55.86, 87.63) --
	( 58.89, 82.39) --
	( 52.83, 82.39) --
	( 55.86, 87.63);

\path[draw=drawColor,line width= 0.4pt,line join=round,line cap=round] ( 62.79, 87.63) --
	( 65.82, 82.39) --
	( 59.76, 82.39) --
	( 62.79, 87.63);

\path[draw=drawColor,line width= 0.4pt,line join=round,line cap=round] ( 69.72, 87.63) --
	( 72.75, 82.39) --
	( 66.69, 82.39) --
	( 69.72, 87.63);

\path[draw=drawColor,line width= 0.4pt,line join=round,line cap=round] ( 76.65, 87.63) --
	( 79.68, 82.39) --
	( 73.62, 82.39) --
	( 76.65, 87.63);

\path[draw=drawColor,line width= 0.4pt,line join=round,line cap=round] ( 83.59,104.46) --
	( 86.62, 99.21) --
	( 80.56, 99.21) --
	( 83.59,104.46);

\path[draw=drawColor,line width= 0.4pt,line join=round,line cap=round] ( 90.52,138.12) --
	( 93.55,132.87) --
	( 87.49,132.87) --
	( 90.52,138.12);

\path[draw=drawColor,line width= 0.4pt,line join=round,line cap=round] ( 97.45, 87.63) --
	(100.48, 82.39) --
	( 94.42, 82.39) --
	( 97.45, 87.63);

\path[draw=drawColor,line width= 0.4pt,line join=round,line cap=round] (104.39, 87.63) --
	(107.42, 82.39) --
	(101.36, 82.39) --
	(104.39, 87.63);

\path[draw=drawColor,line width= 0.4pt,line join=round,line cap=round] (111.32, 87.63) --
	(114.35, 82.39) --
	(108.29, 82.39) --
	(111.32, 87.63);

\path[draw=drawColor,line width= 0.4pt,line join=round,line cap=round] (118.25, 87.63) --
	(121.28, 82.39) --
	(115.22, 82.39) --
	(118.25, 87.63);

\path[draw=drawColor,line width= 0.4pt,line join=round,line cap=round] (125.18,114.31) --
	(128.22,109.06) --
	(122.15,109.06) --
	(125.18,114.31);

\path[draw=drawColor,line width= 0.4pt,line join=round,line cap=round] (132.12,104.46) --
	(135.15, 99.21) --
	(129.09, 99.21) --
	(132.12,104.46);

\path[draw=drawColor,line width= 0.4pt,line join=round,line cap=round] (139.05, 87.63) --
	(142.08, 82.39) --
	(136.02, 82.39) --
	(139.05, 87.63);

\path[draw=drawColor,line width= 0.4pt,line join=round,line cap=round] (145.98, 87.63) --
	(149.01, 82.39) --
	(142.95, 82.39) --
	(145.98, 87.63);

\path[draw=drawColor,line width= 0.4pt,line join=round,line cap=round] (152.92,121.29) --
	(155.95,116.04) --
	(149.89,116.04) --
	(152.92,121.29);

\path[draw=drawColor,line width= 0.4pt,line join=round,line cap=round] (159.85, 87.63) --
	(162.88, 82.39) --
	(156.82, 82.39) --
	(159.85, 87.63);

\path[draw=drawColor,line width= 0.4pt,line join=round,line cap=round] (166.78,143.54) --
	(169.81,138.29) --
	(163.75,138.29) --
	(166.78,143.54);

\path[draw=drawColor,line width= 0.4pt,line join=round,line cap=round] (173.72, 87.63) --
	(176.75, 82.39) --
	(170.69, 82.39) --
	(173.72, 87.63);

\path[draw=drawColor,line width= 0.4pt,line join=round,line cap=round] (180.65, 87.63) --
	(183.68, 82.39) --
	(177.62, 82.39) --
	(180.65, 87.63);

\path[draw=drawColor,line width= 0.4pt,line join=round,line cap=round] (187.58, 87.63) --
	(190.61, 82.39) --
	(184.55, 82.39) --
	(187.58, 87.63);

\path[draw=drawColor,line width= 0.4pt,line join=round,line cap=round] (194.51, 87.63) --
	(197.54, 82.39) --
	(191.48, 82.39) --
	(194.51, 87.63);

\path[draw=drawColor,line width= 0.4pt,line join=round,line cap=round] (201.45,114.31) --
	(204.48,109.06) --
	(198.42,109.06) --
	(201.45,114.31);

\path[draw=drawColor,line width= 0.4pt,line join=round,line cap=round] (208.38, 87.63) --
	(211.41, 82.39) --
	(205.35, 82.39) --
	(208.38, 87.63);

\path[draw=drawColor,line width= 0.4pt,line join=round,line cap=round] (215.31, 87.63) --
	(218.34, 82.39) --
	(212.28, 82.39) --
	(215.31, 87.63);

\path[draw=drawColor,line width= 0.4pt,line join=round,line cap=round] (222.25,172.53) --
	(225.28,167.28) --
	(219.22,167.28) --
	(222.25,172.53);
\definecolor{drawColor}{RGB}{24,116,205}

\path[draw=drawColor,line width= 0.4pt,line join=round,line cap=round] ( 53.61, 65.06) rectangle ( 58.11, 69.56);

\path[draw=drawColor,line width= 0.4pt,line join=round,line cap=round] ( 60.54, 65.06) rectangle ( 65.04, 69.56);

\path[draw=drawColor,line width= 0.4pt,line join=round,line cap=round] ( 67.47, 65.06) rectangle ( 71.97, 69.56);

\path[draw=drawColor,line width= 0.4pt,line join=round,line cap=round] ( 74.40,145.96) rectangle ( 78.90,150.46);

\path[draw=drawColor,line width= 0.4pt,line join=round,line cap=round] ( 81.34, 65.06) rectangle ( 85.84, 69.56);

\path[draw=drawColor,line width= 0.4pt,line join=round,line cap=round] ( 88.27, 91.73) rectangle ( 92.77, 96.23);

\path[draw=drawColor,line width= 0.4pt,line join=round,line cap=round] ( 95.20, 65.06) rectangle ( 99.70, 69.56);

\path[draw=drawColor,line width= 0.4pt,line join=round,line cap=round] (102.14, 81.89) rectangle (106.64, 86.39);

\path[draw=drawColor,line width= 0.4pt,line join=round,line cap=round] (109.07, 65.06) rectangle (113.57, 69.56);

\path[draw=drawColor,line width= 0.4pt,line join=round,line cap=round] (116.00,137.79) rectangle (120.50,142.29);

\path[draw=drawColor,line width= 0.4pt,line join=round,line cap=round] (122.93,194.06) rectangle (127.43,198.56);

\path[draw=drawColor,line width= 0.4pt,line join=round,line cap=round] (129.87,183.05) rectangle (134.37,187.55);

\path[draw=drawColor,line width= 0.4pt,line join=round,line cap=round] (136.80, 65.06) rectangle (141.30, 69.56);

\path[draw=drawColor,line width= 0.4pt,line join=round,line cap=round] (143.73, 65.06) rectangle (148.23, 69.56);

\path[draw=drawColor,line width= 0.4pt,line join=round,line cap=round] (150.67,104.13) rectangle (155.17,108.63);

\path[draw=drawColor,line width= 0.4pt,line join=round,line cap=round] (157.60, 65.06) rectangle (162.10, 69.56);

\path[draw=drawColor,line width= 0.4pt,line join=round,line cap=round] (164.53,196.66) rectangle (169.03,201.16);

\path[draw=drawColor,line width= 0.4pt,line join=round,line cap=round] (171.47, 65.06) rectangle (175.97, 69.56);

\path[draw=drawColor,line width= 0.4pt,line join=round,line cap=round] (178.40, 65.06) rectangle (182.90, 69.56);

\path[draw=drawColor,line width= 0.4pt,line join=round,line cap=round] (185.33, 65.06) rectangle (189.83, 69.56);

\path[draw=drawColor,line width= 0.4pt,line join=round,line cap=round] (192.26,178.28) rectangle (196.76,182.78);

\path[draw=drawColor,line width= 0.4pt,line join=round,line cap=round] (199.20, 65.06) rectangle (203.70, 69.56);

\path[draw=drawColor,line width= 0.4pt,line join=round,line cap=round] (206.13, 65.06) rectangle (210.63, 69.56);

\path[draw=drawColor,line width= 0.4pt,line join=round,line cap=round] (213.06, 65.06) rectangle (217.56, 69.56);

\path[draw=drawColor,line width= 0.4pt,line join=round,line cap=round] (220.00,217.72) rectangle (224.50,222.22);
\definecolor{drawColor}{RGB}{211,211,211}

\path[draw=drawColor,line width= 0.4pt,dash pattern=on 1pt off 3pt ,line join=round,line cap=round] ( 83.59, 61.20) -- ( 83.59,226.08);

\path[draw=drawColor,line width= 0.4pt,dash pattern=on 1pt off 3pt ,line join=round,line cap=round] (118.25, 61.20) -- (118.25,226.08);

\path[draw=drawColor,line width= 0.4pt,dash pattern=on 1pt off 3pt ,line join=round,line cap=round] (152.92, 61.20) -- (152.92,226.08);

\path[draw=drawColor,line width= 0.4pt,dash pattern=on 1pt off 3pt ,line join=round,line cap=round] (187.58, 61.20) -- (187.58,226.08);

\path[draw=drawColor,line width= 0.4pt,dash pattern=on 1pt off 3pt ,line join=round,line cap=round] (222.25, 61.20) -- (222.25,226.08);

\path[draw=drawColor,line width= 0.4pt,dash pattern=on 1pt off 3pt ,line join=round,line cap=round] ( 49.20, 65.09) -- (228.90, 65.09);

\path[draw=drawColor,line width= 0.4pt,dash pattern=on 1pt off 3pt ,line join=round,line cap=round] ( 49.20, 89.37) -- (228.90, 89.37);

\path[draw=drawColor,line width= 0.4pt,dash pattern=on 1pt off 3pt ,line join=round,line cap=round] ( 49.20,113.65) -- (228.90,113.65);

\path[draw=drawColor,line width= 0.4pt,dash pattern=on 1pt off 3pt ,line join=round,line cap=round] ( 49.20,137.93) -- (228.90,137.93);

\path[draw=drawColor,line width= 0.4pt,dash pattern=on 1pt off 3pt ,line join=round,line cap=round] ( 49.20,162.21) -- (228.90,162.21);

\path[draw=drawColor,line width= 0.4pt,dash pattern=on 1pt off 3pt ,line join=round,line cap=round] ( 49.20,186.49) -- (228.90,186.49);

\path[draw=drawColor,line width= 0.4pt,dash pattern=on 1pt off 3pt ,line join=round,line cap=round] ( 49.20,210.77) -- (228.90,210.77);
\definecolor{drawColor}{RGB}{0,0,0}
\definecolor{fillColor}{RGB}{255,255,255}

\path[draw=drawColor,line width= 0.4pt,line join=round,line cap=round,fill=fillColor] ( 58.19,217.83) rectangle (127.62,169.83);
\definecolor{drawColor}{RGB}{69,139,0}

\path[draw=drawColor,line width= 0.4pt,line join=round,line cap=round] ( 67.19,205.83) circle (  2.25);
\definecolor{drawColor}{RGB}{205,38,38}

\path[draw=drawColor,line width= 0.4pt,line join=round,line cap=round] ( 67.19,197.33) --
	( 70.22,192.08) --
	( 64.15,192.08) --
	( 67.19,197.33);
\definecolor{drawColor}{RGB}{24,116,205}

\path[draw=drawColor,line width= 0.4pt,line join=round,line cap=round] ( 64.94,179.58) rectangle ( 69.44,184.08);
\definecolor{drawColor}{RGB}{0,0,0}

\node[text=drawColor,anchor=base west,inner sep=0pt, outer sep=0pt, scale=  1.00] at ( 76.19,202.39) {dHSIC};

\node[text=drawColor,anchor=base west,inner sep=0pt, outer sep=0pt, scale=  1.00] at ( 76.19,190.39) {HSIC};

\node[text=drawColor,anchor=base west,inner sep=0pt, outer sep=0pt, scale=  1.00] at ( 76.19,178.39) {BMR-$1000$};
\end{scope}
\end{tikzpicture}

%% file: Content/conclusions.tex
\section{Summary} \label{sec:conclusions}

We introduce a measure of joint dependence between $d$ variables,
called the d-variable Hilbert-Schmidt independence criterion
(dHSIC). We propose an estimator of dHSIC based on a computationally
efficient V-statistic and derive its asymptotic distribution. This allows to
construct three different hypothesis tests: a permutation test
(Definition~\ref{def:permutation_test}), a 
bootstrap test (Definition~\ref{def:bootstrap_test}) and a test based on a Gamma
approximation (Definition~\ref{def:gammaapproximation_test}).

We prove several properties for these tests. First and
foremost we establish that the bootstrap test achieves pointwise
asymptotic level (Theorem~\ref{thm:level_bootstrap_test}) and that it is
consistent for detecting any fixed alternative with asymptotic power equal
to one (Theorem~\ref{thm:consistencybootstraptest}). 
For the permutation test, we show that it 
achieves valid level (Proposition~\ref{thm:level_permutation_test}) and in
particular, this property carries over to the Monte-Carlo approximated
version of the permutation test.
Regarding the Gamma approximation based test, we derive asymptotic
expansions of the mean and variance of the dHSIC estimator
(Proposition~\ref{thm:meanofmhsic} and 
Proposition~\ref{thm:varianceofmhsic}) which serve as main basis in the
construction of the approximation. Although this test has no guarantees on
level and consistency, it is computationally very fast and was found to
perform well in numerical experiments. 

Various simulations illustrate the advantages of dHSIC over a pairwise
approach with HSIC and a traditional test that we call BMR-$C$.
Notably, dHSIC is computationally less expensive than HSIC and also
BMR-$C$ if $C$ grows larger than $n$. Moreover, when the dimension $d$
is large the pairwise HSIC approach with
Monte-Carlo approximation (for fixed $B$) is not able to reject the
null hypothesis and  BMR-$C$ seems to
suffer substantially from the curse of dimensionality.
We also outline applications for model selection in causal
inference which are based on joint independence testing of error terms in
structural equation models. In our numerical experiments on real and 
simulated data, dHSIC outperforms both other methods.



%% file: Content/appendix.tex
\section{Resampling tests}\label{sec:resamplingtests}

In this section we want to rigorously introduce resampling tests. Most
of this section is based on \citet{lehmann}, we however adjust a lot
of the notation to fit our situation.

Let $\alpha\in (0,1)$, let $\kernelspace$ be a measurable
space, let $T=(T_n)_{n\in\N}$ be a test statistic on $\kernelspace$, let $X$
be a random variable with values in $\kernelspace$ and let
$(X_i)_{i\in\N}$ be a sequence of iid copies.
The main idea behind resampling tests is to construct data sets based
on the original observations $(X_1,\dots,X_n)$. These types of
constructions are formalized by resampling methods.

\begin{definition}[resampling method]
  Let $\kernelspace$ be a measurable space and let
  $(M_n)_{n\in\N}\subseteq\N$ be a sequence. If
  $$g=\left((g_{n,k})_{k\in\{1,\dots,M_n\}}\right)_{n\in\N}$$
  is a family of functions satisfying for all $n\in\N$ and for all
  $k\in\{1,\dots,M_n\}$ that
  $$g_{n,k}:\kernelspace^n\rightarrow\kernelspace^n,$$
  then we call $g$ a resampling method.
\end{definition}
Based on a resampling method $g$ we can construct new observations for
all $n\in\N$ and for all $k\in\{1,\dots,M_n\}$ by defining
\begin{equation*}
  Z_{n,k}\coloneqq g_{n,k}(X_1,\dots,X_n).
\end{equation*}
The new 'resampled' data
$(Z_{n,k})_{k\in\{1,\dots,M_n\}}\subseteq\kernelspace^n$, $n\in\N$ is
called resampling scheme and for each $n\in\N$ the sequence
$Z_{n,1},\dots,Z_{n,M_n}$ should be seen as $M_n$ resampled data sets
constructed from the original observations $(X_1,\dots,X_n)$. A
resampling method is therefore a formalization of the concept of
resampling $M_n$ times from the original observations $(X_1,\dots,X_n)$.
Based on a resampling method we can introduce the resampling
distribution function.

\begin{definition}[resampling distribution function]
  \label{def:resamplingdistfunction}
  Let $\kernelspace$ be a measurable space, let
  $T=(T_n)_{n\in\N}$ be a test statistic on $\kernelspace$, let $g$ be
  a resampling method. For all $n\in\N$, the functions
  $\hat{R}_{T_n}:\kernelspace^n\times\R\rightarrow [0,1]$ defined for
  all $(x_1,\dots,x_n)\in\kernelspace^n$ and for all $t\in\R$ by
  \begin{equation*}
    \hat{R}_{T_n}(x_1,\dots,x_n)(t)\coloneqq\dfrac{1}{M_n}\sum_{k=1}^{M_n}\mathds{1}_{\{T_n(g_{n,k}(x_1,\dots,x_n))\leq
      t\}}
  \end{equation*}
  are called the resampling distribution functions (corresponding to test statistic
  $T$ and resampling method $g$).
\end{definition}
Fixing $n\in\N$ and $(x_1,\dots,x_n)\in\kernelspace^n$ it holds that
\begin{equation*}
  \hat{R}_{T_n}(x_1,\dots,x_n):\R\rightarrow [0,1]
\end{equation*}
is non-decreasing, right-continuous and satisfies
\begin{equation*}
  \lim_{t\rightarrow\text{-}\infty}\hat{R}_{T_n}(x_1,\dots,x_n)(t)=0
  \quad\text{and}\quad
  \lim_{t\rightarrow\infty}\hat{R}_{T_n}(x_1,\dots,x_n)(t)=1.
\end{equation*}
This implies that $\hat{R}_{T_n}(x_1,\dots,x_n)$ is a distribution
function and thus we can define the generalized inverse
\begin{equation*}
  \left(\hat{R}_{T_n}(x_1,\dots,x_n)\right)^{-1}:(0,1)\rightarrow\R
\end{equation*}
satisfying for all $\alpha\in(0,1)$ that
\begin{equation*}
  \left(\hat{R}_{T_n}(x_1,\dots,x_n)\right)^{-1}(\alpha)\coloneqq\inf\{t\in\R\mid
\hat{R}_{T_n}(x_1,\dots,x_n)(t)\geq \alpha\}.
\end{equation*}
Based on the resampling distribution functions we can define a
resampling test as follows.

\begin{definition}[resampling test]
  \label{def:resamplingtest}
  Let $\alpha\in(0,1)$, let $\kernelspace$ be a separable metric
  space, let $T=(T_n)_{n\in\N}$ be a test statistic on $\kernelspace$,
  let $g$ be a resampling method and let $\hat{R}_{T_n}$ be the
  corresponding resampling distribution functions. A hypothesis test
  $\phi=(\phi_n)_{n\in\N}$ defined for all $n\in\N$ and for all $(x_1,\dots,x_n)\in\kernelspace^n$ by
  \begin{equation*}
    \phi_n(x_1,\dots,x_n)\coloneqq\mathds{1}_{\left\{T_n(x_1,\dots,x_n)>(\hat{R}_{T_n}(x_1,\dots,x_n))^{-1}(1-\alpha)\right\}}
  \end{equation*}
  is called $\alpha$-resampling test (corresponding to $g$).
\end{definition}
The advantage of resampling tests is that they can be constructed for
any test statistic. We now define an important subclass of resampling methods.
\begin{definition}[resampling group]
  Let $\kernelspace$ be a measurable space, let
  $(M_n)_{n\in\N}\subseteq\N$ be a sequence and let $g$ be a
  resampling method. If $g$ satisfies that
  \begin{equation*}
    G\coloneqq\{g_{n,1},\dots,g_{n,M_n}\}
  \end{equation*}
  together with concatenation is a group of transformations on
  $\kernelspace^n$, then we call $g$ a resampling group.
\end{definition}
Resampling groups have the important property that for all test
statistics $T=(T_n)_{n\in\N}$ the corresponding resampling
distribution functions satisfy for all $n\in\N$, for all
$k\in\{1,\dots,M_n\}$ and for all $(x_1,\dots,x_n)\in\kernelspace^n$
that
\begin{equation}
  \label{eq:invariance_resamplingdist}
  \hat{R}_{T_n}(x_1,\dots,x_n)=\hat{R}_{T_n}(g_{n,k}(x_1,\dots,x_n)).
\end{equation}
This follows immediately from the group property of $g$. It allows us
to prove, given an appropriate invariance of the resampling group
under the null hypothesis, that the corresponding resampling test
achieves level $\alpha$. The following theorem is a reformulation of
\citet[Theorem 15.2.1]{lehmann}.

\begin{theorem}[level of resampling tests]
  \label{thm:level_resampling_test}
  Let $\alpha\in(0,1)$, let $\kernelspace$ be a separable metric
  space, let $\HO,\HA\subseteq\pmspace{\kernelspace}$ be a null and
  alternative hypothesis respectively, let $g$ be a resampling group
  satisfying under $\HO$ that for all $n\in\N$ and for all
  $k\in\{1,\dots,M_n\}$ it holds that
  \begin{equation*}
    g_{n,k}(X_1,\dots,X_n)\text{ is equal in distribution to }(X_1,\dots,X_n).
  \end{equation*}
  Then, the $\alpha$-resampling test $\phi$ corresponding to $g$ is a test
  at level $\alpha$, when testing $\HO$ against $\HA$.
\end{theorem}

\begin{proof}
  Fix $n\in\N$ and let $K$ be a uniformly distributed random variable
  on $\{1,\dots,M_n\}$ independent of $(X_1,\dots,X_n)$. Let
  $(x_1,\dots,x_n)\in\operatorname{Im}((X_1,\dots,X_n))$ and for all $k\in\{1,\dots,M_n\}$ define
  $z_{n,k}\coloneqq g_{n,k}(x_1,\dots,x_n)$ then it holds that $T_n(z_{n,K})$
  has the distribution function $\hat{R}_{T_n}(x_1,\dots,x_n)$. Hence,
  using \eqref{eq:invariance_resamplingdist} and the the properties of
  the generalized inverse it holds that
  \begin{align*}
    &\dfrac{1}{M_n}\sum_{k=1}^{M_n}\mathds{1}_{\{T_n(z_{n,k})>(\hat{R}_{T_n}(z_{n,k}))^{-1}(1-\alpha)\}}\\
    &\quad=\dfrac{1}{M_n}\sum_{k=1}^{M_n}\mathds{1}_{\{T_n(z_{n,k})>(\hat{R}_{T_n}(x_1,\dots,x_n))^{-1}(1-\alpha)\}}\\
    &\quad=\E\left(\mathds{1}_{\{T_n(z_{n,K})>
        (\hat{R}_{T_n}(x_1,\dots,x_n))^{-1}(1-\alpha)\}}\right)\\
    &\quad\leq\alpha,
  \end{align*}
  which together with the monotonicity of the integral and the
  convention $$Z_{n,k}=g_{n,k}(X_1,\dots,X_n)$$ implies that
  \begin{equation}
    \label{eq:genlevel_estimate}
    \E\left(\dfrac{1}{M_n}\sum_{k=1}^{M_n}\mathds{1}_{\{T_n(Z_{n,k})>(\hat{R}_{T_n}(Z_{n,k}))^{-1}(1-\alpha)\}}\right)\leq\alpha.
  \end{equation}
  Moreover, under $\HO$, i.e. $X_1,X_2,\ldots\iid\P^{X}\in\HO$, it holds by
  assumption for all $k\in\{1,\dots,M_n\}$ that $(X_1,\dots,X_n)$
  is equal in distribution to $Z_{n,k}$. This in particular implies
  that under $\HO$ it holds for all $k\in\{1,\dots,M_n\}$ that
  \begin{equation}
    \label{eq:genlevel_invariance}
    \E\left(\phi_n(Z_{n,k})\right)=\E\left(\phi_n(X_1,\dots,X_n)\right).
  \end{equation}
  Combining \eqref{eq:genlevel_estimate} and
  \eqref{eq:genlevel_invariance} results in
  \begin{align*}
    \P\left(\phi_n(X_1,\dots,X_n)=1\right)&=\E\left(\phi_n(X_1,\dots,X_n)\right)\\
    &=\dfrac{1}{M_n}\sum_{k=1}^{M_n}\E\left(\phi_n(Z_{n,k})\right)\\
    &=\E\left(\dfrac{1}{M_n}\sum_{k=1}^{M_n}\mathds{1}_{\{T_n(Z_{n,k})>(\hat{R}_{T_n}(Z_{n,k}))^{-1}(1-\alpha)\}}\right)\\
    &\leq\alpha,
  \end{align*}
  which completes the proof of Theorem \ref{thm:level_resampling_test}.
\end{proof}

The invariance assumption of the resampling group in the previous
theorem is the same as the randomization hypothesis given by
\citet[Definition 15.2.1]{lehmann}. Unfortunately, there are no
similar guarantees that an arbitrary resampling test controls the type
II error in any way. Results of this type need to be checked on a case
by case basis by analyzing the resampling distribution function for
the specific test statistic.

\section{Monte-Carlo approximated resampling tests}\label{sec:mcresamplingtests}

Finally, we want to discuss a computational difficulty that often
arises in the context of resampling tests. The problem is that in
practical applications the parameter $M_n$ from the definition of a
resampling method grows very fast in $n$ and makes computations
impossible for large $n$. One method of dealing with this is to
approximate the resampling distribution $\hat{R}_n$ using a
Monte-Carlo approximated version.

\begin{definition}[Monte-Carlo approximated resampling distribution]
  \label{def:monte_carlo_resampling_dist}
  Let $\kernelspace$ be a measurable space, let
  $T=(T_n)_{n\in\N}$ be a test statistic on $\kernelspace$, let $g$ be
  a resampling method and let $(K_i)_{i\in\N}$ be a sequence of
  independent uniformly distributed random variables on
  $\{1,\dots,M_n\}$. For all $B\in\N$ let
  $\hat{\mathcal{R}}^B_{T_n}:\kernelspace^n\times\R\rightarrow [0,1]$
  be the functions defined for all $(x_1,\dots,x_n)\in\kernelspace^n$
  and for all $t\in\R$ by
  \begin{equation*}
    \hat{\mathcal{R}}^B_{T_n}(x_1,\dots,x_n)(t)\coloneqq\dfrac{1}{B}\sum_{i=1}^{B}\mathds{1}_{\{T_n(g_{n,K_i}(x_1,\dots,x_n))\leq
      t\}}
  \end{equation*}
  are called the Monte-Carlo approximated resampling distribution functions (corresponding to test statistic
  $T$ and resampling method $g$).
\end{definition}

The following proposition shows that $\hat{\mathcal{R}}^B_{T_n}$ approximates $\hat{R}_{T_n}$ in
an appropriate way.

\begin{proposition}[Monte-Carlo approximation of resampling distribution]
  \label{thm:monte_carlo_resampling_dist}
  Let $\kernelspace$ be a measurable space, let
  $T=(T_n)_{n\in\N}$ be a test statistic on $\kernelspace$, let $g$ be
  a resampling method, let $\hat{R}_{T_n}$ be the resampling
  distribution functions and for all $B\in\N$ let
  $\hat{\mathcal{R}}_{T_n}^B$ be the Monte-Carlo approximated resampling
  distribution functions. Then for all
  $(x_1,\dots,x_n)\in\kernelspace^n$ and for all $t\in\R$ it holds
  $\P$-a.s. that
  \begin{equation*}
    \lim_{B\rightarrow\infty}\hat{\mathcal{R}}_{T_n}^B(x_1,\dots,x_n)(t)=\hat{R}_{T_n}(x_1,\dots,x_n)(t).
  \end{equation*}
\end{proposition}

\begin{proof}
  Let $(K_i)_{i\in\N}$ be the sequence of uniformly distributed random
  variables on $\{1,\dots,M_n\}$ from the definition of
  $\hat{\mathcal{R}}_{T_n}^B$, then introduce for all
  $k\in\{1,\dots,M_n\}$ and for all $i\in\N$ the random variables
  \begin{equation*}
    Y_i^k\coloneqq\mathds{1}_{\{K_i=k\}}.
  \end{equation*}
  $Y_i^k$ has a Bernoulli distribution with parameter
  $\frac{1}{M_n}$. Furthermore, we can write
  \begin{align*}
    \hat{\mathcal{R}}_{T_n}^B(x_1,\dots,x_n)(t)
    &=\dfrac{1}{B}\sum_{i=1}^B\mathds{1}_{\left\{T_n(g_{n,K_i}(x_1,\dots,x_n))\leq
    t\right\}}\\
    &=\sum_{k=1}^{M_n}\dfrac{\sum_{i=1}^BY_i^k}{B}\mathds{1}_{\left\{T_n(g_{n,k}(x_1,\dots,x_n))\leq
    t\right\}}.
  \end{align*}
  By the strong law of large numbers this implies that $\P$-a.s. it
  holds that
  \begin{equation*}
    \lim_{B\rightarrow\infty}\hat{\mathcal{R}}_{T_n}^B(x_1,\dots,x_n)(t)
    =\dfrac{1}{M_n}\sum_{k=1}^{M_n}\mathds{1}_{\left\{T_n(g_{n,k}(x_1,\dots,x_n))\leq
    t\right\}}
    =\hat{R}_{T_n}(x_1,\dots,x_n)(t),
  \end{equation*}
  which completes the proof of Proposition \ref{thm:monte_carlo_resampling_dist}.
\end{proof}

We are now ready to define Monte-Carlo approximated resampling
test. Instead of using the $(1-\alpha)$-quantile of the Monte-Carlo
approximated resampling distribution we use a slightly larger critical
value. Surprisingly, for resampling groups satisfying the invariance
condition in Theorem~\ref{thm:level_resampling_test}, this allows us
to achieve level $\alpha$ for any value of $B$. The trick is that the
slightly larger critical value accounts for the uncertainty due to the
Monte-Carlo approximation.

We define the test using the $p$-value as this leads to easier
calculations. The corresponding critical value can then be calculated
via the standard correspondence between $p$-value and hypothesis test.

\begin{definition}[Monte-Carlo approximated resampling test]
  \label{def:mcresamplingtest}
  Let $\alpha\in(0,1)$, let $\kernelspace$ be a separable metric
  space, let $T=(T_n)_{n\in\N}$ be a test statistic on $\kernelspace$,
  let $g$ be a resampling method, let $B\in\N$, let $(K_i)_{i\in\N}$
  be a sequence of independent uniformly distributed random variables on
  $\{1,\dots,M_n\}$ and let $(k_1,\dots,k_B)$ be a realization of
  $(K_1,\dots,K_B)$. For all $n\in\N$ define the function
  $\hat{p}_n:\kernelspace^n\rightarrow[\tfrac{1}{B+1},1]$ satisfying
  \begin{equation*}
    \hat{p}_n(x_1,\dots,x_n)\coloneqq\frac{1+\abs[\big]{\{i\in\{1,\dots,B\}:\,
    T_n(g_{n,k_i}(x_1,\dots,x_n))\geq T_n(x_1,\dots,x_n)\}}}{1+B}.
  \end{equation*}
  Then the hypothesis test $\phi=(\phi_n)_{n\in\N}$ defined for all
  $n\in\N$ and for all $(x_1,\dots,x_n)\in\kernelspace^n$ by 
  \begin{equation*}
    \phi_n(x_1,\dots,x_n)\coloneqq\mathds{1}_{\{\hat{p}_n(x_1,\dots,x_n)\leq\alpha\}},
  \end{equation*}
  is called $\alpha$-Monte-Carlo approximated resampling test.
\end{definition}

The function $\hat{p}_n$ is called $p$-value of the test $\phi_n$. The
following proposition shows that the Monte-Carlo approximated
resampling test achieves level $\alpha$ given the appropriate invariance
assumptions on $g$.

\begin{proposition}[Monte-Carlo approximated resampling test has valid
  level]
  \label{thm:exactlevelmontecarlo}
  Let $\alpha\in(0,1)$, let $\kernelspace$ be a separable metric
  space, let $\HO,\HA\subseteq\pmspace{\kernelspace}$ be a null and
  alternative hypothesis respectively, let $T=(T_n)_{n\in\N}$ be a
  test statistic on $\kernelspace$, let $B\in\N$ and let $g$ be a resampling
  group satisfying under $\HO$ that for all $n\in\N$ and for all
  $k\in\{1,\dots,M_n\}$ it holds that
  \begin{equation}
    \label{eq:invariance_1}
    g_{n,k}(X_1,\dots,X_n)\text{ is equal in distribution to }(X_1,\dots,X_n),
  \end{equation}
  and for all $k\neq l\in\{1,\dots,M_n\}$ it holds that
  \begin{equation}
    \label{eq:uniqueness_2}
    \P\left(T_n(g_{n,k}(X_1,\dots,X_n))=T_n(g_{n,l}(X_1,\dots,X_n))\right)=0.
  \end{equation}
  Then, the corresponding $\alpha$-Monte-Carlo approximated resampling
  test $\phi=(\phi_n)_{n\in\N}$ has valid level $\alpha$.
\end{proposition}

\begin{proof}
  Begin by defining the function
  $f:\{1,\dots,M_n\}^B\times\kernelspace^n\rightarrow\{0,\dots,B\}$
  satisfying for all $(k_1,\dots,k_B)\in\{1,\dots,M_n\}^B$ and for all
  $(x_1,\dots,x_n)\in\kernelspace^n$ that
  \begin{equation*}
    f(k_1,\dots,k_B)(x_1,\dots,x_n)\coloneqq\abs[\big]{\{i\in\{1,\dots,B\}:\,
    T_n(g_{n,k_i}(x_1,\dots,x_n))\geq T_n(x_1,\dots,x_n)\}},
  \end{equation*}
  and the function $f_{tot}:\kernelspace^n\rightarrow\{1,\dots,M_n\}$
  satisfying for all $(x_1,\dots,x_n)\in\kernelspace^n$ that
  \begin{equation*}
    f_{tot}(x_1,\dots,x_n)\coloneqq\abs[\big]{\{i\in\{1,\dots,M_n\}:\,
    T_n(g_{n,i}(x_1,\dots,x_n))\geq T_n(x_1,\dots,x_n)\}}.
  \end{equation*}
  Then, by the invariance assumption \eqref{eq:invariance_1} it holds
  under $\HO$ for all $k,l\in\{1,\dots,M_n\}$ that
  \begin{equation}
    \label{eq:invariance_3}
    \P\left(f_{tot}(X_1,\dots,X_n)=l\right)=\P\left(f_{tot}(g_{n,k}(X_1,\dots,X_n))=l\right).
  \end{equation}
  Moreover, since $g$ is a group it holds $\P$-a.s. that
  \begin{align*}
    f_{tot}(g_{n,k}(X_1,\dots,X_n))
    &=\sum_{i=1}^{M_n}\mathds{1}_{\{T_n(g_{n,i}(g_{n,k}(X_1,\dots,X_n)))\geq
      T_n(g_{n,k}(X_1,\dots,X_n))\}}\\
    &=\sum_{i=1}^{M_n}\mathds{1}_{\{T_n(g_{n,i}(X_1,\dots,X_n))\geq
      T_n(g_{n,k}(X_1,\dots,X_n))\}},
  \end{align*}
  which implies together with \eqref{eq:uniqueness_2} it holds $\P$-a.s. that
  \begin{equation}
    \label{eq:invariance_4}
    \sum_{k=1}^{M_n}\mathds{1}_{\{f_{tot}(g_{n,k}(X_1,\dots,X_n))=l\}}=1.
  \end{equation}
  Combining \eqref{eq:invariance_3} and \eqref{eq:invariance_4} it
  holds under $\HO$ that
  \begin{align*}
    \P\left(f_{tot}(X_1,\dots,X_n)=l\right)
    &=\frac{1}{M_n}\sum_{k=1}^{M_n}\P\left(f_{tot}(g_{n,k}(X_1,\dots,X_n))=l\right)\\
    &=\frac{1}{M_n}\sum_{k=1}^{M_n}\E\left(\mathds{1}_{\{f_{tot}(g_{n,k}(X_1,\dots,X_n))=l\}}\right)\\
    &=\frac{1}{M_n}\E\left(\sum_{k=1}^{M_n}\mathds{1}_{\{f_{tot}(g_{n,k}(X_1,\dots,X_n))=l\}}\right)\\
    &=\frac{1}{M_n},
  \end{align*}
  which proves that under $\HO$ it holds that $f_{tot}(X_1,\dots,X_n)$
  is uniformly distributed on $\{1,\dots,M_n\}$. Furthermore, conditioned
  on $f_{tot}(X_1,\dots,X_n)=l$ it holds for all $i\in\{1,\dots,B\}$
  that
  \begin{equation*}
    \mathds{1}_{\left\{T_n(g_{n,K_i}(X_1,\dots,X_n))\geq T_n(X_1,\dots,X_n)\right\}}
  \end{equation*}
  is Bernoulli $\frac{l}{M_n}$ distributed which again conditioned on
  $f_{tot}(X_1,\dots,X_n)=l$ implies that
  \begin{equation*}
    f(K_1,\dots,K_B)(X_1,\dots,X_n)=\sum_{i=1}^B\mathds{1}_{\left\{T_n(g_{n,K_i}(X_1,\dots,X_n))\geq T_n(X_1,\dots,X_n)\right\}}
  \end{equation*}
  has binomial distribution with parameters $B$ and $\frac{l}{M_n}$.
  It therefore holds under $\HO$ that
  \begin{align*}
    &\P\left(\hat{p}_n(X_1,\dots,X_n)\leq\alpha\right)\\
    &\quad=\P\left(f(K_1,\dots,K_B)(X_1,\dots,X_n)\leq(B+1)\alpha-1\right)\\
    &\quad=\sum_{l=1}^{M_n}\P\left(f(K_1,\dots,K_B)(X_1,\dots,X_n)\leq(B+1)\alpha-1\,\rvert\,
      f_{tot}(X_1,\dots,X_n)=l\right)\\
    &\qquad\qquad\qquad\qquad\cdot\P\left(f_{tot}(X_1,\dots,X_n)=l\right)\\
    &\quad=\frac{1}{M_n}\sum_{l=1}^{M_n}\sum_{i=0}^{\lfloor(B+1)\alpha-1\rfloor}
      \binom{B}{i}\left(\frac{l}{M_n}\right)^i\left(1-\frac{l}{M_n}\right)^{B-i}\\
    &\quad\leq\int_0^1\sum_{i=0}^{\lfloor(B+1)\alpha-1\rfloor}
      \binom{B}{i}\left(x\right)^i\left(1-x\right)^{B-i}\measure{\lambda}{x}\\
    &\quad=\frac{\lfloor(B+1)\alpha-1\rfloor+1}{B+1}\\
    &\quad\leq\alpha,
  \end{align*}
  where we approximated the sum by an integral and solved the integral
  using integration by parts. This completes the proof of
  Proposition~\ref{thm:exactlevelmontecarlo}.
\end{proof}

The $p$-value is overestimated by the choice we made. In fact, as
described in \citet{phipson2010}, the level of the test would be
preserved even if we chose the $p$-value slightly smaller. This allows
to construct a permutation test which is not only valid in level but
actually achieves exact level.  The next proposition specifies the
critical value that leads to the Monte-Carlo approximated resampling
test.

\begin{proposition}[critical value of Monte-Carlo approximated
  resampling test]
  \label{thm:critvalmontecarlo}
  Let $\alpha\in(0,1)$, let $\kernelspace$ be a measurable space, let
  $T=(T_n)_{n\in\N}$ be a test statistic on $\kernelspace$, let $g$ be
  a resampling method, let $B\in\N$, let $(K_i)_{i\in\N}$
  be a sequence of uniformly distributed random variables on
  $\{1,\dots,M_n\}$ and let $(k_1,\dots,k_B)$ be a realization of
  $(K_1,\dots,K_B)$. For all $n\in\N$ define
  the function $c_n:\kernelspace^n\rightarrow\R$
  satisfying that $c_n(x_1,\dots,x_n)$ is the
  \begin{equation*}
    \lceil(B+1)(1-\alpha)\rceil+\sum_{i=1}^B\mathds{1}_{\{T_n(g_{n,k_i}(x_1,\dots,x_n))
        =T_n(x_1,\dots,x_n)\}}\text{-th largest value}
  \end{equation*}
  in the vector
  $(T_n(g_{n,k_1}(x_1,\dots,x_n)),\dots,T_n(g_{n,k_B}(x_1,\dots,x_n)))$
  if
  \begin{equation*}
    \lceil(B+1)(1-\alpha)\rceil+\sum_{i=1}^B\mathds{1}_{\{T_n(g_{n,k_i}(x_1,\dots,x_n))
      =T_n(x_1,\dots,x_n)\}}\leq B 
  \end{equation*}
  and $\infty$ otherwise. Then the hypothesis test
  $\phi=(\phi_n)_{n\in\N}$ defined for all $n\in\N$ and for all
  $(x_1,\dots,x_n)\in\kernelspace^n$ by
  \begin{equation*}
    \phi(x_1,\dots,x_n)\coloneqq\mathds{1}_{\{T_n(x_1,\dots,x_n)\geq c_n(x_1,\dots,x_n)\}},
  \end{equation*}
  is equal to the $\alpha$-Monte-Carlo approximated resampling test.
\end{proposition}

\begin{proof}
  The following calculation is straight forward:
  \begin{align*}
    &\mathds{1}_{\left\{\hat{p}_n(x_1,\dots,x_n)\leq\alpha\right\}}\\
    &\quad=\mathds{1}_{\left\{\frac{1}{B}\sum_{i=1}^B\mathds{1}_{\{T_n(g_{n,k_i}(x_1,\dots,x_n))\geq
        T_n(x_1,\dots,x_n)\}}\leq\frac{B+1}{B}\alpha-\frac{1}{B}\right\}}\\
    &\quad=\mathds{1}_{\left\{\frac{B+1}{B}(1-\alpha)\leq\frac{1}{B}\sum_{i=1}^B\mathds{1}_{\{T_n(g_{n,k_i}(x_1,\dots,x_n))<T_n(x_1,\dots,x_n)\}}\right\}}\\
    &\quad=\mathds{1}_{\{T_n(x_1,\dots,x_n)\geq c_n(x_1,\dots,x_n)\}}\\
    &\quad=\phi_n(x_1,\dots,x_n).
  \end{align*}
  This completes the prove of Proposition~\ref{thm:critvalmontecarlo}.
\end{proof}

The Monte-Carlo approximated resampling test is closely related to the
Monte-Carlo resampling distribution function. To see this observe that
for large $B$ it holds for all $(x_1,\dots,x_n)\in\kernelspace^n$ that
\begin{equation*}
  c_n(x_1,\dots,x_n)\approx (\hat{R}^B_{T_n}(x_1,\dots,x_n))^{-1}(1-\alpha).
\end{equation*}
As mentioned above $c_n$ approximates the $(1-\alpha)$-quantile of the
Monte-Carlo resampling distribution from above and gets closer as $B$ increases.

An example of how to implement this (conservative) Monte-Carlo approximation
procedure for the dHSIC permutation and bootstrap test is given in
the following algorithm.

\begin{algorithm}
  \caption{computing $p$-value and critical value for the permutation/bootstrap test}\label{alg:pval_critval}
  \begin{algorithmic}[1]
    \Procedure{$\texttt{MonteCarlo-pvalue}$}{$\vx_1,\dots\vx_n,B$}
    \State{initialize empty $B$-dimensional vector $\mathbf{T}$}
    \For{$k=1:B$}
      \State{initialize $d$-dimensional vectors $\tilde{\vx}_1,\dots,\tilde{\vx}_n$}
      \For{$j=1:d$}
        \State $\psi\gets\text{random element from }S_n\text{
          (permutation) or }\{1,\dots,n\}^n\text{ (bootstrap)}$
        \For{$i=1:n$}
          \State $\tilde{\vx}_i[j]\gets \vx_{\psi(i)}[j]$
        \EndFor
      \EndFor
      \State $\mathbf{T}[k]\gets
      \texttt{dHSIC}(\tilde{\vx}_1,\dots,\tilde{\vx}_n)$
    \EndFor
    \State $\operatorname{tmp}\gets
    \#\left\{k\in\{1,\dots,B\}\,\rvert\,\mathbf{T}[k]\geq\texttt{dHSIC}(\vx_1,\dots\vx_n)\right\}$
    \State $\operatorname{pval}\gets(\operatorname{tmp}+1)/(B+1)$
    \State \Return $\operatorname{pval}$
    \EndProcedure \vspace{0.1cm}
    \Procedure{$\texttt{MonteCarlo-critval}$}{$\vx_1,\dots\vx_n,B,\alpha$}
    \State{Perform lines 2. -- 9. from function \texttt{MonteCarlo-pvalue}}
    \State $\operatorname{tmp}\gets\#\left\{k\in\{1,\dots,B\}\,\rvert\,\mathbf{T}[k]=\texttt{dHSIC}(\vx_1,\dots\vx_n)\right\}$
    \State
    $\operatorname{ind}\gets\lceil(B+1)\cdot(1-\alpha)\rceil+\operatorname{tmp}$
    \If{$\operatorname{ind}\leq B$}
      \State $\mathbf{S}\gets\texttt{sort}(\mathbf{T})\text{ (in ascending
      order)}$
      \State $\operatorname{critval}\gets\mathbf{S}[\operatorname{ind}]$
    \Else
      \State $\operatorname{critval}\gets\infty$
    \EndIf
    \State \Return $\operatorname{critval}$
    \EndProcedure
  \end{algorithmic}
\end{algorithm}

\section{Further results on V-statistics}\label{appendix:vstat}

In this section we extend some of the commonly known results about
U-statistics to V-statistics. An overview of the theory of
U-statistics can be found in \citet{serfling}. Most similar
generalizations in literature only apply to V-statistics of degree
$2$; we now extend the results to V-statistics of arbitrary order.
For our convenience and in order to not be repetitive we introduce the
following setting, which we will only use within
Appendix~\ref{appendix:vstat}.
\begin{setting}[U- and V-statistics]
  \label{setting:uvstat}
  Let $n\in\N$, $q\in\{1,\dots,n\}$, $\uspace$ a metric space,
  $(\Omega,\mathcal{F},\P)$ a probability space, $X:\Omega \rightarrow
  \uspace$ a random variable with law $\Plaw{}$ and $(X_i)_{i\in\N}$ a
  sequence of iid copies of $X$, i.e., $(X_i)_{i\in\N}\iid\P^{X}$.
\end{setting}
The sequence $(X_i)_{i\in\N}$ should be seen as the generating process
of observations.

For completeness, we now introduce the U-statistic. Define the set of all 
combinations of $q$ elements on $\{1,\dots,n\}$ by
\begin{equation*}
  \combset{q}{n}\coloneqq\left\{(i_1,\dots,i_q)\in\{1,\dots,n\}:i_1<\cdots<i_q\right\}.
\end{equation*}
The U-statistic
\begin{equation}
  \label{eq:ustat}
  \Ustat{n}{g}\coloneqq\dbinom{n}{q}^{-1}\sum_{\combset{q}{n}}g(X_{i_1},\dots,X_{i_q}),
\end{equation}
estimates the statistical functional
\begin{equation*}
  \theta_g\coloneqq\theta_g\left(\Plaw{}\right)\coloneqq\E\left(g(X_1,\dots,X_q)\right),
\end{equation*}
see~\eqref{eq:vstat} for the corresponding V-statistic. An alternative
representation which also appears in literature is given by setting
\begin{equation*}
  \Ustatb{n}{g}\coloneqq\frac{1}{\permsize{n}{q}}\sum_{\permset{q}{n}}g(X_{i_1},\dots,X_{i_q}),
\end{equation*}
with $\permset{q}{n}\coloneqq\{(i_1,\dots,i_q)\in\{1,\dots,n\}:i_1,\dots,i_q\text{ distinct}\}$
and $\permsize{n}{q}\coloneqq\frac{n!}{(n-q)!}=\abs{\permset{q}{n}}$. It
is straightforward to see that $\Ustatb{n}{g}=\Ustat{n}{g}$.

\subsection{Connection between U-statistics and V-statistics}\label{subsec:comparison_uvstat}

To derive the asymptotic distribution of V-statistics we show that
V-statistics are in an appropriate sense good approximations of
U-statistics. In order to show results of this type we require some
kind of boundedness condition on the core function. The next definition
introduces such a condition.
\begin{definition}[total boundedness condition]
  \label{def:boundednesscond}
  Let $r\in\N$, assume Setting \ref{setting:uvstat} and let
  $g\in\lpspace{r}{\Plawpower{q}}{\abs{\cdot}_{\R}}$ be a core
  function. Then we say that $g$ satisfies the total boundedness
  condition of order $r$ if for all $(i_1,\dots,i_q)\in\mapset{q}{q}$
  it holds that
  \begin{equation*}
    \E\left[\abs{g(X_{i_1},\dots,X_{i_q})}^r\right]<\infty.
  \end{equation*}
\end{definition}
In particular, this condition is fulfilled if the core function $g$ is
a bounded function.

The following result is due to \citet[Lemma, Section 5.7.3]{serfling}.
\begin{lemma}[connection between U- and V-statistics]
  \label{thm:comparison_uvstat}
  Let $r\in\N$, assume Setting \ref{setting:uvstat} and let
  $g\in\lpspace{r}{\Plawpower{q}}{\abs{\cdot}_{\R}}$ be a core
  function satisfying the total boundedness condition of order $r$.
  Then it holds that
  \begin{equation*}
    \E\left[\abs{\Ustat{n}{g}-\Vstat{n}{g}}^r\right]=\landauO{n^{\sm r}}
  \end{equation*}
  as $n\rightarrow\infty$.
\end{lemma}

In order to prove some of the asymptotic statements of V-statistics we
require a stronger way of comparing V-statistics with U-statistics
than that given in Lemma~\ref{thm:comparison_uvstat}. For example,
when computing the asymptotic variance of a V-statistic up to an order
of $n^{-2}$ by comparison with the variance of a U-statistic, we need
to estimate the second moment of the difference to an order of
$n^{-(2+\epsilon)}$. Hence, the result in
Lemma~\ref{thm:comparison_uvstat} is not sufficient. The following
technical lemma gives a decomposition of a V-statistic into the
corresponding U-statistic and some remainder terms. We are not aware
of a similar result in literature.

\begin{lemma}[decomposition of a V-statistic]
  \label{thm:decompositionofmVstat}
  Assume Setting \ref{setting:uvstat} and let
  $g\in\lpspace{1}{\Plawpower{q}}{\abs{\cdot}_{\R}}$ be a core
  function. For all $k\in\{1,\dots,q-1\}$, $l\in\{k+1,\dots,q\}$ let
  $\pi^{kl}:\{1,\dots,q\}\rightarrow\{1,\dots,q-1\}$ be the unique
  surjective functions with the property that
  $\pi^{kl}(k)=\pi^{kl}(l)=1$ and for all
  $i,j\in\{1,\dots,q\}\setminus\{k,l\}$ with $i<j$ it holds that
  $\pi^{kl}(i)<\pi^{kl}(j)$. Define for all $x_1,\dots,x_{q-1}\in\kernelspace$ the function
  \begin{equation*}
    w(x_1,\dots,x_{q-1})\coloneqq\sum_{k=1}^{q\sm 1}\sum_{l=k+1}^{q}g(x_{{\pi^{kl}(1)}},\dots,x_{{\pi^{kl}(q)}}).
  \end{equation*}
  and set
  $B\coloneqq\{(i_1,\dots,i_q)\in\mapset{q}{n}\mid\text{ at most $q-2$ distinct values}\}$.
  Then it holds that
  \begin{equation*}
    \begin{split}
      n\Vstat{n}{g}&=\left(1+\landauO{n^{-1}}\right)\Ustat{n}{w}\\
      &\quad
      +\left(1+\landauO{n^{-1}}\right)\frac{(n-q+1)!}{n!}\sum_{(i_1,\ldots,i_q)\in B}g(X_{i_1},\dots,X_{i_q})\\
      &\quad -\left(\binom{q}{2}+\landauO{n^{-1}}\right)\Ustat{n}{g}\\
      &\quad +n\Ustat{n}{g}
    \end{split}
  \end{equation*}
  and $\abs{B}=\landauO{n^{q-2}}$ as $n\rightarrow\infty$.
\end{lemma}

\begin{proof}
  We begin by introducing
  \begin{equation*}
    S_n=\dfrac{1}{\permsize{n}{q\sm
        1}}\left(\sum_{\mapset{q}{n}}g(X_{i_1},\dots,X_{i_q})-\sum_{\permset{q}{n}}g(X_{i_1},\dots,X_{i_q})\right)
  \end{equation*}
  and
  \begin{equation*}
    A=\{(i_1,\dots,i_q)\in\mapset{q}{n}\mid\text{ at most $q-1$ distinct values}\}.
  \end{equation*}
  Then, observe that $A=\mapset{q}{n}\setminus\permset{q}{n}$ and
  \begin{align*}
    A\setminus
    B&=\{(i_1,\dots,i_q)\in\mapset{q}{n}\mid\text{ exactly $q-1$ distinct values}\}\\
    &=\bigg\{\left(i_{\pi^{kl}(1)},\dots,i_{\pi^{kl}(q)}\right)\Big\rvert\, (i_1,\dots,i_{q\sm 1})\in\permset{q\sm 1}{n},\\
    &\qquad\qquad\qquad\qquad\qquad\qquad k\in\{1,\dots,q-1\},\text{ } l\in\{k+1,\dots,q\} \bigg\}.
  \end{align*}
  Therefore, it holds that $\abs{A}=n^q-\permsize{n}{q}$ and
  $\abs{A\setminus B}=\frac{q(q\sm 1)}{2}\permsize{n}{q\sm
    1}$. Using this we get
  \begin{align*}
    \abs{B}&=\abs{A}-\abs{A\setminus B}\\
    &=n^q-\permsize{n}{q}-\dfrac{q(q-1)}{2}\permsize{n}{q\sm 1}\\
    &=n^q-n(n-1)\cdots(n-(q-1))-\dfrac{q(q-1)}{2}n(n-1)\cdots(n-(q-2))\\
    &=n^q-n^q+\dfrac{q(q-1)}{2}n^{q\sm 1}+\landauO{n^{q\sm
        2}}-\dfrac{q(q-1)}{2}n^{q\sm 1}+\landauO{n^{q\sm 2}}\\
    &=\landauO{n^{q\sm 2}}
  \end{align*}
  as $n\rightarrow\infty$. We can now make the following calculation
  \begin{align}
    S_n&=\dfrac{1}{\permsize{n}{q\sm
        1}}\left(\sum_{\mapset{q}{n}}g(X_{i_1},\dots,X_{i_q})-\sum_{\permset{q}{n}}g(X_{i_1},\dots,X_{i_q})\right)\nonumber\\
    &=\dfrac{1}{\permsize{n}{q\sm 1}}\sum_{A}g(X_{i_1},\dots,X_{i_q})\nonumber\\
    &=\dfrac{1}{\permsize{n}{q\sm 1}}\sum_{\permset{q\sm 1}{n}}w(X_{i_1},\dots,X_{i_{q\sm
        1}})+\dfrac{1}{\permsize{n}{q\sm 1}}\sum_{B}g(X_{i_1},\dots,X_{i_q})\nonumber\\
    &=\Ustatb{n}{w}+\dfrac{1}{\permsize{n}{q\sm 1}}\sum_{B}g(X_{i_1},\dots,X_{i_q}).\label{eq:Smsep}
  \end{align}
  Finally, we can decompose $n\Vstat{n}{g}$ as follows
  \begin{align}
    n\Vstat{n}{g}&=\dfrac{1}{n^{q\sm
        1}}\sum_{\mapset{q}{n}}g(X_{i_1},\dots,X_{i_q}) -
     \dfrac{1}{n^{q\sm 1}}\sum_{\permset{q}{n}}g(X_{i_1},\dots,X_{i_q})\nonumber\\
    &\quad + \dfrac{1}{n^{q\sm
        1}}\sum_{\permset{q}{n}}g(X_{i_1},\dots,X_{i_q})\nonumber\\
    &=\dfrac{\permsize{n}{q\sm 1}}{n^{q\sm
        1}}S_n+\dfrac{\permsize{n}{q}}{n^{q\sm 1}}
    \Ustatb{n}{g}\nonumber\\
    &=\left(1+\landauO{n^{\sm
          1}}\right)S_n+\left(n-\dfrac{q(q-1)}{2}+\landauO{n^{\sm
          1}}\right)\Ustat{n}{g}\nonumber\\
    &=\left(1+\landauO{n^{\sm
          1}}\right)S_n-\left(\dbinom{q}{2}+\landauO{n^{\sm
          1}}\right)\Ustat{n}{g}+n\Ustat{n}{g}.\label{eq:decompositionmVm}
  \end{align}
  Combining \eqref{eq:Smsep} and \eqref{eq:decompositionmVm} completes
  the proof of Lemma \ref{thm:decompositionofmVstat}.
\end{proof}

\subsection{Consistency of  V-statistics}\label{subsec:asymptotic_consistency_vstat}

The following theorem is the counterpart of \citet[Theorem A, Section 5.4]{serfling} for V-statistics. The proof is a
direct application of Lemma \ref{thm:comparison_uvstat} and \citet[Theorem A, Section 5.4]{serfling}.

\begin{theorem}[consistency of a V-statistic]
  \label{thm:asymptotic_consistency_vstat}
  Assume Setting \ref{setting:uvstat} and let
  $g\in\lpspace{1}{\Plawpower{q}}{\abs{\cdot}_{\R}}$ be a symmetric core
  function satisfying the total boundedness condition of order $1$. Then
  \begin{equation*}
    \Vstat{n}{g}\overset{\P}{\longrightarrow}\theta_g
  \end{equation*}
  as $n\rightarrow\infty$.
\end{theorem}

\begin{proof}
  By \citet[Theorem A, Section 5.4]{serfling} it holds that
  \begin{equation*}
    \Ustat{n}{g}\overset{\P}{\longrightarrow}\theta_g
  \end{equation*}
  as $n\rightarrow\infty$.
  Furthermore, by Lemma \ref{thm:comparison_uvstat} we have that
  \begin{equation*}
    \E\abs{\Ustat{n}{g}-\Vstat{n}{g}}=\landauO{n^{\sm 1}}
  \end{equation*}
  as $n\rightarrow\infty$. Since convergence in $L^1$ implies
  convergence in probability we obtain
  \begin{equation*}
    \Vstat{n}{g}\overset{\P}{\longrightarrow}\theta_g
  \end{equation*}
  as $n\rightarrow\infty$, which completes the proof of Theorem \ref{thm:asymptotic_consistency_vstat}.
\end{proof}

\subsection{Variance of V-statistics}\label{subsec:variance_vstat}

In the degenerate setting $\xi_1=0$,
Lemma~\ref{thm:decompositionofmVstat} allows us to show that the variance of
a V-statistic is equal to that of a U-statistic up to a certain order
of $n$. Its proof relies on Lemma \ref{thm:decompositionofmVstat}.

\begin{theorem}[asymptotic variance of a V-statistic]
  \label{thm:var_vstat}
  Assume Setting \ref{setting:uvstat} and let
  $g\in\lpspace{2}{\Plawpower{q}}{\abs{\cdot}_{\R}}$ be a bounded core
  function satisfying $\xi_1=0$. Then it holds that
  \begin{equation*}
    \Var\left(\Vstat{n}{g}\right)=\dbinom{n}{q}^{-1}\dbinom{q}{2}\dbinom{n-q}{q-2}\xi_2+\landauO{n^{-\frac{5}{2}}}
  \end{equation*}
  as $n\rightarrow\infty$.
\end{theorem}

\begin{proof}
  It holds that
  \begin{equation*}
    \Var\left(\Vstat{n}{g}\right)=\Var\left(\Vstat{n}{\tilde{g}}\right),
  \end{equation*}
  which implies that without loss of generality we can assume that $\theta_g=0$. By Lemma
  \ref{thm:decompositionofmVstat} we get that
  \begin{equation}
    \label{eq:varvstat}
    \begin{split}
      n\Vstat{n}{g}&=\left(1+\landauO{n^{-1}}\right)S_n\\
      &\quad -\left(\binom{q}{2}+\landauO{n^{-1}}-n\right)\Ustat{n}{g}
    \end{split}
  \end{equation}
  as $n\rightarrow\infty$, where
  $S_n=\Ustat{n}{w}+\tfrac{1}{\permsize{n}{q-1}}\sum_{B}g(X_{i_1},\dots,X_{i_q})$. Applying
  \citet[Lemma A, Section 5.2.1]{serfling} results in
  \begin{equation}
    \label{eq:varvstat1}
    \Var\left(\Ustat{n}{g}\right)=\dbinom{n}{q}^{-1}\dbinom{q}{2}\dbinom{n-q}{q-2}\xi_2+\landauO{n^{-3}}
  \end{equation}
  and
  \begin{equation}
    \label{eq:varvstat2}
    \Var\left(\Ustat{n}{w}\right)=\landauO{n^{-1}}.
  \end{equation}
  Moreover, using that $g$ is bounded it holds that
  \begin{align}
    &\Var\left(\frac{1}{\permsize{n}{q-1}}\sum_{B}g(X_{i_1},\dots,X_{i_q})\right)\nonumber\\
    &\quad\leq\frac{1}{\permsize{n}{q-1}^2}\E\left(\abs[\Big]{\sum_{B}g(X_{i_1},\dots,X_{i_q})}^2\right)\nonumber\\
    &\quad\leq\frac{1}{\permsize{n}{q-1}^2}\sum_{(i_1,\dots,i_q)\in
      B}\sum_{(j_1,\dots,j_q)\in
      B}\E\left(\abs[\Big]{g(X_{i_1},\dots,X_{j_q})g(X_{i_1},\dots,X_{j_q})}\right)\nonumber\\
    &\quad\leq\frac{C\abs{B}^2}{\permsize{n}{q-1}^2}=\landauO{n^{-2}}.\label{eq:varvstat3}
  \end{align}
  So combining \eqref{eq:varvstat2} and \eqref{eq:varvstat3} shows
  that
  \begin{equation}
    \label{eq:varvstat4}
    \Var\left(S_n\right)=\landauO{n^{-1}}
  \end{equation}
  and
  \begin{equation}
    \label{eq:varvstat5}
    \operatorname{Cov}\left(\Ustat{n}{g},S_n\right)
    \leq\left(\Var\left(\Ustat{n}{g}\right)\Var\left(S_n\right)\right)^{\frac{1}{2}}
    =\landauO{n^{-\frac{3}{2}}}.
  \end{equation}
  Finally, use \eqref{eq:varvstat}, \eqref{eq:varvstat1},
  \eqref{eq:varvstat4} and \eqref{eq:varvstat5} to get
  \begin{align*}
    \Var\left(n\Vstat{n}{g}\right)
    &=\left(1+\landauO{n^{-1}}\right)^2\Var\left(S_n\right)\\
    &\quad +\left(\tbinom{q}{2}+\landauO{n^{-1}}-n\right)^2\Var\left(\Ustat{n}{g}\right)\\
    &\quad
    -2\left(1+\landauO{n^{-1}}\right)\left(n+\tbinom{q}{2}+\landauO{n^{-1}}\right)\operatorname{Cov}\left(\Ustat{n}{g},S_n\right)\\
    &=\landauO{1}\Var\left(S_n\right)+\left(n^2+\landauO{n}\right)\Var\left(\Ustat{n}{g}\right)+\landauO{n}\operatorname{Cov}\left(\Ustat{n}{g},S_n\right)\\
    &=n^2\dbinom{n}{q}^{-1}\dbinom{q}{2}\dbinom{n-q}{q-2}\xi_2+\landauO{n^{-\frac{1}{2}}}.
  \end{align*}
  Dividing by $n^2$ completes the proof of Theorem \ref{thm:var_vstat}.
\end{proof}

It is possible to get a similar result in the non-degenerate case
$\xi_1>0$. In that case similar reasoning as in the proof above can be
used to get a formula for the variance of the V-statistic.

\subsection{Bias of V-statistics}\label{subsec:bias_vstat}

As a further consequence of Lemma \ref{thm:decompositionofmVstat} the
bias of a V-statistic can be explicitly expressed up to order $n^{-2}$.

\begin{theorem}[bias of a V-statistic]
  \label{thm:biasVstat}
  Assume Setting \ref{setting:uvstat} and let
  $g\in\lpspace{2}{\Plawpower{q}}{\abs{\cdot}_{\R}}$ be a core
  function satisfying the total boundedness condition of order
  $2$. Then it holds that
  \begin{equation*}
    \E\left(\Vstat{n}{g}-\theta_g\right)=\dfrac{1}{n}\dbinom{q}{2}\E\left(\tilde{g}_2(X_1,X_1)\right)+\landauO{n^{-2}}
  \end{equation*}
  as $n\rightarrow\infty$.
\end{theorem}

\begin{proof}
  We use Lemma \ref{thm:decompositionofmVstat} to get that
  \begin{equation}
    \label{eq:biasvstat1}
    \begin{split}
      n\Vstat{n}{\tilde{g}}&=\left(1+\landauO{n^{-1}}\right)\Ustat{n}{w}\\
      &\quad +\left(1+\landauO{n^{-1}}\right)\tfrac{1}{\permsize{n}{q-1}}\sum_{B}\tilde{g}(X_{i_1},\dots,X_{i_q})\\
      &\quad -\left(\binom{q}{2}-n+\landauO{n^{-1}}\right)\Ustat{n}{\tilde{g}}.
    \end{split}
  \end{equation}
  Moreover, using the total boundedness condition
  of $g$ we can get a constant $C>0$ such that
  \begin{align}
    \E\abs[\Big]{\dfrac{1}{\permsize{n}{q\sm 1}}\sum_{B}\tilde{g}(X_{i_1},\dots,X_{i_q})}&\leq
     \dfrac{1}{\permsize{n}{q\sm
        1}}\sum_{B}\E\abs[\big]{\tilde{g}(X_{i_1},\dots,X_{i_q})}\nonumber\\
    &\leq C\dfrac{\abs{B}}{\permsize{n}{q\sm 1}}\nonumber\\
    &=\landauO{n^{\sm 1}}\label{eq:biasvstat2}
  \end{align}
  as $n\rightarrow\infty$. Hence, using
  \eqref{eq:biasvstat1},\eqref{eq:biasvstat2} and the unbiasedness of
  U-statistics results in
  \begin{equation}
    \label{eq:biasvstat3}
    \E\left(n\left(\Vstat{n}{g}-\theta_g\right)\right)=\E\left(n\Vstat{n}{\tilde{g}}\right)=\theta_w+\landauO{n^{-1}}.
  \end{equation}
  We can compute $\theta_w$ by using the symmetry of $\tilde{g}$ to
  get
  \begin{equation}
    \label{eq:biasvstat4}
    \theta_w=\E\left(w(X_1,\dots,X_{q-1})\right)=\dbinom{q}{2}\E\left(\tilde{g}_2(X_1,X_1)\right).
  \end{equation}
  Finally, combining \eqref{eq:biasvstat3} and \eqref{eq:biasvstat4}
  and dividing by $m$ concludes the proof of Theorem \ref{thm:biasVstat}.
\end{proof}

\subsection{Asymptotic distribution of V-statistics}\label{subsec:asymptotic_distribution_vstat}

A V-statistic is called degenerate if $\xi_1=\Var(g_1(X_1))=0$ and non-degenerate if
$\xi_1>0$. In this section we analyze the asymptotic distribution of
\begin{itemize}
\item $\sqrt{n}\Vstat{n}{g}$ for the non-degenerate case ($\xi_1>0$) and
\item $n\Vstat{n}{g}$ for a special degenerate case ($\xi_1=0$, $\xi_2>0$).
\end{itemize}
In this section we derive the asymptotic distributions for
V-statistics based on the corresponding results for U-statistics.

\subsubsection{Non-degenerate case}

The following theorem is the counterpart of \citet[Theorem A, Section
5.5.1]{serfling} for V-statistics. The proof is a straightforward
application of both Lemma \ref{thm:comparison_uvstat} and
\citet[Theorem A, Section 5.5.1]{serfling}.

\begin{theorem}[asymptotic distribution of a V-statistic (non-degenerate)]
  \label{thm:asymptoticdist_vstat1}
  Assume Setting \ref{setting:uvstat}, let
  $g\in\lpspace{2}{\Plawpower{q}}{\abs{\cdot}_{\R}}$ be a core
  function satisfying the total boundedness condition of order
  $2$ and assume $\xi_1>0$. Then it holds that
  \begin{equation*}
    \sqrt{n}\left(\Vstat{n}{g}-\theta_g\right)
    \overset{d}{\longrightarrow} \normaldist{0}{q^2\xi_1}
  \end{equation*}
  as $n\rightarrow\infty$.
\end{theorem}

\begin{proof}
  Since convergence in $L^2$ implies convergence in probability
  Lemma \ref{thm:comparison_uvstat} in particular shows that
  \begin{equation*}
    \sqrt{n}\left(\Vstat{n}{g}-\Ustat{n}{g}\right)\overset{\P}{\longrightarrow}0
  \end{equation*}
  as $n\rightarrow\infty$. Combining this with \citet[Theorem A,
  Section 5.5.1]{serfling} and Slutsky's theorem we get
  \begin{equation*}
    \sqrt{n}\left(\Vstat{n}{g}-\theta_g\right)=\sqrt{n}\left(\Ustat{n}{g}-\theta_g\right)
    + \sqrt{n}\left(\Vstat{n}{g}-\Ustat{n}{g}\right) \overset{d}{\longrightarrow} \normaldist{0}{q^2\xi_1}
  \end{equation*}
  as $n\rightarrow\infty$ which completes the proof of Theorem \ref{thm:asymptoticdist_vstat1}.
\end{proof}

\subsubsection{Degenerate case}

In the degenerate case the asymptotic distribution depends on the
eigenvalues of a particular integral operator. For our convenience and
in order to not be repetitive we introduce the following additional setting,
which we will only use within Appendix~\ref{appendix:vstat}.
\begin{setting}[degenerate asymptotic]
  \label{setting:uvstatasymptotic}
  Let $g\in\lpspace{2}{\Plawpower{q}}{\abs{\cdot}_{\R}}$ be a core
  function, let $(Z_j)_{j\in\N}$ be a sequence of independent standard
  normal random variables on $\R$, let $T_{\tilde{g}_2}\in
  L\left(\Lpspace{2}{\Plaw{}}{\abs{\cdot}_{\R}}\right)$ with the property that
  for every $f\in\Lpspace{2}{\Plaw{}}{\abs{\cdot}_{\R}}$ and for every $x\in\uspace$ it holds that
  \begin{equation*}
    \left(T_{\tilde{g}_2}(f)\right)(x)=\int_{\uspace}\tilde{g}_2(x,y)f(y)\measure{\Plaw{}}{y}
  \end{equation*}
  and let $(\lambda_j)_{j\in\N}$ be the eigenvalues of
  $T_{\tilde{g}_2}$.
\end{setting}
Theorem \ref{thm:asymptoticdist_vstat2} is the counterpart of
\citet[Theorem, Section 5.5.2]{serfling} for V-statistics. Similar
statements appear in literature \citep[e.g.][Theorem
2]{gretton08hsic}. However, we are not aware of a complete proof of
the statement. The proof requires the following intermediate result.

\begin{lemma}[eigenvalue representation of the bias]
  \label{thm:lemma_eigensum}
  Assume Setting \ref{setting:uvstat} and Setting \ref{setting:uvstatasymptotic}, let
  $g\in\lpspace{2}{\Plawpower{q}}{\abs{\cdot}_{\R}}$ be a core
  function satisfying the total boundedness condition of order $2$, assume
  $0=\xi_1<\xi_2$ and assume $\tilde{g}_2$ is positive definite and
  continuous. Then it holds that
  \begin{equation*}
    \E\left(g_2(X_1,X_1)\right)=\sum_{i=1}^{\infty}\lambda_i+\theta_g
  \end{equation*}
\end{lemma}

\begin{proof}
  Observe that $\tilde{g}_2$ is a continuous positive definite
  kernel. We can therefore apply Mercer's theorem \citep[see][Theorem
  1.1]{ferreira} to
  get that for all $x,y\in\operatorname{supp}(\Plaw{})$ it holds that
  \begin{equation*}
    \tilde{g}_2(x,y)=\sum_{i=1}^{\infty}\lambda_i\phi_i(x)\phi_i(y)
  \end{equation*}
  converges uniformly. If we now take expectation and use that we can
  exchange the sum and expectation due the uniform convergence we get
  \begin{align*}
    \E\left(\tilde{g}_2(X_1,X_1)\right)&=\E\left(\sum_{i=1}^{\infty}\lambda_i\phi_i(X_1)\phi_i(X_1)\right)\\
    &=\sum_{i=1}^{\infty}\lambda_i\E\left(\abs{\phi_i(X_1)}^2\right)\\
    &=\sum_{i=1}^{\infty}\lambda_i,
  \end{align*}
  where in the last step we used that $(\phi_i)_{i\in\N}$ forms an
  orthonormal basis of $\Lpspace{2}{\Plaw{}}{\abs{\cdot}_{\R}}$. The result
  follows by noting that $g_2\equiv \tilde{g}_2+\theta_g$, which
  completes the proof of Lemma \ref{thm:lemma_eigensum}.
\end{proof}

We are now ready to state and prove the final result of this section.

\begin{theorem}[Asymptotic distribution of a V-statistic (degenerate)]
  \label{thm:asymptoticdist_vstat2}
  Assume Setting \ref{setting:uvstat} and Setting \ref{setting:uvstatasymptotic}, let
  $g\in\lpspace{2}{\Plawpower{q}}{\abs{\cdot}_{\R}}$ be a core
  function satisfying the total boundedness condition of order
  $2$, assume $0=\xi_1<\xi_2$ and assume $\tilde{g}_2$ is positive definite
  and continuous. Then it holds that
  \begin{equation*}
    n\left(\Vstat{n}{g}-\theta_g\right)\overset{d}{\longrightarrow}\dbinom{q}{2}\sum_{i=1}^{\infty}\lambda_iZ_i^2
  \end{equation*}
  as $n\rightarrow\infty$.
\end{theorem}

\begin{proof}
  The idea of the proof is to use Lemma
  \ref{thm:decompositionofmVstat} to get the decomposition
  \begin{equation}
    \label{eq:asymptoticdecomposition1}
    \begin{split}
      n\Vstat{n}{\tilde{g}}&=\left(1+\landauO{n^{-1}}\right)S_n\\
      &\quad
      -\left(\binom{q}{2}+\landauO{n^{-1}}\right)\Ustat{n}{\tilde{g}}\\
      &\quad +n\Ustat{n}{\tilde{g}}
    \end{split}
  \end{equation}
  as $n\rightarrow\infty$, where
  $S_n=\Ustat{n}{w}+\tfrac{1}{\permsize{n}{q-1}}\sum_{B}\tilde{g}(X_{i_1},\dots,X_{i_q})$
  and $w$ is defined as in Lemma \ref{thm:decompositionofmVstat}.  We
  then calculate the asymptotic behavior of $S_n$ and use
  \citet[Theorem, Section 5.5.2]{serfling} to conclude.
  
  Begin by analyzing the asymptotic behavior of $S_n$. To
  this end, note that by symmetry of the core function $g$ it holds
  that
  \begin{equation*}
    \theta_w=\E(w(X_1,\dots,X_{q\sm 1}))=\binom{q}{2}\E(\tilde{g}_2(X_1,X_1)).
  \end{equation*}
  and together with Lemma \ref{thm:lemma_eigensum} it holds that
  \begin{equation*}
    \theta_w=\E(w(X_1,\dots,X_{q\sm 1}))=\binom{q}{2}\sum_{i=1}^{\infty}\lambda_i.
  \end{equation*}
  Combining this with \citet[Theorem A, Section 5.4]{serfling} it follows
  that
  \begin{equation}
    \label{eq:Smterm1}
    \Ustat{n}{w}\overset{\P}{\longrightarrow}\theta_w=\binom{q}{2}\sum_{i=1}^{\infty}\lambda_i
  \end{equation}
  as $n\rightarrow\infty$. Next, we use the total boundedness condition
  of $g$ to get a constant $C>0$ such that
  \begin{align*}
    \E\abs[\Big]{\dfrac{1}{\permsize{n}{q\sm 1}}\sum_{B}\tilde{g}(X_{i_1},\dots,X_{i_q})}&\leq
     \dfrac{1}{\permsize{n}{q\sm
        1}}\sum_{B}\E\abs[\big]{\tilde{g}(X_{i_1},\dots,X_{i_q})}\\
    &\leq C\dfrac{\abs{B}}{\permsize{n}{q\sm 1}}\\
    &=\landauO{n^{\sm 1}}
  \end{align*}
  as $n\rightarrow\infty$. Using that $L^1$ convergence implies
  convergence in probability we get that
  \begin{equation}
    \label{eq:Smterm2}
    \dfrac{1}{\permsize{n}{q\sm 1}}\sum_{A_2}\tilde{g}(X_{i_1},\dots,X_{i_q})\overset{\P}{\longrightarrow}0
  \end{equation}
  as $n\rightarrow\infty$. Finally, combining \eqref{eq:Smterm1} and
  \eqref{eq:Smterm2} this results in
  \begin{equation}
    \label{eq:Sm}
    S_n\overset{\P}{\longrightarrow}\binom{q}{2}\sum_{i=1}^{\infty}\lambda_i
  \end{equation}
  as $n\rightarrow\infty$. Now, by the properties of convergence in
  probability, \eqref{eq:Sm} and \citet[Theorem A, Section
  5.4]{serfling} we have
  \begin{equation}
    \label{eq:asymptoticdecomposition2}
    \left(1+\landauO{n^{\sm 1}}\right)S_n\overset{\P}{\longrightarrow}
    \dbinom{q}{2}\sum_{i=1}^{\infty}\lambda_i
  \end{equation}
  and
  \begin{equation}
    \label{eq:asymptoticdecomposition3}
    \left(\dbinom{q}{2}+\landauO{n^{\sm 1}}\right)\Ustat{n}{\tilde{g}}
    \overset{\P}{\longrightarrow} 0
  \end{equation}
  as $n\rightarrow\infty$. Hence, \eqref{eq:asymptoticdecomposition1},
  \eqref{eq:asymptoticdecomposition2} and
  \eqref{eq:asymptoticdecomposition3} together with Slutsky's theorem and
  \citet[Theorem, Section 5.5.2]{serfling} shows that
  \begin{equation*}
    n\left(\Vstat{n}{g}-\theta_g\right)=n\Vstat{n}{\tilde{g}}\overset{d}{\longrightarrow}\dbinom{q}{2}\sum_{i=1}^{\infty}\lambda_iZ_i^2
  \end{equation*}
  as $n\rightarrow\infty$, which completes the proof of Theorem \ref{thm:asymptoticdist_vstat2}.
\end{proof}

\subsection{Resampling results for U-statistics and V-statistics}\label{subsec:resampling_uvstat}

In this section we want to consider what happens to the asymptotic
behavior of $n\Ustat{n}{g}$ and $n\Vstat{n}{g}$ if instead of the
original data sequence $(X_i)_{i\in\N}$ we consider a sequence of
resampled data. The differences are quite subtle, therefore one needs
to be very precise about what resampling means. Throughout this
section we use the following setting.

\begin{setting}[resampling]
  \label{setting:resampling}
  Let $\uspace$ be a separable metric space, let
  $(\Omega,\mathcal{F},\P)$ be a probability space, let $X:\Omega
  \rightarrow \uspace$ be a random variable and let $(X_i)_{i\in\N}$
  be a sequence of iid copies of $X$. For all $n\in\N$, let
  $(\Omega_n,\mathcal{F}_n,\P_n)$ be probability spaces, let
  $X^*_n:\Omega_n\rightarrow\uspace$ be random variables satisfying
  that $X^*_n\overset{d}{\rightarrow}X$ as $n\rightarrow\infty$
  (i.e. $\lim_{n\rightarrow\infty}\E_n(f(X^*_n))=\E(f(X))$ for all
  bounded and continuous functions $f:\uspace\rightarrow\R$) and let
  $(X^*_{n,i})_{i\in\{1,\dots,n\}}$ be iid copies of $X^*_n$.
\end{setting}
The data $X^*_{n,1},\dots,X^*_{n,n}$ should be interpreted as a new
sample drawn from a distribution which converges to $\plaw$ as $n$
goes to infinity. Resampled data of this type often show up in
different types of bootstrapping or permutation techniques.
We are interested in finding properties of the resampled U-statistc
\begin{equation*}
  \Ustatr{n}{g}\coloneqq\binom{n}{q}^{\sm1}\sum_{\combset{q}{n}}g\left(X^*_{n,i_1},\dots,X^*_{n,i_q}\right)
\end{equation*}
and the resampled V-statistic
\begin{equation*}
  \Vstatr{n}{g}\coloneqq\frac{1}{m^q}\sum_{\mapset{q}{n}}g\left(X^*_{n,i_1},\dots,X^*_{n,i_q}\right).
\end{equation*}
The difference compared to the normal U-and V-statistic is that the
distribution of the sample $X^*_{n,1},\dots,X^*_{n,n}$ depends on
$m$. Therefore, the results of the previous sections only carry over
to the resampled U-and V-statistics if they are results for which $m$
is kept fixed. Results about the asymptotic behavior of the resampled
U-and V-statistics need to be proved separately. A further more
technical difficulty is that for different $m$ the random variables
$\Ustatr{n}{g}$ and $\Vstatr{n}{g}$ are no longer defined on the same
probability space. The following theorem gives us a way of dealing
with this issue and is a slightly modified version of Skorohod's
theorem \citep[see][Theorem 6.7]{billingsley}.

\begin{theorem}[Skorohod's theorem]
  \label{thm:skorohod}
  Assume Setting \ref{setting:resampling}. Then there exists a common
  probability space $(\tilde{\Omega},\tilde{\mathcal{F}},\tilde{\P})$
  and random variables $(\tilde{X}^*_{n,i})_{i\in\{1,\dots,n\}}$, $n\in\N$ and $(\tilde{X}_i)_{i\in\N}$ on
  this probability space satisfying
  \begin{enumerate}[(i)]
  \item for all $n\in\N$, for all $i\in\{1,\dots,n\}$: $\tilde{X}^*_{n,i}\sim\P^{X^*_n}$,
  \item for all $i\in\N$: $\tilde{X}_i\sim\plaw$ and,
  \item
    $\tilde{X}^*_{n,i}\overset{\tilde{\P}\text{-a.s.}}{\longrightarrow}\tilde{X}_i$
    as $n\rightarrow\infty$. 
  \end{enumerate}
\end{theorem}

In order to avoid ambiguity between the resampled and the original
sample we introduce the following notation
\begin{enumerate}[(i)]
\item for all $n\in\N$ and all $c\in\{1,\dots,n\}$
  define $$g^n_c(x_1,\dots,x_c)\coloneqq\E(g(x_1,\dots,x_c,X^*_{n,c+1},\dots,X^*_{n,q})),$$
\item for all $n\in\N$ define $$\theta^n_g\coloneqq\E(g(X^*_{n,1},\dots,X^*_{n,q})),$$
\item for all $n\in\N$ and all $c\in\{1,\dots,n\}$ define $$\xi^n_c(g)\coloneqq\E((g^n_c(X^*_{n,1},\dots,X^*_{n,c})-\theta^n_g)^2).$$
\end{enumerate}

The following theorem shows that $\Ustatr{n}{g}$ is also consistent
with $\theta_g$ in the appropriate sense.

\begin{lemma}[consistency of a resampling U-statistic]
  \label{thm:consistencyresampledUstat}
  Assume Setting \ref{setting:resampling} and let
  $g\in\lpspace{1}{\Plawpower{q}}{\abs{\cdot}_{\R}}$ be a continuous,
  bounded core function. Then it holds that
  \begin{equation*}
    \Ustatr{n}{g}\overset{d}{\longrightarrow}\theta_g
  \end{equation*}
  as $n\rightarrow\infty$.
\end{lemma}

\begin{proof}
  Applying Theorem \ref{thm:skorohod} results in a probability space
  $(\tilde{\Omega},\tilde{\mathcal{F}},\tilde{\P})$ and random
  variables $(\tilde{X}^*_{n,i})_{i\in\{1,\dots,n\}}$, $n\in\N$ and
  $(\tilde{X}_i)_{i\in\N}$ with properties specified in Theorem
  \ref{thm:skorohod}. Next, introduce the resampled U-statistic
  \begin{equation*}
    \tilde{\mathcal{U}}_n(g)\coloneqq\dbinom{n}{q}^{\sm1}\sum_{\combset{q}{n}}g(\tilde{X}^*_{n,i_1},\dots,\tilde{X}^*_{n,i_q}),
  \end{equation*}
  which has the same distribution under $\tilde{\P}$ as $\Ustatr{n}{g}$ under
  $\P_n$ and the U-statistic
  \begin{equation*}
    \mathcal{U}_n(g)\coloneqq\dbinom{n}{q}^{\sm1}\sum_{\combset{q}{n}}g(\tilde{X}_{i_1},\dots,\tilde{X}_{i_q}),
  \end{equation*}
  which has the same distribution under $\tilde{\P}$ as $\Ustat{n}{g}$ under
  $\P$. It holds that
  \begin{align}
    \tilde{\mathcal{U}}_n(g)-\mathcal{U}_n(g)
    &=\dbinom{n}{q}^{\sm1}\sum_{\combset{q}{n}}
      \left(g(\tilde{X}^*_{n,i_1},\dots,\tilde{X}^*_{n,i_q})
      -g(\tilde{X}_{i_1},\dots,\tilde{X}_{i_q})\right)\nonumber\\
    &=\dbinom{n}{q}^{\sm1}\sum_{\combset{q}{n}}
      w((\tilde{X}_{i_1},\tilde{X}^*_{n,i_1}),\dots,(\tilde{X}_{i_q},\tilde{X}^*_{n,i_q})),\label{eq:thisisaustat_consistency}
  \end{align}
  where $w(\vx_1,\dots,\vx_q)\coloneqq
  g(x^2_1,\dots,x^2_q)-g(x^1_1,\dots,x^1_q)$ is a symmetric core
  function. If we define for all $c\in\{1,\dots,q\}$ the functions
    \begin{equation*}
    w^n_c(\vx_1,\dots,\vx_c)\coloneqq\E\left(g(x^2_1,\dots,x^2_c,\tilde{X}^*_{n,c+1},\dots,\tilde{X}^*_{n,q})
      -g(x^1_1,\dots,x^1_c,\tilde{X}_{c+1},\dots,\tilde{X}_{q})\right)
  \end{equation*}
  it holds by the boundedness of $g$ that there exists a constant
  $C\in\R$ such that
  \begin{equation}
    \label{eq:uniformbound_consistency}
    \sup_{n\in\N}\xi^n_c(w)<C.
  \end{equation}
  By \eqref{eq:thisisaustat_consistency}, it holds that for
  fixed $n$ we can apply \citet[Lemma A, Section 5.2.1]{serfling} and together with
  \eqref{eq:uniformbound_consistency} to get
  \begin{equation}
    \label{eq:variance_null}
    \Var\left(\tilde{\mathcal{U}}_n(g)-\mathcal{U}_n(g)\right)
    =\dbinom{n}{q}^{-1}\sum_{c=1}^n\dbinom{q}{c}\dbinom{n-q}{q-c}\xi^n_c(w)
    =\landauO{n^{-1}}.
  \end{equation}
  For $(i_1,\dots,i_q)\in\combset{q}{n}$ it holds
  by continuity of $g$ that
  \begin{equation*}
    g(\tilde{X}^*_{n,i_1},\dots,\tilde{X}^*_{n,i_q})
    \overset{\tilde{\P}\text{-a.s.}}{\longrightarrow}
    g(\tilde{X}_{i_1},\dots,\tilde{X}_{i_q})
  \end{equation*}
  as $n\rightarrow\infty$ and since $g$ is also bounded the dominated
  convergence theorem in particular implies that
  \begin{equation}
    \label{eq:l2convergence_consistency}
    \lim_{n\rightarrow\infty}\E\left(g(\tilde{X}^*_{n,i_1},\dots,\tilde{X}^*_{n,i_q})-g(\tilde{X}_{i_1},\dots,\tilde{X}_{i_q})\right)=0.
  \end{equation}
  Combining \eqref{eq:variance_null} and
  \eqref{eq:l2convergence_consistency} hence proves that
  \begin{align*}
    \lim_{n\rightarrow\infty}\E\left(\left(\tilde{\mathcal{U}}_n(g)-\mathcal{U}_n(g)\right)^2\right)
    &=\lim_{n\rightarrow\infty}\E\left(\tilde{\mathcal{U}}_n(g)-\mathcal{U}_n(g)\right)^2\\
    &=\lim_{n\rightarrow\infty}\E\left(g(\tilde{X}^*_{n,i_1},\dots,\tilde{X}^*_{n,i_q})
      -g(\tilde{X}_{i_1},\dots,\tilde{X}_{i_q})\right)^2\\
    &=0.
  \end{align*}
  Using that convergence in second moment implies convergence in
  probability we have therefore shown that
  \begin{equation*}
    \tilde{\mathcal{U}}_n(g)-\mathcal{U}_n(g)\overset{\P}{\longrightarrow}0
  \end{equation*}
  as $n\rightarrow\infty$. Together with consistency of U-statistics
  \citep[see][Theorem A, Section 5.4]{serfling} it follows that
  \begin{equation*}
    \tilde{\mathcal{U}}_n(g)-\theta_g=(\tilde{\mathcal{U}}_n(g)-\mathcal{U}_n(g))-(\theta_g-\mathcal{U}_n(g))\overset{\P}{\longrightarrow}0
  \end{equation*}
  as $n\rightarrow\infty$. This concludes the proof of Lemma \ref{thm:consistencyresampledUstat}.
\end{proof}

The following two theorems are extensions
of results due to \citet{bootstrap_consistency} that show that
U-and V-statistics based on resampled data keep their respective
asymptotic distributions. In \citet{bootstrap_consistency} only U-and
V-statistics of order 2 (i.e. $q=2$) are considered. We adopted the
proofs to work for arbitrary order.

\begin{theorem}[asymptotic distribution of degenerate resampling
  U-statistic]
  \label{thm:resamplingustat}
  Assume Setting \ref{setting:resampling} and Setting
  \ref{setting:uvstatasymptotic}, let
  $g\in\lpspace{2}{\Plawpower{q}}{\abs{\cdot}_{\R}}$ be a continuous,
  bounded core function. Moreover, assume
  \begin{enumerate}[(i)]
  \item for all $n\in\N$ that $g^n_1\equiv 0$,
  \item $g_1\equiv 0$ (which implies $\xi_1(g)=0$) and
  \item $\theta_g=0$.
  \end{enumerate}
  Then if $\xi_2(g)>0$ it holds that
  \begin{equation*}
    n\Ustatr{n}{g}\overset{d}{\longrightarrow}\dbinom{q}{2}\sum_{i=1}^{\infty}\lambda_i(Z_i^2-1)
  \end{equation*}
  as $n\rightarrow\infty$ and if $\xi_2(g)=0$ it holds that
    \begin{equation*}
    n\Ustatr{n}{g}\overset{d}{\longrightarrow}0
  \end{equation*}
  as $n\rightarrow\infty$.
\end{theorem}

\begin{proof}
  Applying Theorem \ref{thm:skorohod} results in a probability space
  $(\tilde{\Omega},\tilde{\mathcal{F}},\tilde{\P})$ and random
  variables $(\tilde{X}^*_{n,i})_{i\in\{1,\dots,n\}}$, $n\in\N$ and
  $(\tilde{X}_i)_{i\in\N}$ with properties specified in Theorem
  \ref{thm:skorohod}. For $(i_1,\dots,i_q)\in\combset{q}{n}$ it holds
  by continuity of $g$ that
  \begin{equation*}
    g(\tilde{X}^*_{n,i_1},\dots,\tilde{X}^*_{n,i_q})
    \overset{\tilde{\P}\text{-a.s.}}{\longrightarrow}
    g(\tilde{X}_{i_1},\dots,\tilde{X}_{i_q})
  \end{equation*}
  as $n\rightarrow\infty$ and since $g$ is also bounded the dominated
  convergence theorem in particular implies
  \begin{equation}
    \label{eq:l2convergence_asymptoticustat}
    \lim_{n\rightarrow\infty}\E\left(\left(g(\tilde{X}^*_{n,i_1},\dots,\tilde{X}^*_{n,i_q})-g(\tilde{X}_{i_1},\dots,\tilde{X}_{i_q})\right)^2\right)=0.
  \end{equation}
  Next, introduce the resampling U-statistic
  \begin{equation*}
    \tilde{\mathcal{U}}_n(g)\coloneqq\dbinom{n}{q}^{\sm1}\sum_{\combset{q}{n}}g(\tilde{X}^*_{n,i_1},\dots,\tilde{X}^*_{n,i_q}),
  \end{equation*}
  which has the same distribution under $\tilde{\P}$ as $\Ustatr{n}{g}$ under
  $\P_n$ and the U-statistic
  \begin{equation*}
    \mathcal{U}_n(g)\coloneqq\dbinom{n}{q}^{\sm1}\sum_{\combset{q}{n}}g(\tilde{X}_{i_1},\dots,\tilde{X}_{i_q}),
  \end{equation*}
  which has the same distribution under $\tilde{\P}$ as $\Ustat{n}{g}$ under
  $\P$. It holds that
  \begin{align}
    \tilde{\mathcal{U}}_n(g)-\mathcal{U}_n(g)
    &=\dbinom{n}{q}^{\sm1}\sum_{\combset{q}{n}}
      \left(g(\tilde{X}^*_{n,i_1},\dots,\tilde{X}^*_{n,i_q})
      -g(\tilde{X}_{i_1},\dots,\tilde{X}_{i_q})\right)\nonumber\\
    &=\dbinom{n}{q}^{\sm1}\sum_{\combset{q}{n}}
      w((\tilde{X}_{i_1},\tilde{X}^*_{n,i_1}),\dots,(\tilde{X}_{i_q},\tilde{X}^*_{n,i_q})),\label{eq:thisisaustat_asymptoticustat}
  \end{align}
  where $w(\vx_1,\dots,\vx_q)\coloneqq
  g(x^2_1,\dots,x^2_q)-g(x^1_1,\dots,x^1_q)$ is a symmetric core
  function. Define for all $c\in\{1,\dots,q\}$ the functions
  \begin{equation*}
    w^n_c(\vx_1,\dots,\vx_c)\coloneqq\E\left(g(x^2_1,\dots,x^2_c,\tilde{X}^*_{n,c+1},\dots,\tilde{X}^*_{n,q})
      -g(x^1_1,\dots,x^1_c,\tilde{X}_{c+1},\dots,\tilde{X}_{q})\right)
  \end{equation*}
  and the functions
  \begin{equation*}
    \xi^n_c(w)\coloneqq\E\left(w^n_c((\tilde{X}_1,\tilde{X}^*_{n,1}),\dots,(\tilde{X}_c,\tilde{X}^*_{n,c}))^2\right).
  \end{equation*}
  Then, it holds by the boundedness of $g$ that there exists a constant
  $C\in\R$ such that
  \begin{equation}
    \label{eq:uniformbound_asymptoticustat}
    \sup_{n\in\N}\xi^n_c(w)<C.
  \end{equation}
  Moreover, it holds by assumption (i) and (ii) that
  \begin{align*}
    w^n_1(\vx_1)
    &=\E\left(g(x^2_1,\tilde{X}^*_{n,2},\dots,\tilde{X}^*_{n,q})
      -g(x^1_1,\tilde{X}_{2},\dots,\tilde{X}_{q})\right)\\
    &=g^n_1(x^2_1)-g_1(x^1_1)\\
    &=0,
  \end{align*}
  which immediately implies that
  \begin{equation}
    \label{eq:xi1part}
    \xi^n_1(w)=\E\left(w^n_1((\tilde{X}_{1},\tilde{X}^*_{n,1}))^2\right)=0.
  \end{equation}
  Furthermore, by Jensen's inequality it holds that
  \begin{align}
    \xi^n_2(w)
    &=\E\left(w^n_2((\tilde{X}_{1},\tilde{X}^*_{n,1}),(\tilde{X}_{2},\tilde{X}^*_{n,2}))^2\right)\nonumber\\
    &\leq
    \E\left(w((\tilde{X}_{1},\tilde{X}^*_{n,1}),\dots,(\tilde{X}_{q},\tilde{X}^*_{n,q}))^2\right)\nonumber\\
    &=\E\left(\left(g(\tilde{X}^*_{n,1},\dots,\tilde{X}^*_{n,q})-g(\tilde{X}_{1},\dots,\tilde{X}_{q})\right)^2\right).\label{eq:xi2goestozero}
  \end{align}
  By \eqref{eq:thisisaustat_asymptoticustat}, it holds for fixed $n$
  that we can apply the variance formula for a U-statistic
  \citep[see][Lemma A, Section 5.2.1]{serfling} and together with
  \eqref{eq:uniformbound_asymptoticustat} and \eqref{eq:xi1part} we
  get
  \begin{equation}
    \Var\left(\tilde{\mathcal{U}}_n(g)-\mathcal{U}_n(g)\right)
    =\dbinom{n}{q}^{-1}\sum_{c=1}^n\dbinom{q}{c}\dbinom{n-q}{q-c}\xi^n_c(w)
    =\landauO{n^{-2}}\xi^n_2(w)+\landauO{n^{-3}}.
  \end{equation}
  Hence, together with \eqref{eq:xi2goestozero} and
  \eqref{eq:l2convergence_asymptoticustat} it holds that
  \begin{equation*}
    \lim_{n\rightarrow\infty}\Var\left(n\left(\mathcal{U}_n(g)-\tilde{\mathcal{U}}_n(g)\right)\right)=0
  \end{equation*}
  and consequently also that
  \begin{equation}
    \label{eq:convergenceUstatresampling}
    n\left(\mathcal{U}_n(g)-\tilde{\mathcal{U}}_n(g)\right)\overset{\tilde{\P}}{\longrightarrow}0
  \end{equation}
  as $n\rightarrow\infty$. Therefore, if $\xi_2(g)>0$, we can apply Slutsky's theorem
  together with the result about the asymptotic distribution of degenerate
  U-statistics given in \citet[Theorem, Section 5.5.2]{serfling} to
  get that
  \begin{equation*}
     n\Ustatr{n}{g}=n\Ustat{n}{g}+n\left(\Ustatr{n}{g}-\Ustat{n}{g}\right)\overset{d}{\longrightarrow}\dbinom{q}{2}\sum_{i=1}^{\infty}\lambda_i(Z_i^2-1)
  \end{equation*}
  as $n\rightarrow\infty$. If $\xi_2(g)=0$, we apply the variance
  formula of U-statistics \citep[see][Lemma A, Section 5.2.1]{serfling} to get that
  \begin{equation*}
    \lim_{n\rightarrow\infty}\Var(n\Ustat{n}{g})=0.
  \end{equation*}
  Hence, applying Slutsky's theorem together with
  \eqref{eq:convergenceUstatresampling} proves that
  \begin{equation*}
     n\Ustatr{n}{g}=n\Ustat{n}{g}+n\left(\Ustatr{n}{g}-\Ustat{n}{g}\right)\overset{d}{\longrightarrow}0
  \end{equation*}
  as $n\rightarrow\infty$, which completes the proof of Theorem \ref{thm:resamplingustat}.
\end{proof}
 
The same result also holds for V-statistics. The proof uses the same
technique as the proof of Theorem \ref{thm:asymptoticdist_vstat2} and reduces the
V-statistic back to the U-statistic.

\begin{theorem}[asymptotic distribution of degenerate resampling
  V-statistic]
  \label{thm:resamplingvstat}
  Assume Setting \ref{setting:resampling} and Setting
  \ref{setting:uvstatasymptotic}, let
  $g\in\lpspace{2}{\Plawpower{q}}{\abs{\cdot}_{\R}}$ be a continuous,
  bounded core function. Moreover, assume
  \begin{enumerate}[(i)]
  \item for all $n\in\N$ that $g^n_1\equiv 0$,
  \item $g_1\equiv 0$ (which implies $\xi_1(g)=0$) and
  \item $\theta_g=0$.
  \end{enumerate}
  Then if $\xi_2(g)>0$ it holds that
  \begin{equation*}
    n\Vstatr{n}{g}\overset{d}{\longrightarrow}\dbinom{q}{2}\sum_{i=1}^{\infty}\lambda_iZ_i^2
  \end{equation*}
  as $n\rightarrow\infty$ and if $\xi_2(g)=0$ it holds that
  \begin{equation*}
    n\Vstatr{n}{g}\overset{d}{\longrightarrow}\dbinom{q}{2}\sum_{i=1}^{\infty}\lambda_i
  \end{equation*}
  as $n\rightarrow\infty$.
\end{theorem}

\begin{proof}
  Applying Theorem \ref{thm:skorohod} results in a probability space
  $(\tilde{\Omega},\tilde{\mathcal{F}},\tilde{\P})$ and random
  variables $(\tilde{X}^*_{n,i})_{i\in\{1,\dots,n\}}$, $n\in\N$ and
  $(\tilde{X}_i)_{i\in\N}$ with properties specified in Theorem
  \ref{thm:skorohod}. Next, introduce the resampling U-statistic
  \begin{equation*}
    \tilde{\mathcal{U}}_n(g)\coloneqq\dbinom{n}{q}^{\sm1}\sum_{\combset{q}{n}}g(\tilde{X}^*_{n,i_1},\dots,\tilde{X}^*_{n,i_q}),
  \end{equation*}
  which has the same distribution under $\tilde{\P}$ as $\Ustatr{n}{g}$ under
  $\P_n$ and the resampling V-statistic
  \begin{equation*}
    \tilde{\mathcal{V}}_n(g)\coloneqq\dfrac{1}{n^q}\sum_{\mapset{q}{n}}g(\tilde{X}^*_{n,i_1},\dots,\tilde{X}^*_{n,i_q}),
  \end{equation*}
  which has the same distribution under $\tilde{\P}$ as $\Vstatr{n}{g}$
  under $\P_n$. For fixed $n\in\N$ we can view $\tilde{\mathcal{V}}_n(g)$ as a
  V-statistic and apply an adjusted version of Lemma
  \ref{thm:decompositionofmVstat} to get
  \begin{equation}
    \label{eq:decomposition_fun}
    n\tilde{\mathcal{V}}_n(\tilde{g})=\left(1+\landauO{n^{-1}}\right)S_n+\left(n-\dbinom{q}{2}+\landauO{n^{-1}}\right)\tilde{\mathcal{U}}_n(\tilde{g})
  \end{equation}
  as $n\rightarrow\infty$, where
  $S_n=\tilde{\mathcal{U}}_n(w)+\frac{1}{\permsize{n}{q-1}}\sum_B\tilde{g}(\tilde{X}^*_{n,i_1},\dots,\tilde{X}^*_{n,i_q})$. By
  the symmetry of the core function $g$ and the definition of $w$
  given in Lemma \ref{thm:decompositionofmVstat} it holds that
  \begin{equation*}
    \theta_w=\E\left(w(X_1,\dots,X_{q-1})\right)=\dbinom{q}{2}\E\left(\tilde{g}_2(X_1,X_1)\right).
  \end{equation*}
  The consistency of resampled U-statistics given in Lemma
  \ref{thm:consistencyresampledUstat} together with Lemma
  \ref{thm:lemma_eigensum} imply that
  \begin{equation}
    \label{eq:smpart1_fun}
    \tilde{\mathcal{U}}_n(w)\overset{d}{\longrightarrow}\theta_w=\dbinom{q}{2}\sum_{i=1}^{\infty}\lambda_i
  \end{equation}
  as $n\rightarrow\infty$. The boundedness of $g$ combined with the
  size of the set $B$ given in Lemma \ref{thm:decompositionofmVstat}
  shows that
  \begin{equation}
    \label{eq:smpart2_fun}
    \frac{1}{\permsize{n}{q-1}}\sum_B\tilde{g}(\tilde{X}^*_{n,i_1},\dots,\tilde{X}^*_{n,i_q})
    \leq \frac{C\abs{B}}{\permsize{n}{q-1}}=\landauO{n^{-1}}.
  \end{equation}
  Moreover, also by Lemma \ref{thm:consistencyresampledUstat} it
  holds that
  \begin{equation}
    \label{eq:umpart1_fun}
    \tilde{\mathcal{U}}_n(\tilde{g})\overset{d}{\longrightarrow}0
  \end{equation}
  as $n\rightarrow\infty$. By Theorem \ref{thm:resamplingustat} it
  holds if $\xi_2(g)>0$ that
  \begin{equation}
    \label{eq:umpart2_fun}
    n\tilde{\mathcal{U}}_n(\tilde{g})\overset{d}{\longrightarrow}\sum_{i=1}^{\infty}\lambda_i\left(Z_i^2-1\right)
  \end{equation}
  as $n\rightarrow\infty$ and if $\xi_2(g)=0$ that
  \begin{equation}
    \label{eq:umpart3_fun}
    n\tilde{\mathcal{U}}_n(\tilde{g})\overset{d}{\longrightarrow}0
  \end{equation}
  as $n\rightarrow\infty$.
  Finally, we can combine \eqref{eq:decomposition_fun},
  \eqref{eq:smpart1_fun}, \eqref{eq:smpart2_fun},
  \eqref{eq:umpart1_fun}, \eqref{eq:umpart2_fun}, \eqref{eq:umpart3_fun} and use that
  convergence in distribution to a constant implies convergence in
  probability together with Slutsky's theorem to get that if
  $\xi_2(g)>0$ it holds that
  \begin{equation*}
    n\tilde{\mathcal{V}}_n(\tilde{g})\overset{d}{\longrightarrow}\sum_{i=1}^{\infty}\lambda_iZ_i^2
  \end{equation*}
  as $n\rightarrow\infty$ and if $\xi_2(g)=0$ it holds that
    \begin{equation*}
    n\tilde{\mathcal{V}}_n(\tilde{g})\overset{d}{\longrightarrow}\dbinom{q}{2}\sum_{i=1}^{\infty}\lambda_i
  \end{equation*}
  as $n\rightarrow\infty$. This concludes the proof of
  Theorem~\ref{thm:resamplingvstat}.
\end{proof}

\section{Additional proofs and details}

In this section we collect all missing proofs from the main part of
the paper. To make this section more readable we have grouped the
proofs by topic.

\subsection{Properties of the asymptotic distribution of $n\cdot\mhsic_n$
  under $\HO$}\label{proof:properties_of_dist_HO}

This section strongly relies on the theory of empirical processes. We
therefore shortly recall some notation related to this theory, which
mainly builds on \citet{vandervaart}.

Let $\kernelspace$ be a measurable space, let
$\P\in\pmspace{\kernelspace}$ be a probability measure on this space
and let $\mathcal{G}$ be a function class consisting of measurable
functions $h:\kernelspace\rightarrow\R$. Moreover let
$X_1,X_2,\dots\iid\P$, which  due to issues related to the outer integral
defined below \citep[see also][Remark above Section
2.1.1]{vandervaart} are assumed to be canonically defined as
coordinate projections. The empirical measure
$\empjoint(X_1,\dots,X_n)\coloneqq\frac{1}{n}\sum_{i=1}^n\delta_{X_i}$
induces a linear map from $\mathcal{G}$ to $\R$ given by
\begin{equation*}
  h\mapsto\empjoint(X_1,\dots,X_n)h,
\end{equation*}
where we use the convention that
$\mathbb{Q}h\coloneqq\int_{\kernelspace} h(x) \measure{\mathbb{Q}}{x}$
for any measurable function $h$ and any signed measure $\mathbb{Q}$ on
$\kernelspace$.  A natural space to consider in this context is the
space of bounded functions from $\mathcal{G}$ to $\R$, which we denote
by $\linfG$. For all $T\in\linfG$ define the norm
$\norm{T}_{\mathcal{G}}\coloneqq\sup_{h\in\mathcal{G}}\abs{Th}$, then
the space $\linfG$ together with $\norm{\cdot}_{\mathcal{G}}$ forms a
Banach space which is separable if and only if $\mathcal{G}$
is a finite set. Assuming that for all $x\in\kernelspace$ it holds
that $\sup_{h\in\mathcal{G}}\abs{h(x)-\P h}<\infty$ allows us to view
the (normalized) empirical process
\begin{equation}
  \label{eq:empirical_process}
  \sqrt{n}\left(\empjoint(X_1,\dots,X_n)-\P\right),
\end{equation}
as an element of $\linfG$. In modern empirical process theory one
views this process as a random variable on $\linfG$ and is interested
in the asymptotic properties as $n\rightarrow\infty$. It, however,
turns out that the space $\linfG$ is in general too large to ensure
that the empirical process is a Borel measurable function. Therefore
it is common to introduce a specific outer integral, defined for all
functions $T$ from a probability space $(\Omega,\mathcal{F},\P)$ to
the extended real line $\overline{\R}$ (not necessarily measurable) by
\begin{equation*}
  \E^{*}(T)\coloneqq\inf\left\{\E(U):\, U\geq T,
    U:\Omega\rightarrow\overline{R}\text{ measurable and }\E(U)\text{ exists}\right\}.
\end{equation*}
Based on this definition one can extend the notions of convergence in
distribution, probability and almost surely \citep[Chapter
1]{vandervaart}. For example, assume $(X_n)_{n\in\N}$ is a sequence of
possibly non-measurable functions from underlying probability spaces
$(\Omega_n,\mathcal{A}_n,\P_n)$ to a metric space $\kernelspace$ and
let $X$ be a Borel measurable map on $\kernelspace$. Then we say
$(X_n)_{n\in\N}$ converges weakly in $\kernelspace$ to $X$, and write
$X_n\rightsquigarrow_{\kernelspace} X$, if for all
$f\in\operatorname{C}_b(\kernelspace)$ it holds that
\begin{equation*}
  \lim_{n\rightarrow\infty}\E^{*}(f(X_n))=\E(f(X)).
\end{equation*}

An important aspect of empirical process theory is to classify the
sets $\mathcal{G}$ in which the
empirical process defined in \eqref{eq:empirical_process} converges
weakly in $\linfG$. To this end, we need to introduce a particular
Gaussian process. For a probability distribution
$\P\in\pmspace{\kernelspace}$ and for a class of functions
$\mathcal{G}\subset\Lpspace{2}{\P}{\abs{\cdot}_{\R}}$ denote by
$G_{\P}$ the centered Gaussian process indexed by $\mathcal{G}$ with
covariance function given for all $f,g\in\mathcal{G}$ by
\begin{equation*}
  \operatorname{Cov}\left(G_{\P}f,G_{\P}g\right)=\E_{\P}(f(X)g(X))-\E_{\P}(f(X))\E_{\P}(g(X)).
\end{equation*}
The process $G_{\P}$ is sometimes called (generalized)
Brownian bridge process.

It should be clear that convergence of the empirical process is
related to the size of the function class $\mathcal{G}$. We call a
class of functions $\mathcal{G}$ a $\P$-Donsker class if for
$X_1,X_2,\dots\iid\P$ it holds that
\begin{equation*}
  \sqrt{n}\left(\empjoint(X_1,\dots,X_n)-\P\right)\rightsquigarrow_{\linfG} G_{\P}
\end{equation*}
as $n\rightarrow\infty$ and $G_{\P}$ is a tight Borel measurable
element in $\linfG$.

We now turn back to dHSIC. Denote by
$\empprod(\vx_1,\dots,\vx_n)\coloneqq\prod_{j=1}^d\left(\frac{1}{n}\sum_{i=1}^n\delta_{x^j_i}\right)$
the empirical product distribution and define the estimator $T_n:\prodkernelspace^n\rightarrow\R$ for all
$(\vx_1,\dots,\vx_n)\in\prodkernelspace^n$ by
\begin{equation}
  \label{eq:Tn_rep_dhsic}
  T_n(\vx_1,\dots,\vx_n)\coloneqq \sqrt{n}\sup_{f\in\prodRKHS:
    \norm{f}_{\prodRKHS}\leq 1}\abs[\big]{\empjoint(\vx_1,\dots,\vx_n)f-\empprod(\vx_1,\dots,\vx_n)f},
\end{equation}
then it is straightforward to prove that
$n\cdot\mhsicb_n=(T_n)^2$. Next, define for all $j\in\{1,\dots,d\}$ the
sets $\mathcal{F}^j\coloneqq\{f\in\RKHS^j: \norm{f}_{\RKHS^j}\leq 1\}$
and using the tensor product for functions define the set
\begin{equation}
  \label{eq:general_functionclass_F}
  \mathcal{F}\coloneqq\left\{f=\sum_{l=1}^m\lambda_lf^1_l\otimes\cdots\otimes
  f^d_l\,\Big\rvert\,\sum_{l=1}^m\lambda_l\leq 1\text{, }\lambda_l\geq 0\text{
    and } f^j_l\in\mathcal{F}^j\right\}.
\end{equation}
By \citet[Proposition 2.2]{ryan} it holds that $\overline{\mathcal{F}}=\{f\in\prodRKHS:
\norm{f}_{\prodRKHS}\leq 1\}$, which in particular, implies that the
estimator $T_n$ can be expressed as
\begin{equation*}
  T_n(\vx_1,\dots,\vx_n)\coloneqq \sqrt{n}\sup_{f\in\mathcal{F}}\abs[\big]{\empjoint(\vx_1,\dots,\vx_n)f-\empprod(\vx_1,\dots,\vx_n)f}.
\end{equation*}
Introduce the estimator
$S_n:\prodkernelspace^n\rightarrow\linfF$ defined for all
$(\vx_1,\dots,\vx_n)\in\prodkernelspace^n$ by
\begin{equation*}
  S_n(\vx_1,\dots,\vx_n)\coloneqq\sqrt{n}\left(\empjoint(\vx_1,\dots,\vx_n)-\empprod(\vx_1,\dots,\vx_n)\right).
\end{equation*}
The idea is to use the empirical process theory to find a weak limit
for $S_n$. By the continuous mapping theorem this will give us a weak
limit of $T_n=\norm{S_n}_{\mathcal{F}}$. In order to apply the empirical process theory, we show
that $S_n$ can be approximated by the estimator
$Z_n:\prodkernelspace^n\rightarrow\linfF$ which is defined for all
$(\vx_1,\dots,\vx_n)\in\prodkernelspace^n$ by
\begin{equation*}
  Z_n(\vx_1,\dots,\vx_n)\coloneqq\sqrt{n}\left(\empjoint(\vx_1,\dots,\vx_n)-\prodlawbold
  \right)-\sum_{j=1}^d\bigg(\prod_{l\neq
    j}\P^{X^l}\bigg)\sqrt{n}\left(\empjoint^j(\vx_1,\dots,\vx_n)-\P^{X^j}\right).
\end{equation*}
This trick, of approximating $S_n$ by a linear combination of
normalized empirical processes, is due to
\citet{romano1988bootstrap}. We will then be able to show that $Z_n$ converges to
a Gaussian process by making use of \citet[Theorem
4.3]{sriperumbudur2013}, which states that for all $j\in\{1,\dots,d\}$
the function class $\mathcal{F}^j$ is a $\P^{X^j}$-Donsker
class and also that the function class $\overline{\mathcal{F}}$, and
hence also $\mathcal{F}$, is a $\prodlawbold$-Donsker class.

We begin by showing that $Z_n$ and $S_n$ have the same asymptotic
properties on $\linfF$. The proof of this Lemma extends the ideas in
\citet{romano1988bootstrap} to account for classes of functions
instead of sets. Moreover, for the convergence we make use
of the empirical process theory.
\begin{lemma}[decomposition of $S_n$]
  \label{thm:decomposition_Sn}
  Assume Setting~\ref{setting:mHSIC}. Then for all
  $\vX_1,\vX_2,\dots\iid\prodlawbold\in\HO$ it holds that 
  \begin{equation*}
    \norm{S_n(\vX_1,\dots,\vX_n)-Z_n(\vX_1,\dots,\vX_n)}_{\mathcal{F}}\overset{\P}{\longrightarrow}0,
  \end{equation*}
  as $n\rightarrow\infty$.
\end{lemma}

\begin{proof}
  Fix $\vX_1,\vX_2,\dots\iid\prodlawbold\in\HO$ and denote by
  $\empjoint^j(\vx_1,\dots,\vx_n)\coloneqq\frac{1}{n}\sum_{i=1}^n\delta_{x^j_i}$
  the empirical measure on the $j$-th coordinate. Then for all $j\in\{1,\dots,d\}$ it holds by
  \citet[Theorem 4.3]{sriperumbudur2013} that $\mathcal{F}^j$ is a
  $\P^{X^j}$-Donsker class and hence in particular that
  \begin{equation*}
    \norm[\big]{\empjoint^j-\P^{X^j}}_{\mathcal{F}^j}\coloneqq\sup_{h\in\mathcal{F}^j}\abs[\big]{\empjoint^j(\vX_1,\dots,\vX_n)h-\P^{X^j}h}=\landauProbO{n^{-\frac{1}{2}}},
  \end{equation*}
  as $n\rightarrow\infty$.
  Next, let
  $f=f^1\otimes\dots\otimes
  f^d\in\mathcal{F}$ then it holds that
  \begin{align}
    \empprod(\vX_1,\dots,\vX_n)f&=\prod_{j=1}^d\empjoint^j(\vX_1,\dots,\vX_n)f^j\nonumber\\
    &=\prod_{j=1}^d\left[\P^{X^j}f^j+\left(\empjoint^j(\vX_1,\dots,\vX_n)f^j-\P^{X^j}f^j\right)\right]\nonumber\\
    &=\prodlawbold
      f+\sum_{j=1}^d\left(\prod_{k\neq
          j}\P^{X^k}f^k\right)\left(\empjoint^j(\vX_1,\dots,\vX_n)f^j-\P^{X^j}f^j\right)+R\left(\empprod,\prodlawbold,f\right),\label{eq:decompositionSn_1}
  \end{align}
  where $R\left(\empprod,\prodlawbold,f\right)$ is the remainder
  term resulting from multiplying out the product. In order to
  quantify the contribution of this remainder term to the sum,
  observe that for all $j\in\{1,\dots,d\}$ it holds that
  \begin{equation}
    \label{eq:remainderterm_est_1}
    \abs[\big]{\empjoint^j(\vX_1,\dots,\vX_n)f^j-\P^{X^j}f^j}
    \leq\norm[\big]{\empjoint^j(\vX_1,\dots,\vX_n)-\P^{X^j}}_{\mathcal{F}^j}
    =\landauProbO{n^{-\frac{1}{2}}}
  \end{equation}
  as $n\rightarrow\infty$. Moreover, using that $\mathcal{F}^j$
  consists of uniformly bounded functions (clearly true since
  $\mathcal{F}^j$ is the unit ball of the RKHS $\RKHS^j$) it holds that
  \begin{equation}
    \label{eq:remainderterm_est_2}
    \abs[\big]{\P^{X^j}f^j}
    \leq\norm[\big]{\P^{X^j}}_{\mathcal{F}^j}<\infty.
  \end{equation}
  Now, since $R\left(\empprod,\prodlawbold,f\right)$ only contains
  terms with at least two factors of the form
  $\left(\empjoint^kf^k_l-\P^{X^k}f^k_l\right)$ it holds by
  \eqref{eq:remainderterm_est_1} and \eqref{eq:remainderterm_est_2} that
  \begin{equation}
    \label{eq:decompositionSn_2}
    \abs[\Big]{R\left(\empprod,\prodlawbold,f\right)}=\landauProbO{n^{-1}}.
  \end{equation}
  as $n\rightarrow\infty$. Finally, using \eqref{eq:decompositionSn_1} and
  \eqref{eq:decompositionSn_2} we get for all
  $f=\sum_{l=1}^m\lambda_lf^1_l\otimes\dots\otimes
  f^d_l\in\mathcal{F}$ that
  \begin{align*}
    &\abs{S_n(\vX_1,\dots,\vX_n)f-Z_n(\vX_1,\dots,\vX_n)f}\\
    &\quad=\abs[\bigg]{\sum_{l=1}^m\lambda_l\sqrt{n}\Big[\prodlawbold
        f_l-\empprod(\vX_1,\dots,\vX_n) f_l+\sum_{j=1}^d\Big(\prod_{k\neq
          j}\P^{X^k}f^k_l\Big)\Big(\empjoint^j(\vX_1,\dots,\vX_n)f^j_l-\P^{X^j}f^j_l\Big)\Big]}\\
    &\quad=\abs[\bigg]{\sum_{l=1}^m\lambda_l\sqrt{n}R\left(\empprod,\prodlawbold,f_l\right)}\\
    &\quad=\landauProbO{n^{-\frac{1}{2}}},
  \end{align*}
  as $n\rightarrow\infty$. Since the bound is independent of $f$ we
  immediately get
  \begin{equation*}
    \norm{S_n(\vX_1,\dots,\vX_n)-Z_n(\vX_1,\dots,\vX_n)}_{\mathcal{F}}\overset{\P}{\longrightarrow}0,
  \end{equation*}
  as $n\rightarrow\infty$, which completes the proof of Lemma~\ref{thm:decomposition_Sn}.
\end{proof}

\begin{lemma}[asymptotic properties of $Z_n$]
  \label{thm:asymptotic_props_Zn2}
  Assume Setting~\ref{setting:mHSIC}. Let
  $\vX_1,\vX_2\dots\iid\prodlawbold\in\HO$. Then it holds that
  \begin{equation*}
    Z_n(\vX_1,\dots,\vX_n)
  \end{equation*}
  converges in distribution on $\linfF$ to a centered Gaussian
  process $Z$ indexed by $\mathcal{F}$ such that $Z$ is a tight Borel
  measurable element on $\linfF$ and such that $\P\left(\norm{Z}_{\mathcal{F}}=0\right)=0$.
\end{lemma}

\begin{proof}
  Recall that $Z_n$ (first introduced by \citet{romano1988bootstrap}) is a linear combination of empirical processes
  \begin{equation}
    \label{eq:Zn_rep_emp_process}
    Z_n(\vX_1,\dots,\vX_n)=\sqrt{n}\left(\empjoint(\vX_1,\dots,\vX_n)-\prodlawbold
      \right)-\sum_{j=1}^d\bigg(\prod_{l\neq
    j}\P^{X^l}\bigg)\sqrt{n}\left(\empjoint^j(\vX_1,\dots,\vX_n)-\P^{X^j}\right).
  \end{equation}
  Therefore, since by \citet[Theorem 4.3]{sriperumbudur2013} for all
  $j\in\{1,\dots,d\}$ the sets $\mathcal{F}^j$ are $\P^{X^j}$-Donsker
  classes and $\mathcal{F}$ is a $\prodlawbold$-Donsker class it holds
  that
  \begin{equation*}
    Z_n(\vX_1,\dots,\vX_n)\rightsquigarrow_{\linfF} G_{\prodlawbold}-\sum_{j=1}^d\bigg(\prod_{l\neq
    j}\P^{X^l}\bigg)G_{\P^{X^j}}\eqqcolon Z,
  \end{equation*}
  as $n\rightarrow\infty$. Using that the sum of tight
  Borel measurable Gaussian processes is again a tight Borel
  measurable Gaussian process, it follows that $Z$ is a tight Borel
  measurable centered Gaussian process on $\linfF$.
  
  It remains to show that $\P\left(\norm{Z}_{\mathcal{F}}=0\right)=0$. By
  monotonicity of the integral it is enough to show that there exists
  $f_0\in\mathcal{F}$ such that $\P\left(\abs{Zf_0}=0\right)=0$.  To
  this end, define for all $f=f^1\otimes\dots\otimes
  f^d\in\mathcal{F}$ and for all $i\in\N$ the random variables
  \begin{equation*}
    Y_i(f)\coloneqq
    f(\vX_i)-\E\left(f(\vX_i)\right)-\sum_{j=1}^d\left(\prod_{l\neq j}\E\left(f^l(X^l_i)\right)\right)\left(f^j(X^j_i)-\E\left(f^j(X^j_i\right)\right).
  \end{equation*}
  Clearly, it holds that $(Y_i(f))_{i\in\N}$ is an iid sequence of
  centered real-valued random variables and by
  \eqref{eq:Zn_rep_emp_process} it holds that
  $Z_nf=\frac{1}{\sqrt{n}}\sum_{i=1}^nY_i(f)$.
  Hence, if we are able to show that there exists
  $f_0=f^1_0\otimes\dots\otimes f^d_0\in\mathcal{F}$
  such that $\Var\left(Y_i(f_0)\right)>0$ the classical central limit
  theorem proves that $\P\left(\abs{Zf_0}=0\right)=0$. Using that
  $\prodlawbold\in\HO$ we can perform the following calculation.
  \begin{align}
    \Var\left(Y_i(f)\right)
    &=\E\Bigg(\Big[f(\vX_i)-\sum_{j=1}^d\Big(\prod_{l\neq
            j}\E\left(f^l(X^l_i)\right)\Big)f^j(X^j_i)\Big]^2\Bigg)-(d-1)^2\E\left(f(\vX_i)\right)^2\nonumber\\
    &=\E\left(f(\vX_i)^2\right)-\sum_{j=1}^d\Big(\prod_{l\neq
            j}\E\left(f^l(X^l_i)\right)\Big)\E\left(f(\vX_i)f^j(X^j_i)\right)\nonumber\\
    &\quad +\sum_{j_1=1}^d\sum_{j_2=1}^d\Big(\prod_{l_1\neq
      j_1}\E\left(f^{l_1}(X^{l_1}_i)\right)\Big)\Big(\prod_{l_2\neq
      j_2}\E\left(f^{l_2}(X^{l_2}_i)\right)\Big)\E\left(f^{j_1}(X^{j_1}_i)f^{j_2}(X^{j_2}_i)\right)\nonumber\\
    &\quad -(d-1)^2\E\left(f(\vX_i)\right)^2\nonumber\\
    &=\E\left(f(\vX_i)^2\right)-\sum_{j=1}^d\Big(\prod_{l\neq
            j}\E\left(f^l(X^l_i)\right)\Big)^2\E\left(f^j(X^j_i)^2\right)\nonumber\\
    &\quad +\sum_{j=1}^d\Big(\prod_{l\neq
            j}\E\left(f^l(X^l_i)\right)\Big)^2\E\left(f^j(X^j_i)^2\right)+d(d-1)\E\left(f(\vX_i)\right)^2-(d-1)^2\E\left(f(\vX_i)\right)^2\nonumber\\
    &=\Var\left(f(\vX_i)\right)+d\cdot\E\left(f(\vX_i)\right)^2.\label{eq:variance_yi}
  \end{align}
  Assume for the sake of contradiction that
  $\E\left(f(\vX_i)\right)=0$ for all $f=f^1\otimes\dots\otimes
  f^d\in\mathcal{F}$, then by linearity it in particular holds that
  $\E\left(f(\vX_i)\right)=0$ for all $f\in\mathcal{F}$. Using the
  properties of the Bochner integral and the definition of the mean
  embedding function this implies that for all $f\in\mathcal{F}$ it
  holds that
  \begin{equation*}
    \E\left(f(\vX_i)\right)
    =\E\Big(\scalarprod[\big]{\prodk(\vX_i,\cdot),f}_{\prodRKHS}\Big)
    =\scalarprod[\big]{\Pi\left(\prodlawbold\right),f}_{\prodRKHS}=0.
  \end{equation*}
  Since $\overline{\mathcal{F}}$ is the unit ball in $\prodRKHS$ it
  holds that $\Pi\left(\prodlawbold\right)=0$. By
  Setting~\ref{setting:mHSIC}, we have that $\Pi$ is injective (since
  $\prodk$ is characteristic) and hence $\prodlawbold\equiv 0$, which
  contradicts that $\prodlawbold$ is a probability
  distribution. Therefore, there exists $f_0=f^1_0\otimes\dots\otimes
  f^d_0\in\mathcal{F}$ such that $\E\left(f_0(\vX_i)\right)^2>0$. By
  \eqref{eq:variance_yi} this implies that
  $\Var\left(f_0(\vX_i)\right)>0$, which completes the proof of
  Lemma~\ref{thm:asymptotic_props_Zn2}.
\end{proof}

\begin{theorem}[continuous asymptotic distribution]
  \label{thm:cont_strictly_increasing_cdf}
  Assume Setting~\ref{setting:mHSIC}. Then for
  $\vX_1,\vX_2,\dots\iid\prodlawbold\in\HO$ it holds that
  \begin{equation*}
    n\cdot\mhsicb\left(\vX_1,\dots,\vX_n\right)
  \end{equation*}
  converges in distribution to a random variable which has no atom at
  $0$ and a continuous distribution function on $(0,\infty)$.
\end{theorem}

\begin{proof}
  It is sufficient to prove the result for $T_n$, because
  $n\cdot\mhsicb_n=(T_n)^2$. Furthermore, by Slutsky's theorem
  together with Lemma~\ref{thm:decomposition_Sn} and
  Lemma~\ref{thm:asymptotic_props_Zn2} it holds that
  $S_n(\vX_1,\dots,\vX_n)$ converges in distribution on $\linfF$ to a
  mean zero Gaussian process $Z$ indexed by $\mathcal{F}$ such that
  $Z$ is a tight Borel measurable element in $\linfF$.  Next, by
  \citet[Lemma 1.5.9]{vandervaart2} there exists a semi-metric $\rho$
  on $\mathcal{F}$ for which the paths of $Z$ are $\P$-a.s. uniformly
  continuous and $\mathcal{F}$ is totally bounded.

  Denote by $\operatorname{C}_u(\mathcal{F},\rho)$ the space of
  bounded uniformly continuous functions from $\mathcal{F}$ to
  $\R$. It can be shown that this space seen as a subspace of $\linfF$
  with the $\norm{\cdot}_{\mathcal{F}}$ norm is a complete and
  separable space if $\mathcal{F}$ is totally bounded with respect to
  $\rho$ \citep[e.g.][Example 18.7]{vandervaart}. The previous
  considerations therefore imply that it holds $\P$-a.s. that $Z$
  takes values in a separable Banach space.  Therefore we can apply
  \citet[Exercise 2.4.4]{gine} with
  $B=\operatorname{C}_u(\mathcal{F},\rho)$, which proves that
  distribution function of $\norm{Z}_{\mathcal{F}}$ is continuous on
  $(0,\infty)$.  Moreover, the continuous mapping theorem \citep[for
  an extention to outer measures see][Theorem 1.11.1]{vandervaart2}
  together with the continuity of the norm implies that $T_n$
  converges in distribution to $\norm{Z}_{\mathcal{F}}$. Since by
  Lemma~\ref{thm:asymptotic_props_Zn2} it also holds that
  $\norm{Z}_{\mathcal{F}}$ has no atom at zero, the proof of
  Theorem~\ref{thm:cont_strictly_increasing_cdf} is complete.
\end{proof}

\subsection{Asymptotic distribution of $n\cdot\mhsic_n$}\label{proof:asymptotic_dist}

A key step in the following proof is to show by means of contradiction
that $\xi_2(h)>0$. This relies on
Theorem~\ref{thm:cont_strictly_increasing_cdf}, which is proved using
empirical process theory in Appendix~\ref{proof:properties_of_dist_HO}.

\begin{proof}[Theorem~\ref{thm:asymptoticdist_HO}]
  Use Lemma~\ref{thm:properties_h} to observe that $\mhsicb_n$ is simply
  the V-statistic $\Vstat{n}{h}$ with $\theta_h=\mhsic$. By Lemma
  \ref{thm:degeneracy} it holds that $\xi_1(h)=0$ under $\HO$ and
  moreover, again by Lemma \ref{thm:properties_h} it holds that $h$ is
  bounded and continuous.  If $\xi_2(h)>0$, we can apply
  Theorem~\ref{thm:asymptoticdist_vstat2} to see that,
  \begin{equation*}
    n\cdot\mhsicb_n\overset{d}{\longrightarrow}\dbinom{2d}{2}\sum_{i=1}^{\infty}\lambda_iZ_i^2
  \end{equation*}
  as $n\rightarrow\infty$.
  
  Next, assume $\xi_2(h)=0$, we can apply Theorem~\ref{thm:var_vstat} to see
  that $\lim_{n\rightarrow\infty}\Var(n\cdot\mhsicb_n)=0$. Combining
  this with Theorem~\ref{thm:biasVstat} and Lemma~\ref{thm:lemma_eigensum} hence
  leads to
  \begin{equation*}
    n\cdot\mhsicb_n\overset{d}{\longrightarrow}\dbinom{2d}{2}\sum_{i=1}^{\infty}\lambda_i.
  \end{equation*}
  as $n\rightarrow\infty$. However,
  Theorem~\ref{thm:cont_strictly_increasing_cdf} states that the
  limiting distribution of $n\cdot\mhsicb_n$ has a distribution
  function which is continuous on $(0,\infty)$ and has no atom at
  $0$. This is a contradiction and hence it holds that $\xi_2(h)>0$,
  which completes the proof of Theorem~\ref{thm:asymptoticdist_HO}.
\end{proof}

\begin{proof}[Theorem~\ref{thm:corollaryasymptoticdistmhsic}]
  Let $t\in\R$ and $\vX_1,\vX_2,\ldots\iid\jointlawb\in\HA$ fixed, then
  \begin{equation}
    \label{eq:asymptoticdivergence}
    \begin{split}
    \P&\left(n\cdot\mhsicb_n(\vX_1,\dots,\vX_n)\leq t\right)\\
    &=\P\left(\sqrt{n}(\mhsicb_n(\vX_1,\dots,\vX_n)-\mhsic)\leq
      \dfrac{t}{\sqrt{n}}-\sqrt{n}\mhsic\right).
    \end{split}
  \end{equation}
  Moreover, by Lemma~\ref{thm:properties_h} it holds that $\mhsicb$ is
  simply the V-statistic $\Vstat{n}{h}$ with
  $\theta_h=\mhsic$. Additionally, again by Lemma \ref{thm:properties_h}
  it holds that $h$ is bounded and continuous.
  If $\xi_1(h)>0$ then we can apply
  Theorem \ref{thm:asymptoticdist_vstat1} to see that,
  \begin{equation}
    \label{eq:xi_1nonzerocase}
    \sqrt{n}\left(\mhsicb_n-\mhsic\right)
    \overset{d}{\longrightarrow} \normaldist{0}{(2d)^2\xi_1(h)}
  \end{equation}
  as $n\rightarrow\infty$.
  Next assume $\xi_1(h)=0$, then by Theorem~\ref{thm:var_vstat}
  it holds that
  \begin{equation*}
    \E\left(n\left(\mhsicb_n-\mhsic\right)^2\right)=n\Var\left(\mhsicb\right)=\landauO{n^{-1}}
  \end{equation*}
  as $n\rightarrow\infty$ and since convergence in second moment implies convergence in
  distribution this implies
  \begin{equation}
    \label{eq:xi_1zerocase}
    \sqrt{n}\left(\mhsicb_n-\mhsic\right)\overset{d}{\longrightarrow}0
  \end{equation}
  as $n\rightarrow\infty$.  Using the corollary of Slutsky's theorem
  given in \citet[Corollary 11.2.3]{lehmann} and combining \eqref{eq:asymptoticdivergence}
  with \eqref{eq:xi_1nonzerocase} if $\xi_1(h)>0$ and \eqref{eq:asymptoticdivergence} with
  \eqref{eq:xi_1zerocase} if $\xi_1(h)=0$ completes the proof of Corollary
  \ref{thm:corollaryasymptoticdistmhsic}.
\end{proof}

\subsection{Properties of the permutation test}\label{proof:permutation_test}

\begin{proof}[Proposition~\ref{thm:level_permutation_test}]
  Fix $n\in\N$, under $\HO$, i.e. $\vX_1,\vX_2,\ldots\iid\prodlawbold\in\HO$, it holds that the
  individual coordinates of $\vX_i$ are independent. Hence, for all
  $\boldsymbol{\psi}\in (S_n)^d$ it holds that
  $(\vX_1,\dots,\vX_n)$ is equal in distribution to
  $(\vX_1^{\boldsymbol{\psi}},\dots,\vX_n^{\boldsymbol{\psi}})$, so in
  particular, we have that
  \begin{equation}
    \label{eq:invariancecondperm}
    g_{n,\boldsymbol{\psi}}(\vX_1,\dots,\vX_n)\text{ is equal in distribution to }(\vX_1,\dots,\vX_n).
  \end{equation}
  Moreover since $(S_n)^d$ has a group structure we can apply Theorem
  \ref{thm:level_resampling_test} to get that $\phi$ has level
  $\alpha$, which completes the proof of Proposition
  \ref{thm:level_permutation_test}.
\end{proof}

\subsection{Properties of the bootstrap test}\label{proof:bootstrap_test}

In this section we prove that the bootstrap test has pointwise
asymptotic level (Theorem~\ref{thm:level_bootstrap_test}) and is
consistent (Theorem~\ref{thm:consistencybootstraptest}). The proofs
rely on showing that under both $\HO$ and $\HA$ the resampled test
statistic
$n\cdot\mhsicb_n(g_{n,\boldsymbol{\Psi}_n}(\vX_1,\dots,\vX_n))$
converges in distribution to an infinite sum of chi-squared
distributed random variables. This is shown in
Lemma~\ref{thm:asymptotic_dist_resampled_dhsic}. We can use standard
arguments to conclude the proofs, by comparing this
distribution with the results about the asymptotic distribution of the
test statistic $n\cdot\mhsicb(\vX_1,\dots,\vX_n)$ given in
Theorem~\ref{thm:asymptoticdist_HO} and
Theorem~\ref{thm:corollaryasymptoticdistmhsic}.

We begin by formally introducing the empirical product distribution function.
\begin{definition}[empirical product distribution function]
  Assume Setting \ref{setting:mHSIC}, then the function
  $\hat{F}_n:\prodkernelspace^n\times\R^d\rightarrow [0,1]$ satisfying for all
  $(\vx_1,\dots,\vx_n)\in\prodkernelspace^n$ and for all $\vt\in\R^d$ that
  \begin{equation*}
    \hat{F}_n(\vx_1,\dots,\vx_n)(\vt)\coloneqq\prod_{j=1}^d\left(\dfrac{1}{n}\sum_{i=1}^n\mathds{1}_{\{x^j_i\leq t^j\}}\right)
  \end{equation*}
  is called the empirical product distribution function.
\end{definition}
The following lemma is basic but essential for the proof of
Lemma~\ref{thm:asymptotic_dist_resampled_dhsic}. It shows that random draws
from the resampling distribution correspond to independent draws from
the empirical product distribution $\empprodlaw$.

\begin{lemma}[bootstrapping property]
  \label{thm:bootstrap_prop}
  Assume Setting \ref{setting:mHSIC}, let $n\in\N$, and for all
  $\boldsymbol{\psi}\in B_n^d$ (see \eqref{eq:bootstrap_set}) let $g_{n,\boldsymbol{\psi}}$ be defined as in \eqref{eq:gmpsifun},
  let $\boldsymbol{\Psi}$ be a random variable with uniform
  distribution on $B_n^d$ and let $\hat{F}_n$ be the empirical product
  distribution function. Then it holds for all $(\vx_1,\dots,\vx_n)\in\prodkernelspace^n$ that
  \begin{equation*}
    g_{n,\boldsymbol{\Psi}}(\vx_1,\dots,\vx_n)=\left(\vx_{n,1}^{\boldsymbol{\Psi}},\dots,\vx_{n,n}^{\boldsymbol{\Psi}}\right)
  \end{equation*}
  are $n$ iid random variables with distribution function $\hat{F}_n(\vx_1,\dots,\vx_n)$.
\end{lemma}

\begin{proof}
  Let $(\tilde{\Omega},\tilde{\mathcal{F}},\tilde{\P})$ be the
  probability space such that
  $\boldsymbol{\Psi}=(\Psi^1,\dots,\Psi^d):\tilde{\Omega}\rightarrow
  B_n^d$. Then, by the properties of the uniform distribution it holds
  that $\Psi^1,\dots,\Psi^d$ are iid with uniform distribution on
  $B_n$ and moreover
  $\Psi^1(1),\dots,\Psi^1(n),\dots,\Psi^d(1),\dots,\Psi^d(n)$ are iid
  with uniform distribution on $\{1,\dots,n\}$. This implies
  that for all
  $(\vx_1,\dots,\vx_n)\in\prodkernelspace^n$ it holds that
  \begin{equation*}
    g_{n,\boldsymbol{\Psi}}(\vx_1,\dots,\vx_n)=\left(\vx_{n,1}^{\boldsymbol{\Psi}},\dots,\vx_{n,n}^{\boldsymbol{\Psi}}\right)
  \end{equation*}
  is a vector of $n$ independent random variables. Furthermore, we have that for
  all $(\vx_1,\dots,\vx_n)\in\prodkernelspace^n$, for all
  $i\in\{1,\dots,n\}$ and for all $\vt\in\R^d$ it holds that
  \begin{align*}
    \tilde{\P}\left(\vx^{\boldsymbol{\Psi}}_{n,i}\leq\vt\right)
    &=\prod_{j=1}^d\tilde{\P}\left(x^j_{n,\Psi^j(i)}\leq
      t^j\right)\\
    &=\prod_{j=1}^d\left(\dfrac{1}{\abs{B_n}}\sum_{\psi\in B}\mathds{1}_{\{x^j_{n,\psi(i)}\leq
      t^j\}}\right)\\
    &=\prod_{j=1}^d\left(\dfrac{1}{n}\sum_{l=1}^n\mathds{1}_{\{x^j_{n,i}\leq
      t^j\}}\right).
  \end{align*}
  Hence, it holds for
  all $i\in\{1,\dots,n\}$ that $\vx^{\boldsymbol{\Psi}}_{n,i}$ has
  distribution function $\hat{F}_n(\vx_1,\dots,\vx_n)$, which completes the proof of
  Lemma~\ref{thm:bootstrap_prop}.
\end{proof}

We are now ready to show that the resampled test statistic
asymptotically behaves like the test statistic based on the
product law.

\begin{lemma}[asymptotic distribution of the resampled test statistic]
  \label{thm:asymptotic_dist_resampled_dhsic}
  Assume Setting \ref{setting:mHSIC}. Let
  $\vX_1,\vX_2,\dots\iid\prodlawbold\in\HO\cup\HA$ and for all
  $n\in\N$ let $\boldsymbol{\Psi}_n$ be a uniformly distributed random
  variable on $B_n^d$ independent of $(\vX_i)_{i\in\N}$. Moreover, let
  $(Z_i)_{i\in\N}$ be a sequence of independent standard normal random
  variables on $\R$, let $T_{h_2}\in
  L(\Lpspace{2}{\prodlaw}{\abs{\cdot}_{\R}})$ with the property that
  for every $f\in\Lpspace{2}{\prodlaw}{\abs{\cdot}_{\R}}$ and for
  every $\vx\in\prodkernelspace$ it holds that
  \begin{equation}
    \label{eq:lambda_diff}
    \left(T_{h_2}(f)\right)(\vx)=\int_{\prodkernelspace}h_2(\vx,\vy)f(\vy)\measure{\prodlaw}{\vy}
  \end{equation}
  and let $(\lambda_i)_{i\in\N}$ be the eigenvalues of $T_{h_2}$. Then
  there exists a measurable set $A_0\subset\Omega$ with $\P(A_0)=1$ satisfying
  that, for all $\omega\in A_0$ that
  \begin{equation*}
    n\cdot\mhsicb_n\left(g_{n,\boldsymbol{\Psi}_n}(\vX_1(\omega),\dots,\vX_n(\omega))\right)\overset{d}{\longrightarrow}\binom{2d}{2}\sum_{i=1}^{\infty}\lambda_iZ_i^2
  \end{equation*}
  as $n\rightarrow\infty$.
\end{lemma}

\begin{proof}
  Let $\hat{F}_n$ be the empirical product distribution
  function and define for all $\vt\in\R^d$ the population product
  distribution function by
  \begin{equation*}
    F(\vt)\coloneqq
    \prodlaw
    \big((-\infty,t^1]\times\cdots\times (-\infty,t^d]\big)
    =\prod_{j=1}^d\P\left(X^j\leq t^j\right).
  \end{equation*}

  Applying the Glivenko-Cantelli theorem \citep[e.g.][Theorem
  19.1]{vandervaart}, which extends the strong law of large numbers
  for empirical distributions to uniform convergence, shows that there
  exists a subset $A_0\subseteq\Omega$ such that $\P(A_0)=1$ and such
  that for all $\omega\in A_0$ it holds for all $\vt\in\R^d$ that
  \begin{equation}
    \label{eq:convergence_empdist_consistency}
    \lim_{n\rightarrow\infty}\hat{F}_n(\vX_1(\omega),\dots,\vX_n(\omega))(\vt)=F(\vt).
  \end{equation}
  The Glivenko-Cantelli theorem is necessary to get the
  almost sure convergence uniform in $\vt$.

  By Lemma~\ref{thm:bootstrap_prop} it holds for all $n\in\N$
  and for all $(\vx_1,\dots,\vx_n)\in\prodkernelspace^n$ that
  \begin{equation*}
    g_{n,\boldsymbol{\Psi}_n}(\vx_1,\dots,\vx_n)=\left(\vx_{n,1}^{\boldsymbol{\Psi}_n},\dots,\vx_{n,n}^{\boldsymbol{\Psi}_n}\right)
  \end{equation*}
  are iid random
  variables with distribution function $\hat{F}_n(\vx_1,\dots,\vx_n)$.
  
  Fix $\omega\in A_0$, let $(\vX^*_i)_{i\in\N}$ be iid sequence of
  random variables with distribution $\prodlaw$ and for all $i\in\N$ define
  $\vx_i\coloneqq\vX_i(\omega)$. Then, by
  \eqref{eq:convergence_empdist_consistency} it holds that
  \begin{equation*}
    \vx^{\boldsymbol{\Psi}_n}_{n,i}\overset{d}{\longrightarrow}\vX^*_i
  \end{equation*}
  as $n\rightarrow\infty$. Hence, we are in the same setting as described in Setting
  \ref{setting:resampling}.

  Since both $\prodlaw\in\HO$ and $\empprodlaw\in\HO$ it holds by
  Lemma \ref{thm:degeneracy} for all $\vz\in\prodkernelspace$ that
  \begin{equation*}
    h_1(\vz)=\E\left(h(\vz,\vX^*_2,\dots,\vX^*_{2d})\right)=0
  \end{equation*}
  and for all $n\in\{2d,2d+1,\dots\}$ and for all $\vz\in\prodkernelspace$ that
  \begin{equation*}
    h^n_1(\vz)=\E\left(h(\vz,\vx^{\boldsymbol{\Psi}_n}_{n,2},\dots,\vx^{\boldsymbol{\Psi}_n}_{n,2d})\right)=0,
  \end{equation*}
  where $h$ is defined as in \eqref{eq:hcore}. Moreover, it holds by
  Proposition~\ref{thm:mhsic_independence} that
  \begin{equation*}
    \theta_h=\E\left(h(\vX^*_1,\dots,\vX^*_{2d})\right)=\mhsic\left(\prodlaw\right)=0.
  \end{equation*}
  Moreover, since $\prodlaw$ clearly satisfies the null hypothesis,
  Theorem~\ref{thm:asymptoticdist_HO} implies that $\xi_2(h)>0$.  We
  therefore satisfy all requirements of
  Theorem~\ref{thm:resamplingvstat} and get that
  \begin{equation}
    \label{eq:convergence_in_distribution1_consistency}
    n\cdot\mhsicb_n(\vx^{\boldsymbol{\Psi}_n}_{n,1},\dots,\vx^{\boldsymbol{\Psi}_n}_{n,n})
    =n\Vstatr{n}{h}\overset{d}{\longrightarrow}
    \dbinom{2d}{2}\sum_{i=1}^{\infty}\lambda_iZ_i^2
  \end{equation}
  as $n\rightarrow\infty$. This completes the proof of
  Lemma~\ref{thm:asymptotic_dist_resampled_dhsic}.
\end{proof}

Intuitively, it should be clear that
Lemma~\ref{thm:asymptotic_dist_resampled_dhsic} together with
Theorem~\ref{thm:asymptoticdist_HO} is sufficient to show pointwise
asymptotic level. The details are given in the proof below.

\begin{proof}[Theorem~\ref{thm:level_bootstrap_test}, asymptotic level]
  Let $\vX_1,\vX_2,\ldots\iid\prodlawbold\in\HO$ be fixed and use the
  notation defined in
  Lemma~\ref{thm:asymptotic_dist_resampled_dhsic}. Then it holds by
  Theorem~\ref{thm:asymptoticdist_HO} that
  \begin{equation}
    \label{eq:asymptotic_dist_dhsic_null}
    n\cdot\mhsicb_n(\vX_1,\dots,\vX_n)\overset{d}{\longrightarrow}\binom{2d}{2}\sum_{i=1}^{\infty}\lambda_iZ_i^2
  \end{equation}
  as $n\rightarrow\infty$. Let $G:\R\rightarrow(0,1)$ be the distribution
  function of $\binom{2d}{2}\sum_{i=1}^{\infty}\lambda_iZ_i^2$, then
  by Lemma~\ref{thm:asymptotic_dist_resampled_dhsic} it holds for all
  $t\in\R$ and for all $\omega\in A_0$ that
  \begin{align*}
    \lim_{n\rightarrow\infty}\hat{R}_n(\vX_1(\omega),\dots,\vX_n(\omega))(t)
    &=\lim_{n\rightarrow\infty}\dfrac{1}{n^{nd}}\sum_{\boldsymbol{\psi}\in
      B_n^d}
    \mathds{1}_{\{n\cdot\mhsicb_n(g_{n,\boldsymbol{\psi}}(\vX_1(\omega),\dots,\vX_n(\omega)))\leq
      t\}}\\
    &=\lim_{n\rightarrow\infty}\E\left(\mathds{1}_{\{n\cdot\mhsicb_n(g_{n,\boldsymbol{\Psi}_n}(\vX_1(\omega),\dots,\vX_n(\omega)))\leq
      t\}}\right)\\
    &=\lim_{n\rightarrow\infty}\P\left(n\cdot\mhsicb_n(g_{n,\boldsymbol{\Psi}_n}(\vX_1(\omega),\dots,\vX_n(\omega)))\leq
      t\right)\\
    &=G(t).
  \end{align*}
  Since $G$ is continuous it holds for all $t\in (0,1)$ and for all
  $\omega\in A_0$ that
  \begin{equation*}
    \lim_{n\rightarrow\infty}\left(\hat{R}_n(\vX_1(\omega),\dots,\vX_n(\omega))\right)^{-1}(t)=G^{-1}(t)
  \end{equation*}
  \citep[e.g.][Lemma 11.2.1]{lehmann}. Recall that $\P(A_0)=1$ which
  implies that it holds $\P$-a.s. that
  \begin{equation}
    \label{eq:quantileconvergence_level}
    \lim_{n\rightarrow\infty}\left(\hat{R}_n(\vX_1,\dots,\vX_n)\right)^{-1}(1-\alpha)=G^{-1}(1-\alpha).
  \end{equation}
  Finally, we can perform the following calculation
  \begin{align*}
    \limsup_{n\rightarrow\infty}&\,\P\left(\phi_n(\vX_1,\dots,\vX_n)=1\right)\\
    &=\limsup_{n\rightarrow\infty}\,\P\left(n\cdot\mhsicb_n(\vX_1,\dots,\vX_n)>\left(\hat{R}_n(\vX_1,\dots,\vX_n)\right)^{-1}(1-\alpha)\right)\\
    &=1-\liminf_{n\rightarrow\infty}\,\P\left(n\cdot\mhsicb_n(\vX_1,\dots,\vX_n)\leq\left(\hat{R}_n(\vX_1,\dots,\vX_n)\right)^{-1}(1-\alpha)\right)\\
    &=1-G(G^{-1}(1-\alpha))=\alpha,
  \end{align*}
  where in the last step we use the corollary of Slutsky's theorem
  given in \citet[Corollary 11.2.3]{lehmann} together with
  \eqref{eq:asymptotic_dist_dhsic_null} and
  \eqref{eq:quantileconvergence_level}. This completes
  the proof of Theorem~\ref{thm:level_bootstrap_test}.  
\end{proof}

The proof of Theorem~\ref{thm:consistencybootstraptest} uses similar
reasoning as the proof of Theorem~\ref{thm:level_bootstrap_test}.

\begin{proof}[Theorem~\ref{thm:consistencybootstraptest}, consistency]
  Let $\vX_1,\vX_2,\ldots\iid\prodlawbold\in\HA$ be fixed and use the
  notation defined in
  Lemma~\ref{thm:asymptotic_dist_resampled_dhsic}. Let
  $G:\R\rightarrow(0,1)$ be the distribution function of
  $\binom{2d}{2}\sum_{i=1}^{\infty}\lambda_iZ_i^2$, then by
  Lemma~\ref{thm:asymptotic_dist_resampled_dhsic} it holds for all
  $t\in\R$ and for all $\omega\in A_0$ that
  \begin{align*}
    \lim_{n\rightarrow\infty}\hat{R}_n(\vX_1(\omega),\dots,\vX_n(\omega))(t)
    &=\lim_{n\rightarrow\infty}\dfrac{1}{n^{nd}}\sum_{\boldsymbol{\psi}\in
      B_n^d}
    \mathds{1}_{\{n\cdot\mhsicb_n(g_{n,\boldsymbol{\psi}}(\vX_1(\omega),\dots,\vX_n(\omega)))\leq
      t\}}\\
    &=\lim_{n\rightarrow\infty}\E\left(\mathds{1}_{\{n\cdot\mhsicb_n(g_{n,\boldsymbol{\Psi}_n}(\vX_1(\omega),\dots,\vX_n(\omega)))\leq
      t\}}\right)\\
    &=\lim_{n\rightarrow\infty}\P\left(n\cdot\mhsicb_n(g_{n,\boldsymbol{\Psi}_n}(\vX_1(\omega),\dots,\vX_n(\omega)))\leq
      t\right)\\
    &=G(t).
  \end{align*}
  Since $G$ is continuous it holds for all $t\in (0,1)$ and for all
  $\omega\in A_0$ that
  \begin{equation*}
    \lim_{n\rightarrow\infty}\left(\hat{R}_n(\vX_1(\omega),\dots,\vx_n(\omega))\right)^{-1}(t)=G^{-1}(t)
  \end{equation*}
  \citep[e.g.][Lemma 11.2.1]{lehmann}. So in particular for all $\omega\in A_0$ it holds that
  \begin{equation}
    \label{eq:quantileconvergence_consistency}
    \lim_{n\rightarrow\infty}\left(\hat{R}_n(\vX_1(\omega),\dots,\vX_n(\omega))\right)^{-1}(1-\alpha)=G^{-1}(1-\alpha).
  \end{equation}
  Introduce the set
  \begin{equation}
    \label{eq:quantileconvergence_consistency2}
    A_1\coloneqq\left\{\omega\in\Omega \,\big\rvert\, \forall
      t\in\R:\,\lim_{n\rightarrow\infty}\mathds{1}_{\{n\cdot\mhsicb_n(\vX_1(\omega),\dots,\vX_n(\omega))\leq
      t\}}=0\right\}.
  \end{equation}
  By Theorem~\ref{thm:corollaryasymptoticdistmhsic} it holds that
  $\P(A_1)=1$, which implies that $\P(A_0\cap A_1)=1$.  Let $\omega\in
  A_0\cap A_1$, then by \eqref{eq:quantileconvergence_consistency} and
  \eqref{eq:quantileconvergence_consistency2} there exists a constant
  $t^*\in\R$ such that for all $n\in\N$ it holds that
  \begin{equation*}
    \left(\hat{R}_n(\vX_1(\omega),\dots,\vX_n(\omega))\right)^{-1}(1-\alpha)\leq t^*
  \end{equation*}
  and hence
  \begin{align*}
    \lim_{n\rightarrow\infty}&\mathds{1}_{\left\{n\cdot\mhsicb_n(\vX_1(\omega),\dots,\vX_n(\omega))
        \leq\left(\hat{R}_n(\vX_1(\omega),\dots,\vX_n(\omega))\right)^{-1}(1-\alpha)\right\}}\\
    &\leq \lim_{n\rightarrow\infty}\mathds{1}_{\{n\cdot\mhsicb_n(\vX_1(\omega),\dots,\vX_n(\omega))\leq
      t^*\}}=0.
  \end{align*}
  This proves that $\P$-a.s. it holds that
  \begin{equation*}
    \lim_{n\rightarrow\infty}\mathds{1}_{\left\{n\cdot\mhsicb_n(\vX_1,\dots,\vX_n)\leq
      \left(\hat{R}_n(\vX_1,\dots,\vX_n)\right)^{-1}(1-\alpha)\right\}}=0
  \end{equation*}
  and applying the dominated convergence theorem we also get
  \begin{align*}
    \lim_{n\rightarrow\infty}&\P\left(\phi_n(\vX_1,\dots,\vX_n)=0\right)\\
    &=\lim_{n\rightarrow\infty}\E\left(\mathds{1}_{\left\{n\cdot\mhsicb_n(\vX_1,\dots,\vX_n)\leq
      \left(\hat{R}_n(\vX_1,\dots,\vX_n)\right)^{-1}(1-\alpha)\right\}}\right)\\
    &=0,
  \end{align*}
  which completes the proof of Theorem~\ref{thm:consistencybootstraptest}.
\end{proof}

\subsection{Moments of $\mhsicb_n$}\label{proof:momentsofdhsic}
In this section we show how the first two asymptotic moments of
$\mhsicb_n$ can be calculated.
\begin{proof}[Proposition~\ref{thm:meanofmhsic}]
  Due to Lemma \ref{thm:properties_h} we know that $\mhsicb$ is a
  V-statistic with core function $h$. Under $\HO$ it holds that
  $\theta_h=0$ and thus applying Lemma
  \ref{thm:biasVstat} results in
  \begin{equation*}
    \E\left(\mhsicb_n\right)=\dfrac{1}{n}\dbinom{2d}{2}\E\left(h_2(\vX_1,\vX_1)\right)+\landauO{n^{-2}}.
  \end{equation*}
  We can use Lemma \ref{thm:expansionh2} to explicitly calculate
  $\tbinom{2p}{2}\E(h_2(\vX_1,\vX_1))$, which
  together with the independence assumption under $\HO$ simplifies to
  the desired expression. This concludes the proof of Proposition~\ref{thm:meanofmhsic}.
\end{proof}

\begin{proof}[Proposition~\ref{thm:varianceofmhsic}]
  Due to Lemma \ref{thm:properties_h} we know that $\mhsicb$ is a
  V-statistic with core function $h$. Applying Lemma
  \ref{thm:var_vstat} thus results in
  \begin{equation*}
    \Var\left(\mhsicb_n\right)=\binom{n}{2d}^{-1}\dbinom{2d}{2}\dbinom{n-2d}{2d-2}\xi_2+\landauO{n^{-\frac{5}{2}}}.
  \end{equation*}
  Under $\HO$ we can use Lemma~\ref{thm:expansionh2} to get that
  \begin{align*}
    \xi_2
    &=\E\left(h_2(\vX_1,\vX_2)^2\right)\\
    &=\dbinom{2d}{2}^{-2}\E\left(\left(\sum_{i=1}^{10}a_i\right)^2\right)\\
    &=\dbinom{2d}{2}^{-2}\sum_{i,j=1}^{10}\E\left(a_ia_j\right).
  \end{align*}
  Each term $\E(a_ia_j)$ can be explicitly calculated and simplified
  using the independence properties under $\HO$ (very tedious). This concludes the
  proof of Proposition~\ref{thm:varianceofmhsic}.
\end{proof}

\subsection{Kernel representation of $\mhsic$}\label{proof:mhsickernel}

\begin{proof}[Proposition~\ref{thm:mhsickernel}]
  Using the definition of the mean embedding we get
  \begin{align}
    \mhsic&=\norm[\Big]{\Pi\left(\prodlaw\right)-\Pi\left(\jointlaw\right)}_{\scriptscriptstyle \prodRKHS}^2\nonumber\\
    &=\norm[\bigg]{\prod_{j=1}^d\E\left(k^j\left(X^j_1,\cdot\right)\right)-\E\left(\prodk\left(\vX_1,\cdot\right)\right)}_{\scriptscriptstyle \prodRKHS}^2\nonumber\\
    &=\norm[\bigg]{\prod_{j=1}^d\E\left(k^j\left(X^j_1,\cdot\right)\right)}_{\scriptscriptstyle \prodRKHS}^2+\norm[\bigg]{\E\left(\prodk\left(\vX_1,\cdot\right)\right)}_{\scriptscriptstyle \prodRKHS}^2\nonumber\\
    &\qquad
    -2\scalarprod[\Bigg]{\prod_{j=1}^d\E\left(k^j\left(X^j_1,\cdot\right)\right),\E\left(\prodk\left(\vX_1,\cdot\right)\right)}_{\scriptscriptstyle \prodRKHS}\label{eq:mhsicpart0}
  \end{align}
  Next we simplify each term individually using the properties of the
  Bochner integral and the properties
  of tensor Hilbert spaces.
  \begin{align}
    \norm[\bigg]{\prod_{j=1}^d\E\left(k^j\left(X^j_1,\cdot\right)\right)}_{\scriptscriptstyle \prodRKHS}^2&=\prod_{j=1}^d\norm[\bigg]{\E\left(k^j\left(X^j_1,\cdot\right)\right)}_{\scriptscriptstyle \RKHS^j}^2\nonumber\\
    &=\prod_{j=1}^d\scalarprod[\bigg]{\E\left(k^j\left(X^j_1,\cdot\right)\right),\E\left(k^j\left(X^j_1,\cdot\right)\right)}_{\scriptscriptstyle \RKHS^j}\nonumber\\
    &=\prod_{j=1}^d\E\left(\scalarprod[\big]{k^j\left(X^j_1,\cdot\right),k^j\left(X^j_2,\cdot\right)}_{\scriptscriptstyle \RKHS^j}\right)\nonumber\\
    &=\prod_{j=1}^d\E\left(k^j\left(X^j_1,X^j_2\right)\right)\nonumber\\
    &=\E\left(\prod_{j=1}^dk^j\left(X^j_{2j-1},X^j_{2j}\right)\right)\label{eq:mhsicpart1}
  \end{align}
  \begin{align}
    \norm[\bigg]{\E\left(\prodk\left(\vX_1,\cdot\right)\right)}_{\scriptscriptstyle \prodRKHS}^2&=\scalarprod[\bigg]{\E\left(\prodk\left(\vX_1,\cdot\right)\right),\E\left(\prodk\left(\vX_1,\cdot\right)\right)}_{\scriptscriptstyle \prodRKHS}\nonumber\\
    &=\E\Big(\scalarprod[\big]{\prodk\left(\vX_1,\cdot\right),\prodk\left(\vX_2,\cdot\right)}_{\scriptscriptstyle \prodRKHS}\Big)\nonumber\\
    &=\E\left(\prodk\left(\vX_1,\vX_2\right)\right)\nonumber\\
    &=\E\left(\prod_{j=1}^dk^j\left(X^j_1,X^j_2\right)\right)\label{eq:mhsicpart2}
  \end{align}
  \begin{align}
    \scalarprod[\Bigg]{\prod_{j=1}^d\E\left(k^j\left(X^j_1,\cdot\right)\right),\E\left(\prodk\left(\vX_1,\cdot\right)\right)}_{\scriptscriptstyle
      \prodRKHS}
    &=\E\left(\scalarprod[\bigg]{\E\left(\prod_{j=1}^d k^j\left(X^j_{j+1},\cdot\right)\right),\prod_{i=1}^dk^j\left(X^j_1,\cdot\right)}_{\scriptscriptstyle \prodRKHS}\right)\nonumber\\
    &=\E\left(\scalarprod[\bigg]{\prod_{j=1}^dk^j\left(X^j_{j+1},\cdot\right),\prod_{j=1}^dk^j\left(X^j_1,\cdot\right)}_{\scriptscriptstyle \prodRKHS}\right)\nonumber\\
    &=\E\left(\prod_{j=1}^d\scalarprod[\big]{k^j\left(X^j_{j+1},\cdot\right),k^j\left(X^j_1,\cdot\right)}_{\scriptscriptstyle \RKHS^j}\right)\nonumber\\
    &=\E\left(\prod_{j=1}^dk^j\left(X^j_1,X^j_{j+1}\right)\right)\label{eq:mhsicpart3}
  \end{align}
  Combining \eqref{eq:mhsicpart0}, \eqref{eq:mhsicpart1}, \eqref{eq:mhsicpart2} and
  \eqref{eq:mhsicpart3} completes the proof of Proposition~\ref{thm:mhsickernel}.
\end{proof}

\subsection{Properties of $h$}\label{proof:properties_h}

\begin{proof}[Lemma~\ref{thm:properties_h}]
  We prove the 5 properties separately.
  
  $\bullet$ {\bfseries $h$ is symmetric:}\\ This is immediate by
    construction.
    
  $\bullet$ {\bfseries $h$ is continuous:}\\ This follows from the continuity of the kernels
    $k^j$, which is assumed in Setting \ref{setting:mHSIC}.
    
  $\bullet$ {\bfseries $h$ is bounded:}\\ Under Setting \ref{setting:mHSIC} we assume
  that all $k^j$'s are bounded. Hence for all $j\in\{1,\dots,d\}$ let
  $C^j>0$ such that for all $z_1,z_2\in\kernelspace$ it holds
  that $\abs{k^j(z_1,z_2)}<C^j.$
  Thus it is clear that for all
  $\vz_1,\dots,\vz_{2d}\in\prodkernelspace$ it holds that
  \begin{equation*}
    \abs{h(\vz_1,\dots,\vz_{2d})}<4\prod_{j=1}^dC^j=:C.
  \end{equation*}
  
  $\bullet$ {\bfseries $\mhsicb$ is a V-statistic with core function $h$:}\\ Compute directly,
  \begin{align*}
    \Vstat{n}{h}&=\dfrac{1}{n^{2d}}\sum_{\mapset{2d}{n}}h\left(\vX_1,\dots,\vX_{2p}\right)\\
    &=\dfrac{1}{(2d)!}\sum_{\pi\in
      S_{2d}}\dfrac{1}{n^{2d}}\sum_{\mapset{2d}{n}}\Bigg[\prod_{j=1}^dk^j\left(X^j_{\pi(i_1)},X^j_{\pi(i_2)}\right)+
    \prod_{j=1}^dk^j\left(X^j_{\pi(_{2j-1})},X^j_{\pi(i_{2j})}\right)\\
    &\qquad\qquad\qquad\qquad\qquad\qquad-2\prod_{j=1}^dk^j\left(X^j_{\pi(i_1)},X^j_{\pi(i_{j+1})}\right)\Bigg]\\
    &=\dfrac{1}{(2d)!}\sum_{\pi\in
      S_{2d}}\Bigg[\dfrac{1}{n^{2}}\sum_{\mapset{2}{n}}\prod_{j=1}^dk^j\left(X^j_{\pi(i_1)},X^j_{\pi(i_2)}\right)\\
    &\qquad\qquad\qquad\qquad\qquad\qquad+\dfrac{1}{n^{2d}}\sum_{\mapset{2d}{n}}\prod_{j=1}^dk^j\left(X^j_{\pi(_{2j-1})},X^j_{\pi(i_{2j})}\right)\\
    &\qquad\qquad\qquad\qquad\qquad\qquad-\dfrac{2}{n^{d+1}}\sum_{\mapset{d+1}{n}}\prod_{j=1}^dk^j\left(X^j_{\pi(i_1)},X^j_{\pi(i_{j+1})}\right)\Bigg]\\
    &=\mhsicb_n.
  \end{align*}
  
  $\bullet$ {\bfseries $\theta_h=\mhsic$:}\\ Again computing directly,
  \begin{align*}
    \E\left(h(\vX_1,\dots,\vX_{2d})\right)&=\dfrac{1}{(2d)!}\sum_{\pi\in
      S_{2d}}\Bigg[\E\left(\prod_{j=1}^dk^j\left(X^j_{\pi(1)},X^j_{\pi(2)}\right)\right)\\
    &\qquad\qquad\qquad+\E\left(\prod_{j=1}^dk^j\left(z^j_{\pi(2j-1)},z^j_{\pi(2j)}\right)\right)\\
    &\qquad\qquad\qquad-2\E\left(\prod_{j=1}^dk^j\left(z^j_{\pi(1)},z^j_{\pi(j+1)}\right)\right)\Bigg]\\
    &=\dfrac{1}{(2d)!}\sum_{\pi\in
      S_{2d}}\Bigg[\E\left(\prod_{j=1}^dk^j\left(X^j_{1},X^j_{2}\right)\right)+
    \E\left(\prod_{j=1}^dk^j\left(z^j_{2j-1},z^j_{2j}\right)\right)\\
    &\qquad\qquad\qquad-2\E\left(\prod_{j=1}^dk^j\left(z^j_{1},z^j_{j+1}\right)\right)\Bigg]\\
    &=\mhsic.
  \end{align*}

  This completes the proof of Lemma \ref{thm:properties_h}
\end{proof}

\subsection{Further technical results related to $h$}

In order to make the calculations in this section more readable we use
the following conventions.
\begin{itemize}
\item For all $j\in\{1,\dots,d\}$ and for all
  $i_1,i_2\in\{1,\dots,n\}$ we
  set $$k^j_{i_1,i_2}:=k^j(X^j_{i_1},X^j_{i_2}).$$
\item For all $q,n\in\N$, for all functions
  $g:\prodkernelspace^n\rightarrow\R$ and for all
  $i_1,\dots,i_q,j_1,\dots,j_n\in\{1,\dots,n\}$ we
  set $$\E_{i_1,\dots,i_q}\left(g(\vX_{j_1},\dots,\vX_{j_n})\right)=\int_{\prodkernelspace}\cdots\int_{\prodkernelspace}g(\vX_{j_1},\dots,\vX_{j_n})\measure{\P^{\vX}}{\vX_{i_1}}\cdots\measure{\P^{\vX}}{\vX_{i_q}}.$$
\end{itemize}

\begin{lemma}[expansion of $h_1$]
  \label{thm:expansionh1}
  Assume Setting \ref{setting:mHSIC}. Then it holds for all
  $\vz\in\prodkernelspace$ that,
  \begin{equation*}
    \begin{split}
      h_1(\vz)&=\dfrac{1}{d}\left[\E\left(\prod_{j=1}^dk^j(z^j,X^j_1)\right)-\E\left(\prod_{j=1}^dk^j(z^j,X^j_{j})\right)\right]\\
      &\quad
      +\dfrac{d-1}{d}\left[\E\left(\prod_{j=1}^dk^j(X^j_1,X^j_2)\right)-\E\left(\prod_{j=1}^dk^j(X^j_1,X^j_{j+1})\right)\right]\\
      &\quad +\dfrac{1}{d}\left[\sum_{r=1}^d\E\left(\left(\prod_{j\neq
            r}^dk^j(X^j_{2j\sm
            1},X^j_{2j})\right)k^r(z^r,X^r_{2r})\right)\right.\\
      &\qquad\qquad\qquad\left.-\sum_{r=1}^d\E\left(\left(\prod_{j\neq
          r}^dk^j(X^j_1,X^j_{j+1})\right)k^r(z^r,X^r_{r+1})\right)\right]
    \end{split}
  \end{equation*}
\end{lemma}

\begin{proof}
  Recall that
  \begin{equation*}
    h_1(\vz)=\E\left(h(\vz,\vX_1,\dots,\vX_{2d\sm 1})\right).
  \end{equation*}
  Next we separate $h$ into 3 terms as follows.
  \begin{align*}
    h(\vz_1,\dots,\vz_{2d})&=\dfrac{1}{(2d)!}\sum_{\pi\in
      S_{2d}}\left[\prod_{j=1}^dk^j\left(z^j_{\pi(1)},z^j_{\pi(2)}\right)\right](=:b_1)\\
    &\quad+\dfrac{1}{(2d)!}\sum_{\pi\in S_{2d}}\left[\prod_{j=1}^dk^j\left(z^j_{\pi(2j-1)},z^j_{\pi(2j)}\right)\right](=:b_2)\\
    &\quad-\dfrac{2}{(2d)!}\sum_{\pi\in S_{2d}}\left[\prod_{j=1}^dk^j\left(z^j_{\pi(1)},z^j_{\pi(j+1)}\right)\right](=:b_3).
  \end{align*}
  Now we calculate
  $\E_{2,\dots,2d}\left(h(\vX_1,\dots,\vX_{2d})\right)$ by considering
  these three terms separately.

  $\mathbf{b_1}$:
  Begin by letting $\pi\in S_{2d}$, then
  \begin{equation*}
    \E_{2,\dots,2d}\left(\prod_{j=1}^dk^j_{\pi(1),\pi(2)}\right)=
    \begin{cases}
      \E_{2,3}\left(\prod_{j=1}^dk^j_{2,3}\right) &\text{if }
      \pi(1)\neq 1 \wedge \pi(2)\neq 1\\
      \E_2\left(\prod_{j=1}^dk^j_{1,2}\right) &\text{if }
      \pi(1)=1 \vee \pi(2)=1.
    \end{cases}
  \end{equation*}
  Counting how often each of these cases can occur for $\pi\in S_{2d}$
  leads to
  \begin{align}
    &\dfrac{1}{(2d)!}\sum_{\pi\in
      S_{2d}}\E_{2,\dots,2d}\left(\prod_{j=1}^dk^j_{\pi(1),\pi(2)}\right)\nonumber\\
    &\quad
    =\dfrac{(2d-2)(2d-1)!}{(2d)!}\E_{2,3}\left(\prod_{j=1}^dk^j_{2,3}\right)+\dfrac{2(2d-1)!}{(2d)!}\E_2\left(\prod_{j=1}^dk^j_{1,2}\right)\nonumber\\
    &\quad
    =\dfrac{d-1}{p}\E_{2,3}\left(\prod_{j=1}^dk^j_{2,3}\right)+\dfrac{1}{p}\E_2\left(\prod_{j=1}^dk^j_{1,2}\right)\label{eq:partb1}
  \end{align}

  $\mathbf{b_2}$:
  Begin by letting $\pi\in S_{2d}$, $r\in\{1,\dots,p\}$ such that $\pi(2r-1)=1$ or $\pi(2r)=1$ then
  \begin{equation*}
    \E_{2,\dots,2d}\left(\prod_{j=1}^dk^j_{\pi(2j-1),\pi(2j)}\right)=\E_{2,\dots,2d+1}\left(\left(\prod_{j\neq
        r}^dk^j_{2j,2j+1}\right)k^r_{1,2r}\right)
  \end{equation*}
  Counting how many combinations are possible for each $r$ and adding all different combinations up gives us
  \begin{align}
    &\dfrac{1}{(2d)!}\sum_{\pi\in
      S_{2d}}\E_{2,\dots,2d}\left(\prod_{j=1}^dk^j_{\pi(2j-1),\pi(2j)}\right)\nonumber\\
    &\quad
    =\dfrac{2(2d-1)!}{(2d)!}\sum_{r=1}^d\E_{2,\dots,2d+1}\left(\left(\prod_{j\neq
          r}^dk^j_{2j,2j+1}\right)k^r_{1,2r}\right)\nonumber\\
    &\quad
    =\dfrac{1}{d}\sum_{r=1}^d\E_{2,\dots,2d+1}\left(\left(\prod_{j\neq
          r}^dk^j_{2j,2j+1}\right)k^r_{1,2r}\right)\label{eq:partb2}
  \end{align}
  
  $\mathbf{b_3}$:
  Begin by letting $\pi\in S_{2d}$, then
  \begin{equation*}
    \E_{2,\dots,2d}\left(\prod_{j=1}^dk^j_{\pi(1),\pi(j+1)}\right)=
    \begin{cases}
      \E_{2,\dots,d+2}\left(\prod_{j=1}^dk^j_{2,j+2}\right) &\text{if }
      \pi(1)\neq 1 \wedge\cdots\wedge \pi(d+1)\neq 1\\
      \E_{2,\dots,d+1}\left(\prod_{j=1}^dk^j_{1,j+1}\right) &\text{if }
      \pi(1)=1\\
      \E_{2,\dots,d+2}\left(\prod_{j\neq r}^dk^j_{2,j+2}k^r_{1,2}\right) &\text{if }
      \pi(r+1)=1 \text{ for }r\in\{1,\dots,d\}\\
    \end{cases}
  \end{equation*}
  Counting how often each of these cases can occur for different
  $\pi\in S_{2d}$ and adding all cases up results in
  \begin{align}
    &\dfrac{1}{(2d)!}\sum_{\pi\in
      S_{2d}}\E_{2,\dots,2d}\left(\prod_{j=1}^dk^j_{\pi(1),\pi(j+1)}\right)\nonumber\\
    &\quad
    =\dfrac{d-1}{2d}\E_{2,\dots,p+2}\left(\prod_{j=1}^dk^j_{2,j+2}\right)+\dfrac{1}{2d}\E_{2,\dots,p+1}\left(\prod_{j=1}^dk^j_{1,j+1}\right)\nonumber\\
    &\qquad\quad+\dfrac{1}{2d}\sum_{r=1}^d\E_{2,\dots,p+2}\left(\prod_{j\neq r}^dk^j_{2,j+2}k^r_{1,2}\right)\label{eq:partb3}
  \end{align}

  Finally combining \eqref{eq:partb1}, \eqref{eq:partb1} and
  \eqref{eq:partb1} completes the proof of Lemma \ref{thm:expansionh1}.
\end{proof}

\begin{lemma}[expansion of $h_2$ under $\HO$]
  \label{thm:expansionh2}
  Assume Setting \ref{setting:mHSIC}. Then under $\HO$ it holds for all
  $\vz_1,\vz_2\in\prodkernelspace$ that,
  \begin{align*}
    \dbinom{2d}{2}h_2(\vz_1,\vz_2)&=\prod_{r=1}^dk^r(z^r_1,z^r_2) \quad (=:a_1)\\
  &\quad+(d-1)^2\prod_{r=1}^d\E\left(k^r(X^r_1,X^r_2)\right) \quad (=:a_2)\\
  &\quad+(d-1)\prod_{r=1}^d\E\left(k^r(z^r_1,X^r_1)\right) \quad (=:a_3)\\
  &\quad+(d-1)\prod_{r=1}^d\E\left(k^r(z^r_2,X^r_1)\right) \quad (=:a_4)\\
  &\quad+\sum_{r=1}^dk^r(z^r_1,z^r_2)\prod_{l\neq r}\E\left(k^l(X^l_1,X^l_2)\right) \quad (=:a_5)\\
  &\quad-\sum_{r=1}^dk^r(z^r_1,z^r_2)\prod_{l\neq r}\E\left(k^l(z^l_1,X^l_1)\right) \quad (=:a_6)\\
  &\quad-\sum_{r=1}^dk^r(z^r_1,z^r_2)\prod_{l\neq r}\E\left(k^l(z^r_2,X^r_1)\right) \quad (=:a_7)\\
  &\quad+\sum_{r\neq s}\E\left(k^r(z^r_1,X^r_1)\right)\E\left(k^s(z^s_2,X^s_1)\right)\prod_{l\neq
    r,s}\E\left(k^l(X^l_1,X^l_2)\right) \quad (=:a_8)\\
  &\quad-(d-1)\sum_{r=1}^d\E\left(k^r(z^r_1,X^r_1)\right)\prod_{l\neq
    r}\E\left(k^l(X^l_1,X^l_2)\right) \quad (=:a_9)\\
  &\quad-(d-1)\sum_{r=1}^d\E\left(k^r(z^r_2,X^r_1)\right)\prod_{l\neq r}\E\left(k^l(X^l_1,X^l_2)\right) \quad (=:a_{10}).
  \end{align*}
\end{lemma}

\begin{proof}
  Begin by setting,
  \begin{align*}
    A&:=\sum_{\pi\in S_{2d}}\E_{3,\dots,2d}\left(\prod_{j=1}^dk^j_{\pi(1),\pi(2)}\right)\\
    B&:=\sum_{\pi\in S_{2d}}\E_{3,\dots,2d}\left(\prod_{j=1}^dk^j_{\pi(2j-1),\pi(2j)}\right)\\
    C&:=\sum_{\pi\in S_{2d}}\E_{3,\dots,2d}\left(\prod_{j=1}^dk^j_{\pi(1),\pi(j+1)}\right).
  \end{align*}
  Then it holds that,
  \begin{equation}
    \label{eq:expansionh2}
    h_2(\vX_1,\vX_2)=\E_{3,\dots,2d}\left(h(\vX_1,\dots,\vX_{2d})\right)=\dfrac{1}{(2d)!}\left(A+B-2C\right).
  \end{equation}
  Under the null hypothesis $\HO$ the terms $A$,$B$ and $C$ can be
  simplified using combinatorial arguments (similar to the ones used
  in the proof of Lemma \ref{thm:expansionh1}).
  \begin{align*}
    A&=2(2d-2)!\prod_{r=1}^dk^r_{1,2}\\
    &\quad+(2d-2)(2d-3)(2d-2)!\prod_{r=1}^d\E_{3,4}\left(k^r_{3,4}\right)\\
    &\quad+2(2d-2)(2d-2)!\prod_{r=1}^d\E_{3}\left(k^r_{1,3}\right)\\
    &\quad+2(2d-2)(2d-2)!\prod_{r=1}^d\E_{3}\left(k^r_{2,3}\right)
  \end{align*}
\begin{align*}
    B&=2(2d-2)!\sum_{r=1}^dk^r_{1,2}\prod_{l\neq
      r}\E_{3,4}\left(k^l_{3,4}\right)\\
    &\quad+4(2d-2)!\sum_{r\neq
      s}\E_{3}\left(k^r_{1,3}\right)\E_{3}\left(k^s_{2,3}\right)\prod_{l\neq r,s}\E_{3,4}\left(k^l_{3,4}\right)
  \end{align*}
\begin{align*}
    C&=2(2d-2)!\sum_{r=1}^dk^r_{1,2}\prod_{l\neq
      r}\E_{3}\left(k^l_{1,3}\right)\\
    &\quad+(d-1)(2d-2)!\prod_{r=1}^d\E_{3}\left(k^r_{1,3}\right)+(d-1)(2d-2)!\prod_{r=1}^d\E_{3}\left(k^r_{2,3}\right)\\
    &\quad+(d-1)(d-2)(2d-2)!\prod_{r=1}^d\E_{3,4}\left(k^r_{3,4}\right)\\
    &\quad+(2d-2)!\sum_{r\neq
      s}\E_{3}\left(k^r_{1,3}\right)\E_{3}\left(k^s_{2,3}\right)\prod_{l\neq
      r,s}\E_{3,4}\left(k^l_{3,4}\right)\\
    &\quad+(d-1)(2d-2)!\sum_{r=1}^d\E_{3}\left(k^r_{1,3}\right)\prod_{l\neq
      r}\E_{3,4}\left(k^l_{3,4}\right)\\
    &\quad+(d-1)(2d-2)!\sum_{r=1}^d\E_{3}\left(k^r_{2,3}\right)\prod_{l\neq r}\E_{3,4}\left(k^l_{3,4}\right)
  \end{align*}
  Plugging these expressions for $A$, $B$ and $C$ into
  \eqref{eq:expansionh2} completes the proof of Lemma \ref{thm:expansionh2}.
\end{proof}

\begin{lemma}[degeneracy under $\HO$]
  \label{thm:degeneracy}
  Assume Setting \ref{setting:mHSIC}. Then under $\HO$ it holds for
  all $\vz\in\prodkernelspace$ that
  \begin{equation*}
    h_1(\vz)=0,
  \end{equation*}
  and therefore in particular that $\xi_1(h)=0$.
\end{lemma}

\begin{proof}
  Observe that under $\HO$ it holds for all $\vz\in\prodkernelspace$ that
  \begin{align}
    &\mathbf{\bullet}\quad\E\left(\prod_{j=1}^dk^j(z^j,X^j_1)\right)=\E\left(\prod_{j=1}^dk^j(z^j,X^j_{j})\right)\label{eq:h1_parta}\\
    &\mathbf{\bullet}\quad\E\left(\prod_{j=1}^dk^j(X^j_1,X^j_2)\right)=\E\left(\prod_{j=1}^dk^j(X^j_1,X^j_{j+1})\right)\label{eq:h1_partb}\\
    &\mathbf{\bullet}\quad\E\left(\left(\prod_{j\neq r}^dk^j(X^j_{2j\sm 1},X^j_{2j})\right)k^r(z^r,X^r_{2r})\right)\nonumber\\
    &\qquad\qquad\qquad\qquad\qquad=\E\left(\left(\prod_{j\neq r}^dk^j(X^j_1,X^j_{j+1})\right)k^r(z^r,X^r_{r+1})\right).\label{eq:h1_partc}
  \end{align}
  Plugging \eqref{eq:h1_parta}, \eqref{eq:h1_partb} and
  \eqref{eq:h1_partc} into the explicit form of $h_1$ given in Lemma
  \ref{thm:expansionh1} yields
  \begin{equation*}
    h_1(\vz)=0.
  \end{equation*}
  This completes the proof of Lemma \ref{thm:degeneracy}.
\end{proof}

\section{Dealing with estimated residuals}\label{sec:approx_resid}

There are two ways of dealing with the fact that the residuals
resulting from the regression in the DAG verification method in
Section~\ref{subsec.causal} are estimated and inherently
dependent. Firstly, it can be shown that even though the estimated
residuals are dependent the asymptotic ordering corresponds to the
ordering resulting from the exact residuals, see
Section~\ref{sec:order_preserve}. While this allows us to use the DAG
verification method as a model selection procedure, we need to be more
careful if we want to use it as a goodness of fit test. To ensure
that such a test has asymptotic level (or equivalently that the
$p$-values are asymptotically correct) we can employ a sample splitting
procedure, see Section~\ref{sec:sample_split}.

We first introduce the non-parametric regression setting used in
the DAG verification method in a more formal manner. Let
$(\vX_i,\vR_i)_{i\in\{1,\dots,n\}}\in\R^d\times\R^d$ 
be 
iid 
random variables satisfying
\begin{equation}
  \label{eq:regression_setting}
  \begin{pmatrix}
    X^1_i\\
    \vdots\\
    X^d_i
  \end{pmatrix}
  =
  \begin{pmatrix}
    f_1(\parents^1_i)\\
    \vdots\\
    f_d(\parents^d_i)
  \end{pmatrix}
  +
  \begin{pmatrix}
    R^1_i\\
    \vdots\\
    R^d_i
  \end{pmatrix}
  \quad\text{with}\quad \parents^j_i\independent R^j_i
\end{equation}
for some functions $f_j:\R^{\abs{\parents^j}}\rightarrow\R$. The functions
$f_j$ are called regression functions and can be expressed for all
$\vx\in\R^{\abs{\parents^j}}$ as
\begin{equation*}
  f_j(\vx)=\E\left(X^j\mid\parents^j=\vx\right).
\end{equation*}
Given a data set $\mathcal{Z}_n$ based on $n$ 
iid
random variables, a
regression method provides an estimate of the regression function,
which we denote by $\hat{f}_{j,\mathcal{Z}_n}$. The corresponding
estimated residuals $\hat{\vR}_{i,n}$ are defined for all
$j\in\{1,\dots,d\}$ by
\begin{equation*}
  \hat{R}^j_{i,n}\coloneqq X^j_i-\hat{f}_{j,\mathcal{Z}_n}(\parents^j_i).
\end{equation*}
Similar as in \citet[Definition 18]{dataset} we call a regression method
suitable if 
\begin{equation*}
  \lim_{n\rightarrow\infty}\E\left(\frac{1}{n}\sum_{i=1}^n\norm{\hat{\vR}_{i,n}-\vR_i}_{\R^d}^2\right)
  =\sum_{j=1}^d\lim_{n\rightarrow\infty}\E\left(\frac{1}{n}\sum_{i=1}^n\abs{\hat{R}^j_{i,n}-R^j_i}^2\right)
  =0.
\end{equation*}
Such non-parametric regression methods exists \citep[see
e.g.][Lemma 5]{kpotufe2014}.

\subsection{dHSIC preserves rank}\label{sec:order_preserve}
In order to prove that the ordering is preserved, we use a Lipschitz
property of $\mhsicb$ that holds whenever the kernel function
is Lipschitz \citep[see also][Lemma 16]{dataset}.
\begin{lemma}[Lipschitz-property of $\mhsicb$]
  \label{thm:dhsic_lipschitz}
  Let $\vx_1,\dots,\vx_n,\vy_1\dots,\vy_n\in\R^d$ and for all
  $j\in\{1,\dots,d\}$ let $k^j$ be a positive semi-definite
  kernel on $\R$ satisfying that there exists a constant
  $L^j>0$ such that for all $x,x',y\in\kernelspace^j$ it holds that
  $\abs{k^j(x,y)-k^j(x',y)}\leq L^j\abs{x-x'}$. Then
  it holds that
  \begin{equation*}
    \abs[\big]{\mhsicb_n(\vx_1,\dots,\vx_n)-\mhsicb_n(\vy_1,\dots,\vy_n)}^2
    \leq\frac{K}{n}\sum_{i=1}^n\norm{\vx_i-\vy_i}_{\R^d}^2.
  \end{equation*}
\end{lemma}

\begin{proof}
  For all $i\in\{1,\dots,n\}$, $j\in\{2,\dots,d\}$ define
  $\vz_{i,j}\coloneqq(y^1_i,\dots,y^{(j-1)}_i,x^j_i,\dots,x^d_i)$,
  moreover for all $i\in\{1,\dots,n\}$ define
  $\vz_{i,1}\coloneqq\vx_i$ and 
  $\vz_{i,d+1}\coloneqq\vy_i$. Fix a coordinate
  $j\in\{1,\dots,d\}$, we want to show that there exists a constant
  $K_j>0$ such that
  \begin{equation*}
    \abs[\big]{\mhsicb_n(\vz_{i,j},\dots,\vz_{n,j})-\mhsicb_n(\vz_{i,j+1},\dots,\vz_{n,j+1})}^2
    \leq\frac{K_j}{n}\sum_{i=1}^n\abs{x^j_i-y^j_i}^2.
  \end{equation*}
  To this end, use $\mhsicb_n=\Vstat{n}{h}$ where $h$
  is given in \eqref{eq:hcore}, to get that,
  \begin{align*}
    &\abs[\big]{\mhsicb_n(\vz_{i,j},\dots,\vz_{n,j})-\mhsicb_n(\vz_{i,j+1},\dots,\vz_{n,j+1})}^2\\
    &\quad=\Bigg(\frac{1}{n^{2d}}\frac{1}{(2d)!}\sum_{\mapset{2d}{n}}\sum_{S_{2d}}
      \bigg[\prod_{l\neq j}k^l(z^l_{i_{\pi(1)},j},z^l_{i_{\pi(2)},j})\left(k^j(x^j_{i_{\pi(1)}},x^j_{i_{\pi(2)}})-k^j(y^j_{i_{\pi(1)}},y^j_{i_{\pi(2)}})\right)\\
    &\qquad\qquad +\prod_{l\neq
      j}k^l(z^l_{i_{\pi(2j-1)},j},z^l_{i_{\pi(2j)},j})\left(k^j(x^j_{i_{\pi(2j-1)}},x^j_{i_{\pi(2j)}})-k^j(y^j_{i_{\pi(2j-1)}},y^j_{i_{\pi(2j)}})\right)\\
    &\qquad\qquad -2\prod_{l\neq j}k^l(z^l_{i_{\pi(1)},j},z^l_{i_{\pi(j+1)},j})\left(k^j(x^j_{i_{\pi(1)}},x^j_{i_{\pi(j+1)}})-k^j(y^j_{i_{\pi(1)}},y^j_{i_{\pi(j+1)}})\right)\bigg]\Bigg)^2.
  \end{align*}
  Hence, using $(a+b+c)^2\leq 4a^2+4b^2+4c^2$ we get that
  \begin{align*}
    &\abs[\big]{\mhsicb_n(\vz_{i,j},\dots,\vz_{n,j})-\mhsicb_n(\vz_{i,j+1},\dots,\vz_{n,j+1})}^2\\
    &\quad\leq\left(\frac{1}{n^{2d}}\frac{1}{(2d)!}\right)^2\Bigg(
      4\sum_{\mapset{2d}{n}}\sum_{S_{2d}}\bigg[\prod_{l\neq j}k^l(z^l_{i_{\pi(1)},j},z^l_{i_{\pi(2)},j})\left(k^j(x^j_{i_{\pi(1)}},x^j_{i_{\pi(2)}})-k^j(y^j_{i_{\pi(1)}},y^j_{i_{\pi(2)}})\right)\bigg]^2\\
    &\qquad\qquad +4\sum_{\mapset{2d}{n}}\sum_{S_{2d}}\bigg[\prod_{l\neq
      j}k^l(z^l_{i_{\pi(2j-1)},j},z^l_{i_{\pi(2j)},j})\left(k^j(x^j_{i_{\pi(2j-1)}},x^j_{i_{\pi(2j)}})-k^j(y^j_{i_{\pi(2j-1)}},y^j_{i_{\pi(2j)}})\right)\bigg]^2\\
    &\qquad\qquad +8\sum_{\mapset{2d}{n}}\sum_{S_{2d}}\bigg[\prod_{l\neq j}k^l(z^l_{i_{\pi(1)},j},z^l_{i_{\pi(j+1)},j})\left(k^j(x^j_{i_{\pi(1)}},x^j_{i_{\pi(j+1)}})-k^j(y^j_{i_{\pi(1)}},y^j_{i_{\pi(j+1)}})\right)\bigg]^2\Bigg).
  \end{align*}
  We analyze each of the three parts of the sum separately. The
  Cauchy-Schwartz inequality together with the Lipschitz property of
  the kernel $k^j$, the boundedness of the kernels and the inequality
  $(a+b)^2\leq 2a^2+2b^2$ lead to
  \begin{align*}
    &\sum_{\mapset{2d}{n}}\sum_{S_{2d}}
      \bigg[\prod_{l\neq
    j}k^l(z^l_{i_{\pi(1)},j},z^l_{i_{\pi(2)},j})\left(k^j(x^j_{i_{\pi(1)}},x^j_{i_{\pi(2)}})-k^j(y^j_{i_{\pi(1)}},y^j_{i_{\pi(2)}})\right)\bigg]^2\\
    &\quad\leq \sum_{\mapset{2d}{n}}\sum_{S_{2d}}
      \bigg[\prod_{l\neq
      j}k^l(z^l_{i_{\pi(1)},j},z^l_{i_{\pi(2)},j})\bigg]^2
      \sum_{\mapset{2d}{n}}\sum_{S_{2d}}
      \bigg[\left(k^j(x^j_{i_{\pi(1)}},x^j_{i_{\pi(2)}})-k^j(y^j_{i_{\pi(1)}},y^j_{i_{\pi(2)}})\right)\bigg]^2\\
    &\quad\leq n^{2d}(2d)!C^{2(d-1)}L_j^2\sum_{\mapset{2d}{n}}\sum_{S_{2d}}
      \bigg[\abs{x^j_{i_{\pi(1)}}-y^j_{i_{\pi(1)}}}+\abs{x^j_{i_{\pi(2)}}-y^j_{i_{\pi(2)}}}\bigg]^2\\
    &\quad\leq 2n^{2d}(2d)!C^{2(d-1)}L_j^2\bigg[\sum_{\mapset{2d}{n}}\sum_{S_{2d}}
      \abs{x^j_{i_{\pi(1)}}-y^j_{i_{\pi(1)}}}^2+\sum_{\mapset{2d}{n}}\sum_{S_{2d}}\abs{x^j_{i_{\pi(2)}}-y^j_{i_{\pi(2)}}}^2\bigg]\\
    &\quad=2n^{2d}(2d)!C^{2(d-1)}L_j^2\bigg[n^{2d-1}(2d)!\sum_{i=1}^n
      \abs{x^j_{i}-y^j_{i}}^2+n^{2d-1}(2d)!\sum_{i=1}^n\abs{x^j_{i}-y^j_{i}}^2\bigg]\\
    &\quad=4n^{4d-1}(2d)!^2C^{2(d-1)}L_j^2\sum_{i=1}^n
      \abs{x^j_{i}-y^j_{i}}^2.
  \end{align*}
  The same argument can be applied to the two remaining terms, which
  finally results in
  \begin{align}
    &\abs[\big]{\mhsicb_n(\vz_{i,j},\dots,\vz_{n,j})-\mhsicb_n(\vz_{i,j+1},\dots,\vz_{n,j+1})}^2\nonumber\\
    &\quad\leq \frac{16n^{4d-1}(2d)!^2C^{2(d-1)}L_j^2}{n^{4d}(2d)!^2}\sum_{i=1}^n
      \abs{x^j_{i}-y^j_{i}}^2 \nonumber\\
    &\quad
    \leq \frac{16C^{2(d-1)}L_j^2}{n}\sum_{i=1}^n
      \abs{x^j_{i}-y^j_{i}}^2.\label{eq:jthinequality}
  \end{align}
  Using \eqref{eq:jthinequality} and applying the
  inequality $(\sum_{j=1}^da_j)^2\leq 2^{d-1}\sum_{j=1}^da_j^2$
yields
  \begin{align*}
    &\abs[\big]{\mhsicb_n(\vx_1,\dots,\vx_n)-\mhsicb_n(\vy_1,\dots,\vy_n)}^2\\
    &\quad=\abs[\bigg]{\sum_{j=1}^d\mhsicb_n(\vz_{i,j},\dots,\vz_{n,j})-\mhsicb_n(\vz_{i,j+1},\dots,\vz_{n,j+1})}^2\\
    &\quad\leq\sum_{j=1}^d\abs[\big]{\mhsicb_n(\vz_{i,j},\dots,\vz_{n,j})-\mhsicb_n(\vz_{i,j+1},\dots,\vz_{n,j+1})}^2\\
    &\quad\leq\sum_{j=1}^d\frac{16C^{2(d-1)}L_j^2}{n}\sum_{i=1}^n
      \abs{x^j_i-y^j_i}^2\\
    &\quad\leq\frac{K}{n}\sum_{i=1}^n\norm{\vx_i-\vy_i}_{\R^d}^2,
  \end{align*}
  where $K\coloneqq 16C^{2(d-1)}(\max_{j\in\{1,\dots,d\}}L_j)^2$. This
  completes the proof of Lemma~\ref{thm:dhsic_lipschitz}.
\end{proof}

Using the Lipschitz property 
given in Lemma~\ref{thm:dhsic_lipschitz} 
we can now prove that the difference
between $\mhsicb$ based on the estimated residuals and $\mhsicb$ based on the
exact residuals converges to $0$ in probability, see Theorem~\ref{thm:noniid_resid} below.
This shows that (asymptotically) both quantities yield the same order of causal models. 

\begin{theorem}[consistency of $\mhsicb$ for estimated residuals]
  \label{thm:noniid_resid}
  Let $(\vX_i,\vR_i)_{i\in\{1,\dots,n\}}\in\R^d\times\R^d$ be iid
  random variables satisfying \eqref{eq:regression_setting}. Moreover, assume $\prodk$ satisfies
  the Lipschitz condition of Lemma~\ref{thm:dhsic_lipschitz}. Then,
  given a suitable regression method for this problem it holds that
  \begin{equation*}
    \abs[\big]{\mhsicb_n(\hat{\vR}_{1,n},\dots,\hat{\vR}_{n,n})-\mhsicb_n(\vR_1,\dots,\vR_n)}\overset{\P}{\longrightarrow}0, 
    \quad  \text{ as } n\rightarrow\infty.
  \end{equation*}
\end{theorem}

\begin{proof}
  By Lemma~\ref{thm:dhsic_lipschitz} it holds that
  \begin{equation*}
    \abs[\big]{\mhsicb_n(\hat{\vR}_{1,n},\dots,\hat{\vR}_{n,n})-\mhsicb_n(\vR_1,\dots,\vR_n)}^2
    \leq\frac{K}{n}\sum_{i=1}^n \norm{\hat{\vR}^j_{i,n}-\vR_i^j}_{\R^d}^2,
  \end{equation*}
  from which the result immediately follows, since $L^2$-convergence
  implies convergence in probability. This completes the proof of
  Theorem~\ref{thm:noniid_resid}.
\end{proof}

\subsection{Sample splitting}\label{sec:sample_split}

In order to justify the DAG verification method as a goodness of fit
test we need to ensure that it achieves asymptotic
level.
This can be achieved by splitting the
sample and using one part to perform the regression and the other part to
perform the independence test. While this guarantees that the
residuals are independent, we still need to show that the asymptotic
distribution of the estimated residuals converges to the same
distribution as the exact residuals. In order to do so, we use our
results on resampling V-statistics derived in
Section~\ref{subsec:resampling_uvstat}.

\begin{proposition}[asymptotic distribution of test statistic for estimated residual]
  \label{thm:asymptotic_dist_approx_resid}
  Let $(\vX_i,\vR_i)_{i\in\{1,\dots,n\}}\in\R^d\times\R^d$ be iid
  random variables satisfying \eqref{eq:regression_setting}. Let
  $\mathcal{Z}^1_n\coloneqq(\vX_1,\dots,\vX_{n})$ and
  $\mathcal{Z}^2_n\coloneqq(\vX_{n+1},\dots,\vX_{2n})$ be two
  sub-samples, assume $\hat{f}_{j,\mathcal{Z}^2_n}$ are the estimates
  of the regression functions due to a suitable regression method
  based on the sample $\mathcal{Z}^2_n$ and define for all
  $i\in\{1,\dots,n\}$ and for all $j\in\{1,\dots,d\}$ the residuals
  \begin{equation*}
    \hat{R}^j_{i,n}\coloneqq X^j_i-\hat{f}_{j,\mathcal{Z}^2_n}(\parents^j_i).
  \end{equation*}
  Moreover, let $(Z_i)_{i\in\N}$ be a sequence of independent standard
  normal random variables on $\R$, let
  $T_{h_2}\in L(\Lpspace{2}{\prodlawR}{\abs{\cdot}_{\R}})$ with the
  property that for every
  $f\in\Lpspace{2}{\prodlawR}{\abs{\cdot}_{\R}}$ and for every
  $\vx\in\prodkernelspace$ it holds that
  \begin{equation}
    \label{eq:lambda_diff}
    \left(T_{h_2}(f)\right)(\vx)=\int_{\prodkernelspace}h_2(\vx,\vy)f(\vy)\measure{\prodlawR}{\vy}
  \end{equation}
  and let $(\lambda_i)_{i\in\N}$ be the eigenvalues of $T_{h_2}$. Then
  it holds that
  \begin{equation*}
    \begin{rcases}
      n\cdot\mhsicb_n(\hat{\vR}_{1,n},\dots,\hat{\vR}_{n,n})\\
      n\cdot\mhsicb_n\left(\vR_{1},\dots,\vR_{n}\right)
    \end{rcases}
    \overset{d}{\longrightarrow}\binom{2d}{2}\sum_{i=1}^{\infty}\lambda_iZ_i^2
    \quad \text{ as } n\rightarrow\infty.
  \end{equation*}
\end{proposition}

\begin{proof}
  The convergence of
  $n\cdot\mhsicb_n\left(\vR_{1},\dots,\vR_{n}\right)$ follows
  immediately from Theorem~\ref{thm:asymptoticdist_HO}. Due to the
  sample splitting it follows that
  $\hat{\vR}_{1,n},\dots,\hat{\vR}_{n,n}$ are iid random variables and
  given that the regression method is \textit{suitable} and since
  $L^2$-convergence implies convergence in distribution it in
  particular holds that $\hat{\vR}_{i,n}\overset{d}{\rightarrow}\vR_i$
  as $n\rightarrow\infty$. Hence,
  $(\hat{\vR}_{i,n})_{i\in\{1,\dots,n\}}$ satisfies the same
  properties as $(\vX^*_{n,i})_{i\in\{1,\dots,n\}}$ in
  Setting~\ref{setting:resampling}. The convergence of
  $n\cdot\mhsicb_n(\hat{\vR}_{1,n},\dots,\hat{\vR}_{n,n})$ thus
  follows directly from Theorem~\ref{thm:resamplingvstat}.
\end{proof}

Using a similar argument as in Theorem~\ref{thm:level_bootstrap_test}
we can hence show that the goodness of fit test based on the DAG
verification method with sample splitting has asymptotic level.
